\documentclass{amsart}
\usepackage[utf8]{inputenc}
\usepackage[T1]{fontenc}
\usepackage{amsmath,amssymb,amsthm}
\usepackage[margin=1in]{geometry}
\usepackage{graphicx}
\usepackage{enumitem}
\usepackage{array,booktabs,longtable}
\usepackage[colorlinks=true,linkcolor=blue,citecolor=blue,urlcolor=blue]{hyperref}
\setlist[enumerate,1]{label=\textup{(\alph*)},ref=(\alph*),leftmargin=2.2em}

\graphicspath{{figures/}}
\newcommand{\TCDevice}{4 CPU processes with one thread each}
\newcommand{\TCMode}{full}

\newcommand{\TCStarts}{64}
\newcommand{\TCSteps}{800}
\newcommand{\TCGridSize}{14}
\newcommand{\TCDigits}{40}
\newcommand{\TCDims}{2, 3, 4}

\newcommand{\TCWorkers}{4}
\newcommand{\TCSeconds}{475}
\newcommand{\TCVzero}{1\times10^{-16}}
\newcommand{\TCCertDev}{5\times10^{-6}}
\newcommand{\TCSharpDev}{1\times10^{-7}}
\newcommand{\TCBoundExcess}{-1\times10^{-13}}
\newcommand{\TCNDev}{4\times10^{-3}}
\newcommand{\TCSectorCentral}{1\times10^{-16}}
\newcommand{\TCSectorExcess}{-7\times10^{-14}}
\newcommand{\TCSectorReached}{63}
\newcommand{\TCSectorTotal}{72}
\newcommand{\TCSectorDet}{2\times10^{-11}}
\newcommand{\TCSciRatioMin}{0.994}
\newcommand{\TCSciRatioMax}{1.097}
\newcommand{\TCSciStarts}{40}

\newcommand{\TCSGMExcessV}{0.847}
\newcommand{\TCSGMExcessKappa}{0.667}
\newcommand{\TCSGMExcessDelta}{0.8}
\newcommand{\TCSGMExcessN}{4}
\newcommand{\TCSharpDevNew}{3\times10^{-4}}
\newcommand{\TCSharpDevApproached}{2\times10^{-6}}
\newcommand{\TCNDevNew}{4\times10^{-7}}

\newcommand{\TCNDevProj}{3\times10^{-10}}
\newcommand{\TCNDevMon}{3\times10^{-8}}
\newcommand{\TCVPowersNineThree}{0.3064}
\newcommand{\TCVSgmNineThree}{1.2372}

\newcommand{\TCRefSgmNine}{1.2942}

\newcommand{\TCVPowersNine}{0.2197}

\newcommand{\TCRefExpqpqHalf}{0.0802}

\newcommand{\TCWmonLead}{1.020}
\newcommand{\TCExpOddHalf}{0.0694}
\newcommand{\TCWmonFirst}{0.3}

\newcommand{\TCWmonReach}{0.7}
\newcommand{\TCSGMRatioHalf}{0.924}

\newcommand{\TCALTRatioHalf}{0.796}

\newcommand{\TCRows}{%
1 & $|\varphi(X)|\le\varphi(|X|)$ & \cite{gardner1979} & 0.0050 & 0.1547 & 1.2942 & 0.1547 & $\delta^{2}$ & 2.03 & $4\times10^{-10}$ \\
2 & $\varphi(|X+Y|)\le\varphi(|X|)+\varphi(|Y|)$ & \cite{gardner1979,stolyarov2002} & 0.0050 & 0.1547 & 1.2942 & 0.1547 & $\delta^{2}$ & 2.03 & $1\times10^{-9}$ \\
3 & $\varphi(A+A^*)\le2\varphi(|A^*|)$ & \cite{bikchentaev2019abed} & 0.0050 & 0.1547 & 1.2942 & 0.1547 & $\delta^{2}$ & 2.03 & $5\times10^{-10}$ \\
4 & $|\varphi(A^2)|\le\varphi(A^*A)$ & \cite{petz1988,bikchentaev2026moslehian} & 0.0050 & 0.1547 & 1.2942 & 0.1547 & $\delta^{2}$ & 2.03 & $1\times10^{-8}$ \\
5 & $\varphi(|P-R|)\le\varphi(|P-Q|)+\varphi(|Q-R|)$ & \cite{bikchentaev2019metric} & 0.0050 & 0.1547 & 1.2942 & 0.1547 & $\delta^{2}$ & 2.03 & $1\times10^{-8}$ \\
6 & $\varphi(A\natural B)\le\sqrt{\varphi(A)\varphi(B)}$ & \cite{bikchentaev2026sgm} & n.r. & 0.1430 & 1.1979 & 0.1547$^{b}$ & $\delta^{2}$ & -- & $n\ge3$ \\
7 & $\varphi((B^{1/2}AB^{1/2})^2)\le\varphi(BA^2B)$ & \cite{bikchentaev2011pb} & 0.0059 & 0.2653 & 3.4322 & 0.3333$^{b}$ & $\delta^{2}$ & 2.31 & $9\times10^{-2}$ \\
8 & $\varphi(|AB^*|)\le\frac12\varphi(|A|^2)+\frac12\varphi(|B|^2)$ & \cite{bikchentaev2005young} & 0.2222 & 2.0000 & 18.00 & 2.0000 & $\delta^{1}$ & 1.17 & $9\times10^{-12}$ \\
9 & $\varphi(PQ+QP)\le2\varphi(QPQ)$ & \cite{bikchentaev2019abed} & 0.0477 & 0.2056 & 0.3469 & 0.2056$^{a}$ & $\delta^{1}$ & 0.94 & $3\times10^{-10}$ \\
10 & $0\le A\le B\Rightarrow\varphi(A^2)\le\varphi(B^2)$ & \cite{bikchentaev2007mono,sano2006,hoa2010} & $6.3\times10^{-6}$ & 0.0052 & 0.1825 & 0.0052 & $\delta^{4}$ & 4.04 & $3\times10^{-8}$ \\
11 & $(\det A)^{1/n}\le\frac1n\varphi(A)$, $\varphi(1)=n$ & \cite{bikchentaev2020det} & 0.0050 & 0.1340 & 0.5641 & 0.1340 & $\delta^{2}$ & 2.01 & yes \\
12 & $\varphi((A+B)^{1/2})\le\varphi(A^{1/2})+\varphi(B^{1/2})$ & \cite{tikhonov2005} & 0.0050 & 0.1338 & 0.9678 & -- & -- & 2.00 & $5\times10^{-4}$ \\
13 & $\varphi(A+B-|A-B|)\le2\varphi(A^{1/4}B^{1/2}A^{1/4})$ & \cite{hoa2013} & 0.0048 & 0.0911 & 0.2197 & -- & -- & 1.93 & $4\times10^{-1}$ \\
14 & $|\varphi(PQ)|\le\varphi(P)^{1/3}\varphi(Q)^{2/3}$ & \cite{bikchentaev2011mn,bikchentaev2019abed} & $2.8\times10^{-4}$ & 0.0086 & 0.0711 & -- & -- & 2.03 & $4\times10^{-3}$ \\
15 & $\varphi(e^{P+Q})\le\varphi(e^{P/2}e^Qe^{P/2})$ & \cite{bikchentaev2011mn} & $6.4\times10^{-4}$ & 0.0142 & 0.0420 & -- & -- & 1.98 & $2\times10^{-1}$ \\
16 & $\varphi(e^A)e^{\varphi(e^{A/2}Be^{A/2})/\varphi(e^A)}\le\varphi(e^{A+B})$ & \cite{bikchentaev2011pb} & $9.7\times10^{-4}$ & 0.0561 & 0.8890 & -- & -- & 2.35 & $2\times10^{-3}$ \\
17 & $\varphi(PQP)\le\varphi(Q)$ & \cite{bikchentaev2010smj} & 0.0101 & 0.3333 & 4.2632 & 0.3333 & $\delta^{2}$ & 2.04 & $7\times10^{-9}$ \\
18 & $\varphi(|P-Q|)\le\varphi(P+Q)$ & \cite{bikchentaev2010smj} & 0.0050 & 0.1547 & 1.2942 & 0.1547 & $\delta^{2}$ & 2.03 & $1\times10^{-8}$ \\
19 & $\varphi(P+Q)\le\varphi(1+P\wedge Q)$ & \cite{bikchentaev2010smj} & 0.1000 & 0.5000 & 0.9000 & 0.5000$^{c}$ & $\delta^{1}$ & 1.00 & odd $n$ \\
20 & $\varphi(e^{QPQ})\le\varphi(e^P)$ & \cite{bikchentaev2011mn} & 0.0029 & 0.0802 & 0.3115 & 0.0802 & $\delta^{2}$ & 2.02 & odd $n$ \\
21 & $\varphi(\mathrm{Re}\,PQ)\ge0$ & \cite{bikchentaev2011mn} & 0.0013 & 0.0359 & 0.1964 & 0.0359 & $\delta^{2}$ & 2.02 & $8\times10^{-9}$ \\
22 & $\varphi(\mathrm{Re}\,PQ)\le\varphi(|PQ|)$ & \cite{bikchentaev2011mn} & 0.0050 & 0.1547 & 1.2942 & 0.1547 & $\delta^{2}$ & 2.03 & $5\times10^{-9}$ \\
23 & $\varphi(P\vee Q)\le\varphi(P)+\varphi(Q)$ & \cite{petz1988} & 0.1111 & 1.0000 & 9.0000 & 1.0000$^{c}$ & $\delta^{1}$ & 1.17 & $4\times10^{-7}$ \\
24 & $\varphi(A^2)^{1/2}\le\varphi(A)$, $\varphi(1)=n$ & \cite{bikchentaev2020det} & 0.0541 & 0.4142 & 2.1623 & 0.4142 & $\delta^{1}$ & 1.13 & yes \\
25 & $0\le A\le B\Rightarrow\varphi(A^4)\le\varphi(AB^2A)$ & \cite{hoa2010} & n.r. & $1.5\times10^{-6}$ & 0.0017 & $1.7\times10^{-6}{}^{b}$ & $\delta^{8}$ & -- & $3\times10^{-6}$ \\
}
\newcommand{\TCClassRows}{%
8 & $\kappa-1=2\delta/(1-\delta)$ & $2\delta$ & 2.0000 & 2.0000 & sharp \\
9 & $s/(1+s)$, $s=\sin\frac\psi2$ & $\delta/2$ & 0.2056 & 0.2056 & sharp for $n=2$ \\
19 & $\delta$ & $\delta$ & 0.5000 & 0.5000 & sharp, approached \\
23 & $\delta/(1-\delta)$ & $\delta$ & 1.0000 & 1.0000 & sharp, approached \\
24 & $(1-\delta)^{-1/2}-1$ & $\delta/2$ & 0.4142 & 0.4142 & exact, every $n$ ($n=2$ shown) \\
\midrule
1--5, 18, 22 & $G-1$ & $\delta^2/2$ & 0.1547 & 0.1547 & sharp \\
17 & $G^2-1$ & $\delta^2$ & 0.3333 & 0.3333 & sharp \\
11 & $1-\sqrt{1-\delta^2}$ & $\delta^2/2$ & 0.1340 & 0.1340 & exact, every $n$ ($n=2$ shown) \\
20 & max.\ of one variable & $\frac{(e-1)^2}{e(e+1)}\delta^2$ & 0.0802 & 0.0802 & sharp \\
21 & $\frac{1}{2}\tan^2(\psi/2)$ & $\delta^2/8$ & 0.0359 & 0.0359 & sharp \\
6 & $\ge G-1$ & $\ge\delta^2/2$ & 0.1547 & 0.1430 & lower bound \\
7 & $\ge G^2-1$ & $\ge\delta^2$ & 0.3333 & 0.2653 & lower bound \\
12--16 & -- & $c\,\delta^2$, $c\in[\frac1{36},\frac12]$ & -- & 0.0086--0.1338 & search; lower bounds for 12, 14, 16 \\
\midrule
10 & $(G-1)^2/(4G)$ & $\delta^4/16$ & 0.0052 & 0.0052 & sharp \\
25 & $\ge g_{25}(\delta)$ & $\ge\delta^8/4096$ & $1.7\times10^{-6}$ & $1.5\times10^{-6}$ & lower bound \\
}

\newcommand{\TCCStarts}{64}
\newcommand{\TCCSteps}{800}
\newcommand{\TCCDigits}{40}
\newcommand{\TCCSeconds}{760}
\newcommand{\TCCRowsA}{%
1 & 0.1547 & 0.1547 & 0.1547 & 1.2942 & 1.2942 & 1.2942 & 0.1547 & 0.1547 & 0.1547 & 1.2942 & 1.2942 & 1.2942 \\
2 & 0.1547 & 0.1547 & 0.1547 & 1.2942 & 1.2941 & 1.2942 & 0.1547 & 0.1547 & 0.1547 & 1.2942 & 1.2941 & 1.2942 \\
3 & 0.1547 & 0.1547 & 0.1547 & 1.2942 & 1.2942 & 1.2942 & 0.1547 & 0.1547 & 0.1547 & 1.2942 & 1.2942 & 1.2942 \\
4 & 0.1547 & 0.1547 & 0.0774 & 1.2942 & 1.2942 & 0.6471 & 0.1547 & 0.1328 & 0.0774 & 1.2942 & 1.2612 & 0.7393 \\
6 & 0.1427 & 0 & 0.1501 & 1.2139 & 1.1707 & 1.2924 & 0.1359 & 0 & 0.1491 & 1.2125 & 1.1405 & 1.2922 \\
7 & 0.2656 & 0.2730 & 0 & 3.4387 & 3.5241 & 0 & 0.2733 & 0.2703 & 0 & 3.4965 & 3.4358 & 0 \\
8 & 2.0000 & 2.0000 & 2.0000 & 18.00 & 18.00 & 18.00 & 2.0000 & 2.0000 & 2.0000 & 18.00 & 18.00 & 18.00 \\
10 & 0.0052 & 0.0052 & 0.0052 & 0.1825 & 0.1825 & 0.1825 & 0.0052 & 0.0052 & 0.0052 & 0.1825 & 0.1825 & 0.1825 \\
11 & 0.1340 & 0.1340 & -- & 0.5641 & 0.5641 & -- & 0.1347 & 0.1347 & -- & 0.6188 & 0.6188 & -- \\
12 & 0.1338 & 0.1338 & 0.1338 & 0.9678 & 0.9679 & 0.9680 & 0.1338 & 0.1338 & 0.1338 & 0.9678 & 0.9675 & 0.9680 \\
13 & 0.0911 & 0.0329 & 0.0911 & 0.2197 & 0.2171 & 0.2197 & 0.0911 & 0.0911 & 0.0911 & 0.3048 & 0.2057 & 0.2197 \\
16 & 0.0560 & 0.0561 & 0.0560 & 0.8885 & 0.8885 & 0.8873 & 0.0561 & 0.0561 & 0.0560 & 0.8894 & 0.8887 & 0.8864 \\
24 & 0.4142 & 0.4142 & 0.4142 & 2.1623 & 2.1623 & 2.1623 & 0.5275 & 0.5275 & 0.5275 & 2.6056 & 2.6056 & 2.6056 \\
25 & $1.6\cdot10^{-6}$ & $1.4\cdot10^{-6}$ & 0 & 0.0017 & 0.0017 & 0 & $5.3\cdot10^{-7}$ & $2.2\cdot10^{-8}$ & 0 & 0.0017 & 0.0017 & 0 \\
}
\newcommand{\TCCGardnerDev}{1\times10^{-9}}
\newcommand{\TCCGardnerDevEven}{9\times10^{-10}}
\newcommand{\TCCTriangleDevOdd}{-2\times10^{-4}}
\newcommand{\TCCRealpartDevOdd}{-2\times10^{-9}}
\newcommand{\TCCGardnerTwoDDev}{5\times10^{-10}}
\newcommand{\TCCPowerMeanDev}{2\times10^{-11}}
\newcommand{\TCCProjGap}{0.73}
\newcommand{\TCCTwoLevelShort}{2 ($\delta'\in\{0.4, 0.6\}$); 4 ($\delta'\ge0.2$); 6 ($\delta'\ge0.8$); 7 ($\delta'\ge0.6$); 11 ($\delta'\in\{0.4, 0.8, 0.9, 0.99\}$); 13 ($\delta'\ge0.9$); 25 ($\delta'\ge0.99$)}
\newcommand{\TCCTwoLevelDim}{3}
\newcommand{\TCCSubaddOnset}{0.999 for $\delta=0.3$, 0.99 for $\delta=0.6$, 0.99 for $\delta=0.9$ ($M_2$ and $M_3$ alike)}

\newcommand{\TCCSubaddNone}{0.95}

\newcommand{\TCCMonOnsetTwo}{0.99 for $\delta=0.3$, 0.95 for $\delta=0.6$, 0.8 for $\delta=0.9$}

\newcommand{\TCCThresholds}{6 ($n=2$): 0.9; 6 ($n=3$): 0.8; 10 ($n=2$): 0.8; 10 ($n=3$): 0.8; 12 ($n=2$): 0.99; 12 ($n=3$): 0.99; 13 ($n=2$): 0.9; 13 ($n=3$): 0.9; 24 ($n=2$): 0.8; 24 ($n=3$): 0.9; 25 ($n=2$): --; 25 ($n=3$): 0.99}
\newcommand{\TCCThresholdMin}{0.8}

\newcommand{\TMCSamples}{200000}
\newcommand{\TMCDims}{2, 3, 4}
\newcommand{\TMCWorstRel}{1.003}
\newcommand{\TMCRows}{%
1 & Gardner & 4.4 & 19.0 & $<0.1$ & 2.9 & $<0.1$ & 0.6 \\
2 & triangle & 1.1 & 6.3 & $<0.1$ & 6.2 & $<0.1$ & 4.5 \\
3 & $\varphi(A{+}A^*)$ & 0.6 & 3.1 & $<0.1$ & 0.4 & 0 & $<0.1$ \\
4 & $|\varphi(A^2)|$ & 6.6 & 6.8 & 0.2 & 0.4 & $<0.1$ & $<0.1$ \\
5 & metric & 0 & -- & $<0.1$ & -- & $<0.1$ & -- \\
6 & geom. mean & $<0.1$ & -- & 0 & -- & 0 & -- \\
7 & ALT & 19.5 & 0 & 14.3 & 0 & 11.0 & 0 \\
8 & Young (ext.) & 7.6 & 8.4 & 3.0 & 2.1 & 1.4 & 0.6 \\
9 & Bikchentaev-Abed & 50.2 & -- & 50.1 & -- & 49.8 & -- \\
10 & monotonicity & $<0.1$ & 0.7 & 0 & $<0.1$ & 0 & $<0.1$ \\
11 & AM-GM & 20.5 & -- & 13.2 & -- & 10.0 & -- \\
12 & subadditivity & $<0.1$ & 21.3 & 0 & 9.1 & 0 & 5.1 \\
13 & Powers-Stormer & $<0.1$ & 9.0 & 0 & 3.0 & 0 & 1.4 \\
14 & Hoelder & 1.3 & -- & $<0.1$ & -- & 0 & -- \\
15 & Golden-Thompson & 25.6 & -- & 16.8 & -- & 11.5 & -- \\
16 & Peierls-Bogoliubov & 2.4 & 1.8 & 0.8 & 0.4 & 0.4 & 0.1 \\
17 & $\varphi(PQP)$ & 12.0 & -- & 4.0 & -- & 1.9 & -- \\
18 & $\varphi(|P-Q|)$ & 13.0 & -- & 4.3 & -- & 2.1 & -- \\
19 & $P\wedge Q$ & 49.9 & -- & 37.8 & -- & 33.3 & -- \\
20 & $e^{QPQ}$ & 8.0 & -- & 2.2 & -- & 1.0 & -- \\
21 & Re$PQ\geq0$ & 5.6 & -- & 1.4 & -- & 0.6 & -- \\
22 & Re$PQ\leq|PQ|$ & 14.6 & -- & 7.7 & -- & 4.3 & -- \\
23 & $P\vee Q$ & 49.8 & -- & 37.7 & -- & 33.5 & -- \\
24 & power mean & 7.2 & 50.0 & 0.2 & 25.1 & $<0.1$ & 16.8 \\
25 & weighted mon. & 0 & 0 & 0 & 0 & 0 & 0 \\
}

\newcommand{\COSamples}{100000}
\newcommand{\COWorst}{-5\times10^{-8}}
\newcommand{\COIndefInside}{0}
\newcommand{\COJordanTwo}{0.8839}
\newcommand{\COJordanThree}{0.8963}
\newcommand{\COJordanThreeOpt}{0.8963}
\newcommand{\COJordanFour}{0.8963}
\newcommand{\COJordanDev}{3\times10^{-7}}
\newcommand{\COJordanSin}{0.9600}
\newcommand{\COJordanTanh}{0.8824}

\newtheorem{theorem}{Theorem}[section]
\newtheorem{proposition}[theorem]{Proposition}
\newtheorem{lemma}[theorem]{Lemma}
\newtheorem{corollary}[theorem]{Corollary}
\newtheorem{conjecture}[theorem]{Conjecture}
\newtheorem*{thmA}{Theorem A}
\newtheorem*{thmB}{Theorem B}
\newtheorem*{thmC}{Theorem C}
\theoremstyle{definition}
\newtheorem{definition}[theorem]{Definition}

\theoremstyle{remark}
\newtheorem{remark}[theorem]{Remark}

\newcommand{\cM}{\mathcal{M}}
\newcommand{\M}{\mathcal{M}}
\newcommand{\cZ}{\mathcal{Z}}
\newcommand{\cH}{\mathcal{H}}
\newcommand{\cT}{\mathcal{T}}
\newcommand{\C}{\mathbb{C}}
\newcommand{\R}{\mathbb{R}}
\newcommand{\cL}{\mathcal{L}}
\newcommand{\cR}{\mathcal{R}}
\newcommand{\cN}{\mathcal{N}}
\DeclareMathOperator{\OLC}{OLC}
\DeclareMathOperator{\OMC}{OMC}
\DeclareMathOperator{\LNC}{LNC}
\DeclareMathOperator{\NC}{NC}
\DeclareMathOperator{\Tr}{Tr}
\DeclareMathOperator{\diag}{diag}
\DeclareMathOperator{\Ad}{Ad}
\DeclareMathOperator{\re}{Re}
\renewcommand{\Re}{\mathrm{Re}\,}
\newcommand{\ceil}[1]{\lceil #1\rceil}
\newcommand{\flr}[1]{\lfloor #1\rfloor}

\newcommand{\one}{\mathbf 1}
\newcommand{\yp}{\mathsf y_p}
\newcommand{\tp}{\mathsf t_p}
\newcommand{\mpp}{\mathsf m_p}

\begin{document}

\title[Violation constants, contrast and noncentrality]{How characterisations of traces and central
elements fail:\\ sharp violation constants via Michelson contrast\\ and noncentrality}

\author{Sami Abdullah Abed}
\address{College of Administration and Economics, Diyala University, Baquba, Diyala Governorate, Iraq}
\email{samialbarkish@gmail.com}

\author{Irina Nikolaeva}
\address{Lean Technologies Laboratory LLC, 189a-202 Krasnokokshayskaia str., Kazan, 420000, Russia}

\author{Andrej Novikov}
\address{Novikov Laboratories LLC, 14-124 Absaliamova str., Kazan, 420124, Russia}
\email{a.hobukob@gmail.com}
\thanks{Authors are listed in alphabetical order. Corresponding author: Andrej Novikov,
\texttt{a.hobukob@gmail.com}.}

\begin{abstract}
A positive functional is a trace if and only if it satisfies any one of many inequalities, the first
being Gardner's $|\varphi(X)|\le\varphi(|X|)$, and positive central elements are characterised by
similar inequalities. We ask how strongly these characterisations fail. For
$\varphi=\Tr(\rho\,\cdot\,)$ on a matrix algebra the natural scale is the Michelson contrast
$\delta=\OMC(\rho)$ of the density. For Gardner's inequality the sharp constant depends on $\delta$
and on the contrast $\delta'$ of $|X|$ only through their product: it is
$(1-\delta^2\delta'^2)^{-1/2}$, the secant of the angle whose sine is $\delta\delta'$, and it is
attained in even dimension. On a finite von Neumann algebra $\cM$ with a faithful normal trace
$\tau$ the contrast is the wrong parameter, since a central density gives a trace whatever its
contrast. For $\varphi=\tau(\rho\,\cdot\,)$ with $\rho$ positive and invertible, the governing
parameter is the noncentrality, defined through the central condition number
$\kappa_\cZ(\rho)=\|\ceil\rho\flr\rho^{-1}\|$ of the central envelopes of $\rho$ (the extreme
elements of Halpern's central spectrum). A transfer principle turns sharp $2\times2$ examples and
upper bounds in terms of the condition number into exact constants in terms of $\kappa_\cZ$. For
38 trace characterisations the constant is exact in 23 cases (in 18, one of them $+\infty$, as an
explicit function of $\kappa_\cZ$ in every finite von Neumann algebra; in five only on matrix algebras,
two of these only for $n=2$ or even $n$), with proved bounds in 13 and known only numerically in two.
For example, it is $\cosh\LNC(\rho)$, $\LNC=\frac12\ln\kappa_\cZ$, for Gardner's inequality and the
subadditivity of the modulus, $(\cosh\LNC+1)^2/(4\cosh\LNC)$ for the monotonicity of
$A\mapsto\varphi(A^2)$, and $\kappa_\cZ^{1/p}$ for Young's inequality, while the Bikchentaev--Abed
inequality holds exactly for $-1\le K\le\coth^2(\LNC(\rho)/2)$. On matrix algebras the relative
violations vanish at the trace to the orders $\delta$, $\delta^2$ or $\delta^4$ where they are known
exactly and finite (the constant of the Peierls--Bogoliubov inequality in the form with $\varphi(Be^A)$ is $+\infty$ for every non-central $\rho$); for the weighted monotonicity we prove a lower bound of order $\delta^8$. Outside
factors and abelian algebras no function of the contrast describes these constants. We also find
the sharp constants of the centre characterisations of Novikov and Tikhonov, sharp bounds for three
unnormalised trace functionals, and sharp contrast thresholds below which subadditivity of the
square root and the power-mean inequality cannot fail. A numerical search, certified as a lower
bound for each supremum, supports conjectured values for the remaining inequalities.
\end{abstract}


\maketitle

\section{Introduction}\label{sec:intro}

\subsection{The question}
\enlargethispage{6pt}
Let $\varphi$ be a positive linear functional on a C$^*$-algebra. Gardner~\cite{gardner1979} proved
that $\varphi$ is a trace, $\varphi(XY)=\varphi(YX)$, if and only if
\begin{equation}\label{eq:gardner}
  |\varphi(X)|\le\varphi(|X|)\qquad\text{for all }X .
\end{equation}
Since then many inequalities that hold for every trace have been shown to hold \emph{only} for
traces: the triangle inequality for $\varphi(|\cdot|)$, the inequalities of Petz--Zem\'anek,
monotonicity of $A\mapsto\varphi(A^2)$, Young's, Golden--Thompson's and Peierls--Bogoliubov's
inequalities, the anticommutator inequality of Bikchentaev and Abed, and a family of inequalities in
which only projections enter. Centrality of a positive element admits similar characterisations
(Novikov and Tikhonov~\cite{novikov2015}). Section~\ref{sec:related} reviews these results. Each of
them answers a yes/no question. For a functional that is not a trace the inequality fails for some
operators; but how strongly, and does the size of the failure measure how far the functional is
from being a trace? This paper answers these questions for 38 characterisations of traces and for
the characterisations of central elements of~\cite{novikov2015}.

The known proofs are qualitative. They reduce to $M_2(\C)$ with $\varphi=\Tr(S\,\cdot\,)$,
$S=\diag(1/2+s,1/2-s)$, and exhibit a one-parameter family of test operators for which the
inequality fails at first order in the parameter as soon as $s\ne0$; no constant results. For
Gardner's inequality on $M_n(\C)$ the sharp constant is contained in an inequality of
Bourin~\cite[Cor.~1.4, Rem.~1.9]{Bou06}, which is not formulated as a trace characterisation
(Remark~\ref{rem:bourin}), and the mixed Schwarz inequality of Kato~\cite{Kat52}, in the form
$|\varphi(x)|^2\le\varphi(|x|)\varphi(|x^*|)$ (see~\cite{Dra14}), gives the non-sharp bound
$|\varphi(x)|\le\kappa(\rho)^{1/2}\varphi(|x|)$ for $\varphi=\Tr(\rho\,\cdot\,)$, where $\kappa(\rho)$ is
the condition number. Apart from this we are not aware of quantitative versions of these
characterisations in the literature.

\subsection{Contrast and noncentrality}
On $M_n$ every positive functional is $\varphi_\rho=\Tr(\rho\,\cdot\,)$ with a density $\rho\ge0$, and
$\varphi_\rho$ is a trace if and only if $\rho$ is a multiple of~$1$. The natural scale-invariant
distance of $\rho$ from the scalars is its \emph{Michelson contrast}
\[
  \delta=\OMC(\rho)=\frac{\lambda_{\max}(\rho)-\lambda_{\min}(\rho)}{\lambda_{\max}(\rho)+\lambda_{\min}(\rho)}
  =\frac{\kappa(\rho)-1}{\kappa(\rho)+1}\in[0,1],
\]
which was studied for operators in~\cite{abed2024}, together with its logarithmic form
$\OLC(\rho)=\frac12\ln\kappa(\rho)=\operatorname{artanh}\delta$~\cite{NN26}. Gibbs states show that this
is a physically meaningful scale. For a Hamiltonian $H$ with extreme eigenvalues $E_{\min}\le E_{\max}$
and the inverse temperature $\beta$, the density $\rho=e^{-\beta H}/\Tr e^{-\beta H}$ has
\[
  \OMC(\rho)=\tanh\frac{\beta(E_{\max}-E_{\min})}2 ,\qquad \OLC(\rho)=\frac{\beta(E_{\max}-E_{\min})}2,
\]
so $\delta\to0$ at infinite temperature, where the Gibbs state becomes the normalised trace, and
$\delta\to1$ at zero temperature. Every bound in terms of $\delta$ is therefore a law in the
temperature: for Gardner's inequality we shall find the largest relative violation
$\cosh(\beta(E_{\max}-E_{\min})/2)-1$. In the language of modular theory, $\varphi_\rho$ (for
invertible $\rho$) is a trace if and only if its modular group $\sigma_t(X)=\rho^{it}X\rho^{-it}$ is
trivial; the derivation $X\mapsto[\log \rho,X]$ that generates it has the spectrum
$\{\log w_i-\log w_j\}$, where $w_i$ are the eigenvalues of $\rho$, whose largest point is
$\log\kappa(\rho)=2\OLC(\rho)$. The contrast thus measures how far the modular dynamics is from being
trivial.

Beyond factors the contrast is the wrong parameter. Let $\cM$ be a finite von Neumann algebra with a
faithful normal tracial state $\tau$ and $\varphi=\tau(\rho\,\cdot\,)$ with $\rho\in\cM$ positive and
invertible. If $\rho$ is central, then $\varphi$ is a trace and nothing is violated, however large
$\OMC(\rho)$ is; already on $M_2\oplus M_2$ the density $1\oplus 4$ has contrast $3/5$ and gives a
trace. The relevant parameter is the \emph{noncentrality} of $\rho$, the distance of $\rho$ from the
centre $\cZ$ rather than from the scalars. It is defined through the central envelopes of $\rho$:
$\ceil{\rho}$ is the least central element $\ge\rho$ and $\flr{\rho}$ the greatest central element
$\le\rho$ (they exist by Lemma~\ref{lem:env}; they are the largest and the smallest elements of
Halpern's central spectrum~\cite{Hal72}), and
\[
\kappa_\cZ(\rho)=\bigl\|\ceil{\rho}\,\flr{\rho}^{-1}\bigr\|,\qquad
\LNC(\rho)=\tfrac12\ln\kappa_\cZ(\rho),\qquad \delta_\cZ=\NC(\rho)=\tanh\LNC(\rho)=\frac{\kappa_\cZ(\rho)-1}{\kappa_\cZ(\rho)+1},
\]
in analogy with $\OMC=\tanh\OLC$. Thus $\kappa_\cZ(\rho)=1$ if and only if $\rho$ is central,
$\kappa_\cZ(\rho)\le\kappa(\rho)$, $\kappa_\cZ=\kappa$ in factors, and
$\kappa_\cZ(\rho)=\min_z\kappa(z\rho)$ over positive invertible central $z$ (Section~\ref{sec:nc}). On
a direct sum $\bigoplus_cM_{k_c}$ the noncentrality of $\rho=\bigoplus_c\rho_c$ is the largest
contrast \emph{inside} the blocks, $\max_c\OMC(\rho_c)$.

The constants are expressed through the \emph{Kantorovich constant}
\[
  G=G(\rho)=\cosh\OLC(\rho)=\frac{1+\kappa(\rho)}{2\sqrt{\kappa(\rho)}}=\frac1{\sqrt{1-\delta^2}}=\sec\psi,
  \qquad G_\cZ=\cosh\LNC(\rho)=\frac1{\sqrt{1-\delta_\cZ^2}}=\sec\psi_\cZ,
\]
where $\psi=\arcsin\delta$ is the angle of $\rho$ in the sense of Krein and Gustafson (the largest
angle between a vector $x$ and its image $\rho x$) and $\psi_\cZ=\arcsin\delta_\cZ$; $G^2$ is the
best constant in Kantorovich's inequality. Table~\ref{tab:notation} in Section~\ref{sec:setting}
collects the notation.

\subsection{Envelopes and violation constants}
In~\cite{abed2024} we plotted the violation of Gardner's inequality against $\OMC(\rho)$ for random
real $2\times2$ and $3\times3$ matrices~\cite[Sect.~5.1, Fig.~1]{abed2024} and asked whether such
violations are bounded by functions of the contrasts~\cite[Sect.~7]{abed2024}. To make the question
precise, write an inequality as $\cL\le\cR$, where $\cL$ and $\cR$ depend on $\varphi$ and on operator
variables. The \emph{violation constant} $c(\rho)$ is the best constant $c$ in the weakened
inequality $\cL\le c\,\cR$ (for instance $|\varphi(X)|\le c\,\varphi(|X|)$); it equals $1$ exactly when
$\varphi$ is a trace. On $M_n$ we also fix a normalisation $\cN>0$ that makes the quotient
$(\cL-\cR)/\cN$ invariant under the scaling of $\rho$ and of the variables (mostly $\cN=\cR$, and then
$(\cL-\cR)/\cN=\cL/\cR-1$), and call
\[
  V_n(\delta)=\sup\Bigl\{\frac{\cL-\cR}{\cN}:\ \rho\ge0\ \text{in}\ M_n,\ \OMC(\rho)=\delta,\ \text{all
  operator variables free}\Bigr\}
\]
the \emph{violation envelope}. By the characterisation theorems $V_n(\delta)>0$ for $\delta>0$, and
$V_n(0)=0$. A violation needs more than a non-tracial functional: if the operators are scalars,
nothing can go wrong. So we also prescribe the contrast $\delta'$ of the operator variables (of their
moduli, for operators that are not positive) and write $V_n(\delta,\delta')$.

\subsection{Gardner's inequality with two contrasts}
For Gardner's inequality both contrasts can be prescribed, and the answer is a function of their
product alone.

\begin{thmA}
Let $\rho>0$ and $X$ be in $M_n$, $n\ge2$, with $\OMC(\rho)=\delta$ and $\OMC(|X|)=\delta'$,
$\delta\delta'<1$. Then
\[
  |\varphi_\rho(X)|\le\frac{1}{\sqrt{1-\delta^2\delta'^2}}\,\varphi_\rho(|X|),
\]
and for even $n$ the constant is attained, so $V_n(\delta,\delta')=(1-\delta^2\delta'^2)^{-1/2}-1$. For
odd $n$ the envelope is given by an explicit one-dimensional maximum; it is smaller when $\delta'<1$
and tends to the even value as $n\to\infty$. The same upper bound holds for the triangle inequality
$\varphi(|X+Y|)\le\varphi(|X|)+\varphi(|Y|)$ and for $\varphi(A+A^*)\le2\varphi(|A^*|)$
(Theorem~\ref{thm:gardnerD}).
\end{thmA}

With the operator free ($\delta'=1$) the constant is $G=\sec\psi$, and in a finite von Neumann
algebra it is $G_\cZ=\cosh\LNC(\rho)$ (Theorem~\ref{thm:gardner}).

\begin{figure}[t]
\centering\includegraphics[width=\textwidth]{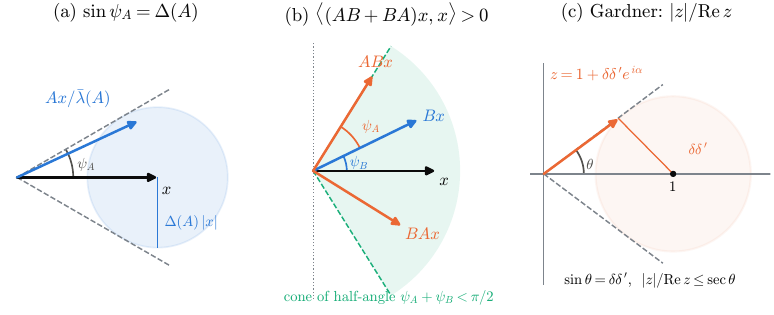}
\caption{(a) The angle of a positive operator: $Ax/\bar\lambda(A)$ lies in the ball of radius
$\OMC(A)|x|$ about $x$, so $\angle(x,Ax)\le\psi_A=\arcsin\OMC(A)$. (b) For a Jordan product the
angles add; if $\psi_A+\psi_B<\pi/2$, then $ABx$ and $BAx$ stay in a cone around $x$ that lies in the
half-space of $x$, and $AB+BA>0$ (Theorem~\ref{thm:jordan}). (c) On a two-dimensional block,
Gardner's quotient is $|z|/\Re z=\sec\arg z$ for a point $z$ on the circle of radius $\delta\delta'$
about $1$ (Remark~\ref{rem:geom}); the largest argument is $\arcsin(\delta\delta')$. In the figure the
contrast is written $\Delta(\cdot)$, so $\Delta(A)=\OMC(A)$.}
\label{fig:angle}
\end{figure}

The theorem has a simple geometric meaning (Figure~\ref{fig:angle}). Krein~\cite{krein1969} and
Gustafson~\cite{gustafson1968} introduced the angle $\psi(A)$ of an operator, the largest angle
between a vector $x$ and its image $Ax$; for $A>0$, $\sin\psi(A)=\OMC(A)$, since $A$ divided by the
midpoint $\bar\lambda(A)$ of its spectrum is $1+E$ with $\|E\|=\OMC(A)$. If $|X|=vv^*$ has rank one
($\delta'=1$), Gardner's quotient is $|\langle \rho v,Uv\rangle|/\langle \rho v,v\rangle\le\|\rho
v\|/\langle \rho v,v\rangle=\sec\angle(v,\rho v)\le\sec\psi(\rho)$, which is the Kantorovich constant
of $\rho$. For general $|X|$, the proof reduces to two-dimensional blocks on which $\rho$ and $|X|$
are, up to scale, $1+\delta S$ and $1+\delta'S'$ with symmetries $S,S'$; there $\Tr(|X|\rho)$ and the
trace norm $\||X|\rho\|_1$ are $2\Re z$ and $2|z|$ for $z=1+\delta\delta'e^{i\alpha}$, where
$e^{\pm i\alpha}$ are the eigenvalues of the rotation $S'S$. Gardner's quotient is $\sec\arg z$, and
the circle described by $z$ is seen from the origin under the angle $\theta$ with
$\sin\theta=\delta\delta'=\sin\psi(\rho)\sin\psi(|X|)$. The product of the contrasts is the radius of
this circle, and the constant of Theorem~A is $\sec\theta$.

\subsection{The catalogue}
Table~\ref{tab:main} summarises the constants. Two general statements organise them. The first
(Theorem~\ref{thm:transfer}) is a transfer principle: if an inequality is compatible with central
rescaling, its upper bound is proved with the ambient condition number in arbitrary finite von
Neumann algebras, and a matching example exists in $M_2(\C)$, then the violation constant is the
same function of $\kappa_\cZ$ in every finite von Neumann algebra. The matrix envelopes of
Section~\ref{sec:catalogue} supply such examples and, conversely, the transfer principle turns
several matrix results into results on von Neumann algebras. The constants that appear are $G$, its
square, the condition number, the Specht ratio and $\coth^2(\LNC/2)$; on $M_n$ the finite proved envelopes
vanish at the trace to the orders $\delta$, $\delta^2$ and $\delta^4$ (item~37 is $+\infty$ for every non-central $\rho$, i.e.\ non-scalar on $M_n$), and
the weighted monotonicity (item~25) has a proved lower bound of order $\delta^8$, while only the trivial
upper bound $V_n\le\kappa-1=2\delta/(1-\delta)$ (constant $\le\kappa_\cZ$ in finite von Neumann algebras) is known for it; the
computed envelope ($\delta\ge0.3$) does not exceed the lower bound
(Theorem~B and Figure~\ref{fig:ladder}). So the characterisations differ greatly in sensitivity: two inequalities
that are equivalent for traces may fail at different rates. In particular, the monotonicity of
$A\mapsto\varphi(A^2)$ fails only to the order $\delta^4$: its constant is
$(G_\cZ+1)^2/(4G_\cZ)$ (Theorem~\ref{thm:mono2}), and on $M_n$ an explicit two-parameter family of
rank-one perturbations attains it (Remark~\ref{rem:g10}).

Of the 38 trace characterisations, 18 have an exact constant in every finite von Neumann algebra
(items 1--5, 8, 10, 17, 18, 21--23, 26, 28, 29, 31, 33, and item~37, whose constant is $+\infty$),
five have an exact envelope on matrix algebras only (items 11, 19 and 24, item~9 for $n=2$ and
item~20 for even $n$), 13 have proved bounds (items 6, 7, 12, 14, 16, 25, 27, 30, 32, 34--36 and~38),
and items 13 and 15 have numerical values only.

Returning to the unnormalised quantities of~\cite{abed2024}, we find the sharp bound of
$|\Tr(XZ)|-\Tr(X|Z|)$ with both contrasts prescribed and those of the two other trace functionals
sampled there (Theorem~C), which answers the question of~\cite{abed2024}.

{\footnotesize
\renewcommand{\arraystretch}{1.2}
\begin{longtable}{@{}>{\raggedright}p{0.9cm}>{\raggedright}p{4.1cm}>{\raggedright}p{4.4cm}>{\raggedright\arraybackslash}p{4.5cm}@{}}
\caption{Violation constants. Items refer to Table~\ref{tab:items}. Third column: the constant $c(\rho)$
in every finite von Neumann algebra (``exact'' when proved to be the stated function of $\kappa_\cZ$);
in factors $\kappa_\cZ=\kappa$ and $G_\cZ=G$. Fourth column: the envelope $V_n(\delta)$ on $M_n$ with
the normalisation $\cN$ of Table~\ref{tab:items} ($V=c-1$ when $\cN=\cR$). $\yp$, $\tp$, $\mpp$,
$b_2$, $g_{25}$ and $s$ are defined in the cited results. Items 39--42: every von Neumann algebra
(item~42 exact for finite $\cM$).}\label{tab:main}\\
\toprule
item & inequality & $c(\rho)$, finite $\cM$ & $V_n(\delta)$ on $M_n$; status\\
\midrule
\endfirsthead
\toprule
item & inequality & $c(\rho)$, finite $\cM$ & $V_n(\delta)$ on $M_n$; status\\
\midrule
\endhead
\bottomrule
\endfoot
\multicolumn{4}{@{}l}{\emph{Traces: $\varphi=\tau(\rho\,\cdot)$ on a finite von Neumann algebra}}\\
1--5, 18, 22 & Gardner; subadditivity of $\varphi(|\cdot|)$ (four versions); $\varphi(A+A^*)\le2\varphi(|A^*|)$; $|\varphi(A^2)|\le\varphi(A^*A)$; metric on projections; $\varphi(\Re PQ)\le\varphi(|PQ|)$ & $G_\cZ$, exact (Thms.~\ref{thm:gardner}, \ref{thm:subadd}, \ref{thm:pz2}, Cors.~\ref{cor:item3}, \ref{cor:item5}, Prop.~\ref{prop:item22}) & $G-1\sim\delta^2/2$, attained in every $n$; with $\OMC(|X|)=\delta'$: Theorem~A\\
17 & $\varphi(PQP)\le\varphi(Q)$ & $G_\cZ^2$, exact (Thm.~\ref{thm:proj}) & $G^2-1\sim\delta^2$, attained\\
23 & $\varphi(P\vee Q)\le\varphi(P)+\varphi(Q)$ & $\frac{1+\kappa_\cZ}2$, exact (Thm.~\ref{thm:proj}) & $\frac{\kappa-1}2=\frac{\delta}{1-\delta}$, approached\\
26 & $\varphi(|X^*|^q)\le\varphi(|X|^q)$ & $\kappa_\cZ$, exact (Thm.~\ref{thm:modulus}) & $\kappa-1$, attained\\
33 & $\varphi(PQP)=\varphi(QPQ)$, $\varphi(s(PQP))=\varphi(s(QPQ))$ & $\kappa_\cZ$, exact (Thm.~\ref{thm:proj}) & $\kappa-1$, approached\\
8 & $\varphi(|AB^*|)\le\frac12\varphi(|A|^2+|B|^2)$, all $A,B$ & $\kappa_\cZ$, exact (Cor.~\ref{cor:item8}) & $\kappa-1=\frac{2\delta}{1-\delta}$, attained; with $\delta'$: Prop.~\ref{prop:youngtwo}\\
29 & Young, all $A,B$ & $\kappa_\cZ^{1/p}$, exact (Thm.~\ref{thm:young}) & $\kappa^{1/p}-1$, attained\\
30 & Young, $A,B\ge0$ & in $[\yp(\kappa_\cZ),\kappa_\cZ^{1/p}]$; $=\yp$ conjectured for $p\le2$, $>\yp$ for $p>2$ (Prop.~\ref{prop:youngpos}) & $p=2$: numerically $G-1$\\
31 & Young, $p=2$, $B=B^*$ & $(4G_\cZ^2-3)^{1/2}$, exact (Prop.~\ref{prop:youngsa}) & $(4G^2-3)^{1/2}-1$, attained\\
21 & $\varphi(\Re PQ)\ge0$ & $\cN$-normalised: $\frac12\tanh^2\frac{\LNC}2=\frac{G_\cZ-1}{2(G_\cZ+1)}$, exact (Prop.~\ref{prop:item21}) & $\frac12\tan^2\frac\psi2\sim\frac{\delta^2}8$, attained\\
28 & $\varphi(A^2+B^2+K(AB+BA))\ge0$ & holds iff $-1\le K\le\frac{G_\cZ+1}{G_\cZ-1}=\coth^2\frac{\LNC}2$, exact (Thm.~\ref{thm:ba}) & the same with $G$\\
10 & $0\le A\le B\Rightarrow\varphi(A^2)\le\varphi(B^2)$ & $\frac{(G_\cZ+1)^2}{4G_\cZ}$, exact (Thm.~\ref{thm:mono2}) & $\frac{(G-1)^2}{4G}\sim\frac{\delta^4}{16}$, attained (Rem.~\ref{rem:g10})\\
32 & monotonicity for $t^p$ ($p>1$) and $\exp$ & in $[\mpp(\kappa_\cZ),\kappa_\cZ]$; $\le K(\kappa_\cZ^{1/p},p)$, resp.\ $\le S(\kappa_\cZ)$, if $\cM$ is AFD (Prop.~\ref{prop:monop}) & --\\
25 & $0\le A\le B\Rightarrow\varphi(A^4)\le\varphi(AB^2A)$ & $\le\kappa_\cZ$ (Prop.~\ref{prop:wmon}) & in $[g_{25}(\delta),\kappa-1]$, $g_{25}\sim\frac{\delta^8}{4096}$ (Prop.~\ref{prop:wmon})\\
9 & $\varphi(PQ+QP)\le2\varphi(QPQ)$ & -- & $V_2=\frac s{1+s}$, $s=\sin\frac\psi2$, exact for $n=2$ only (Prop.~\ref{prop:abed})\\
19 & $\varphi(P+Q)\le\varphi(1+P\wedge Q)$ & --; diluted by other blocks (Rem.~\ref{rem:dilute}) & $\delta$ ($n$ even), $\frac{(n-1)\delta}{n-\delta}$ ($n$ odd), approached (Thm.~\ref{thm:projP})\\
20 & $\varphi(e^{QPQ})\le\varphi(e^P)$ & $c^{\rm pr}_{\exp}\in(1,\kappa_\cZ]$, not a function of $\kappa_\cZ$ (Prop.~\ref{prop:bikexp}) & max.\ of an explicit function, $\sim\frac{(e-1)^2}{e(e+1)}\delta^2$, exact for even $n$ only (Thm.~\ref{thm:projP})\\
36 & $\varphi(e^{PQP})\le\varphi(e^Q)$, $Q\ge0$ & in $(1,\kappa_\cZ]$, not a function of $\kappa_\cZ$ (Prop.~\ref{prop:bikexp}) & --\\
6 & $\varphi(A\natural B)\le\sqrt{\varphi(A)\varphi(B)}$ & in $[G_\cZ,\kappa_\cZ]$ (Prop.~\ref{prop:sgm}) & in $[G-1,\kappa-1]$; $>G-1$ for $n=3,4$ at some $\delta$ (numerical)\\
7, 38 & Araki--Lieb--Thirring, $(r,q)=(2,1)$; general $r>1$, $q>0$ & $c_{2,1}\in[G_\cZ^2,\kappa_\cZ]$ (Prop.~\ref{prop:altlow}); $c_{r,q}\in(1,\kappa_\cZ]$, $\sup_qc_{r,q}=\kappa_\cZ$ (Prop.~\ref{prop:alt}) & item 7: $\ge G^2-1$; the value $G^2-1$ is approached by an explicit family\\
27 & $|\varphi(X^k)|\le\varphi(|X|^k)$, $k\ge3$ & in $[G_\cZ,\kappa_\cZ^{1-1/k}]$; not a function of $\kappa$, already on $M_k(\C)$ (Props.~\ref{prop:pzk}--\ref{prop:pzk4}) & --\\
12, 34 & $\varphi((A+B)^p)\le\varphi(A^p)+\varphi(B^p)$, $0<p<1$ (item 12: $p=\frac12$) & in $[\tp(\kappa_\cZ),\frac{1+\kappa_\cZ}2]$ (Thm.~\ref{thm:tik}) & item 12: search $\approx\mathsf t_{1/2}(\kappa)-1\approx\frac{\delta^2}2$; no violation if $\OMC(A)=\OMC(B)\le\frac{2\sqrt2}3$ (Prop.~\ref{prop:item12})\\
13, 15 & Powers--St{\o}rmer; Golden--Thompson for projections & -- & numerics only; search: $\approx0.5\,\delta^2$, $\approx0.064\,\delta^2$\\
14 & $|\varphi(PQ)|\le\varphi(P)^{1/3}\varphi(Q)^{2/3}$ & -- & $\ge\frac{\delta^2}{36}+O(\delta^3)$ (Prop.~\ref{prop:holder})\\
16, 37 & Peierls--Bogoliubov, symmetric form; form with $\varphi(Be^A)$ & $c_{\rm PB}\in[b(\rho),\kappa_\cZ]$, $b\ge b_2(\kappa_\cZ)$; item 37: exact ($+\infty$ for $\rho\notin\cZ$) (Prop.~\ref{prop:pb}) & item 16: search $\approx b_2(\kappa)-1$\\
35 & $\varphi(A)\le\varphi(|A+iB|)$ & in $(1,G_\cZ]$; $<2$ on $M_2$; not a function of $\kappa_\cZ$ (Prop.~\ref{prop:af}) & --\\
11, 24 & $(\det A)^{1/n}\le\frac1n\varphi(A)$, $\varphi(A^2)^{1/2}\le\varphi(A)$, canonical trace & -- & exact for every $n$ (Props.~\ref{prop:det}, \ref{prop:power}); $n=2$ values shown: $1-\sqrt{1-\delta^2}$, $(1-\delta)^{-1/2}-1$; they grow with $n$\\
\midrule
\multicolumn{4}{@{}l}{\emph{Central elements: $a$ positive invertible in a von Neumann algebra}}\\
39 & $|\varphi(a)|\le|\varphi|(a)$ & $\cosh\LNC(a)$, exact (Thm.~\ref{thm:nt2}) & \\
40 & $pap\le a$ & $\cosh^2\LNC(a)$, exact (Thm.~\ref{thm:pap}) & \\
41 & other items of \cite[Thm.~1, Cor.~1]{novikov2015} & $\cosh^2\LNC(a)$ or $\cosh\LNC(a)$, exact (Thm.~\ref{thm:ntrest}) & \\
42 & $|\varphi|(a)=\|a^{1/2}\varphi a^{1/2}\|$ & ratio in $[\kappa_\cZ^{-1/2},\kappa_\cZ^{1/2}]$; exact for finite $\cM$ (Thm.~\ref{thm:l1}) & \\
\end{longtable}}

\subsection{Noncentrality and the comparison theorem}
The second general statement (Theorem~\ref{thm:compare}) compares contrast and noncentrality: the
contrast always gives a valid but weaker bound; it gives the exact constant if and only if
$\max\ceil{\rho}$ and $\min\flr{\rho}$ are attained at a common point of the spectrum of the centre;
the loss is unbounded; and in every finite von Neumann algebra that is neither a factor nor abelian
(every such von Neumann algebra, finite or not, for the constants of Section~\ref{sec:centre} other
than Theorem~\ref{thm:l1}) the violation constant is not a function of the contrast, while for the
exact cases it is a function of the noncentrality. On a direct sum $\bigoplus_cM_{k_c}$ this means
that only the contrast inside the blocks counts (Corollary~\ref{cor:sectors}): a density that varies
only between the blocks violates none of the inequalities except the two that characterise the
\emph{canonical} trace, items 11 and 24.

\subsection{Central elements}
The dual problem, in which the functional is arbitrary and an element is to be recognised as
central, has the same answer. For $a\ge0$ invertible the best constant $\lambda$ in $pap\le\lambda a$
over all projections $p$ is $\cosh^2\LNC(a)$, and the best constant in $|\varphi(a)|\le
C|\varphi|(a)$ over all normal functionals is $\cosh\LNC(a)$, in every von Neumann algebra
(Section~\ref{sec:centre}).

\subsection{Composite operators, thresholds and numerics}
Section~\ref{sec:composite} records how contrasts behave under the composites through which the
operators enter the inequalities: the rapidities $\OLC=\operatorname{artanh}\OMC$ add under products and
means (the triangle inequality of Hilbert's projective metric), the angles $\arcsin\OMC$ add under
Jordan products (Gustafson's criterion), and the commutator test $\varphi(AB)=\varphi(BA)$ fails by exactly
$2\OMC(\rho)\OMC(A)\OMC(B)$ in even dimension. These rules prove sharp thresholds: if
$\OMC(A)=\OMC(B)\le2\sqrt2/3$, then $(A+B)^{1/2}\le A^{1/2}+B^{1/2}$ as operators, so no positive functional
violates subadditivity of the square root, while above $2\sqrt2/3$ suitable operators violate it for
every density of contrast close enough to $1$; and the power-mean inequality $\varphi(A^2)^{1/2}\le\varphi(A)$
is never violated when $\OMC(A)\le\sqrt{(n-1)/n}$, sharply (Section~\ref{sec:thresholds} and
Figure~\ref{fig:square}).

For the inequalities without a closed form we computed the envelopes on $M_2$, $M_3$ and $M_4$ by
constrained optimisation (Section~\ref{sec:numerics}), using the open-source library of~\cite{lib}. Every
reported value is recomputed in 40-digit arithmetic at an exactly admissible point, so it is a rigorous
\emph{lower} bound for the supremum, not a certified value of it. Where the envelope is known, the
search reproduces it to $\TCSharpDev$--$\TCSharpDevNew$; it reaches the proved value of item~10 and the
proved lower bound of item~12, and comes within $3\cdot10^{-3}$ of the proved lower bound of item~16
(Section~\ref{sec:numcompare}), and elsewhere it supports the conjectures of
Section~\ref{sec:open}. Random sampling, the method of~\cite{abed2024}, rarely comes close to the
envelope: for monotonicity and subadditivity fewer than one random sample in a thousand violates the
inequality at all. No proof in this paper depends on a floating-point computation.

\subsection{What is new and what is classical}
The characterisation theorems are due to the cited authors. The constants $G$, $\kappa$ and the
Specht ratio, and the inequalities behind them, are classical: the Kantorovich
inequality~\cite{kantorovich1948} in the operator form of Greub and Rheinboldt~\cite{greub1959}
(equivalently, the Krein inequality and Gustafson's operator angle~\cite{gustafson1968,krein1969}),
Wielandt's inequality in the Bauer--Householder form~\cite{BH60}, H\"older's inequality~\cite{FK86},
the Araki--Lieb--Thirring inequality~\cite{Kos92} and Specht's reverse arithmetic--geometric mean
inequality~\cite{specht1960}. In factors several of the upper bounds below are short consequences of
these inequalities; on $M_n(\C)$ the Gardner constant and the operator inequality $P\rho P\le
G^2\rho$ are contained in inequalities of Bourin~\cite{Bou06}, and Bourin's reverse of Araki's
inequality~\cite{Bou05} gives our upper bound for monotonicity on approximately finite-dimensional
algebras. Several statements of Section~\ref{sec:composite} are also classical: the bounds for
congruences are the submultiplicativity of the condition number, the positivity criterion for
Jordan products is Gustafson's~\cite{gustafson1968}, and the norm bound for commutators of positive
operators follows from Kittaneh's inequality~\cite{kittaneh2008}. In Section~\ref{sec:nc}, the central
envelopes are the extreme elements of Halpern's central spectrum~\cite{Hal72}, and the identity
$\|\Ad_A\|=\kappa_\cZ(|A|)$ for the norm of an inner automorphism follows from the description of
norms of elementary operators in~\cite{CS93,Som98,Tim07} (Remark~\ref{rem:ad}).

What we add is that the classical constants are the exact violation constants; the determination of
these constants in finite von Neumann algebras and their central versions; the identification of
the noncentrality as the governing parameter; the dependence on the product of the contrasts in
Theorem~A; the exact envelopes of the projection inequalities 17--23 (items 19 and 20 on $M_n$, item~20
for even $n$), of the Bikchentaev--Abed inequality 28 and, for $n=2$, of item~9, and of the
monotonicity (item~10); the constructions with violations of the orders $\delta^4$ and $\delta^8$; the sharp
bounds of Theorem~C; the sharp trace version of the commutator bound; and the threshold results of
Section~\ref{sec:thresholds}. For the Petz--Zem\'anek inequality with $k\ge3$ the situation is
different: the constant depends on more of the spectrum than its extreme points
(Propositions~\ref{prop:pz3} and~\ref{prop:pzk4}). We know of no previous study of these constants;
the closest work is~\cite{abed2024}, which samples the unnormalised Gardner violation but defines no
envelope and proves no bound.

\subsection{Organisation}
Section~\ref{sec:related} reviews related work. Section~\ref{sec:setting} fixes the notation, lists
the inequalities and collects the tools. Section~\ref{sec:nc} recalls central envelopes and
noncentrality and proves the transfer principle. Section~\ref{sec:catalogue} determines the violation
constants item by item, Section~\ref{sec:gardner2} proves Theorems~A and~C,
Section~\ref{sec:centre} treats the centre characterisations, and Section~\ref{sec:compare} compares
contrast and noncentrality, including direct sums. Section~\ref{sec:composite} treats composite
operators and the thresholds, Section~\ref{sec:numerics} the numerical evidence, and
Section~\ref{sec:open} lists conjectures and open problems. Appendix~\ref{app:mc} shows the computed
envelopes and random samples item by item.

\section{Related work}\label{sec:related}

\paragraph{Characterisations of the trace.} Gardner~\cite{gardner1979} proved that a linear functional
on a $C^*$-algebra with $|\varphi(x)|\le\varphi(|x|)$ for all $x$ is a trace. Petz and
Zem\'anek~\cite{petz1988} showed that $|\varphi(x^k)|\le\varphi(|x|^k)$ for all $x$ and a single
$k\in\mathbb N$ already characterises positive traces, and characterised traces by subadditivity on
the projection lattice~\cite[Theorem~4]{petz1988}. Stolyarov, Tikhonov and
Sherstnev~\cite{stolyarov2002} characterised traces among normal semifinite weights by the
subadditivity $\varphi(|a_1+a_2|)\le\varphi(|a_1|)+\varphi(|a_2|)$ for self-adjoint $a_1,a_2$, and
Tikhonov~\cite{tikhonov2005} studied when the subadditivity $\varphi(f(A+B))\le\varphi(f(A))+\varphi(f(B))$
for positive $A,B$ characterises tracial functionals; formally, for $f=\one_{(0,\infty)}$ and
projections this inequality becomes $\varphi(P\vee Q)\le\varphi(P)+\varphi(Q)$. Bikchentaev and Tikhonov
characterised the trace on $M_n$ by Young's inequality $\varphi(|AB|)\le\varphi(A^p)/p+\varphi(B^q)/q$ for
positive $A,B$~\cite{bikchentaev2005young} and by the monotonicity $0\le A\le
B\Rightarrow\varphi(A^p)\le\varphi(B^p)$~\cite{bikchentaev2007mono}; Sano and Yatsu~\cite{sano2006}
treated monotonicity and convexity for general functions, Cho and Sano~\cite{cho2009} Young's
inequality for arbitrary operators and general Young pairs, and Dinh Trung Hoa and Tikhonov the
weighted forms of monotonicity~\cite{hoa2007,hoa2010}. Bikchentaev obtained a long series of
characterisations by projection inequalities such as $\varphi(PQP)=\varphi(QPQ)$, $\varphi(PQP)\le\varphi(Q)$
and $\varphi(|P-Q|)\le\varphi(P+Q)$~\cite{bikchentaev2009rm,bikchentaev2010smj,bikchentaev2011mn,Bik15}, and
by many other inequalities, see~\cite{Bik16} and the survey~\cite{Bik12}; Bikchentaev and
Abed~\cite{bikchentaev2019abed} characterised traces by the anticommutator inequality
$\varphi(A^2+B^2+K(AB+BA))\ge0$ for $A,B\ge0$ and all large $K$ and by inequalities for two projections.
Bikchentaev~\cite{bikchentaev2011pb} characterised tracial functionals by the Peierls--Bogoliubov and
the Araki--Lieb--Thirring inequalities, and Alhasan and Fawwaz~\cite{AF21} by
$\varphi(A)\le\varphi(|A+iB|)$ for $A\ge0$ and $B=B^*$. Further characterisations use the generalised
Powers--St{\o}rmer inequality~\cite{hoa2013}, inequalities for determinants and power
means~\cite{bikchentaev2020det}, the spectral geometric mean~\cite{bikchentaev2026sgm}, inequalities
for $|\varphi(A^2)|$~\cite{bikchentaev2026moslehian} and metrics on
projections~\cite{bikchentaev2019metric}. Recent contributions include~\cite{BM25b,moslehian2025jmaa};
the canonical trace on $M_n$ is characterised in~\cite{moslehian2025jmaa,bikchentaev2020det}. Bikchentaev
and Moslehian~\cite{bikchentaev2025survey} survey these results; see also~\cite{AM25,BKMS24}.

\paragraph{Characterisations of central elements.} Novikov and Tikhonov~\cite{novikov2015} proved that
$a\ge0$ is central if and only if $pap\le a$ for all projections $p$, if and only if
$|\varphi(a)|\le|\varphi|(a)$ for all normal functionals $\varphi$, and gave several equivalent forms;
Novikov~\cite{Nov17} extended such results to unbounded operators affiliated with the algebra.
Characterisations of central elements by local monotonicity or convexity of functions are due to
Moln\'ar~\cite{Mol17} and Virosztek~\cite{Vir17}. All these results are qualitative.

\paragraph{Reverse inequalities.} Our upper bounds are reverse inequalities of Kantorovich type.
Kantorovich's inequality~\cite{kantorovich1948} and its operator form due to Greub and
Rheinboldt~\cite{greub1959} bound $\omega(\rho^2)\omega(1)$ by $G^2\omega(\rho)^2$; the Mond--Pe\v{c}ari\'c
method~\cite{furuta2005} produces such constants systematically, Specht's ratio~\cite{specht1960} is
the sharp constant of the reverse arithmetic--geometric mean inequality, and Bourin~\cite{Bou05,Bou06}
and Fujii and Seo~\cite{FS05} proved reverse inequalities of Araki and Davis--Hansen--Pedersen type.
The question here is different: not whether such a constant bounds the violation, but whether it is
the smallest one.

\paragraph{Operator angles.} Krein~\cite{krein1969} introduced the deviation of an operator and
Gustafson~\cite{gustafson1968} the equivalent angle $\psi(A)$,
$\cos\psi(A)=\inf_{x\ne0}\Re\langle Ax,x\rangle/(\|Ax\|\|x\|)$, the first antieigenvalue; for
$0<m\le A\le M$ with $m,M$ in the spectrum, $\cos\psi(A)=2\sqrt{Mm}/(M+m)$ and
$\sin\psi(A)=(M-m)/(M+m)$~\cite{gustafsonrao1997,gustafson2012}. Thus the Michelson contrast of a
positive operator is the sine of its angle, and the Kantorovich constant is $G(\rho)=\sec\psi(\rho)$.
Gustafson used the angle to decide when products of positive operators are
accretive~\cite{gustafson1968}.

\paragraph{Jordan products.} The product $AB$ of positive operators is accretive exactly when the
Jordan product $AB+BA=2\Re(AB)$ is positive. Bounds for the eigenvalues of $AB+BA$ go back to
Strang~\cite{strang1962} and Nicholson~\cite{nicholson1979}; Alikakos and Bates~\cite{alikakos1984}
bounded all eigenvalues and showed that Strang's lower bound for the smallest one is best
possible~\cite{alikakos1985}, and Conley, Pucci and Serrin~\cite{conley2005} gave necessary and
sufficient conditions for the positivity of the symmetrised product of two positive definite
matrices. Moslehian and Najafi~\cite[Theorem~2.1]{moslehian2011} proved that $AB+BA\ge0$ if and only if
$f(A+B)\le f(A)+f(B)$ for every non-negative operator monotone function $f$ on $[0,\infty)$. This
links Jordan products to the subadditivity inequalities that characterise traces~\cite{tikhonov2005},
and it is the background of Section~\ref{sec:thresholds}.

\paragraph{Commutators, derivations and Hilbert's metric.} Kittaneh~\cite{kittaneh2008} proved
$\|XY-YX\|\le\|X\|\|Y\|$ when $X\ge0$, which gives $\|AB-BA\|\le\frac12\|A\|\|B\|$ for positive $A,B$;
versions for unitarily invariant norms are in~\cite{kittaneh2007,bhatia2008}. The norm of an inner
derivation is twice the distance to the centre~\cite{Sta70,KLR67,Gaj72,Zsi73}, and the noncentrality is
its multiplicative counterpart: the norm of the inner automorphism $\Ad_A$ is $\kappa_\cZ(|A|)$, which
follows from~\cite{CS93,Som98,Tim07} (Remark~\ref{rem:ad}). On the cone of positive definite matrices,
Thompson's metric~\cite{thompson1963} is $d_T(A,B)=\|\log(A^{-1/2}BA^{-1/2})\|$ and Hilbert's
projective metric~\cite{bushell1973,lemmens2012} is
$d_H(A,B)=\log\bigl(\lambda_{\max}/\lambda_{\min}\bigr)(A^{-1/2}BA^{-1/2})$. The logarithmic contrast of a
positive definite $A$ is
\[
  \OLC(A)=\operatorname{artanh}\OMC(A)=\tfrac12\log\frac{\lambda_{\max}(A)}{\lambda_{\min}(A)}=\tfrac12\,d_H(A,1)
  =\min_{t>0}d_T(A,t1),
\]
the Thompson distance of $A$ to the scalars; this is why it behaves additively under products
(Proposition~\ref{prop:rapid}), and why its central version $\LNC$ is the Thompson distance to the
centre (Proposition~\ref{prop:char}).

\section{Setting and tools}\label{sec:setting}

\subsection{Von Neumann algebras and densities}
$\cM\subset B(\cH)$ is a von Neumann algebra with centre $\cZ$; $\cM_{sa}$, $\cM_+$ and
$\cM_{++}$ denote the self-adjoint, positive, and positive invertible elements. By the Gelfand
representation $\cZ\cong C(\Omega)$ with $\Omega$ extremally disconnected; we identify $z\in\cZ$
with its Gelfand transform. Central projections correspond to clopen subsets of $\Omega$, which
form a base of the topology. For a projection $e$, $c(e)$ is its central support, and $s(x)$ is
the support projection of $x\in\cM_+$. We use the comparison fact \cite[Ch.~6]{KR2}: if $E,F$
are projections with $c(E)c(F)\ne0$, then there is a nonzero partial isometry $v\in\cM$ with
$v^*v\le F$ and $vv^*\le E$.

A von Neumann algebra is \emph{finite} here if it carries a faithful normal tracial state
$\tau$. Then $\|X\|_p=\tau(|X|^p)^{1/p}$, H\"older's inequality
$|\tau(XY)|\le\tau(|XY|)\le\|X\|_p\|Y\|_q$ ($1/p+1/q=1$) and the duality
$\|X\|_p=\sup\{|\tau(XY)|:\|Y\|_q\le1\}$ hold \cite{FK86}. Every $X\in\cM$ has a polar
decomposition $X=U|X|$ with $U$ unitary, and $|X^*|=U|X|U^*$. On $L^2(\cM,\tau)$ with
$\langle X,Y\rangle=\tau(Y^*X)$ the left and right multiplications $L_X h=Xh$, $R_Xh=hX$ are
bounded; for $D\in\cM_{++}$, $L_D$ and $R_D$ are commuting positive operators with spectra in
$\sigma(D)$. We use the Araki--Lieb--Thirring inequality in Kosaki's form \cite{Kos92}: for
$A,B\in\cM_+$ and $r\ge1$,
\begin{equation}\label{eq:alt}
\tau\bigl((B^{1/2}AB^{1/2})^r\bigr)\le\tau\bigl(B^{r/2}A^rB^{r/2}\bigr).
\end{equation}

For $\rho\in\cM_{++}$ put $\varphi=\varphi_\rho=\tau(\rho\,\cdot)$. Then $\varphi$ is tracial
if and only if $\rho\in\cZ$: if $\tau(\rho XY)=\tau(\rho YX)=\tau(X\rho Y)$ for all $X,Y$, then
$\tau((\rho X-X\rho)Y)=0$ for all $Y$, so $\rho X=X\rho$.

On $M_n=M_n(\C)$ we use the usual trace $\Tr$ in place of $\tau$ (all constants are invariant under
$\tau\mapsto s\tau$), $|X|=(X^*X)^{1/2}$, $\|X\|_1=\Tr|X|$, $\Re X=\frac12(X+X^*)$, and
$\varphi=\varphi_\rho=\Tr(\rho\,\cdot\,)$ with $\rho\ge0$; $\varphi$ is a trace iff $\rho\in\C1$. We write
$a\ge b$ for the extreme eigenvalues of $\rho$, $e_1,e_2$ for corresponding unit eigenvectors, and
$\delta=\OMC(\rho)$; unless stated otherwise $\rho>0$, i.e.\ $\delta<1$. Variables called $P,Q,R$ are
orthogonal projections, $A,B\ge0$ unless said otherwise, and $X,Y$ arbitrary; $P\wedge Q$ and $P\vee Q$
are the projections onto $\mathrm{ran}\,P\cap\mathrm{ran}\,Q$ and onto $\mathrm{ran}\,P+\mathrm{ran}\,Q$.

\begin{table}[ht]
\centering\small
\renewcommand{\arraystretch}{1.15}
\begin{tabular}{@{}>{\raggedright}p{4.6cm}>{\raggedright\arraybackslash}p{10.6cm}@{}}
\toprule
symbol & meaning\\
\midrule
$\cM$, $\cZ$, $\tau$ & von Neumann algebra, its centre, a faithful normal tracial state ($\Tr$ on $M_n$)\\
$\rho\in\cM_{++}$, $\varphi=\varphi_\rho=\tau(\rho\,\cdot\,)$ & density and functional\\
$\kappa(T)=\max\sigma(T)/\min\sigma(T)$ & condition number; $\kappa=\kappa(\rho)$\\
$\OMC(T)=\frac{\kappa(T)-1}{\kappa(T)+1}$ & Michelson contrast; $\delta=\OMC(\rho)$, and $\delta'$ the contrast of the operator variables\\
$\OLC(T)=\frac12\ln\kappa(T)=\operatorname{artanh}\OMC(T)$ & logarithmic contrast (rapidity; $\eta_A=\OLC(A)$ in Section~\ref{sec:composite})\\
$\psi(T)=\arcsin\OMC(T)$ & angle of Krein and Gustafson; $\psi=\psi(\rho)$, $\psi_A=\psi(A)$\\
$G(T)=\cosh\OLC(T)=\sec\psi(T)=\frac{1+\kappa(T)}{2\sqrt{\kappa(T)}}$ & Kantorovich constant; $G=G(\rho)=(1-\delta^2)^{-1/2}$, and $G(x)=(1-x^2)^{-1/2}$ for a number $x\in[0,1)$\\
$\ell=\frac1\kappa=\frac{1-\delta}{1+\delta}$, $\ell'=\frac{1-\delta'}{1+\delta'}$ & smallest eigenvalue after normalising the largest to $1$\\
$\ceil\rho$, $\flr\rho$ & central envelopes (Lemma~\ref{lem:env})\\
$\kappa_\cZ(\rho)=\|\ceil\rho\flr\rho^{-1}\|$ & central condition number\\
$\delta_\cZ=\NC(\rho)=\frac{\kappa_\cZ-1}{\kappa_\cZ+1}$ & noncentrality\\
$\LNC(\rho)=\frac12\ln\kappa_\cZ(\rho)=\operatorname{artanh}\delta_\cZ$ & logarithmic noncentrality\\
$\psi_\cZ=\arcsin\delta_\cZ$, $G_\cZ=\cosh\LNC(\rho)=\sec\psi_\cZ$ & central angle and central Kantorovich constant\\
$c(\rho)=c_\cM(\rho)$ & violation constant (Definition~\ref{def:viol})\\
$\cL$, $\cR$, $\cN$; $V(\rho)$, $V_n(\delta)$, $V_n(\delta,\delta')$ & sides and normalisation of an inequality; violation envelopes on $M_n$ (Section~\ref{sec:envelopes})\\
\bottomrule
\end{tabular}
\caption{Notation. In factors, in particular on $M_n$, $\kappa_\cZ=\kappa$, $\delta_\cZ=\delta$,
$\psi_\cZ=\psi$ and $G_\cZ=G$.}
\label{tab:notation}
\end{table}

\subsection{Contrast, angle and the Kantorovich constant}
For a positive invertible operator $T$ on a Hilbert space with $m=\min\sigma(T)$, $M=\max\sigma(T)$
let
\begin{gather*}
\kappa(T)=\frac Mm,\qquad \OLC(T)=\tfrac12\ln\kappa(T),\\
\OMC(T)=\tanh\OLC(T)=\frac{M-m}{M+m},\qquad G(T)=\cosh\OLC(T)=\frac{M+m}{2\sqrt{Mm}}.
\end{gather*}
$\OMC$ is the operator Michelson contrast of~\cite{abed2024} and $\OLC$ its logarithmic
form~\cite{NN26}. In Gustafson's operator trigonometry~\cite{gustafson1968} the angle $\psi(T)$ of $T$
is defined by $\cos\psi(T)=\inf_{h\ne0}\langle Th,h\rangle/(\|Th\|\,\|h\|)$, and
$\cos\psi(T)=2\sqrt{Mm}/(M+m)$ by the Krein inequality~\cite{gustafson1968,krein1969}. Hence
\[
\OMC(T)=\sin\psi(T),\qquad G(T)=\cosh\OLC(T)=\sec\psi(T).
\]

\begin{lemma}[Kantorovich, Krein]\label{lem:krein}
Let $0<m\le T\le M$ and $G=(M+m)/(2\sqrt{Mm})$.
\begin{enumerate}
\item\label{k:a} $\omega(T^2)\,\omega(1)\le G^2\,\omega(T)^2$ for every positive functional $\omega$ on a
$C^*$-algebra containing $T$~\cite{greub1959}.
\item\label{k:b} $\|Th\|\,\|h\|\le G\,\langle Th,h\rangle$ for every vector $h$.
\end{enumerate}
\end{lemma}
\begin{proof}
$(M-T)(T-m)\ge0$ gives $\omega(T^2)\le(M+m)x-Mm\,y$ with $x=\omega(T)$, $y=\omega(1)$, and
$((M+m)x-Mmy)y\le G^2x^2$ because the difference is $(Gx-\sqrt{Mm}\,y)^2$. Part~\ref{k:b} is
\ref{k:a} for the vector functional $\omega=\langle\cdot\,h,h\rangle$.
\end{proof}

With $x=T^{1/2}h$, Lemma~\ref{lem:krein}\ref{k:b} is the Kantorovich inequality
$\langle Tx,x\rangle\langle T^{-1}x,x\rangle\le G(T)^2\,\|x\|^4$~\cite{kantorovich1948}; equivalently,
it is the case $B=1$ of the Greub--Rheinboldt form~\cite{greub1959}
\[
\langle A^2x,x\rangle\langle B^2x,x\rangle\le\frac{(M_1M_2+m_1m_2)^2}{4m_1m_2M_1M_2}\,\langle ABx,x\rangle^2
\]
for commuting positive $A,B$ with $0<m_1\le A\le M_1$ and $0<m_2\le B\le M_2$. We call $G(T)$ the
\emph{Kantorovich constant} of $T$; its square $G(T)^2=(M+m)^2/(4Mm)$ is the best constant in
Kantorovich's inequality, and the expression $\langle Tx,x\rangle\langle T^{-1}x,x\rangle$ for a unit
vector $x$ will be called the Kantorovich expression.

\begin{lemma}[Wielandt]\label{lem:wielandt}
Let $\delta=\OMC(T)$.
\begin{enumerate}
\item\label{w:one} If $\re\langle x,y\rangle\ge0$, then
$-\re\langle Tx,y\rangle\le\delta\,\langle Tx,x\rangle^{1/2}\langle Ty,y\rangle^{1/2}$.
\item\label{w:two} If $x\perp y$, then $|\langle Tx,y\rangle|\le\delta\,\langle
Tx,x\rangle^{1/2}\langle Ty,y\rangle^{1/2}$.
\end{enumerate}
\end{lemma}

\begin{proof}
\ref{w:one} Put $u=T^{1/2}x$, $v=T^{1/2}y$ and $c=\frac12(m^{-1}+M^{-1})$. Then
$T^{-1}=c(1+E)$ with $E$ self-adjoint and $\|E\|\le(m^{-1}-M^{-1})/(m^{-1}+M^{-1})=\delta$. Hence
$0\le\re\langle x,y\rangle=\re\langle T^{-1}u,v\rangle=c(\re\langle u,v\rangle+\re\langle
Eu,v\rangle)$, so $-\re\langle Tx,y\rangle=-\re\langle u,v\rangle\le\|E\|\,\|u\|\,\|v\|$.
\ref{w:two} If $x\perp y$, then \ref{w:one} applies to $x$ and $-e^{i\theta}y$ for every $\theta$.
\end{proof}

Part~\ref{w:two} is Wielandt's inequality in the form of Bauer and Householder \cite{BH60}; part
\ref{w:one} is a one-sided version that we need for the Bikchentaev--Abed inequality.

\begin{lemma}\label{lem:pTp}
For every projection $p$ on $\cH$, $pTp\le\cosh^2\OLC(T)\,T$.
\end{lemma}

\begin{proof}
For $k\in\cH$ let $x=pk$, $y=(1-p)k$, $X=\langle Tx,x\rangle$, $Y=\langle Ty,y\rangle$ and
$\delta=\OMC(T)$. By Lemma~\ref{lem:wielandt}\ref{w:two},
$\langle Tk,k\rangle\ge X+Y-2\delta\sqrt{XY}\ge(1-\delta^2)X$ (minimise over $\sqrt Y$), and
$(1-\delta^2)^{-1}=\cosh^2\OLC(T)$.
\end{proof}

Lemma~\ref{lem:pTp} is classical. Since $pTp\le\lambda T$ if and only if
$\|T^{1/2}pT^{-1/2}\|^2\le\lambda$, the best constant for a rank-one projection
$p=|e\rangle\langle e|$ is the Kantorovich expression $\langle Te,e\rangle\langle T^{-1}e,e\rangle$,
and the lemma is then the Kantorovich inequality \cite{kantorovich1948}. For general $p$ it is the
Wielandt (Bauer--Householder) bound \cite{BH60} on the norm of the oblique projection
$T^{1/2}pT^{-1/2}$, whose range $T^{1/2}p\cH$ and kernel $T^{1/2}(1-p)\cH$ make an angle with
cosine at most $\delta$ by Lemma~\ref{lem:wielandt}\ref{w:two}. For matrices it is also the case
$A=p$ of Bourin's inequality $AZA\le\frac{(a+b)^2}{4ab}Z$ for $0\le A\le1$ and $Z>0$ with extremal
eigenvalues $a,b$ \cite[Cor.~1.5]{Bou06}.

\subsection{Thompson metric}
For $B,C\in\cM_{++}$ let $M(B/C)=\inf\{t>0:B\le tC\}$ and
$d_T(B,C)=\ln\max\{M(B/C),M(C/B)\}$, the Thompson metric \cite{thompson1963}.

\subsection{Violation constants and envelopes}\label{sec:envelopes}
For each inequality $\cL\le\cR$ of Table~\ref{tab:items} the \emph{violation constant} $c_\cM(\rho)$ is
the best constant $c$ in $\cL\le c\,\cR$ on the algebra $\cM$, i.e.\ the supremum of $\cL/\cR$ over the
admissible variables (Definition~\ref{def:viol} makes this precise). On $M_n$ we measure the violation
also by the quotient $(\cL-\cR)/\cN$ with the normalisation $\cN$ of Table~\ref{tab:items}; it is
invariant under $\rho\mapsto t\rho$, $t>0$. For a fixed density $V(\rho)$ denotes its supremum, and the
envelope $V_n(\delta)$ is the supremum of $V(\rho)$ over all $\rho$ in $M_n$ with $\OMC(\rho)=\delta$. If
$\cN=\cR$, then $V(\rho)=c_{M_n}(\rho)-1$; since $\kappa_\cZ=\kappa$ on $M_n$, every constant of the form
$c(\rho)=f(\kappa_\cZ(\rho))$ gives $V_n(\delta)=f(\frac{1+\delta}{1-\delta})-1$, the same in every
dimension $n\ge2$. We call a value of a constant or an envelope \emph{attained} if the supremum is a
maximum and \emph{approached} otherwise.

\subsection{The inequalities}\label{sec:items}
Table~\ref{tab:items} lists the inequalities. Items 1--25 are those for which we compute matrix
envelopes; items 26--38 are further characterisations of traces, several of them families that
contain an item among 1--25; items 39--42 characterise central elements.

{\small
\setlength{\tabcolsep}{3pt}
\begin{longtable}{@{}r>{\raggedright}p{7.4cm}>{\raggedright}p{2.5cm}>{\raggedright}p{1.9cm}>{\raggedright\arraybackslash}p{3.2cm}@{}}
\caption{The inequalities studied, with the normalisation $\cN$ of the violation $(\cL-\cR)/\cN$ on
$M_n$ and the source of the characterisation. Projections are $P,Q,R$; positive operators $A,B$
unless stated otherwise. Remarks on the forms are in Section~\ref{sec:forms}.}\label{tab:items}\\
\toprule
 & inequality $\cL\le\cR$ & variables & $\cN$ & source\\
\midrule
\endfirsthead
\toprule
 & inequality $\cL\le\cR$ & variables & $\cN$ & source\\
\midrule
\endhead
\bottomrule
\endfoot
\multicolumn{5}{@{}l}{\emph{Traces}}\\
1 & $|\varphi(X)|\le\varphi(|X|)$ & $X$ & $\cR$ & \cite{gardner1979}\\
2 & $\varphi(|X+Y|)\le\varphi(|X|)+\varphi(|Y|)$ & $X,Y$ & $\cR$ & \cite{gardner1979,stolyarov2002}\\
3 & $\varphi(A+A^*)\le2\varphi(|A^*|)$ & $A$ arbitrary & $\cR$ & \cite{bikchentaev2019abed}\\
4 & $|\varphi(A^2)|\le\varphi(A^*A)$ & $A$ arbitrary & $\cR$ & \cite[Thm.~1, $k=2$]{petz1988}, \cite{bikchentaev2026moslehian}\\
5 & $\varphi(|P-R|)\le\varphi(|P-Q|)+\varphi(|Q-R|)$ & $P,Q,R$ & $\cR$ & \cite{bikchentaev2019metric}\\
6 & $\varphi(A\natural B)\le\sqrt{\varphi(A)\varphi(B)}$ & $A,B>0$ & $\cR$ & \cite{bikchentaev2026sgm}\\
7 & $\varphi((B^{1/2}AB^{1/2})^2)\le\varphi(BA^2B)$ & $A,B$ & $\cR$ & \cite[Lemma~2]{bikchentaev2011pb}; item~38 with $(r,q)=(2,1)$\\
8 & $\varphi(|AB^*|)\le\frac12\varphi(|A|^2)+\frac12\varphi(|B|^2)$ & $A,B$ arbitrary & $\cR$ & extension of \cite{bikchentaev2005young}\\
9 & $\varphi(PQ+QP)\le2\varphi(QPQ)$ & $P,Q$ & $\varphi(P)+\varphi(Q)$ & \cite{bikchentaev2019abed}\\
10 & $0\le A\le B\Rightarrow\varphi(A^2)\le\varphi(B^2)$ & $A\le B$ & $\cR$ & \cite{bikchentaev2007mono,sano2006}\\
11 & $(\det A)^{1/n}\le\frac1n\varphi(A)$, $\varphi(1)=n$ & $A$ & $(\det A)^{1/n}$ & \cite[Thm.~3(iii)]{bikchentaev2020det}\\
12 & $\varphi((A+B)^{1/2})\le\varphi(A^{1/2})+\varphi(B^{1/2})$ & $A,B$ & $\cR$ & \cite[Cor.~1]{tikhonov2005}\\
13 & $\varphi(A+B-|A-B|)\le2\varphi(A^{1/4}B^{1/2}A^{1/4})$ & $A,B$ & $\varphi(A+B)$ & \cite[Lemma~3.1]{hoa2013}\\
14 & $|\varphi(PQ)|\le\varphi(P)^{1/3}\varphi(Q)^{2/3}$ & $P,Q$ & $\cR$ & \cite[Thm.~2(i)]{bikchentaev2011mn}\\
15 & $\varphi(e^{P+Q})\le\varphi(e^{P/2}e^Qe^{P/2})$ & $P,Q$ & $\cR$ & \cite[Thm.~2(iv)]{bikchentaev2011mn}\\
16 & $\varphi(e^A)\exp\frac{\varphi(e^{A/2}Be^{A/2})}{\varphi(e^A)}\le\varphi(e^{A+B})$ & $A,B$ self-adj.\ & $\cR$ & \cite[Lemma~1]{bikchentaev2011pb}\\
17 & $\varphi(PQP)\le\varphi(Q)$ & $P,Q$ & $\cR$ & \cite[Thm.~3.4(iv)]{bikchentaev2010smj}\\
18 & $\varphi(|P-Q|)\le\varphi(P+Q)$ & $P,Q$ & $\cR$ & \cite[Thm.~3.4(v)]{bikchentaev2010smj}\\
19 & $\varphi(P+Q)\le\varphi(1+P\wedge Q)$ & $P,Q$ & $\cR$ & \cite[Thm.~3.4(iii)]{bikchentaev2010smj}\\
20 & $\varphi(e^{QPQ})\le\varphi(e^P)$ & $P,Q$ & $\cR$ & \cite[Thm.~2(v)]{bikchentaev2011mn}\\
21 & $0\le\varphi(\Re PQ)$ & $P,Q$ & $\varphi(P)+\varphi(Q)$ & \cite[Thm.~2(vi)]{bikchentaev2011mn}\\
22 & $\varphi(\Re PQ)\le\varphi(|PQ|)$ & $P,Q$ & $\cR$ & \cite[Thm.~2(vii)]{bikchentaev2011mn}\\
23 & $\varphi(P\vee Q)\le\varphi(P)+\varphi(Q)$ & $P,Q$ & $\cR$ & \cite[Thm.~4]{petz1988}\\
24 & $\varphi(A^2)^{1/2}\le\varphi(A)$, $\varphi(1)=n$ & $A$ & $\cR$ & \cite[Thm.~3(vii)]{bikchentaev2020det}\\
25 & $0\le A\le B\Rightarrow\varphi(A^4)\le\varphi(AB^2A)$ & $A\le B$ & $\cR$ & \cite[Thms.~4, 5, $r=1$]{hoa2010}, \cite{hoa2007}\\
26 & $\varphi(|X^*|^q)\le\varphi(|X|^q)$, $q>0$ & $X$ & $\cR$ & $|X^*|=U|X|U^*$\\
27 & $|\varphi(X^k)|\le\varphi(|X|^k)$, $k\ge3$ & $X$ & $\cR$ & \cite[Thm.~1]{petz1988}; cf.\ 4\\
28 & $\varphi(A^2+B^2+K(AB+BA))\ge0$ & $A,B$; $K\in\R$ & threshold $K^*$ & \cite[Thm.~4.6]{bikchentaev2019abed}\\
29 & $\varphi(|AB|)\le\frac1p\varphi(|A|^p)+\frac1q\varphi(|B|^q)$, $\frac1p+\frac1q=1$ & $A,B$ arbitrary & $\cR$ & \cite{cho2009}; cf.\ 8\\
30 & the same & $A,B\ge0$ & $\cR$ & \cite{bikchentaev2005young}\\
31 & the same with $p=q=2$ & $A$ arbitrary, $B=B^*$ & $\cR$ & --\\
32 & $0\le A\le B\Rightarrow\varphi(A^p)\le\varphi(B^p)$ ($p>1$); $A\le B\Rightarrow\varphi(e^A)\le\varphi(e^B)$ & $A\le B$ & $\cR$ & \cite{bikchentaev2007mono,sano2006}; cf.\ 10\\
33 & $\varphi(PQP)=\varphi(QPQ)$; $\varphi(s(PQP))=\varphi(s(QPQ))$ & $P,Q$ & $\cR$ & \cite{bikchentaev2009rm,bikchentaev2011mn}\\
34 & $\varphi((A+B)^p)\le\varphi(A^p)+\varphi(B^p)$, $0<p<1$ & $A,B$ & $\cR$ & \cite{tikhonov2005}; cf.\ 12\\
35 & $\varphi(A)\le\varphi(|A+iB|)$ & $A\ge0$, $B=B^*$ & $\cR$ & \cite{AF21}\\
36 & $\varphi(e^{PQP})\le\varphi(e^Q)$ & $P$, $Q\ge0$ & $\cR$ & cf.\ \cite{Bik12}; cf.\ 20\\
37 & $\varphi(e^A)\exp\frac{\re\varphi(Be^A)}{\varphi(e^A)}\le\varphi(e^{A+B})$ & $A,B$ self-adj. & $\cR$ & \cite{bikchentaev2011pb}; cf.\ 16\\
38 & $\varphi((B^{1/2}AB^{1/2})^{rq})\le\varphi((B^{r/2}A^rB^{r/2})^q)$, $r>1$, $q>0$ & $A,B$ & $\cR$ & \cite{bikchentaev2011pb}; cf.\ 7\\
\midrule
\multicolumn{5}{@{}l}{\emph{Central elements: $a$ positive and invertible, $\varphi,\omega$ normal functionals, $p$ projections}}\\
39 & $|\varphi(a)|\le|\varphi|(a)$ & $\varphi$ & $\cR$ & \cite[Thm.~1]{novikov2015}\\
40 & $pap\le a$ & $p$ & best $\lambda$ in $pap\le\lambda a$ & \cite[Thm.~1]{novikov2015}\\
41 & $\varphi^+(a)\le\varphi_1(a)$ ($\varphi=\varphi_1-\varphi_2$); $\varphi^+(a)\le\omega^+(a)$ ($\varphi\le\omega$); $(\varphi+\omega)^+(a)\le\varphi^+(a)+\omega^+(a)$; $|\varphi+\omega|(a)\le|\varphi|(a)+|\omega|(a)$ & $\varphi,\omega$ & $\cR$ & \cite[Thm.~1, Cor.~1]{novikov2015}\\
42 & $|\varphi|(a)=\|a^{1/2}\varphi a^{1/2}\|$ & $\varphi$ & ratio & \cite{novikov2015,Nov17}\\
\end{longtable}}

\subsection{Remarks on the forms}\label{sec:forms}
Several sources state families of inequalities; we use one member. In item~4, $|\varphi(A^k)|\le\varphi(|A|^k)$ for all
$A$ and one fixed $k\ge1$ characterises the trace on $M_n$~\cite[Theorem~1]{petz1988}; $k=1$ is Gardner's inequality.
Item~7 is the Araki--Lieb--Thirring inequality $\varphi((B^{1/2}AB^{1/2})^{rq})\le\varphi((B^{r/2}A^rB^{r/2})^q)$ (item~38) with
$r=2$ and outer exponent $1$~\cite[Lemma~2]{bikchentaev2011pb}; we call the outer exponent $q$ throughout, as in
item~38, because $p$ denotes the Young exponent of items 29--30. The published form of item~8 is $\varphi(|AB|)\le\frac1p\varphi(A^p)+
\frac1q\varphi(B^q)$ for $A,B\ge0$~\cite{bikchentaev2005young}; we use $p=q=2$ and allow arbitrary $A,B$, in the form of
Ando's matrix Young inequality~\cite{ando1995}, and Corollary~\ref{cor:item8} refers to this extension; it differs from item~29
(Remark~\ref{rem:youngforms}). Item~9 is the
member $p=1$ of the family $\varphi(PQ+QP)\le2\varphi((QPQ)^p)$, $p\in(0,1]$~\cite{bikchentaev2019abed}. Items 10 and 25
are the cases $r=0$ and $r=1$ of $\varphi(A^{2r+2})\le\varphi(A^rB^2A^r)$ for $0\le A\le B$, which characterises traces for
every fixed $r\ge0$~\cite[Lemma~1, Theorems~4,~5]{hoa2010}; on $M_n$ the case $r=0$ is
in~\cite{bikchentaev2007mono,sano2006} and $r>0$ in~\cite[Theorem~6]{hoa2007}. Items 11 and 24 characterise the
canonical trace among the positive functionals with $\varphi(1)=n$~\cite[Theorem~3]{bikchentaev2020det}; item~24 is (vii)
there with $p=2$, $q=1$ and needs $\varphi$ faithful. Item~12 is the case $f(t)=t^{1/2}$ of Tikhonov's
characterisation by $\varphi(f(A+B))\le\varphi(f(A))+\varphi(f(B))$ for concave $f$ with
$f(0)=0$~\cite[Corollary~1]{tikhonov2005}. Item~13 is the generalised Powers--St{\o}rmer inequality
$\varphi(A+B)-\varphi(|A-B|)\le2\varphi(f(A)^{1/2}g(B)f(A)^{1/2})$, $g(t)=t/f(t)$, with $f(t)=t^{1/2}$~\cite[Lemma~3.1,
Remark~4]{hoa2013}; with $\cN=\cR$ its envelope would be infinite (Remark~\ref{rem:13}), so we normalise by $\varphi(A+B)$.
Item~14 is Hölder's inequality for projections with the exponent $\alpha=3$; for $\alpha=2$ it holds for every positive
$\varphi$~\cite[remark after Theorem~2]{bikchentaev2011mn}. Item~16 is the Peierls--Bogoliubov inequality, stated
in~\cite{bikchentaev2011pb} for positive $H,K$; both sides scale by $e^c$ under $A\mapsto A+c$, $B\mapsto B+c$, so the
envelope over self-adjoint arguments is the same. Items 17--19 are~\cite[Theorem~3.4]{bikchentaev2010smj}, announced
in~\cite{bikchentaev2009rm}, and items 14, 15, 20--22 are~\cite[Theorem~2]{bikchentaev2011mn}.
The remaining items are taken as in Table~\ref{tab:items}. In item~2 the four versions with
arbitrary, self-adjoint and positive variables and with projections (item~18) have the same
constant (Theorem~\ref{thm:subadd}); the self-adjoint version is the form in which~\cite{stolyarov2002}
characterises traces. The equality $\varphi(|X^*|)=\varphi(|X|)$ for all $X$ characterises traces,
since $|X^*|=U|X|U^*$ (item~26). Item~28 is~\cite[Thm.~4.6]{bikchentaev2019abed}, in which the
inequality is required for all $K$ larger than some $K_0>-1$; Corollary~\ref{cor:ba} shows that
arbitrarily large $K$ suffice. In item~29 the Cho--Sano form $\varphi(|A^*B|)$ gives the same constant
(Section~\ref{ssec:young}). Bikchentaev's characterisation by $\varphi(e^{PQP})\le\varphi(e^Q)$
concerns projections~\cite{Bik12}; item~20 is this inequality with the roles of $P$ and $Q$
interchanged, and item~36 allows $Q\ge0$ (for self-adjoint $Q$ it fails already for traces,
Section~\ref{ssec:further}). Items 16 and 37 are the two forms of the Peierls--Bogoliubov inequality
of~\cite{bikchentaev2011pb}; for non-tracial $\varphi$ the number $\varphi(Be^A)$ need not be real, and
item~37 uses its real part. Items 39--42 are~\cite[Thm.~1 and Cor.~1]{novikov2015}; the convexity
items (vii) and (viii) of~\cite[Remark~1]{novikov2015} are not treated here.

\subsection{Two projections}
\begin{lemma}[Two projections; Halmos~\cite{halmos1969}]\label{lem:halmos}
For projections $P,Q$ on a finite-dimensional Hilbert space $H$, $H$ is the orthogonal sum of subspaces of dimension
$1$ or $2$ invariant under $P$ and $Q$. On each two-dimensional one there is an orthonormal basis $f,f'$ with $P=ff^*$
and $Q=qq^*$, $q=\cos\theta\,f+e^{i\chi}\sin\theta\,f'$, $0<\theta<\pi/2$.
\end{lemma}

\subsection{Contrasts of the operator variables}\label{sec:twocontrasts}

The envelope $V_n(\delta)$ lets the operator variables range over all operators. To see how the violation depends on
the operators themselves, we also prescribe their contrast. For a positive variable $A$ we use $\OMC(A)$, for an
arbitrary or self-adjoint variable $X$ the contrast $\OMC(|X|)$ of its modulus (for the Peierls--Bogoliubov
inequality, item~16, the variables are positive, as in~\cite{bikchentaev2011pb}), and we let every operator variable of
an item have the same contrast $\delta'\in[0,1]$:
\[
  V_n(\delta,\delta')=\sup\Bigl\{\frac{\cL-\cR}{\cN}:\ \OMC(\rho)=\delta,\ \text{every operator variable of contrast }\delta'\Bigr\}.
\]
Projections other than $0$ and $1$ have contrast $1$, so for items 5, 9, 14, 15 and 17--23 only $\delta'=1$ occurs, and
$V_n(\delta,1)\le V_n(\delta)$ with equality whenever the supremum over all operators is approached by singular ones.
A positive operator of contrast $\delta'$ whose spectrum has only two points is, up to a positive factor,
\[
  A=\ell'+(1-\ell')P,\qquad \ell'=\frac{1-\delta'}{1+\delta'},\qquad P\ \text{a projection},\ P\ne0,1;
\]
we call such operators \emph{two-level}. They are the affine images of projections, and as $\delta'\to1$ they become
scaled projections.

\begin{proposition}\label{prop:scalar}
If every operator variable is a positive multiple of the identity (for arbitrary and self-adjoint variables: if every
modulus is), then $\cL\le\cR$ in every item 1--25 of Table~\ref{tab:items} and for every positive $\varphi$. Hence
$V_n(\delta,0)\le0$ for all items; and $V_n(0,\delta')\le0$, since for $\delta=0$ the functional is a multiple of the trace
(the canonical trace in items 11 and 24, where $\varphi(1)=n$).
\end{proposition}
\begin{proof}
Item by item, with $|X|=x$, $|Y|=y$ scalar and $X=xU$, $Y=yV$ with unitaries (symmetries for self-adjoint $X,Y$), and
$A=\alpha$, $B=\beta$ scalars: (1) $|\varphi(U)|\le\varphi(1)$; (2) $\varphi(|xU+yV|)\le\|xU+yV\|\varphi(1)\le(x+y)\varphi(1)$; (3) and (4)
as (1); (6) and (7) are equalities; (8) $|AB^*|=xy$ and $xy\le\frac12(x^2+y^2)$; (10) and (25) follow from $\alpha\le\beta$; (11) is an
equality for $\varphi(1)=n$; (12) $\sqrt{\alpha+\beta}\le\sqrt\alpha+\sqrt\beta$; (13) $\alpha+\beta-|\alpha-\beta|=2\min(\alpha,\beta)\le2\sqrt{\alpha\beta}$; (16) is an equality;
(24) $\varphi(\alpha^2)^{1/2}=\alpha\sqrt n\le\alpha n$.
\end{proof}

So a violation needs both contrasts: the functional must be far from a trace and the operators must be far from scalars.
The next lemma is the mechanism behind the two-level operators: whenever the violation depends convexly on the spectrum of a
variable, the extremal spectra have only two points.

\begin{lemma}[Vertex principle]\label{lem:vertex}
Fix all variables except the eigenvalues $\lambda\in[\ell',1]^n$ of one positive variable (its eigenbasis fixed), and let
$F(\lambda)$ and $\cN(\lambda)>0$ be the resulting values of $\cL-\cR$ and $\cN$. If $\lambda\mapsto F(\lambda)-c\cN(\lambda)$ is convex for
every $c\in\R$, then the supremum of $F/\cN$ over the spectra with $\max_i\lambda_i=1$ and $\min_i\lambda_i=\ell'$ equals its
supremum over the two-level spectra $\lambda\in\{\ell',1\}^n$ (both values occurring).
\end{lemma}
\begin{proof}
The admissible set is the union of the faces $\{\lambda_i=1,\ \lambda_j=\ell'\}$ of the cube, and the vertices of these faces are
the two-level spectra. If $F/\cN>c$ at an admissible point, then $F-c\cN>0$ there, and by convexity $F-c\cN>0$ at a vertex
of the same face, where therefore $F/\cN>c$.
\end{proof}

\noindent The hypothesis holds for Gardner's inequality ($F-c\cN=|\varphi(UT)|-(1+c)\varphi(T)$ for $X=UT$: the modulus of a
linear function minus a linear one), for item~3, for item~24 ($\varphi(A^2)^{1/2}=\|\rho^{1/2}A\|_2$ is a norm of a linear
function of $\lambda$) and for the unnormalised Gardner violation of Section~\ref{sec:gardner2}; for these items two-level
operators suffice. For the other items, whether they suffice is a question we examine numerically
(Section~\ref{sec:projsuffice}).

\section{Central envelopes, noncentrality and the transfer principle}\label{sec:nc}

\subsection{Central envelopes and noncentrality}\label{ssec:nc}

We collect, with short proofs, the facts on noncentrality that we use. The envelopes in Lemma~\ref{lem:env} are the largest and smallest elements
$C_u,C_l$ of Halpern's central spectrum \cite[Prop.~3.11, 3.12]{Hal72} (taken with respect to the
zero central ideal); we include short proofs of the elementary properties we need.

\begin{lemma}[Central envelopes; cf.\ \cite{Hal72}]\label{lem:env}
Let $B\in\cM_{sa}$. The set $\{z\in\cZ_{sa}:z\ge B\}$ has a least element $\ceil{B}$ and
$\{z\in\cZ_{sa}:z\le B\}$ has a greatest element $\flr{B}$. Moreover:
\begin{enumerate}
\item\label{env:a} $\flr{B}\le B\le\ceil{B}$, $\flr{B}=-\ceil{-B}$, $\ceil{B+z}=\ceil{B}+z$ for
$z\in\cZ_{sa}$, $\ceil{zB}=z\ceil{B}$ for $z\in\cZ_{++}$, and $\ceil{UBU^*}=\ceil{B}$ for unitaries
$U\in\cM$;
\item\label{env:b} if $B\in\cM_{++}$, then $\ceil{B},\flr{B}\in\cZ_{++}$, $\ceil{B^s}=\ceil{B}^s$
and $\flr{B^s}=\flr{B}^s$ for $s>0$, and $\flr{B}=\ceil{B^{-1}}^{-1}$;
\item\label{env:c} if $\cM$ is a factor, then $\ceil{B}=\max\sigma(B)$ and
$\flr{B}=\min\sigma(B)$; if $\cM$ is abelian, then $\ceil{B}=\flr{B}=B$; for a projection $e$,
$\ceil{e}=c(e)$.
\end{enumerate}
\end{lemma}

\begin{proof}
The set $S=\{z\in\cZ_{sa}:z\ge B\}$ is nonempty, bounded below and directed downwards: for
$z_1,z_2\in S$ and the central projection $q=\one_{(-\infty,0]}(z_1-z_2)$, the element
$z_1q+z_2(1-q)$ lies in $S$, since $q$ commutes with $B$. A decreasing net in $\cZ_{sa}$ converges
strongly to its infimum, which is central and still dominates $B$. The statements in
\ref{env:a} follow from the definition; for instance $w\ge zB$ iff $z^{-1}w\ge B$ for
$z\in\cZ_{++}$. For \ref{env:b}: $B\ge\|B^{-1}\|^{-1}$ gives $\flr{B}\ge\|B^{-1}\|^{-1}$; for
central $w>0$, $w$ and $B$ commute, so $w\ge B$ iff $w^s\ge B^s$, and $w\le B$ iff
$w^{-1}\ge B^{-1}$. For \ref{env:c}, only the last claim needs comment: if $z\in\cZ$, $z\ge e$, and
$q=\one_{(-\infty,1-t)}(z)$ with $t>0$, then $eq\le(1-t)q$ forces the projection $eq$
to vanish, so $q\le1-c(e)$ and $z\ge(1-t)c(e)$ (note $z\ge e\ge0$); hence $z\ge c(e)$, and
since $c(e)\ge e$ is central, $\ceil e=c(e)$.
\end{proof}

\begin{lemma}[Local attainment]\label{lem:local}
Let $B\in\cM_{sa}$, $\omega_0\in\Omega$, $\varepsilon>0$, $\alpha=\ceil{B}(\omega_0)$ and
$\beta=\flr{B}(\omega_0)$. Let $P$ be a central projection whose clopen set contains $\omega_0$
and on which $\ceil{B}$ and $\flr{B}$ oscillate by less than $\varepsilon$ (such $P$ exists).
Then the spectral projections $E=\one_{[\alpha-2\varepsilon,\infty)}(B)P$ and
$F=\one_{(-\infty,\beta+2\varepsilon]}(B)P$ satisfy $c(E)=c(F)=P$.
\end{lemma}

\begin{proof}
$E$ commutes with $\cZ$, so $c(E)\le P$. If $Q\le P$ is a nonzero central projection with
$EQ=0$, then $BQ\le(\alpha-2\varepsilon)Q$, so $w=\ceil{B}(1-Q)+(\alpha-2\varepsilon)Q$ is central
and $w\ge B$, hence $w\ge\ceil{B}$; on $Q$ this gives $\alpha-2\varepsilon\ge\ceil B>\alpha-\varepsilon$, a
contradiction. Apply this to $-B$ for $F$. $P$ exists by continuity of Gelfand transforms.
\end{proof}

\begin{corollary}\label{cor:norm}
For $B\in\cM_{sa}$, $\max\sigma(B)=\max_\Omega\ceil{B}$ and $\min\sigma(B)=\min_\Omega\flr{B}$.
\end{corollary}

\begin{proof}
$B\le\ceil{B}\le\max\ceil{B}$; conversely, at a maximum point of $\ceil{B}$ the projection $E$ of
Lemma~\ref{lem:local} is nonzero for every $\varepsilon$. Apply this to $-B$ as well.
\end{proof}

\begin{definition}\label{def:nc}
For $B\in\cM_{++}$ let $u=\ceil{B}$, $l=\flr{B}$ and
\[
\kappa_\cZ(B)=\|ul^{-1}\|=\max_\Omega\frac ul,\qquad \LNC(B)=\tfrac12\ln\kappa_\cZ(B),\qquad
\NC(B)=\tanh\LNC(B)=\frac{\kappa_\cZ(B)-1}{\kappa_\cZ(B)+1}.
\]
\end{definition}

By Corollary~\ref{cor:norm}, $\kappa(B)=\max_\Omega u/\min_\Omega l$, whereas
$\kappa_\cZ(B)=\max_\Omega(u/l)$.

\begin{proposition}\label{prop:char}
Let $B\in\cM_{++}$, $u=\ceil{B}$, $l=\flr{B}$ and $z_0=(ul)^{-1/2}$.
\begin{enumerate}
\item\label{ch:resc} $\kappa_\cZ(B)=\min\{\kappa(zB):z\in\cZ_{++}\}$, attained at $z=z_0$.
\item\label{ch:thomp} $\LNC(B)=\min\{d_T(B,w):w\in\cZ_{++}\}$, attained at $w=(ul)^{1/2}$.
\item\label{ch:mich} $\NC(B)=\min\{\|1-zB\|:z\in\cZ_+\}$, attained at $z=2(u+l)^{-1}$.
\item\label{ch:prop} $\kappa_\cZ(B)\ge1$ with equality iff $B\in\cZ$;
$\kappa_\cZ(zB)=\kappa_\cZ(B^{-1})=\kappa_\cZ(UBU^*)=\kappa_\cZ(B)$ and
$\kappa_\cZ(B^s)=\kappa_\cZ(B)^s$ for $z\in\cZ_{++}$, unitaries $U\in\cM$ and $s>0$;
$\kappa_\cZ(B)\le\kappa(B)$, with equality if $\cM$ is a factor; $\kappa_\cZ\equiv1$ if $\cM$ is
abelian; and $\kappa_\cZ(\bigoplus_iB_i)=\sup_i\kappa_\cZ(B_i)$ in $\bigoplus_i\cM_i$.
\end{enumerate}
\end{proposition}

\begin{proof}
\ref{ch:resc} By Lemma~\ref{lem:env} and Corollary~\ref{cor:norm}, $\|zB\|=\max zu$ and
$\|(zB)^{-1}\|=\max(zl)^{-1}$, and $\max(zu)\cdot\max(zl)^{-1}\ge\max(u/l)$, with equality for
$z=z_0$, where $z_0u=(z_0l)^{-1}=(u/l)^{1/2}$.
\ref{ch:thomp} For $w\in\cZ_{++}$, $B\le tw$ iff $u\le tw$, and $w\le sB$ iff $w\le sl$, so
$M(B/w)=\|uw^{-1}\|$ and $M(w/B)=\|wl^{-1}\|$; the larger of the two is at least
$\|ul^{-1}\|^{1/2}$, with equality for $w=(ul)^{1/2}$.
\ref{ch:mich} For $z\in\cZ_+$, $zl\le zB\le zu$ gives
$\|1-zB\|\le\max_\Omega\max\{|1-zu|,|1-zl|\}$. Conversely, if $|1-zu|$ is maximal at $\omega_0$
and $P,E$ are as in Lemma~\ref{lem:local} with $z$ also oscillating by less than $\varepsilon$ on
$P$, then for a unit vector $\xi\in E\cH$ we have $\langle B\xi,\xi\rangle\in[\alpha-2\varepsilon,\alpha+\varepsilon]$
and $\langle zB\xi,\xi\rangle=z(\omega_0)\langle B\xi,\xi\rangle+O(\varepsilon)$, so
$\|1-zB\|\ge|1-z(\omega_0)u(\omega_0)|-O(\varepsilon)$; similarly for $l$. Hence
$\|1-zB\|=\max_\Omega\max\{|1-zu|,|1-zl|\}$, and pointwise
$\max\{|1-tu|,|1-tl|\}\ge(u-l)/(u+l)$ with equality at $t=2/(u+l)$.
\ref{ch:prop} If $u=l$, then $B=u\in\cZ$. The invariances follow from Lemma~\ref{lem:env}
(for $B^{-1}$: $\ceil{B^{-1}}=l^{-1}$, $\flr{B^{-1}}=u^{-1}$). The rest follows from
$\max(u/l)\le\max u/\min l$, Lemma~\ref{lem:env}\ref{env:c}, and the fact that envelopes are
computed summandwise.
\end{proof}

\begin{remark}[Direct sums]\label{rem:blocks}
Let $\cM=\bigoplus_{c=1}^CM_{k_c}$ and $\rho=\bigoplus_c\rho_c\in\cM_{++}$. By
Proposition~\ref{prop:char}\ref{ch:prop}, $\kappa_\cZ(\rho)=\max_c\kappa(\rho_c)$, hence
\[
\delta_\cZ=\NC(\rho)=\max_c\OMC(\rho_c),\qquad G_\cZ(\rho)=\max_cG(\rho_c):
\]
on a finite-dimensional algebra the noncentrality is the largest contrast \emph{inside} the blocks,
whereas $\OMC(\rho)$ also sees the differences between the blocks. Corollary~\ref{cor:sectors} uses
this.
\end{remark}

\begin{remark}\label{rem:halpern}
By \cite[Prop.~3.11, 3.12]{Hal72}, $\ceil{B}(\zeta)$ and $\flr{B}(\zeta)$ are the maximum and the
minimum of the spectrum of the image of $B$ in the quotient of $\cM$ by the Glimm ideal of
$\zeta\in\Omega$. Hence $\LNC(B)$ is the supremum of the contrasts of these images; for
$\cM=L^\infty(X)\bar\otimes M_n(\C)$ it is $\operatorname{ess\,sup}_x\OLC(B(x))$.
\end{remark}

\begin{remark}\label{rem:ann}
In \cite[\S4]{abed2024} the quantity
$\Delta_\cM(x)=\inf\{\|z-x/\lambda\|:\lambda>0,\ z\in\cZ,\ \|z\|=1\}$ was proposed as a central
version of the contrast. With $u=\ceil x$, $l=\flr x$, the choice $\lambda=\frac12\|u+l\|$,
$z=(u+l)/\|u+l\|$ gives $\Delta_\cM(x)\le\|u-l\|/\|u+l\|\le\NC(x)$. The inequality can be strict,
and $\Delta_\cM$ is not invariant under central rescaling: in $M_2(\C)\oplus M_2(\C)$,
$x=\diag(1,4)\oplus100\cdot1$ has $\NC(x)=3/5$ and $\Delta_\cM(x)\le0.015$ (take $\lambda=100$,
$z=(0.025,1)$), while $\Delta_\cM(\diag(1,4)\oplus1)=3/5$.
\end{remark}

\begin{remark}\label{rem:ad}
For $A$ invertible, the norm of the inner automorphism $\Ad_A(X)=AXA^{-1}$ of $\cM$ equals
$\kappa_\cZ(|A|)$. This is known: it follows from the description of norms of elementary operators
$X\mapsto aXb$ on von Neumann algebras through Glimm quotients \cite{CS93,Som98,Tim07},
combined with Remark~\ref{rem:halpern}. It also follows in two lines from our tools. Writing
$A=U|A|$ with $U$ unitary, $\Ad_A=\Ad_U\circ\Ad_{|A|}$, so we may assume $A=B\in\cM_{++}\setminus\cZ$.
Since $\Ad_{zB}=\Ad_B$ for $z\in\cZ_{++}$, rescaling by $z_0$ (Proposition~\ref{prop:char}\ref{ch:resc})
gives $\|\Ad_B\|\le\|z_0B\|\,\|(z_0B)^{-1}\|=\kappa_\cZ(B)$; conversely, for the corner partial
isometry $v$ and a unit vector $\xi\in v^*v\cH$ as in Lemma~\ref{lem:corner}\ref{c:a},\ref{c:b}
with $a=B$, the unit vector $h=B\xi/\|B\xi\|$ gives
$\|\Ad_B(v)\|\ge\|Bv\xi\|/\|B\xi\|\ge(\alpha-2\varepsilon)/(\beta+2\varepsilon)\to\kappa_\cZ(B)$.
We do not use this. It is the multiplicative counterpart of
the formula $\|D_T\|=2\operatorname{dist}(T,\cZ)$ for the norm of an inner derivation
$D_T=[T,\cdot\,]$, which goes back to Stampfli \cite{Sta70} for $B(\cH)$, to Kadison, Lance
and Ringrose \cite{KLR67} for self-adjoint $T$, to Gajendragadkar \cite{Gaj72}, and to Zsid\'o
\cite{Zsi73} without separability assumptions. For self-adjoint $B$ the additive quantity is
$\|\ceil B-\flr B\|$ and the multiplicative one is $\|\ceil B\flr B^{-1}\|$; only the latter is
invariant under central rescaling.
\end{remark}

\begin{remark}\label{rem:upgrade}
The rescaling argument behind Proposition~\ref{prop:char}\ref{ch:resc} upgrades every sufficient
condition in terms of contrasts for a property invariant under central rescaling to the same
condition in terms of noncentralities. For example, the classical criterion
$\OMC(X)^2+\OMC(Y)^2<1\Rightarrow XY+YX>0$ (see Gustafson \cite{gustafson1968} and Theorem~\ref{thm:jordan})
holds with $\NC$ in place of $\OMC$, since $XY+YX>0$ is invariant under $X\mapsto zX$, $z\in\cZ_{++}$.
We do not pursue such applications here.
\end{remark}

\subsection{The transfer principle}\label{sec:transfer}

\subsubsection*{Violation constants}
We fix the class of quantities to which the transfer principle applies. An \emph{admissible
expression} in noncommuting variables $X_1,\dots,X_r$ is built from the variables by sums, scalar
multiples, products, adjoints, and Borel functions $g:\R\to\C$ with $g(0)=0$ applied to
self-adjoint expressions. Examples are $|X_1|^p$, $X_1^k$, $|X_1X_2|$, $s(X_1)$ for $X_1\ge0$, and
$X_1\vee X_2=s(X_1+X_2)$ for projections.

\begin{definition}\label{def:viol}
A \emph{violation constant} is given by a test class and two functionals. The test class assigns
to every finite von Neumann algebra $\cM$ the set $\cT(\cM)\subset\cM^r$ of tuples satisfying
finitely many conditions of the form $X_i\in\cM_{sa}$, $X_i\in\cM_+$, $X_i$ is a projection,
$X_i\le X_j$. The functionals are $N_\varphi(X)=\Psi_N(\varphi(F_1(X)),\dots,\varphi(F_s(X)))$ and
$D_\varphi(X)=\Psi_D(\varphi(F_1(X)),\dots,\varphi(F_s(X)))$, where the $F_j$ are admissible
expressions and $\Psi_N:\C^s\to\R$, $\Psi_D:\C^s\to[0,\infty)$ are continuous and positively
homogeneous of the same degree; for instance, for the constant $\gamma$ of Theorem~\ref{thm:ba}
one takes $\Psi_N(z)=-\re z_1$ and $\Psi_D(z)=(|z_2||z_3|)^{1/2}$ with $F_1=X_1X_2$, $F_2=X_1^2$,
$F_3=X_2^2$. For a faithful normal positive functional $\varphi$ on $\cM$,
\[
c_\cM(\varphi)=\sup\{N_\varphi(X)/D_\varphi(X):X\in\cT(\cM),\ D_\varphi(X)>0\},
\]
and $c_\cM(\rho)=c_\cM(\varphi_\rho)$ for $\rho\in\cM_{++}$. By homogeneity
$c_\cM(s\varphi)=c_\cM(\varphi)$ for $s>0$, so the normalisation of $\tau$ is irrelevant; on
$M_n(\C)$ we use $\Tr$.
\end{definition}

The constant $m_{\exp}$ of Section~\ref{ssec:mono} and the exponential constants of
Section~\ref{ssec:further} fall outside this definition, since $\exp(0)\ne0$, and so does the
spectral geometric mean (item~6), whose variables are invertible; they are treated
directly. Condition (R) below and Lemma~\ref{lem:resc} apply verbatim to these suprema, and the
lower bounds of Theorem~\ref{thm:transfer}\ref{t:low} are obtained there by a direct modification
of its proof.

The constant satisfies the \emph{rescaling condition} (R) if for every $z\in\cZ_{++}$ there is a
bijection $\sigma_z$ of $\cT(\cM)$ with $N_{z\rho}(X)=N_\rho(\sigma_zX)$ and
$D_{z\rho}(X)=D_\rho(\sigma_zX)$ for all $\rho$ and $X$.

\begin{lemma}[Central rescaling]\label{lem:resc}
If (R) holds, then $c_\cM(z\rho)=c_\cM(\rho)$ for all $z\in\cZ_{++}$. In particular
$c_\cM(\rho)=c_\cM(z_0\rho)$, where $z_0=(\ceil\rho\flr\rho)^{-1/2}$ and
$\kappa(z_0\rho)=\kappa_\cZ(\rho)$.
\end{lemma}

\begin{proof}
The first claim is immediate from (R); the second is Proposition~\ref{prop:char}\ref{ch:resc}.
\end{proof}

\subsubsection*{The corner construction}

\begin{lemma}\label{lem:hom}
Let $\pi:M_n(\C)\to\cM$ be an injective $*$-homomorphism and $e=\pi(1)$. Then
$g(\pi(x))=\pi(g(x))$ for every self-adjoint $x\in M_n(\C)$ and every Borel function $g$ with
$g(0)=0$; consequently
$F(\pi(X_1),\dots,\pi(X_r))=\pi(F(X_1,\dots,X_r))$ for every admissible expression $F$. Moreover
$\pi(x)\ge0$ iff $x\ge0$, $\pi(x)$ is a projection iff $x$ is, and $\varphi\circ\pi$ is a
faithful positive functional on $M_n(\C)$ for every faithful positive $\varphi$ on $\cM$.
\end{lemma}

\begin{proof}
Write $x=\sum_\lambda\lambda q_\lambda$ with spectral projections $q_\lambda$. The
$\pi(q_\lambda)$ are orthogonal projections with sum $e$, so $\pi(x)=\sum\lambda\pi(q_\lambda)+0\cdot(1-e)$
and $g(\pi(x))=\sum g(\lambda)\pi(q_\lambda)=\pi(g(x))$. The remaining claims follow from this and
the injectivity of $\pi$.
\end{proof}

\begin{lemma}[Corner construction]\label{lem:corner}
Let $a\in\cM_{++}\setminus\cZ$, let $\omega_0\in\Omega$ be a point where $\ceil a/\flr a$ attains
its maximum $\kappa_\cZ(a)$, $\alpha=\ceil a(\omega_0)$, $\beta=\flr a(\omega_0)$, so that
$\alpha/\beta=\kappa_\cZ(a)$, and let $0<\varepsilon<(\alpha-\beta)/4$,
$I_1=[\alpha-2\varepsilon,\alpha+\varepsilon]$, $I_2=[\beta-\varepsilon,\beta+2\varepsilon]$.
\begin{enumerate}
\item\label{c:a} There is a nonzero partial isometry $v\in\cM$ with $f_1:=vv^*\le\one_{I_1}(a)$
and $f_2:=v^*v\le\one_{I_2}(a)$. Hence $f_1f_2=0$, $f_1g(a)f_2=0$ for every bounded Borel $g$, and
$(\inf_{I_i}g)f_i\le f_ig(a)f_i\le(\sup_{I_i}g)f_i$ for real $g$.
\item\label{c:b} $e_{11}=f_1$, $e_{12}=v$, $e_{21}=v^*$, $e_{22}=f_2$ are matrix units, and
$\pi(w)=\sum_{i,j}w_{ij}e_{ij}$ is an injective $*$-homomorphism $M_2(\C)\to\cM$. For a unit
vector $\xi\in f_2\cH$ the vectors $\eta_1=v\xi$, $\eta_2=\xi$ are orthonormal,
$\pi(w)\eta_j=\sum_iw_{ij}\eta_i$, $\langle g(a)\eta_1,\eta_2\rangle=0$ and
$\langle g(a)\eta_i,\eta_i\rangle\in[\inf_{I_i}g,\sup_{I_i}g]$.
\item\label{c:c} If $\cM$ is finite, then $\tau(f_1)=\tau(f_2)=:t>0$ and
$\tau(g(a)\pi(w))=t\,(r_1w_{11}+r_2w_{22})$ with $r_i=\tau(g(a)f_i)/t\in[\inf_{I_i}g,\sup_{I_i}g]$.
\item\label{c:d} If $\cM$ is a II$_1$ factor, $\lambda_1,\dots,\lambda_n\in\sigma(a)$ and
$\varepsilon>0$, then there are matrix units $\{e_{ij}\}_{i,j=1}^n\subset\cM$ with
$e_{ii}\le\one_{[\lambda_i-\varepsilon,\lambda_i+\varepsilon]}(a)$, $\tau(e_{ii})=t>0$ and
$|\tau(ae_{ij})-\delta_{ij}\lambda_it|\le\varepsilon t$.
\end{enumerate}
\end{lemma}

\begin{proof}
\ref{c:a} Choose $P$ as in Lemma~\ref{lem:local} for $B=a$. Since $aP\le\ceil aP\le(\alpha+\varepsilon)P$
and $aP\ge\flr aP\ge(\beta-\varepsilon)P$, the projections $E,F$ of that lemma satisfy
$E\le\one_{I_1}(a)$ and $F\le\one_{I_2}(a)$; they are orthogonal because $I_1\cap I_2=\emptyset$,
and $c(E)c(F)=P\ne0$. The comparison fact gives $v$. Since $E,F$ commute with $a$,
$f_1g(a)f_2=f_1g(a)EFf_2=0$, and $f_ig(a)f_i=f_i(g(a)\one_{I_i}(a))f_i$ gives the bounds.
\ref{c:b} is a direct computation: $v^*v\xi=\xi$, $e_{ij}=e_{ij}f_j$, and $\langle
g(a)\eta_1,\eta_2\rangle=\langle f_2g(a)f_1\eta_1,\eta_2\rangle=0$.
\ref{c:c} $\tau(vv^*)=\tau(v^*v)$; $\tau(g(a)e_{12})=\tau(g(a)f_1vf_2)=\tau(f_2g(a)f_1v)=0$, and
likewise for $e_{21}$.
\ref{c:d} Shrinking $\varepsilon$ we may assume that the windows of distinct values among the
$\lambda_i$ are disjoint. For each distinct value $\mu$ of multiplicity $n_\mu$ the spectral
projection $E_\mu=\one_{[\mu-\varepsilon,\mu+\varepsilon]}(a)$ is nonzero. Let
$t=\min_\mu\tau(E_\mu)/n_\mu$. In a II$_1$ factor we can choose orthogonal subprojections of the
$E_\mu$ of trace $t$, $n_\mu$ of them under $E_\mu$; projections of equal trace are equivalent,
which gives matrix units with $e_{ii}\le E_{\lambda_i}$. Then
$\tau(ae_{ij})=\tau(aE_{\lambda_i}e_{ij})=\lambda_i\tau(e_{ij})+\tau((a-\lambda_i)E_{\lambda_i}e_{ij})$,
$\tau(e_{ij})=\delta_{ij}t$ and $|\tau(Ye_{ij})|\le\|Y\|\tau(|e_{ij}|)=\|Y\|t$.
\end{proof}

\subsubsection*{The transfer theorem}

\begin{theorem}\label{thm:transfer}
Let $c$ be a violation constant.
\begin{enumerate}
\item\label{t:low} For every finite $\cM$ and $\rho\in\cM_{++}\setminus\cZ$,
$c_\cM(\rho)\ge c_{M_2}(\diag(\kappa_\cZ(\rho),1))$. If $\cM$ is a II$_1$ factor, then
$c_\cM(\rho)\ge c_{M_n}(\diag(\lambda_1,\dots,\lambda_n))$ for all $\lambda_1,\dots,\lambda_n\in\sigma(\rho)$.
\item\label{t:up} Suppose (R) holds and $f:[1,\infty)\to[0,\infty]$ satisfies
$c_\cM(\rho)\le f(\kappa(\rho))$ for all finite $\cM$ and all $\rho\in\cM_{++}$. Then
$c_\cM(\rho)\le f(\kappa_\cZ(\rho))$.
\item\label{t:eq} If, in addition, $c_{M_2}(\diag(\kappa,1))\ge f(\kappa)$ for every $\kappa>1$, and
$c_\cM(\rho)\ge f(1)$ for central $\rho$, then $c_\cM(\rho)=f(\kappa_\cZ(\rho))$ for every
finite $\cM$ and every $\rho\in\cM_{++}$. In particular $c=f(\kappa)$ in every finite factor.
\end{enumerate}
\end{theorem}

\begin{proof}
\ref{t:low} Let $\pi$ be the homomorphism of Lemma~\ref{lem:corner}\ref{c:b} for $a=\rho$. By
Lemma~\ref{lem:corner}\ref{c:c} with $g(t)=t$, $\varphi\circ\pi=t\Tr(R_\varepsilon\,\cdot)$ with
$R_\varepsilon=\diag(r_1,r_2)\to\Lambda=\diag(\alpha,\beta)$ as $\varepsilon\to0$. By
Lemma~\ref{lem:hom}, $\pi(\cT(M_2))\subset\cT(\cM)$, $N_\varphi(\pi X)=N_{\varphi\circ\pi}(X)$ and
$D_\varphi(\pi X)=D_{\varphi\circ\pi}(X)$. Hence, for every $X\in\cT(M_2)$ with
$D_{\Lambda}(X)>0$ and all small $\varepsilon$,
$c_\cM(\rho)\ge N_{R_\varepsilon}(X)/D_{R_\varepsilon}(X)$, and the right side tends to
$N_\Lambda(X)/D_\Lambda(X)$ because $N$ and $D$ depend continuously on the functional. Thus
$c_\cM(\rho)\ge c_{M_2}(\Lambda)=c_{M_2}(\diag(\kappa_\cZ(\rho),1))$ by homogeneity. In a II$_1$
factor use Lemma~\ref{lem:corner}\ref{c:d}: $\varphi\circ\pi=t\Tr(R_\varepsilon\,\cdot)$ with
$(R_\varepsilon)_{ji}=\tau(\rho e_{ij})/t\to\delta_{ij}\lambda_i$.
\ref{t:up} By Lemma~\ref{lem:resc}, $c_\cM(\rho)=c_\cM(z_0\rho)\le f(\kappa(z_0\rho))=f(\kappa_\cZ(\rho))$.
\ref{t:eq} combines \ref{t:low} and \ref{t:up}.
\end{proof}

\begin{remark}\label{rem:attain}
If $\cM=M_n(\C)$ and $e_{\max},e_{\min}$ are unit eigenvectors of $\rho$ for $\max\sigma(\rho)$ and
$\min\sigma(\rho)$, the homomorphism $\pi$ defined by these eigenvectors satisfies
$\varphi\circ\pi=\Tr(\diag(\max\sigma(\rho),\min\sigma(\rho))\,\cdot)$ exactly. Hence every
supremum in Section~\ref{sec:catalogue} that is attained in $M_2(\C)$ is attained in $M_n(\C)$.
\end{remark}

In the applications all conditions are checked as follows: the constant is a violation constant in
the sense of Definition~\ref{def:viol} by inspection; (R) is verified by exhibiting $\sigma_z$;
the upper bound $f(\kappa)$ is proved in an arbitrary finite von Neumann algebra; and an example
in $M_2(\C)$ gives the lower bound. For central $\rho$ the value $1=f(1)$ is attained at a
trivial tuple (e.g.\ $X=1$).

\section{The catalogue of violation constants}\label{sec:catalogue}

This section determines the violation constants item by item. Where a constant is known in finite
von Neumann algebras we prove it there, with one general upper bound and one $2\times2$ example
transferred by Theorem~\ref{thm:transfer}; the matrix envelopes follow by
Section~\ref{sec:envelopes}. With the operator variables free, the matrix envelopes are as follows,
where $a\ge b>0$ are the extreme eigenvalues of $\rho$ and
\[
  \kappa=\frac ab=\frac{1+\delta}{1-\delta},\qquad G=\frac{a+b}{2\sqrt{ab}}=\frac1{\sqrt{1-\delta^2}}=\sec\psi,\qquad
  \psi=\arcsin\delta .
\]

\begin{thmB}
Let $n\ge2$ and $\delta=\OMC(\rho)$ for densities $\rho$ in $M_n$.
\begin{enumerate}
\item[(a)] $V_n=G-1$ for Gardner's inequality, the triangle inequalities for $\varphi(|\cdot|)$ on operators and on
projections, $\varphi(A+A^*)\le2\varphi(|A^*|)$, $|\varphi(A^2)|\le\varphi(A^*A)$, $\varphi(|P-Q|)\le\varphi(P+Q)$ and
$\varphi(\Re PQ)\le\varphi(|PQ|)$ (items 1--5, 18, 22), in every dimension
(Theorems~\ref{thm:gardner}, \ref{thm:subadd}, \ref{thm:pz2}, Corollaries~\ref{cor:item3}, \ref{cor:item5},
Proposition~\ref{prop:item22}).
\item[(b)] $V_n=G^2-1$ for $\varphi(PQP)\le\varphi(Q)$; $V_n=\frac12\tan^2(\psi/2)$ for $\varphi(\Re PQ)\ge0$;
$V_n=\delta/(1-\delta)$ for $\varphi(P\vee Q)\le\varphi(P)+\varphi(Q)$; $V_n=\delta$ for
$\varphi(P+Q)\le\varphi(1+P\wedge Q)$ ($n$ even); $V_n$ is the maximum of an explicit function of one
variable, $\frac{(e-1)^2}{e(e+1)}\delta^2+O(\delta^3)$, for $\varphi(e^{QPQ})\le\varphi(e^P)$ ($n$ even); and
$V_n=(G-1)^2/(4G)=\delta^4/16+O(\delta^6)$ for the monotonicity $0\le A\le B\Rightarrow\varphi(A^2)\le\varphi(B^2)$
(Theorems~\ref{thm:proj}, \ref{thm:projP}, \ref{thm:mono2} and Proposition~\ref{prop:item21}).
\item[(c)] $V_n=\kappa-1$ for the extension of Young's inequality to arbitrary operators; $V_2=s/(1+s)$, $s=\sin(\psi/2)$, for
$\varphi(PQ+QP)\le2\varphi(QPQ)$; for the inequalities that characterise the \emph{canonical} trace, $V_2=1-\sqrt{1-\delta^2}$
for the inequality $(\det A)^{1/n}\le\frac1n\varphi(A)$, increasing with $n$ to $1-1/S(\kappa)$ with the Specht ratio $S$,
and $V_n=\sqrt{(\ell+n-1)/(n\ell)}-1$, $\ell=1/\kappa$, for $\varphi(A^2)^{1/2}\le\varphi(A)$
(Corollary~\ref{cor:item8}, Propositions~\ref{prop:abed}, \ref{prop:det} and~\ref{prop:power}).
\item[(d)] $V_n\ge G-1$ for the spectral geometric mean inequality, $V_n\ge G^2-1$ for the
Araki--Lieb--Thirring inequality with $(r,q)=(2,1)$, and $V_n\ge\delta^8/4096\,(1+O(\delta^2))$ for the
weighted monotonicity with $r=1$ (Propositions~\ref{prop:sgm}, \ref{prop:altlow} and~\ref{prop:wmon}); and
$V_n\le\kappa-1$ for items 1--8, 10 and~25: for items 1--5 by (a), since $G\le\kappa$, for item~8 by (c)
(Corollary~\ref{cor:item8}), for item~10 by (b) (Theorem~\ref{thm:mono2}), and for items 6, 7 and~25 by
Propositions~\ref{prop:sgm}, \ref{prop:alt} and~\ref{prop:wmon}.
\end{enumerate}
\end{thmB}

\begin{figure}[t]
\centering\includegraphics[width=0.72\textwidth]{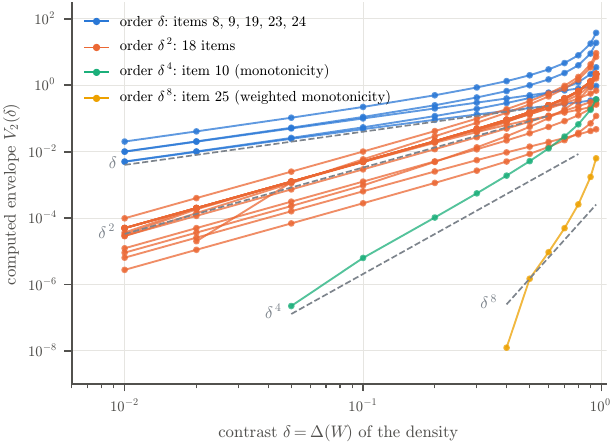}
\caption{The ladder of orders: the computed envelopes $V_2(\delta)$ of items 1--25 (certified lower bounds for the
suprema, Section~\ref{sec:numerics}) on logarithmic axes, coloured by the order of vanishing at the trace, with guide
lines of slopes $1$, $2$, $4$ and~$8$. The horizontal axis is the contrast $\delta=\OMC(\rho)$ of the
density, written $\Delta(W)$ in the figure. Item~25 is shown for $\delta\ge0.4$ (the search resolves it for
$\delta\ge\TCWmonFirst$).}
\label{fig:ladder}
\end{figure}

The proved envelopes thus vanish at the trace to different orders, $\delta$, $\delta^2$ and $\delta^4$;
for item~25 the lower bound is of order $\delta^8$, while only the trivial upper bound
$\kappa-1=2\delta/(1-\delta)$ is known (Proposition~\ref{prop:wmon}), and the computed envelope is consistent with $g_{25}$ for $\delta\ge0.7$ (Figure~\ref{fig:ladder} and
Table~\ref{tab:classes}). Two inequalities that are equivalent for traces may fail at
different rates. For instance, $\varphi(P\vee Q)+\varphi(P\wedge Q)=\varphi(P)+\varphi(Q)$ for every trace, yet the lattice
subadditivity fails as $\delta/(1-\delta)$, while $\varphi(|P-Q|)\le\varphi(P+Q)$, with the same right-hand side, fails only as
$G-1\sim\delta^2/2$. The orders refer to the normalisations $\cN$ of Table~\ref{tab:items}; a normalisation that
vanishes at the extremisers can change them (Remark~\ref{rem:13}).

Throughout this section, except where $M_n$ is named, $\cM$ is a von Neumann algebra with a faithful normal tracial state
$\tau$, $\rho\in\cM_{++}$, $\varphi=\tau(\rho\,\cdot)$, $\kappa=\kappa(\rho)$,
$\kappa_\cZ=\kappa_\cZ(\rho)$, $u=\ceil\rho$, $l=\flr\rho$ and $G_\cZ=\cosh\LNC(\rho)$. All suprema
are over tuples for which the denominator is nonzero. We repeatedly use that, for $Y\in\cM_+$,
\begin{equation}\label{eq:sandwich}
\varphi(Y)\le\tau(uY)\quad\text{and}\quad\tau(lY)\le\varphi(Y),\qquad\text{and}\quad u\le\kappa_\cZ l,
\end{equation}
since $\tau(\rho Y)=\tau(Y^{1/2}\rho Y^{1/2})$, $l\le\rho\le u$, and $u/l\le\kappa_\cZ$ on $\Omega$.

\subsection{Gardner's inequality}

\begin{theorem}[Item 1]\label{thm:gardner}
$\displaystyle\sup_{X\ne0}\frac{|\varphi(X)|}{\varphi(|X|)}=G_\cZ=\frac{1+\kappa_\cZ}{2\sqrt{\kappa_\cZ}}$.
In $M_n(\C)$ the supremum is $G(\rho)=\cosh\OLC(\rho)=\sec\psi(\rho)$ and is attained; for $M_n(\C)$
cf.\ Bourin \cite[Cor.~1.4, Rem.~1.9]{Bou06} and Remark~\ref{rem:bourin}.
\end{theorem}

\begin{proof}
\emph{Upper bound with $\kappa$.} Let $X=U|X|$ and $h=|X|^{1/2}$. Then
$\varphi(X)=\tau(h\rho Uh)=\langle Uh,\rho h\rangle_{L^2}$, so by Lemma~\ref{lem:krein} for
$L_\rho$, whose spectrum lies in $\sigma(\rho)$,
$|\varphi(X)|\le\|h\|\,\|L_\rho h\|\le\cosh\OLC(\rho)\langle L_\rho h,h\rangle=\cosh\OLC(\rho)\varphi(|X|)$.
\emph{(R).} $\sigma_z(X)=zX$: $\tau(z\rho X)=\tau(\rho zX)$ and $|zX|=z|X|$.
\emph{Example.} In $M_2$ with $\rho=\diag(\kappa,1)$ let $b=(\kappa+1)^{-1/2}(1,\sqrt\kappa)$,
$a=\rho b/\|\rho b\|$ and $X=|a\rangle\langle b|$. Then $|X|=|b\rangle\langle b|$,
$\varphi(|X|)=\langle\rho b,b\rangle=2\kappa/(\kappa+1)$ and
$\varphi(X)=\langle a,\rho b\rangle=\|\rho b\|=\sqrt\kappa$, so the ratio is
$(1+\kappa)/(2\sqrt\kappa)$. Theorem~\ref{thm:transfer} and Remark~\ref{rem:attain} complete the
proof.
\end{proof}

\begin{remark}\label{rem:bourin}
On $M_n(\C)$ the upper bound of Theorem~\ref{thm:gardner} is contained in a result of Bourin: if
$A\ge0$ and $Z>0$ has extremal eigenvalues $a,b$, then
$\|AZ\|_1\le\frac{a+b}{2\sqrt{ab}}\Tr(AZ)$, and this is sharp \cite[Cor.~1.4, Rem.~1.9]{Bou06}.
Indeed, with $X=U|X|$ we have $|\Tr(\rho X)|=|\Tr(U\,|X|\rho)|\le\||X|\rho\|_1$, and Bourin's
inequality with $A=|X|$ and $Z=\rho$ gives $|\varphi(X)|\le\cosh\OLC(\rho)\,\varphi(|X|)$. The
extension to finite von Neumann algebras and the replacement of $\kappa$ by $\kappa_\cZ$ are not
contained in \cite{Bou06}.
\end{remark}

\noindent By Theorem~\ref{thm:gardner}, the constant on $M_n$ depends on $\rho$ only through its
extreme eigenvalues, not on the dimension or on the rest of the spectrum. Item~3 is a consequence.

\begin{corollary}[Item 3]\label{cor:item3}
For every $A\in\cM$, $\varphi(A+A^*)\le2G_\cZ\,\varphi(|A^*|)$, and the constant is optimal. In $M_n(\C)$
it is $G$ and attained, so $V_n=G-1$ for item~3.
\end{corollary}

\begin{proof}
Since $\rho\ge0$, $\varphi(A^*)=\overline{\tau(A\rho)}=\overline{\varphi(A)}$, so
$\varphi(A+A^*)=2\re\varphi(A)\le2|\varphi(A^*)|\le2G_\cZ\varphi(|A^*|)$ by Theorem~\ref{thm:gardner} for
$X=A^*$. For the lower bound let $X$ be the example in the proof of Theorem~\ref{thm:gardner}, for
which $\varphi(X)=\sqrt\kappa>0$, and $A=X^*$. Then $\varphi(A+A^*)=2\varphi(X)$ and $\varphi(|A^*|)=\varphi(|X|)$,
so the ratio is $(1+\kappa)/(2\sqrt\kappa)$. The constant is a violation constant (one arbitrary variable,
$F_1=X_1+X_1^*$, $F_2=|X_1^*|$), so Theorem~\ref{thm:transfer}\ref{t:low} and Remark~\ref{rem:attain}
complete the proof.
\end{proof}

\subsection{Subadditivity of the modulus}
Consider
\begin{gather*}
c_{\rm gen}=\sup_{A,B\in\cM}\frac{\varphi(|A+B|)}{\varphi(|A|)+\varphi(|B|)},\qquad
c_{\rm sa}=\sup_{A,B\in\cM_{sa}}\frac{\varphi(|A+B|)}{\varphi(|A|)+\varphi(|B|)},\\
c_{\rm pd}=\sup_{A,B\in\cM_+}\frac{\varphi(|A-B|)}{\varphi(A+B)},\qquad
c_{\rm proj}=\sup_{P,Q}\frac{\varphi(|P-Q|)}{\varphi(P+Q)},
\end{gather*}
the last over projections. Clearly $c_{\rm proj}\le c_{\rm pd}\le c_{\rm sa}\le c_{\rm gen}$ (for
$c_{\rm pd}\le c_{\rm sa}$ take $a_1=A$, $a_2=-B$). The self-adjoint version is the form in which
\cite{stolyarov2002} characterises traces; $c_{\rm proj}$ corresponds to Bikchentaev's inequality
$\varphi(|P-Q|)\le\varphi(P+Q)$ from the same series \cite{bikchentaev2009rm,bikchentaev2010smj,bikchentaev2011mn,Bik15}.

\begin{theorem}[Items 2 and 18]\label{thm:subadd}
$c_{\rm gen}=c_{\rm sa}=c_{\rm pd}=c_{\rm proj}=G_\cZ$. In $M_n(\C)$ all four suprema are attained,
by a pair of rank-one projections.
\end{theorem}

\begin{proof}
Let $A+B=U|A+B|$. Then $|A+B|=U^*A+U^*B$, so by Theorem~\ref{thm:gardner}
$\varphi(|A+B|)\le|\varphi(U^*A)|+|\varphi(U^*B)|\le G_\cZ(\varphi(|U^*A|)+\varphi(|U^*B|))$, and
$|U^*A|=|A|$ because $|U^*A|^2=A^*UU^*A=|A|^2$ ($U$ is unitary). Hence $c_{\rm gen}\le G_\cZ$. For the lower bound
take, in $M_2$ with $\rho=\diag(\kappa,1)$, $s=(\kappa+1)^{-1/2}$, $c=(\kappa/(\kappa+1))^{1/2}$,
$b_\pm=(s,\pm c)$ and $P_\pm=|b_\pm\rangle\langle b_\pm|$. Then $P_+-P_-=2sc\,\sigma$ with
$\sigma=\bigl(\begin{smallmatrix}0&1\\1&0\end{smallmatrix}\bigr)$, $|P_+-P_-|=2sc\cdot1$,
$\varphi(|P_+-P_-|)=2sc(\kappa+1)=2\sqrt\kappa$ and $\varphi(P_\pm)=2\kappa/(\kappa+1)$, so the
ratio is $(1+\kappa)/(2\sqrt\kappa)$. Theorem~\ref{thm:transfer}\ref{t:low} applied to $c_{\rm
proj}$ gives $c_{\rm proj}\ge G_\cZ$.
\end{proof}

\begin{corollary}\label{cor:subadd}
$\varphi(|A+B|)\le(1-\NC(\rho)^2)^{-1/2}(\varphi(|A|)+\varphi(|B|))$ for all $A,B\in\cM$, and the
constant is optimal already for pairs $(P,-Q)$ of projections.
\end{corollary}

\begin{corollary}[Items 5 and 18]\label{cor:item5}
For projections $P,Q,R\in\cM$,
\[
\varphi(|P-R|)\le G_\cZ\bigl(\varphi(|P-Q|)+\varphi(|Q-R|)\bigr),\qquad \varphi(|P-Q|)\le G_\cZ\,\varphi(P+Q),
\]
and both constants are optimal; in $M_n(\C)$ they are attained (by rank-one $P,Q$, and for the first
inequality with $Q=0$ and rank-one $P,R$). Hence
$V_n=G-1$ for items 5 and~18.
\end{corollary}

\begin{proof}
The first inequality is $c_{\rm gen}\le G_\cZ$ for $A=P-Q$, $B=Q-R$; the second is $c_{\rm proj}=G_\cZ$.
For $Q=0$ the first inequality becomes the second with $(P,R)$ in place of $(P,Q)$, so its constant is
at least $c_{\rm proj}$. Attainment follows from Theorem~\ref{thm:subadd}.
\end{proof}

\subsection{The modulus of the adjoint}
The equality $\varphi(|X^*|)=\varphi(|X|)$ for all $X$ characterises traces, since
$|X^*|=U|X|U^*$.

\begin{theorem}[Item 26]\label{thm:modulus}
For every partial isometry $U$ and $Y\in\cM_+$ with $U^*U\ge s(Y)$,
$\varphi(UYU^*)\le\kappa_\cZ\varphi(Y)$. Consequently, for every $q>0$,
\[
\sup_{X\ne0}\frac{\varphi(|X^*|^q)}{\varphi(|X|^q)}=\kappa_\cZ,\qquad
\inf_{X\ne0}\frac{\varphi(|X^*|^q)}{\varphi(|X|^q)}=\kappa_\cZ^{-1}.
\]
In $M_n(\C)$ both are attained at $X=|e_{\max}\rangle\langle e_{\min}|$ and its adjoint.
\end{theorem}

\begin{proof}
By \eqref{eq:sandwich} and centrality of $u$,
$\varphi(UYU^*)\le\tau(uUYU^*)=\tau(uY^{1/2}U^*UY^{1/2})=\tau(uY)\le\kappa_\cZ\tau(lY)\le\kappa_\cZ\varphi(Y)$.
With $X=U|X|$ we have $|X^*|^q=U|X|^qU^*$, which gives $\le$; the infimum follows by
$X\leftrightarrow X^*$. For the lower bound, $X=e_{12}$ in $M_2$ with $\rho=\diag(\kappa,1)$ has
$|X|^q=e_{22}$, $|X^*|^q=e_{11}$ and ratio $\kappa$; apply Theorem~\ref{thm:transfer}\ref{t:low}.
\end{proof}

\subsection{The Petz--Zem\'anek inequality}
For $k\ge2$ let $c_k(\rho)=\sup_{X\ne0}|\varphi(X^k)|/\varphi(|X|^k)$. For $k=2$ the inequality
$|\varphi(X^2)|\le\varphi(|X|^2)$ is also one of the characterisations in \cite{bikchentaev2026moslehian}.

\begin{theorem}[Item 4]\label{thm:pz2}
$c_2(\rho)=G_\cZ$; moreover $|\re\varphi(X^2)|\le|\varphi(X^2)|\le G_\cZ\varphi(|X|^2)$, and both constants
are optimal. In $M_n(\C)$ the supremum is attained.
\end{theorem}

\begin{proof}
\emph{Upper bound.} On $L^2(\cM,\tau)$ let $Th=\rho^{1/2}h\rho^{-1/2}$, a positive invertible
operator with $\sigma(T)\subset[\kappa^{-1/2},\kappa^{1/2}]$, and $Jh=h^*$. Then $JTJ=T^{-1}$,
hence $JT^{1/2}J=T^{-1/2}$. For $h=X\rho^{1/2}$ we have $\|h\|^2=\varphi(|X|^2)$ and $Th=\rho^{1/2}X$,
so, using $\tau(xy)=\langle x,Jy\rangle$,
\[
\varphi(X^2)=\tau(\rho^{1/2}X\cdot X\rho^{1/2})=\langle Th,Jh\rangle=\langle
T^{1/2}h,JT^{-1/2}h\rangle,
\]
and $|\varphi(X^2)|\le\|T^{1/2}h\|\,\|T^{-1/2}h\|\le\frac12\langle(T+T^{-1})h,h\rangle\le\cosh\OLC(\rho)\|h\|^2$.
\emph{(R).} $\sigma_z(X)=z^{1/2}X$. \emph{Example.} For $\rho=\diag(\kappa,1)$ and
$X=\rho^{1/2}\sigma$ with $\sigma$ the flip, $X^2=\sqrt\kappa\cdot1$ and $|X|^2=\sigma\rho\sigma=\diag(1,\kappa)$, so
the ratio is $(\kappa+1)\sqrt\kappa/(2\kappa)$. Theorem~\ref{thm:transfer} and
Remark~\ref{rem:attain} apply. Since $X^2$ is self-adjoint in the example, the constant for
$|\re\varphi(X^2)|$ is the same.
\end{proof}

\begin{proposition}[Item 27]\label{prop:pzk}
Let $k\ge3$.
\begin{enumerate}
\item\label{pz:inv} $c_k(z\rho)=c_k(\rho)$ for $z\in\cZ_{++}$, and $c_k(\rho)\le\kappa_\cZ^{1-1/k}$.
\item\label{pz:low} If $\rho\notin\cZ$, then $c_k(\rho)\ge c_k^{M_2}(\diag(\kappa_\cZ,1))>1$;
moreover $c_k(\rho)\ge G_\cZ$ for every $k$ (Proposition~\ref{prop:pzG}).
\item\label{pz:cyc} In $M_n(\C)$, $c_k(\rho)\ge\frac1k\sum_i\lambda_i/(\prod_i\lambda_i)^{1/k}$ for the
eigenvalues $\lambda_1,\dots,\lambda_k$ of any $k\le n$ orthonormal eigenvectors of $\rho$. In a
II$_1$ factor, $c_k(\rho)\ge\max_{1\le j<k}(j\kappa+k-j)/(k\kappa^{j/k})$; in particular
$\frac1k\kappa^{1-1/k}\le c_k(\rho)\le\kappa^{1-1/k}$ there.
\end{enumerate}
\end{proposition}

\begin{proof}
\ref{pz:inv} $\sigma_z(X)=z^{1/k}X$ gives (R). Let $\epsilon=\rho^{1/k}$ and $Z=X\epsilon$. Moving
the last factor $\epsilon^{-1}$ of $(Z\epsilon^{-1})^k$ to the front,
$\varphi(X^k)=\tau(\epsilon^{k-1}Z\epsilon^{-1}Z\cdots\epsilon^{-1}Z)$ with $k$ factors $Z$ and $k-1$
factors $\epsilon^{-1}$. H\"older's inequality ($Z$ in $L^k$, the rest in $L^\infty$) gives
$|\varphi(X^k)|\le\|\rho\|^{(k-1)/k}\|\rho^{-1}\|^{(k-1)/k}\|Z\|_k^k$, and by \eqref{eq:alt} with
$r=k/2$, $\|Z\|_k^k=\tau((\epsilon X^*X\epsilon)^{k/2})\le\tau(\rho^{1/2}|X|^k\rho^{1/2})=\varphi(|X|^k)$.
So $c_k\le\kappa^{1-1/k}$; apply Theorem~\ref{thm:transfer}\ref{t:up}.
\ref{pz:low} The first inequality is Theorem~\ref{thm:transfer}\ref{t:low}, and
$c_k^{M_2}(\diag(\kappa,1))>1$ for $\kappa>1$ by the theorem of Petz and Zem\'anek \cite{petz1988}
applied in $M_2(\C)$. The bound $c_k\ge G_\cZ$ is Proposition~\ref{prop:pzG} below.
\ref{pz:cyc} With $W$ the cyclic shift $e_i\mapsto e_{i+1}$ (indices mod $k$) on the span of the
eigenvectors and $X=W\rho^{-1/k}$ there, $X^k=(\prod\lambda_i)^{-1/k}\cdot1$ and $|X|^k=\rho^{-1}$ on the
span, which gives the ratio. In a II$_1$ factor apply Theorem~\ref{thm:transfer}\ref{t:low} with
$j$ copies of $\max\sigma(\rho)$ and $k-j$ copies of $\min\sigma(\rho)$.
\end{proof}

\begin{proposition}\label{prop:pzG}
\begin{enumerate}
\item\label{pzG:a} For every $k\ge2$, $c_k(\rho)\ge G_\cZ$. In $M_2(\C)$ with $\rho=\diag(\kappa,1)$
the ratio $G_\cZ$ is attained for every $k$.
\item\label{pzG:b} $\liminf_{k\to\infty}c_k(\rho)\ge c_\infty(\rho):=\sup_{A\in\cM}|\varphi(e^A)|/\varphi(e^{\re A})$.
\end{enumerate}
\end{proposition}

\begin{proof}
\ref{pzG:a} Let $\rho=\diag(\kappa,1)$ with $\kappa>1$, $\theta=\pi/k$, $s=\kappa^{-1/(2k)}$ and
$c=2\cos\theta/(s+s^{-1})$, so that $0\le c<1$. Put
\[
V=\begin{pmatrix}c&-\sqrt{1-c^2}\\ \sqrt{1-c^2}&c\end{pmatrix},\qquad D=\diag(s,s^{-1}),\qquad X=VD.
\]
Since $V$ is unitary, $X^*X=D^2$, so $|X|^k=D^k$ and $\varphi(|X|^k)=\kappa s^k+s^{-k}=2\sqrt\kappa$.
Moreover $\det X=1$ and $\operatorname{tr}X=c(s+s^{-1})=2\cos\theta$, so $X$ has the distinct
eigenvalues $e^{\pm i\theta}$. Hence $X$ is diagonalisable and $X^k=-1$, so
$|\varphi(X^k)|=\kappa+1$ and the ratio is $(\kappa+1)/(2\sqrt\kappa)$. For $\rho\notin\cZ$ the
claim follows from Theorem~\ref{thm:transfer}\ref{t:low}; for $\rho\in\cZ$ both sides equal $1$.

\ref{pzG:b} For $X_k=1+A/k$ we have $X_k^k\to e^A$ and
$|X_k|^k=\exp\bigl(\frac k2\ln(1+\frac{A+A^*}k+\frac{A^*A}{k^2})\bigr)\to e^{\re A}$ in norm, so
$\liminf_kc_k(\rho)\ge|\varphi(e^A)|/\varphi(e^{\re A})$ for every $A$.
\end{proof}

For $k=2$ the matrix in \ref{pzG:a} is, up to a scalar and a diagonal unitary, the example in the
proof of Theorem~\ref{thm:pz2}.

For $k=2$ the bounds in \ref{pz:inv} and \ref{pz:low} meet. For $k\ge3$ they do not, and the
constant depends on more than the extreme points of the spectrum: Proposition~\ref{prop:pz3}
proves this for $k=3$ and Proposition~\ref{prop:pzk4} for $k\ge4$.

\begin{proposition}\label{prop:pz3}
Let $\mu=(2/27)^{1/2}\approx0.2722$, $\kappa>1$ and $\theta=\kappa^{1/3}$.
\begin{enumerate}
\item\label{pz3:a} In $M_2(\C)$, $c_3(\diag(\kappa,1))\le\mu\theta^2+2\theta+2+2\theta^{-1}+\theta^{-2}$.
\item\label{pz3:b} In $M_3(\C)$, $c_3(\diag(\kappa,\kappa,1))\le\mu\theta^2+8\theta+9+4\theta^{-1}+4\theta^{-2}$.
\item\label{pz3:c} In $M_3(\C)$, $c_3(\diag(\kappa,1,1))\ge(\kappa+2)/(3\kappa^{1/3})=\frac13\theta^2+\frac2{3\theta}$.
\end{enumerate}
The bound in \ref{pz3:c} exceeds the bound in \ref{pz3:a} for $\kappa\ge4\cdot10^4$ and the bound in
\ref{pz3:b} for $\kappa\ge2.5\cdot10^6$. Hence $c_3$ is not a function of the condition number
(which equals $\kappa_\cZ$ here) even on the single algebra $M_3(\C)$: the densities
$\diag(\kappa,\kappa,1)$ and $\diag(\kappa,1,1)$ have the same condition number but different
constants.
\end{proposition}

\begin{proof}
First, every $W\in M_2(\C)$ with singular values $s_1\ge s_2$ satisfies
\begin{equation}\label{eq:w2}
|W_{12}W_{22}W_{21}|\le\Bigl(\tfrac13(|W_{12}|^2+|W_{22}|^2+|W_{21}|^2)\Bigr)^{3/2}
\le\Bigl(\tfrac{s_1^2+s_2^2}3\Bigr)^{3/2}\le\mu\,(s_1^3+s_2^3),
\end{equation}
by the AM--GM inequality, $\sum|W_{ij}|^2=s_1^2+s_2^2$, and the power-mean inequality
$(s_1^2+s_2^2)^{3/2}\le\sqrt2(s_1^3+s_2^3)$.

Let $\rho$ be diagonal, $\epsilon=\rho^{1/3}$ and $Z=X\epsilon$. As in
Proposition~\ref{prop:pzk}, $\varphi(X^3)=\sum_{i,j,l}\epsilon_i^2\epsilon_j^{-1}\epsilon_l^{-1}Z_{ij}Z_{jl}Z_{li}$
and $\varphi(|X|^3)\ge\Tr|Z|^3$; moreover $|Z_{ij}Z_{jl}Z_{li}|\le\|Z\|^3\le\Tr|Z|^3$.

\ref{pz3:a} Here $\epsilon=\diag(\theta,1)$. The weight $\epsilon_i^2/(\epsilon_j\epsilon_l)$ is
$\theta^2$ for $(i,j,l)=(1,2,2)$, and the other seven weights are
$1,\theta,\theta,\theta^{-2},\theta^{-1},\theta^{-1},1$. The term with weight $\theta^2$ is bounded by
\eqref{eq:w2} with $W=Z$, the others by $\Tr|Z|^3$.

\ref{pz3:b} Here $\epsilon=\diag(\theta,\theta,1)$. The weight $\theta^2$ occurs for $i\in\{1,2\}$,
$j=l=3$; the weight $\theta$ for $8$ triples, $1$ for $9$, $\theta^{-1}$ for $4$ and $\theta^{-2}$ for
$4$. The $\theta^2$-part is $T=Z_{33}(Z_{13}Z_{31}+Z_{23}Z_{32})$, so
$|T|\le|Z_{33}|\beta\gamma$ with $\beta=(|Z_{13}|^2+|Z_{23}|^2)^{1/2}$,
$\gamma=(|Z_{31}|^2+|Z_{32}|^2)^{1/2}$. If $\beta\gamma\ne0$, let $u,u'\in\operatorname{span}\{e_1,e_2\}$
be the unit vectors with coordinates $\bar Z_{3k}/\gamma$ and $Z_{k3}/\beta$ ($k=1,2$), and let
$W$ be the matrix of the compression of $Z$ from $\operatorname{span}\{u,e_3\}$ to
$\operatorname{span}\{u',e_3\}$. Then $W_{12}=\beta$, $W_{21}=\gamma$, $W_{22}=Z_{33}$, and the singular values
of $W$ do not exceed those of $Z$. By \eqref{eq:w2}, $|T|\le\mu\Tr|Z|^3$.

\ref{pz3:c} is Proposition~\ref{prop:pzk}\ref{pz:cyc}.

The difference between \ref{pz3:c} and \ref{pz3:a} is
$h_1(\theta)=(\frac13-\mu)\theta^2-2\theta-2-\frac4{3\theta}-\frac1{\theta^2}$, which is increasing for
$\theta>(\frac13-\mu)^{-1}\approx16.35$ and positive at $\theta=(4\cdot10^4)^{1/3}\approx34.20$. The
difference between \ref{pz3:c} and \ref{pz3:b} is
$h_2(\theta)=(\frac13-\mu)\theta^2-8\theta-9-\frac{10}{3\theta}-\frac4{\theta^2}$, which is increasing for
$\theta>4(\frac13-\mu)^{-1}\approx65.4$ and positive at $\theta=(2.5\cdot10^6)^{1/3}\approx135.7$.
\end{proof}

The thresholds $4\cdot10^4$ and $2.5\cdot10^6$ are conservative: they come from bounding all
lower-order terms crudely by $\Tr|Z|^3$. Numerically the constants already differ at $\kappa=10$
(Remark~\ref{rem:pznum}).

For $k\ge4$ the same comparison works, with a different $2\times2$ estimate in place of
\eqref{eq:w2}.

\begin{proposition}\label{prop:pzk4}
Let $k\ge4$, $n\ge k$, $\kappa>1$ and $\theta=\kappa^{1/k}$. Put $\nu_4=1/(2\cdot3^{3/4})\approx0.2193$
and $\nu_k=\frac2k(1-\frac2k)^{(k-2)/2}$ for $k\ge5$. Then $\nu_k<\frac1k$, and in $M_n(\C)$:
\begin{enumerate}
\item\label{pz4:a} $c_k(\diag(\kappa,\dots,\kappa,1))\le\nu_k\theta^{k-1}+n^k\theta^{k-2}$;
\item\label{pz4:b} $c_k(\diag(\kappa,1,\dots,1))\ge(\kappa+k-1)/(k\kappa^{1/k})>\frac1k\theta^{k-1}$.
\end{enumerate}
Hence $c_k(\diag(\kappa,1,\dots,1))>c_k(\diag(\kappa,\dots,\kappa,1))$ for
$\kappa\ge\bigl(n^k/(\frac1k-\nu_k)\bigr)^k$. Both densities have condition number $\kappa$ ($=\kappa_\cZ$
here), so $c_k$ is not a function of the condition number on $M_n(\C)$ for any $n\ge k$.
\end{proposition}

\begin{proof}
\emph{The constants.} $4\nu_4=2\cdot3^{-3/4}<1$. For $k\ge5$, $k\nu_k=2e^{g(k/2)}$ with
$g(m)=(m-1)\ln(1-\frac1m)$; since $g'(m)=\ln(1-\frac1m)+\frac1m<0$ and $g(2)=-\ln2$, we get
$k\nu_k<1$.

\emph{A $2\times2$ estimate.} We show that every $W\in M_2(\C)$ satisfies
\begin{equation}\label{eq:wk}
|W_{22}|^{k-2}|W_{12}|\,|W_{21}|\le\nu_k\Tr|W|^k.
\end{equation}
Let $x,a,b,c$ be the moduli of $W_{11},W_{22},W_{12},W_{21}$. Let $k\ge5$. Replacing $W$ by its
transpose, which has the same singular values, we may assume $c\le b$. Then
$\|W\|^2\ge\|We_2\|^2=a^2+b^2$, so $\Tr|W|^k\ge(a^2+b^2)^{k/2}$, and
$a^{k-2}bc\le a^{k-2}b^2=(a^2+b^2)^{k/2}(1-s)^{(k-2)/2}s$ with $s=b^2/(a^2+b^2)$; the maximum over
$s\in[0,1]$ is attained at $s=2/k$ and equals $\nu_k$. For $k=4$ this argument gives \eqref{eq:wk}
only with the constant $1/4=1/k$ in place of $\nu_4$, whereas the final comparison of~\ref{pz4:a}
with~\ref{pz4:b} needs a constant $<1/k$; so we argue differently. Let $k=4$. With singular values $s_1,s_2$,
$\Tr|W|^4=(s_1^2+s_2^2)^2-2s_1^2s_2^2=(\sum|W_{ij}|^2)^2-2|\det W|^2$, and
$|\det W|\le xa+bc\le P^{1/2}Q^{1/2}$ with $P=x^2+b^2$, $Q=a^2+c^2$. Hence
$\Tr|W|^4\ge(P+Q)^2-2PQ=P^2+Q^2\ge b^4+(a^2+c^2)^2$. For fixed $a,c$ and $K=(a^2+c^2)^2$,
$\max_{b>0}b/(b^4+K)=\frac{3^{3/4}}4K^{-3/4}$ (at $b^4=K/3$), and
$a^2c\,(a^2+c^2)^{-3/2}\le2\cdot3^{-3/2}$ (at $c^2=a^2/2$). The product of the two constants
is $\nu_4$.

\ref{pz4:a} Let $\rho=\diag(\kappa,\dots,\kappa,1)$, $\epsilon=\rho^{1/k}=\diag(\theta,\dots,\theta,1)$ and
$Z=X\epsilon$. As in the proof of Proposition~\ref{prop:pzk}\ref{pz:inv},
\[
\varphi(X^k)=\sum_{i_1,\dots,i_k}\epsilon_{i_1}^{k-1}\epsilon_{i_2}^{-1}\cdots\epsilon_{i_k}^{-1}
Z_{i_1i_2}Z_{i_2i_3}\cdots Z_{i_ki_1},\qquad \varphi(|X|^k)\ge\Tr|Z|^k,
\]
and every product of $k$ entries of $Z$ is bounded by $\|Z\|^k\le\Tr|Z|^k$. The weight is
$\theta^{k-1}$ if $i_1<n$ and $i_2=\dots=i_k=n$, and at most $\theta^{k-2}$ for each of the fewer
than $n^k$ other index tuples. The $\theta^{k-1}$-part is $T=Z_{nn}^{k-2}\sum_{i<n}Z_{in}Z_{ni}$, so
$|T|\le|Z_{nn}|^{k-2}\beta\gamma$ with $\beta=(\sum_{i<n}|Z_{in}|^2)^{1/2}$ and
$\gamma=(\sum_{i<n}|Z_{ni}|^2)^{1/2}$. If $\beta\gamma\ne0$, the compression $W$ of $Z$ constructed as
in the proof of Proposition~\ref{prop:pz3}\ref{pz3:b}, with $\operatorname{span}\{e_1,\dots,e_{n-1}\}$
and $e_n$ in place of $\operatorname{span}\{e_1,e_2\}$ and $e_3$, has $W_{12}=\beta$, $W_{21}=\gamma$,
$W_{22}=Z_{nn}$, and singular values not exceeding those of $Z$. Hence $\Tr|W|^k\le\Tr|Z|^k$, and
\eqref{eq:wk} gives $|T|\le\nu_k\Tr|Z|^k$. Altogether
$|\varphi(X^k)|\le(\nu_k\theta^{k-1}+n^k\theta^{k-2})\varphi(|X|^k)$.

\ref{pz4:b} is Proposition~\ref{prop:pzk}\ref{pz:cyc} for $\kappa,1,\dots,1$, and $\theta^k+k-1>\theta^k$.

Finally, if $\theta\ge n^k/(\frac1k-\nu_k)$, then $(\frac1k-\nu_k)\theta^{k-1}\ge n^k\theta^{k-2}$, and the
claim follows from \ref{pz4:a} and \ref{pz4:b}.
\end{proof}

As for $k=3$, the threshold is conservative (for $k=n=4$ it is about $5\cdot10^{15}$), while
numerically the constants differ already at $\kappa=10$.

\begin{remark}\label{rem:pznum}
By Proposition~\ref{prop:pzG}, $c_k\ge c_2=G_\cZ$ for every $k$ and $\liminf_kc_k\ge c_\infty$.
Numerical optimisation (not used in the proofs) indicates that $c_2<c_3<c_4<\dots$ (in $M_2(\C)$
for $1.05\le\kappa\le1000$ and $k\le7$; near $\kappa=1$, $c_k-G_\cZ=O(\LNC^4)$), that the
limit is $c_\infty$, and that for fixed $\kappa$ the constant $c_k$,
$k\ge3$, is nondecreasing in the multiplicity of the bottom of the spectrum and does not depend on that of
the top; in all cases computed it increases strictly from multiplicity $1$ to $2$ (provably for
$k=3$ and large $\kappa$, by Proposition~\ref{prop:pz3}; see the examples below), but not always beyond (for $k=4$,
$\kappa=10$, multiplicities $2$ and $3$ give the same value).
For $\kappa=10$, for instance, $c_3(\diag(10,1))\approx c_3(\diag(10,10,1))\approx1.7790$ while
$c_3(\diag(10,1,1))\ge(10+2)/(3\cdot10^{1/3})\approx1.8566$; and
$c_4(\diag(10,1))\approx c_4(\diag(10,10,1))\approx c_4(\diag(10,10,10,1))\approx1.8050$ while
$c_4(\diag(10,1,1))\approx c_4(\diag(10,1,1,1))\approx1.8725$.
\end{remark}

\subsection{The Bikchentaev--Abed inequality}
Let
\begin{gather*}
\gamma(\rho)=\sup_{A,B\in\cM_+\setminus0}\frac{-\re\varphi(AB)}{\sqrt{\varphi(A^2)\varphi(B^2)}},\\
\mathcal K(\rho)=\{K\in\R:\varphi(A^2+B^2+K(AB+BA))\ge0\ \text{ for all }A,B\in\cM_+\}.
\end{gather*}

\begin{theorem}[Item 28]\label{thm:ba}
Let $\cM\ne\C1$. Then $\gamma(\rho)=\tanh^2\frac{\LNC(\rho)}2=\Bigl(\frac{\sqrt{\kappa_\cZ}-1}{\sqrt{\kappa_\cZ}+1}\Bigr)^2=\frac{G_\cZ-1}{G_\cZ+1}$,
and $\mathcal K(\rho)=[-1,K^*(\rho)]$ with
\[
K^*(\rho)=\frac1{\gamma(\rho)}=\coth^2\frac{\LNC(\rho)}2=\frac{G_\cZ+1}{G_\cZ-1}\qquad(K^*=+\infty\text{ iff }\rho\in\cZ).
\]
Equivalently, $\re\varphi(AB)\ge-\tanh^2(\LNC(\rho)/2)\sqrt{\varphi(A^2)\varphi(B^2)}$ for all
$A,B\in\cM_+$, with the optimal constant. In $M_n(\C)$ the supremum is attained.
\end{theorem}

\begin{proof}
\emph{Reduction.} $\varphi(AB+BA)=2\re\varphi(AB)$, since
$\overline{\tau(\rho AB)}=\tau(BA\rho)$. Replacing $(A,B)$ by $(sA,B/s)$ and minimising
$s^2\varphi(A^2)+s^{-2}\varphi(B^2)$ shows that $K\in\mathcal K(\rho)$ iff
$-K\re\varphi(AB)\le\sqrt{\varphi(A^2)\varphi(B^2)}$ for all $A,B\ge0$. For $K\in[-1,0]$ this holds
by the Cauchy--Schwarz inequality for $(X,Y)\mapsto\varphi(Y^*X)$; for $K<-1$ it fails at $A=B$;
for $K>0$ it holds iff $K\gamma\le1$. Since $\gamma\ge0$ (take nonzero $A,B\ge0$ with $AB=0$, which
exist as $\cM\ne\C1$), $\mathcal K=[-1,1/\gamma]$.

\emph{Upper bound with $\kappa$.} Let $D=\rho^{1/2}$, $L=\OLC(\rho)$,
$S=\rho^{1/4}A\rho^{1/4}$, $T=\rho^{1/4}B\rho^{1/4}$; $A\mapsto S$ is a bijection of $\cM_+$. By
cyclicity $\varphi(AB)=\tau(DSD^{-1}T)$, and
$\overline{\tau(DSD^{-1}T)}=\tau(D^{-1}SDT)$, so
\[
\re\varphi(AB)=\tau(\Phi(S)T),\quad\varphi(A^2)=\tau(\Phi(S)S),\quad\varphi(B^2)=\tau(\Phi(T)T),
\]
where $\Phi=\frac12(L_DR_{D^{-1}}+L_{D^{-1}}R_D)$ acts on $L^2(\cM,\tau)$.
$\Phi$ is a function $g(L_D,R_D)$ of two commuting positive operators with spectra in
$[\|D^{-1}\|^{-1},\|D\|]$, where $g(s,t)=\frac12(s/t+t/s)\in[1,\cosh L]$, so $\sigma(\Phi)\subset[1,\cosh L]$
on $L^2(\cM,\tau)$. Since $\langle S,T\rangle=\tau(TS)=\tau(S^{1/2}TS^{1/2})\ge0$,
Lemma~\ref{lem:wielandt}\ref{w:one} gives
$-\re\varphi(AB)\le\frac{\cosh L-1}{\cosh L+1}\sqrt{\varphi(A^2)\varphi(B^2)}$, and
$\frac{\cosh L-1}{\cosh L+1}=\tanh^2\frac L2$.

\emph{(R).} $\sigma_z(A,B)=(z^{1/2}A,z^{1/2}B)$: $\tau(\rho z^{1/2}Az^{1/2}B)=\tau(z\rho AB)$, and
similarly for $A^2$, $B^2$.

\emph{Example.} In $M_2$ with $\rho=\diag(\kappa,1)$ let $A=|a\rangle\langle a|$,
$B=|b\rangle\langle b|$ with $a=\rho^{-1/4}(1,1)$, $b=\rho^{-1/4}(1,-1)$, i.e.\
$S=\bigl(\begin{smallmatrix}1&1\\1&1\end{smallmatrix}\bigr)$,
$T=\bigl(\begin{smallmatrix}1&-1\\-1&1\end{smallmatrix}\bigr)$. $\Phi$ acts as the Schur multiplier
with entries $1$ on the diagonal and $C=\cosh L$ off it, so $\tau(\Phi(S)T)=2-2C$ and
$\tau(\Phi(S)S)=\tau(\Phi(T)T)=2+2C$, and the ratio is $(C-1)/(C+1)$. The theorem follows from
Theorem~\ref{thm:transfer} applied to $\gamma$ (with $f(\kappa)=\tanh^2(\frac14\ln\kappa)$, $f(1)=0$,
attained at $AB=0$).
\end{proof}

\begin{corollary}\label{cor:ba}
Let $K>1$. Then $\varphi(A^2+B^2+K(AB+BA))\ge0$ for all $A,B\in\cM_+$ if and only if
\[
\kappa_\cZ(\rho)\le\Bigl(\frac{\sqrt K+1}{\sqrt K-1}\Bigr)^2.
\]
In particular, for $\varphi=\tau(\rho\,\cdot)$ with $\rho\in\cM_{++}$, condition (iii) of
\cite[Thm.~4.6]{bikchentaev2019abed} (there is $K_0>-1$ such that the inequality holds for all $K>K_0$ and all
$A,B\ge0$) can be weakened to: the inequality holds for arbitrarily large $K$. Every non-tracial
$\varphi$ satisfies the inequality exactly on the bounded interval $[-1,K^*(\rho)]$.
\end{corollary}

\begin{proof}
$K\le1/\gamma$ iff $\tanh(\LNC/2)=(\sqrt{\kappa_\cZ}-1)/(\sqrt{\kappa_\cZ}+1)\le K^{-1/2}$.
\end{proof}

\noindent For projections, $\varphi(P^2)=\varphi(P)$, and Theorem~\ref{thm:ba} becomes a sharp lower bound
for $\varphi(\Re PQ)$; this is item~21.

\begin{proposition}[Item 21]\label{prop:item21}
For all projections $P,Q\in\cM$,
\[
-\varphi(\Re PQ)\le\gamma(\rho)\sqrt{\varphi(P)\varphi(Q)}\le\frac{\gamma(\rho)}2\bigl(\varphi(P)+\varphi(Q)\bigr),\qquad
\gamma(\rho)=\tanh^2\frac{\LNC(\rho)}2=\frac{G_\cZ-1}{G_\cZ+1},
\]
where $\gamma(\rho)$ is the $\gamma$ of Theorem~\ref{thm:ba} if $\cM\ne\C1$ (on $\C1$ we use the right-hand side),
and $\gamma(\rho)/2$ is the supremum of $-\varphi(\Re PQ)/(\varphi(P)+\varphi(Q))$. In $M_n(\C)$ it is
attained and equals $\frac12\tan^2(\psi/2)$, so $V_n=\frac12\tan^2(\psi/2)=\frac{G-1}{2(G+1)}$ for item~21.
\end{proposition}

\begin{proof}
If $\cM=\C1$, then $P,Q\in\{0,1\}$, $\varphi(\Re PQ)\ge0$ and $\LNC(\rho)=0$, so all claims hold
trivially (the supremum is $0$, attained at $P=1$, $Q=0$); let $\cM\ne\C1$, so that
Theorem~\ref{thm:ba} applies. Since $\varphi(\Re PQ)=\re\varphi(PQ)$, $P^2=P$ and $Q^2=Q$, the first inequality is Theorem~\ref{thm:ba}, and the second
is the arithmetic--geometric mean inequality. For the lower bound let $\rho=\diag(\kappa,1)$ in $M_2$ and
let $P,Q$ be the projections onto the vectors $a=\rho^{-1/4}(1,1)$ and $b=\rho^{-1/4}(1,-1)$ of the
example in the proof of Theorem~\ref{thm:ba}. The ratio $-\re\varphi(AB)/\sqrt{\varphi(A^2)\varphi(B^2)}$ is
invariant under positive scalings of $A$ and of $B$, so it equals $\gamma$ also for $(P,Q)$; moreover
$\|a\|=\|b\|$ and $\langle\rho a,a\rangle=\langle\rho^{1/2}(1,1),(1,1)\rangle=\langle\rho b,b\rangle$, so
$\varphi(P)=\varphi(Q)$ and the arithmetic--geometric mean inequality is an equality. Hence the ratio
$-\varphi(\Re PQ)/(\varphi(P)+\varphi(Q))$ equals $\gamma(\diag(\kappa,1))/2$. This ratio is a violation
constant (two projections, $F_1=X_1X_2+X_2X_1$, $F_2=X_1$, $F_3=X_2$, $\Psi_N(z)=-\frac12\re z_1$,
$\Psi_D(z)=|z_2|+|z_3|$), and Theorem~\ref{thm:transfer}\ref{t:low} and Remark~\ref{rem:attain} give the
lower bound in $\cM$ and the attainment in $M_n$. Finally $\tan^2\frac\psi2=\frac{1-\cos\psi}{1+\cos\psi}=\frac{G-1}{G+1}$,
since $\cos\psi=1/G$.
\end{proof}

\noindent In the example, $|\langle a,b\rangle|/(\|a\|\|b\|)=(\sqrt\kappa-1)/(\sqrt\kappa+1)=\tan(\psi/2)$: the
two projections make the angle whose cosine is $\tan(\psi/2)$, and their Bloch vectors are symmetric
about the eigenvector of $\rho$ for its smaller eigenvalue.

The second inequality of~\cite{bikchentaev2019abed} that we consider, item~9, involves two projections
and the compression $QPQ$; we determine its envelope on $M_2$.

\begin{proposition}[Item 9]\label{prop:abed}
For $n=2$, item~9 has the envelope $V_2=s/(1+s)$, $s=\sin(\psi/2)$. Thus $V_2\sim\delta/2$ as $\delta\to0$ and
$V_2\to\sqrt2-1$ as $\delta\to1$.
\end{proposition}
\begin{proof}
If $P$ or $Q$ is $0$ or $1$, or $P=Q$, then $\cL=\cR$. Otherwise $P=uu^*$, $Q=vv^*$ with unit vectors; choose phases so that
$u=cv+\sigma v_\perp$ with $c=\cos\theta$, $\sigma=\sin\theta$, $0<\theta<\pi/2$, and a unit vector $v_\perp\perp v$.
Normalise $\Tr \rho=2$, so that $\rho$ has the eigenvalues $1\pm\delta$, and put $\beta=v^*\rho v$, $\beta'=v_\perp^*\rho v_\perp=2-\beta$,
$t=\Re(v^*\rho v_\perp)$ and $\mu=|v^*\rho v_\perp|$; then $\beta\beta'-\mu^2=\det \rho=1-\delta^2$. Since $PQ+QP=c(uv^*+vu^*)$
and $QPQ=c^2Q$,
\[
  \cL-\cR=2c\Re(v^*\rho u)-2c^2\beta=2c\sigma t,\qquad \cN=u^*\rho u+\beta=\beta(1+c^2)+\sigma^2\beta'+2c\sigma t .
\]
With $D=\beta(1+c^2)+\sigma^2\beta'>0$ the quotient $2c\sigma t/(D+2c\sigma t)$ increases with $t\le\mu$, so it is at most
$r/(1+r)$ with
\[
  r=\frac{2c\sigma\mu}{D}=\frac{\mu\sin2\theta}{\beta+1+(\beta-1)\cos2\theta}\le\frac{\mu}{\sqrt{(\beta+1)^2-(\beta-1)^2}}
  =\frac\mu{2\sqrt\beta},
\]
using $\beta'=2-\beta$ and $\max_x\sin x/(p+q\cos x)=(p^2-q^2)^{-1/2}$ for $p>|q|$. Finally
$r^2\le\mu^2/(4\beta)=\frac14\bigl(2-\beta-(1-\delta^2)/\beta\bigr)$, which is maximal at $\beta=\sqrt{1-\delta^2}=\cos\psi$
with the value $\frac12(1-\cos\psi)=\sin^2(\psi/2)$. All steps are equalities for $v$ with $v^*\rho v=\cos\psi$, a phase of
$v_\perp$ with $v^*\rho v_\perp=\mu>0$, and the optimal $\theta$, and $r\mapsto r/(1+r)$ is increasing.
\end{proof}

\subsection{Young's inequality}\label{ssec:young}
Let $1<p<\infty$, $1/p+1/q=1$, and
\[
c_p(\rho)=\sup_{A,B\in\cM}\frac{\varphi(|AB|)}{\varphi(|A|^p)^{1/p}\varphi(|B|^q)^{1/q}},\qquad
c_p^+(\rho)=\sup_{A,B\in\cM_+}\frac{\varphi(|AB|)}{\varphi(A^p)^{1/p}\varphi(B^q)^{1/q}}.
\]
Replacing $(A,B)$ by $(sA,B/s)$ and using $\inf_s(s^pa/p+s^{-q}b/q)=a^{1/p}b^{1/q}$, $c_p$ is also the
best constant $c$ in $\varphi(|AB|)\le c\,\varphi(|A|^p/p+|B|^q/q)$, and similarly for $c_p^+$. The
Cho--Sano form $\varphi(|A^*B|)$ gives the same constant: with $A=U|A|$,
$|A^*B|=||A|U^*B|=|A(U^*B)|$ and $|U^*B|=|B|$, so $B\mapsto U^*B$ transforms one supremum into the
other.

\begin{lemma}\label{lem:mixedholder}
For $c,b\in\cM_+$, $\tau(\rho|cb|)\le\|\rho\|^{1/p}\,\|c\|_p\,\varphi(b^q)^{1/q}$.
\end{lemma}

\begin{proof}
By norm continuity of $b\mapsto|cb|$ we may assume $b$ invertible. Let $Y=|cb|$, $cb=WY$ the polar
decomposition, $S=b^{-1/2}Yb^{-1/2}\ge0$ and $Z=b^{1/2}Sb^{-1/2}=Yb^{-1}=W^*c$, so
$\|Z\|_p\le\|c\|_p$. The function $F(z)=b^{-z}Zb^{z}$ is analytic and bounded on the strip
$0\le\re z\le1$, $F(\frac12)=S$, $\|F(iy)\|_p=\|Z\|_p$, and $F(1+iy)=b^{-iy}Z^*b^{iy}$ because
$b^{-1}Zb=b^{-1/2}Sb^{1/2}=Z^*$. By the three lines lemma applied to $z\mapsto\tau(F(z)V)$ with
$\|V\|_q\le1$ and by duality, $\|S\|_p\le\|Z\|_p\le\|c\|_p$. Next,
$\tau(\rho Y)=\tau(S\,b^{1/2}\rho b^{1/2})\le\|S\|_p\|b^{1/2}\rho b^{1/2}\|_q$, and by \eqref{eq:alt}
and $\rho^q\le\|\rho\|^{q-1}\rho$,
$\|b^{1/2}\rho b^{1/2}\|_q^q=\tau((\rho^{1/2}b\rho^{1/2})^q)\le\tau(\rho^qb^q)\le\|\rho\|^{q-1}\varphi(b^q)$.
Since $(q-1)/q=1/p$, the lemma follows.
\end{proof}

\begin{theorem}[Item 29]\label{thm:young}
$c_p(\rho)=\kappa_\cZ^{1/p}=e^{2\LNC(\rho)/p}$; that is,
$\varphi(|AB|)\le\kappa_\cZ^{1/p}\varphi(|A|^p)^{1/p}\varphi(|B|^q)^{1/q}\le\kappa_\cZ^{1/p}\varphi(|A|^p/p+|B|^q/q)$
for all $A,B\in\cM$, with the optimal constant. In $M_n(\C)$ it is attained.
\end{theorem}

\begin{proof}
Since $|AB|^2=B^*|A|^2B$, $|AB|=||A|B|$. With $B=V|B|$, $V$ unitary, $||A|V|B||=|c\,|B||$ for
$c=V^*|A|V$, and $\tau(c^p)=\tau(|A|^p)\le\|\rho^{-1}\|\varphi(|A|^p)$. Lemma~\ref{lem:mixedholder}
gives $c_p\le\kappa^{1/p}$. (R) holds with $\sigma_z(A,B)=(z^{1/p}A,z^{1/q}B)$, since
$|z^{1/p}Az^{1/q}B|=z|AB|$. In $M_2$ with $\rho=\diag(\kappa,1)$, $A=e_{12}$, $B=e_{21}$ give
$|AB|=e_{11}$, $|A|^p=e_{22}$, $|B|^q=e_{11}$ and ratio $\kappa/\kappa^{1/q}=\kappa^{1/p}$.
Theorem~\ref{thm:transfer} and Remark~\ref{rem:attain} complete the proof.
\end{proof}

\begin{corollary}[Item 8]\label{cor:item8}
For all $A,B\in\cM$,
\[
\varphi(|AB^*|)\le\kappa_\cZ\,\varphi(|A|^2)^{1/2}\varphi(|B|^2)^{1/2}\le\kappa_\cZ\cdot\tfrac12\bigl(\varphi(|A|^2)+\varphi(|B|^2)\bigr),
\]
and $\kappa_\cZ$ is optimal in both inequalities. In $M_n(\C)$ the constant $\kappa$ is attained at
$A=B=e_1e_2^*$, so $V_n=\kappa-1=2\delta/(1-\delta)$ for item~8 in every dimension.
\end{corollary}

\begin{proof}
Theorem~\ref{thm:young} with $p=2$, applied to $A$ and $B^*$, gives
$\varphi(|AB^*|)\le\kappa_\cZ^{1/2}\varphi(|A|^2)^{1/2}\varphi(|B^*|^2)^{1/2}$, and Theorem~\ref{thm:modulus}
with $q=2$ gives $\varphi(|B^*|^2)\le\kappa_\cZ\varphi(|B|^2)$; the second inequality is the
arithmetic--geometric mean inequality. For the lower bound let $\rho=\diag(\kappa,1)$ in $M_2$ and
$A=B=e_{12}$. Then $|AB^*|=e_{11}$ and $|A|^2=|B|^2=e_{22}$, so both ratios equal $\kappa$. The constants
are violation constants (two arbitrary variables, $F_1=|X_1X_2^*|$, $F_2=|X_1|^2$, $F_3=|X_2|^2$), so
Theorem~\ref{thm:transfer}\ref{t:low} and Remark~\ref{rem:attain} complete the proof.
\end{proof}

\begin{remark}[Two forms of Young's inequality]\label{rem:youngforms}
Items 8 and 29 (with $p=2$) differ in the way the second operator is measured: item~29 compares
$\varphi(|AB|)$ with $\varphi(|A|^2)$ and $\varphi(|B|^2)$, item~8 compares $\varphi(|AB^*|)$ with
$\varphi(|A|^2)$ and $\varphi(|B|^2)$, that is, with $\varphi(B^*B)$ instead of $\varphi(|B^*|^2)=\varphi(BB^*)$.
For traces the two are the same. For $\varphi_\rho$ the constant of item~29 is $\kappa_\cZ^{1/2}$ and that of
item~8 is $\kappa_\cZ$; the additional factor $\kappa_\cZ^{1/2}$ is the square root of the constant
$\kappa_\cZ$ of item~26 with $q=2$, which measures the discrepancy between $\varphi(|B^*|^2)$ and
$\varphi(|B|^2)$; it enters as a square root because $\varphi(|B^*|^2)$ appears as $\varphi(|B^*|^2)^{1/2}$. Both fail linearly at the trace, item~8 as
$\kappa-1\sim2\delta$ and item~29 as $\kappa^{1/2}-1\sim\delta$ on $M_n$. For positive $A,B$ the two forms
coincide ($|AB^*|=|AB|$, $|A|^2=A^2$) and become item~30 with $p=2$, whose constant is conjecturally
$G_\cZ$ (Conjecture~\ref{conj:young}).
\end{remark}

\noindent With the contrast of the operators prescribed, the violation of item~8 is bounded below as
follows.

\begin{proposition}[Item 8 with two contrasts]\label{prop:youngtwo}
Let $\OMC(\rho)=\delta$ and let the moduli of $A$ and $B$ have contrast $\delta'$. On $M_2$, for $\rho=\mathrm{diag}(a,b)$,
$T=\mathrm{diag}(\ell',1)$ with $\ell'=\frac{1-\delta'}{1+\delta'}$, $U$ the swap of the two basis vectors and $A=B=UT$,
the extended Young inequality has the relative violation $\frac{\kappa+\ell'^2}{1+\kappa\ell'^2}-1$. Hence, for $n=2$ and, by direct
sums, for every even $n$,
\[
  V_n(\delta,\delta')\ \ge\ \frac{\kappa+\ell'^2}{1+\kappa\ell'^2}-1=\frac{4\delta\delta'}{1+\delta'^2-2\delta\delta'}.
\]
For $\delta'=1$ this is the value of Corollary~\ref{cor:item8}.
\end{proposition}
\begin{proof}
For $A=B$, $|AB^*|=|AA^*|=AA^*=UT^2U^*=\mathrm{diag}(1,\ell'^2)$ and $|A|^2=T^2=\mathrm{diag}(\ell'^2,1)$, so $\cL=a+b\ell'^2$ and
$\cR=a\ell'^2+b$, and $\cL/\cR=(\kappa+\ell'^2)/(\kappa\ell'^2+1)$. Inserting $\kappa=\frac{1+\delta}{1-\delta}$ and $\ell'$ gives the second form.
\end{proof}

\noindent The numerical search over all $A,B$ with this contrast finds no larger value on $M_2$ (agreement to four
digits, Figure~\ref{fig:mc-young}), and we conjecture equality for even $n$. Unlike Gardner's inequality, the extended Young
inequality fails \emph{linearly} in each contrast, $V\sim4\delta\delta'$.

For positive operators the problem is different. For $\kappa\ge1$ put
\[
\yp(\kappa)=\max_{v>0}\ (1+\kappa)\kappa^{1/p}\,v\,(1+v^2)^{\frac1p-\frac12}(v^2+\kappa^2)^{-\frac1p-\frac12}.
\]

\begin{proposition}[Item 30]\label{prop:youngpos}
\begin{enumerate}
\item\label{yp:a} $c_p^+(z\rho)=c_p^+(\rho)$ for $z\in\cZ_{++}$, and $\yp(\kappa_\cZ)\le c_p^+(\rho)\le\kappa_\cZ^{1/p}$.
\item\label{yp:b} In $M_2(\C)$ with $\rho=\diag(\kappa,1)$, $\yp(\kappa)$ is the supremum of the ratio over
rank-one $A,B$. Moreover $\yp(\kappa)>1$ for $\kappa>1$, and $\mathsf y_2(\kappa)=\cosh(\frac12\ln\kappa)$.
\item\label{yp:c} If $1<p\le2$, then in $M_2(\C)$ with $\rho=\diag(\kappa,1)$ the supremum over $A\ge0$
and rank-one $B\ge0$ equals $\yp(\kappa)$.
\item\label{yp:d} If $p>2$ and $\rho\notin\cZ$, then $c_p^+(\rho)>\yp(\kappa_\cZ)$.
\end{enumerate}
\end{proposition}

\begin{proof}
\ref{yp:b} For unit vectors $x,y$ and $A=|x\rangle\langle x|$, $B=|y\rangle\langle y|$ we have
$|AB|=|\langle x,y\rangle|\,|y\rangle\langle y|$, so the ratio is
$R(x,y)=|\langle x,y\rangle|(\langle\rho y,y\rangle/\langle\rho x,x\rangle)^{1/p}$. Replacing the
coordinates by their moduli does not decrease $R$, so let $x=(\cos s,\sin s)$, $y=(\cos t,\sin t)$
with $s,t\in[0,\pi/2]$, and $g(t)=\kappa\cos^2t+\sin^2t$. On the boundary of the square $R$ is not
maximal unless $R\le1$: e.g.\ at $s=\pi/2$, $\partial_s\ln R=-\cot t<0$, and similarly on the other
sides. At an interior critical point, $\tan(t-s)=\frac1p\frac{g'(s)}{g(s)}=\frac1p\frac{g'(t)}{g(t)}$, and
$g'(t)/g(t)=-2(\kappa-1)h(\tan t)$ with $h(v)=v/(\kappa+v^2)$. Now $h(v_1)=h(v_2)$ iff $v_1=v_2$ or
$v_1v_2=\kappa$; the first case means $s=t$ and $R=1$. Substituting $\tan s=v$, $\tan t=\kappa/v$
gives $\cos(t-s)=(1+\kappa)v/\sqrt{(1+v^2)(v^2+\kappa^2)}$ and $g(t)/g(s)=\kappa(1+v^2)/(v^2+\kappa^2)$,
i.e.\ $R$ equals the function maximised in the definition of $\yp$. That function tends to $0$ as
$v\to0$ or $v\to\infty$, equals $1$ at $v=\sqrt\kappa$, and its logarithmic derivative there is
$\frac2p\frac{\kappa-1}{\sqrt\kappa(\kappa+1)}>0$; hence its maximum exceeds $1$, so the maximum of $R$
is interior and equals $\yp(\kappa)$. For $p=2$ the maximum is at $v=\kappa$ and equals
$(1+\kappa)/(2\sqrt\kappa)$.

\ref{yp:a} (R) holds with $\sigma_z(A,B)=(z^{1/p}A,z^{1/q}B)$. The upper bound follows from
Theorem~\ref{thm:young}, the lower bound from \ref{yp:b} and Theorem~\ref{thm:transfer}\ref{t:low}.

\ref{yp:c} For $B=|y\rangle\langle y|$ and $A=\sum_ia_i|x_i\rangle\langle x_i|$ (orthonormal $x_i$),
with $w_i=|\langle x_i,y\rangle|^2$ and $r_i=\langle\rho x_i,x_i\rangle$, the $p$-th power of the
ratio is $(\sum a_i^2w_i)^{p/2}\langle\rho y,y\rangle/\sum a_i^pr_i$. Since $p/2\le1$,
$(\sum a_i^2w_i)^{p/2}\le\sum a_i^pw_i^{p/2}\le\max_i(w_i^{p/2}/r_i)\sum a_i^pr_i$, so the ratio is at
most $\max_iR(x_i,y)\le\yp(\kappa)$.

\ref{yp:d} Let $(x,y)$ be a maximiser of $R$ in $M_2$ with $\kappa=\kappa_\cZ$, and $x'\perp x$. Then
$w_1=|\langle x,y\rangle|^2\in(0,1)$, $w_2=1-w_1>0$. For $A=|x\rangle\langle x|+a|x'\rangle\langle
x'|$ and $B=|y\rangle\langle y|$ the $p$-th power of the ratio equals
$R(x,y)^p(1+\frac p2\frac{w_2}{w_1}a^2+o(a^2))(1-\frac{r_2}{r_1}a^p+o(a^p))$, which exceeds $R(x,y)^p$
for small $a>0$ because $p>2$. Hence $c_p^+(\diag(\kappa_\cZ,1))>\yp(\kappa_\cZ)$; apply
Theorem~\ref{thm:transfer}\ref{t:low}.
\end{proof}

\begin{conjecture}\label{conj:young}
For $1<p\le2$, $c_p^+(\rho)=\yp(\kappa_\cZ(\rho))$. In particular $c_2^+(\rho)=\cosh\LNC(\rho)$:
for positive operators, the $p=2$ Young inequality is violated exactly as much as Gardner's.
\end{conjecture}

Numerical optimisation over general positive $A,B$ in $M_2$, $M_3$ and $M_4$ agrees with $\yp$ to
about seven digits for $p\le2$, and shows a small excess over $\yp$ for $p>2$, in accordance
with \ref{yp:d}. By Proposition~\ref{prop:youngpos}, the theorem of Bikchentaev and Tikhonov
\cite{bikchentaev2005young} holds in every finite von Neumann algebra for $\varphi=\tau(\rho\,\cdot)$ with
$\rho\in\cM_{++}$; densities that are unbounded or have unbounded inverses are not covered.

For $p=2$ and the normalisation of item~8, an independent BFGS search over $A,B\ge0$ on $M_2$ (30
random starts; a test of the repository) comes within $0.1\%$ of $G-1$ and does not exceed it
($0.1546$ at $\delta=0.5$ and $1.2938$ at $\delta=0.9$, against $G-1=0.1547$ and $1.2942$), in accordance
with the case $c_2^+=G_\cZ$ of Conjecture~\ref{conj:young}: for positive operators the violation stays at
the order $\delta^2$, and the linear failure of item~8 is a feature of the extension to arbitrary
operators.

For $p=2$ and self-adjoint $B$ the constant is $(4G_\cZ^2-3)^{1/2}>G_\cZ$. So if Conjecture~\ref{conj:young}
holds, the positivity of $B$ cannot be weakened to self-adjointness without changing the constant.

\begin{proposition}[Item 31]\label{prop:youngsa}
For $A\in\cM$ and $B\in\cM_{sa}$,
\[
\varphi(|AB|)\le\bigl(\kappa_\cZ-1+\kappa_\cZ^{-1}\bigr)^{1/2}\varphi(|A|^2)^{1/2}\varphi(B^2)^{1/2}.
\]
The constant $(\kappa_\cZ-1+\kappa_\cZ^{-1})^{1/2}=(4G_\cZ^2-3)^{1/2}$ is optimal, already for $A\ge0$,
and in $M_n(\C)$ it is attained.
\end{proposition}

\begin{proof}
Let $AB=W|AB|$ be the polar decomposition. Since $|AB|=W^*AB$ and $\tau(WYW^*)\le\tau(Y)$ for
$Y\ge0$,
\[
\varphi(|AB|)=\tau\bigl(A\rho^{1/2}\cdot\rho^{-1/2}B\rho W^*\bigr)
\le\|A\rho^{1/2}\|_2\,\|\rho^{-1/2}B\rho W^*\|_2\le\varphi(|A|^2)^{1/2}\,\tau(\rho^2B\rho^{-1}B)^{1/2}.
\]
On $L^2(\cM,\tau)$ let $Jh=h^*$, so that $JL_xJ=R_{x^*}$. For $B=B^*$ and the positive operator
$Y=L_{\rho^2}R_{\rho^{-1}}$ we have $\langle YB,B\rangle=\langle JYJB,B\rangle$. Hence
$\tau(\rho^2B\rho^{-1}B)=\langle YB,B\rangle=\langle g(L_\rho,R_\rho)B,B\rangle$ with
$g(s,t)=\frac12(s^2/t+t^2/s)$, while $\varphi(B^2)=\langle h(L_\rho,R_\rho)B,B\rangle$ with
$h(s,t)=\frac12(s+t)$. On $[\|\rho^{-1}\|^{-1},\|\rho\|]^2$ we have
$g/h=s/t+t/s-1\le\kappa-1+\kappa^{-1}$, which gives the bound with $\kappa$ in place of
$\kappa_\cZ$. (R) holds with $\sigma_z(A,B)=(z^{1/2}A,z^{1/2}B)$, so Lemma~\ref{lem:resc} gives the
bound with $\kappa_\cZ$. For the lower bound let $\rho=\diag(\kappa,1)$ in $M_2$, $A=\diag(\kappa^{-1},\kappa)$
and $B=\sigma$. Then $|AB|=\diag(\kappa,\kappa^{-1})$, $\varphi(|AB|)=\varphi(A^2)=\kappa^2+\kappa^{-1}$
and $\varphi(B^2)=\kappa+1$, so the squared ratio is
$(\kappa^2+\kappa^{-1})/(\kappa+1)=\kappa-1+\kappa^{-1}$. Theorem~\ref{thm:transfer} and
Remark~\ref{rem:attain} complete the proof, and $\kappa+\kappa^{-1}=4G_\cZ^2-2$.
\end{proof}

\begin{remark}\label{rem:youngstar}
Since $|AB|=||A|B|$ and $\varphi(|A|^2)$ depends only on $|A|$, $c_2^+$ is also the supremum over
arbitrary $A$ and $B\ge0$; only the positivity of $B$ matters. The constant is conjecturally $G_\cZ$ for
$B\ge0$, it is $(4G_\cZ^2-3)^{1/2}>G_\cZ$ for $B=B^*$ (Proposition~\ref{prop:youngsa}), and it is
$\kappa_\cZ^{1/2}$ for arbitrary $B$ (Theorem~\ref{thm:young}). The case $p=2$ of
Conjecture~\ref{conj:young} is equivalent to
\begin{equation}\label{eq:youngstar}
\varphi(BXB)^2\le G_\cZ^2\,\varphi(XB^2X)\,\varphi(B^2)\qquad(X\in\cM_{sa},\ B\in\cM_+).
\end{equation}
Indeed, by continuity we may take $B$ invertible; put $\Gamma=B^{-1}\rho B^{-1}$. As $A$ runs
through $\cM$, $Y=AB$ runs through $\cM$ and $\varphi(|A|^2)=\tau(\Gamma|Y|^2)$, so
$\sup_A\varphi(|AB|)^2/\varphi(|A|^2)=\sup_{P\in\cM_+}\tau(\rho P)^2/\tau(\Gamma P^2)$. On the real
Hilbert space of self-adjoint elements of $L^2(\cM,\tau)$, $\tau(\Gamma P^2)=\langle\Lambda P,P\rangle$
with $\Lambda=\frac12(L_\Gamma+R_\Gamma)$. So the supremum over $P\in\cM_{sa}$ is
$\langle\Lambda^{-1}\rho,\rho\rangle$. It is attained at
$\Lambda^{-1}\rho=2\int_0^\infty e^{-t\Gamma}\rho e^{-t\Gamma}\,dt\ge0$, hence it is also the
supremum over $P\ge0$. Substituting $P=BXB$ gives the claim. For rank-one
$B=|y\rangle\langle y|$ with $\|y\|=1$, \eqref{eq:youngstar} reads
$\langle Xy,y\rangle^2\langle\rho y,y\rangle\le G_\cZ^2\langle\rho Xy,Xy\rangle$. In $M_n(\C)$ this
follows from the Cauchy--Schwarz inequality
$\langle Xy,y\rangle^2\le\langle\rho Xy,Xy\rangle\langle\rho^{-1}y,y\rangle$ and the Kantorovich
inequality. Numerically (not used in the proofs), the supremum of the quotient of the two sides of
\eqref{eq:youngstar} is $G_\cZ^2$ to seven digits in $M_2(\C)$ and $M_3(\C)$.
\end{remark}

\subsection{Monotonicity}\label{ssec:mono}
For $p>1$ let
\[
m_p(\rho)=\sup\Bigl\{\frac{\varphi(A^p)}{\varphi(B^p)}:0\le A\le B,\ B\ne0\Bigr\},\qquad
m_{\exp}(\rho)=\sup\Bigl\{\frac{\varphi(e^A)}{\varphi(e^B)}:A\le B\text{ in }\cM_{sa}\Bigr\}.
\]

\begin{theorem}[Item 10]\label{thm:mono2}
$\displaystyle m_2(\rho)=\frac1{1-\gamma(\rho)^2}=\frac{(G_\cZ+1)^2}{4G_\cZ}$, where
$\gamma(\rho)=\tanh^2(\LNC(\rho)/2)$, which is the constant of Theorem~\ref{thm:ba} if $\cM\ne\C1$. In
$M_n(\C)$ the supremum is attained.
\end{theorem}

\begin{proof}
If $\cM=\C1$, then $\re\varphi(AC)\ge0$ for $A,C\ge0$, so $\varphi(A^2)\le\varphi(B^2)$ whenever $0\le A\le B$, and
$m_2=1=(G_\cZ+1)^2/(4G_\cZ)$ since $G_\cZ=1$; let $\cM\ne\C1$.
Let $0\le A\le B$, $C=B-A\ge0$, $a=\varphi(A^2)$, $c=\varphi(C^2)$. By Theorem~\ref{thm:ba},
$\varphi(B^2)=a+2\re\varphi(AC)+c\ge a-2\gamma\sqrt{ac}+c\ge(1-\gamma^2)a$. For the converse, in
$M_2$ with $\rho=\diag(\kappa,1)$, take the pair $A=|a\rangle\langle a|$, $B'=|b\rangle\langle b|$ of
the proof of Theorem~\ref{thm:ba}, for which $\re\varphi(AB')=-\gamma\varphi(A^2)$ and
$\varphi(B'^2)=\varphi(A^2)$, and put $B=A+\gamma B'$. Then $\varphi(B^2)=(1-\gamma^2)\varphi(A^2)$.
(R) holds with $\sigma_z(A,B)=(z^{1/2}A,z^{1/2}B)$, and Theorem~\ref{thm:transfer} applies. Finally
$1-\gamma^2=4G_\cZ/(G_\cZ+1)^2$ since $\gamma=(G_\cZ-1)/(G_\cZ+1)$.
\end{proof}

\begin{remark}[Item 10: an explicit extremal family on $M_2$]\label{rem:g10}
Normalise $a+b=2$, so $\rho=\mathrm{diag}(1+\delta,1-\delta)$, and take $A=uu^*$, $B=A+tvv^*$, $t\ge0$, with
$u=(\cos\alpha,\sin\alpha)$, $v=(-\cos\alpha,\sin\alpha)$ and $y=-\cos2\alpha$. Then $u^*v=y$, $v^*\rho u=y-\delta$ and
$u^*\rho u=v^*\rho v=1-\delta y$, so
\[
  \varphi(A^2)-\varphi(B^2)=-2t\Re\bigl((u^*v)(v^*\rho u)\bigr)-t^2v^*\rho v=2ty(\delta-y)-t^2(1-\delta y).
\]
For $0<y<\delta$ the best $t=y(\delta-y)/(1-\delta y)>0$ gives $\cL-\cR=\Gamma:=y^2(\delta-y)^2/(1-\delta y)$ and
\[
\frac{\cL-\cR}\cR=\frac{\Gamma}{1-\delta y-\Gamma}=g(y):=\frac{y^2(\delta-y)^2}{(1-2\delta y+y^2)(1-y^2)} .
\]
Let $y=\tan(\psi/2)$, so that $\delta=\sin\psi=2y/(1+y^2)$ and $0<y<\delta$. Then
$\delta-y=y(1-y^2)/(1+y^2)$ and $1-2\delta y+y^2=(1-y^2)^2/(1+y^2)$, hence
\[
g\bigl(\tan\tfrac\psi2\bigr)=\frac{y^4}{(1+y^2)(1-y^2)}=\frac{(G-1)^2}{4G},\qquad\text{since } G=\sec\psi=\frac{1+y^2}{1-y^2}.
\]
By Theorem~\ref{thm:mono2}, $(G-1)^2/(4G)=m_2(\rho)-1$ is the supremum, so
$\max_{0<y<\delta}g(y)=(G-1)^2/(4G)$ and $V_n(\delta)=(G-1)^2/(4G)=\delta^4/16+O(\delta^6)$ for item~10 in every
dimension $n\ge2$, attained at $y=\tan(\psi/2)$. The numerical search of Section~\ref{sec:numerics} had
suggested that this family is extremal: it reproduces $\max_yg$ to the digits of Table~\ref{tab:values},
and the values for $n=3,4$ agree with $n=2$ to $\TCNDevMon$. The theorem proves it. The leading order comes from $y\approx\delta/2$. The pair $(A,\gamma B'+A)$ in the proof of
Theorem~\ref{thm:mono2} is a second extremal family of the same type, a rank-one $A$ perturbed by a
rank-one positive operator.
\end{remark}

Let $\mpp(\kappa)=m_p(\diag(\kappa,1))$ computed in $M_2(\C)$, and $C_p=(p-1)^{p-1/p}/p^{p+1}$. For
$h>1$ and $p>1$ let
\[
K(h,p)=\frac{h^p-h}{(p-1)(h-1)}\Bigl(\frac{p-1}p\,\frac{h^p-1}{h^p-h}\Bigr)^p,\qquad
S(h)=\frac{(h-1)h^{1/(h-1)}}{e\ln h},
\]
and $K(1,p)=S(1)=1$. These are the generalised Kantorovich constant (Ky Fan's constant, see
\cite{Bou05}) and the Specht ratio. We call $\cM$ \emph{approximately finite-dimensional} (AFD)
if it contains an increasing net of finite-dimensional $*$-subalgebras containing $1$ whose union is
weakly dense. Examples are finite-dimensional algebras and the hyperfinite II$_1$ factor.

\begin{proposition}[Item 32]\label{prop:monop}
\begin{enumerate}
\item\label{m:a} $m_p(z\rho)=m_p(\rho)$ and $m_{\exp}(z\rho)=m_{\exp}(\rho)$ for $z\in\cZ_{++}$.
\item\label{m:b} $p\mapsto m_p(\rho)$ is nondecreasing, and $m_{\exp}(\rho)=\sup_pm_p(\rho)=\lim_{p\to\infty}m_p(\rho)$.
\item\label{m:c} $\mpp(\kappa_\cZ)\le m_p(\rho)\le\kappa_\cZ$, and $\mpp(\kappa)>1$ for $\kappa>1$.
Moreover $\mpp(\kappa)\ge C_p\kappa^{1-1/p}(1+o(1))$ as $\kappa\to\infty$.
\item\label{m:d} For $\kappa_\cZ\ge e$,
\[
\frac{\kappa_\cZ}{e\,(\ln\kappa_\cZ+\ln\ln\kappa_\cZ)(1+1/\ln\kappa_\cZ)}\le m_{\exp}(\rho)\le\kappa_\cZ;
\]
in particular $m_{\exp}(\rho)\ge(1-o(1))\kappa_\cZ/(e\ln\kappa_\cZ)$ as $\kappa_\cZ\to\infty$.
\item\label{m:e} If $\cM$ is AFD, then $m_p(\rho)\le K(\kappa_\cZ^{1/p},p)$ and
$m_{\exp}(\rho)\le S(\kappa_\cZ)$. Consequently, on AFD algebras, as $\kappa_\cZ\to\infty$,
\[
C_p\kappa_\cZ^{1-1/p}(1+o(1))\le m_p(\rho)\le\frac{(p-1)^{p-1}}{p^p}\,\kappa_\cZ^{1-1/p}(1+o(1)),
\qquad m_{\exp}(\rho)=(1+o(1))\frac{\kappa_\cZ}{e\ln\kappa_\cZ}.
\]
\end{enumerate}
\end{proposition}

\begin{proof}
\ref{m:a} Use $(A,B)\mapsto(z^{1/p}A,z^{1/p}B)$, resp.\ $(A+\ln z,B+\ln z)$.
\ref{m:b} If $1<p<r$ and $0\le A\le B$, then $A^{p/r}\le B^{p/r}$ by the L\"owner--Heinz
inequality, so $\varphi(A^p)/\varphi(B^p)\le m_r$. If $0<A\le B$, then $p\ln A\le p\ln B$ and
$\varphi(A^p)/\varphi(B^p)=\varphi(e^{p\ln A})/\varphi(e^{p\ln B})\le m_{\exp}$; singular $A$ is
handled by $A+\varepsilon$. Conversely, for $A\le B$ and large $p$, $0\le1+A/p\le1+B/p$ and
$(1+A/p)^p\to e^A$ in norm, so $\varphi(e^A)\le(\sup_pm_p)\varphi(e^B)$.
\ref{m:c} The lower bound is Theorem~\ref{thm:transfer}\ref{t:low}, and $\mpp(\kappa)>1$ for $\kappa>1$
by the theorem of Bikchentaev and Tikhonov \cite{bikchentaev2007mono}. For the upper bound, $\tau(\one_QA^p)\le\tau(\one_QB^p)$
for every central projection $Q$ by the monotonicity of $\tau(f(\cdot))$ for increasing $f$ in the finite
algebra $Q\cM$ \cite{FK86}; approximating central positive elements by positive combinations of central
projections, $\tau(uA^p)\le\tau(uB^p)$, and by \eqref{eq:sandwich}
$\varphi(A^p)\le\tau(uA^p)\le\tau(uB^p)\le\kappa_\cZ\tau(lB^p)\le\kappa_\cZ\varphi(B^p)$. For the
asymptotics, let $\rho=\diag(\kappa,1)$, $B=\diag(\beta,1)$, $w=(c',s')$ a unit vector and
$A=\lambda|w\rangle\langle w|$ with $\lambda=\langle B^{-1}w,w\rangle^{-1}$, the largest $\lambda$ with
$A\le B$. Then
\begin{equation}\label{eq:ansatz}
\frac{\varphi(A^p)}{\varphi(B^p)}=\frac{\kappa c'^2+s'^2}{(c'^2/\beta+s'^2)^p(\kappa\beta^p+1)}.
\end{equation}
With $c'^2=\theta\beta$, $\kappa\beta^p=y$, $\theta=y=1/(p-1)$ and $\kappa\to\infty$ this is
$\theta y^{1/p}\kappa^{1-1/p}(1+\theta)^{-p}(1+y)^{-1}(1+o(1))=C_p\kappa^{1-1/p}(1+o(1))$.
\ref{m:d} The upper bound follows from \ref{m:b} and \ref{m:c}. For the lower bound take
$\beta=e^{-\sigma/p}$ in \eqref{eq:ansatz} with fixed $\sigma\ge1$ and $c'^2=1/\sigma$; as $p\to\infty$,
$(c'^2e^{\sigma/p}+s'^2)^p\to e^{\sigma c'^2}=e$, so by \ref{m:b} and \ref{m:c}
$m_{\exp}(\rho)\ge\sup_p\mpp(\kappa_\cZ)\ge(\kappa_\cZ/\sigma+1-1/\sigma)/(e(\kappa_\cZ e^{-\sigma}+1))$.
Choose $e^{-\sigma}=1/(\kappa_\cZ\ln\kappa_\cZ)$, which gives $\sigma\ge1$ for $\kappa_\cZ\ge e$.

\ref{m:e} By \ref{m:a} and Proposition~\ref{prop:char}\ref{ch:resc} we may replace $\rho$ by
$z_0\rho$ and assume $\kappa(\rho)=\kappa_\cZ$. Let $\mu_-=\min\sigma(\rho)$ and $\mu_+=\max\sigma(\rho)$. We
first show that
\begin{equation}\label{eq:bourintrace}
\varphi(A^p)\le K(\kappa^{1/p},p)\,\tau\bigl((\rho^{1/2p}A\rho^{1/2p})^p\bigr)\qquad(A\in\cM_+)
\end{equation}
for every finite-dimensional $\cM$ and every $\rho$ with $\sigma(\rho)\subset[\mu_-,\mu_+]$, where
$\kappa=\mu_+/\mu_-$. Then $\cM=\bigoplus_iM_{n_i}(\C)$ and $\tau=\sum_it_i\Tr_i$ with $t_i>0$. Bourin's
reverse of Araki's inequality \cite[Thm.~2]{Bou05} states the following. If $A'\ge0$ and $Z>0$ in
$M_N(\C)$, $Z$ has extremal eigenvalues $a,b$, and $p>1$, then there is a unitary $V$ with
$A'^pZ^pA'^p\le K(b/a,p)\,V(A'ZA')^pV^*$ (the constant there is written $K(a,b,p)=K(b/a,p)$). Apply
it in $M_{n_i+2}(\C)$ to $A'=A_i^{1/2}\oplus0$ and $Z=\rho_i^{1/p}\oplus\diag(\mu_-^{1/p},\mu_+^{1/p})$,
whose extremal eigenvalues are $\mu_-^{1/p}$ and $\mu_+^{1/p}$, and take traces. This gives
$\Tr_i(\rho_iA_i^p)=\Tr_i(A_i^{p/2}\rho_iA_i^{p/2})\le K(\kappa^{1/p},p)\Tr_i((A_i^{1/2}\rho_i^{1/p}A_i^{1/2})^p)$,
and the last trace is $\Tr_i((\rho_i^{1/2p}A_i\rho_i^{1/2p})^p)$. Multiplying by $t_i$ and summing
gives \eqref{eq:bourintrace}.
If $\cM$ is AFD with finite-dimensional subalgebras $N_\alpha$, let $E_\alpha$ be the
$\tau$-preserving conditional expectations onto $N_\alpha$. Then $E_\alpha(x)\to x$ in
$\|\cdot\|_2$, hence strongly, with $\|E_\alpha(x)\|\le\|x\|$, and $\sigma(E_\alpha(\rho))\subset[\mu_-,\mu_+]$.
Apply \eqref{eq:bourintrace} in $N_\alpha$ to $E_\alpha(\rho)$ and $E_\alpha(A)$, and pass to the
limit; products, continuous functional calculus and $\tau$ are strongly continuous on bounded sets.
Now let $0\le A\le B$. By \eqref{eq:bourintrace}, by the monotonicity of $X\mapsto\tau(f(X))$ for
increasing $f$ \cite{FK86}, and by \eqref{eq:alt} with $r=p$,
\[
\varphi(A^p)\le K\tau\bigl((\rho^{1/2p}A\rho^{1/2p})^p\bigr)\le K\tau\bigl((\rho^{1/2p}B\rho^{1/2p})^p\bigr)
\le K\tau(\rho^{1/2}B^p\rho^{1/2})=K\varphi(B^p),
\]
with $K=K(\kappa_\cZ^{1/p},p)$. For $h>1$ and $x=h^{1/p}$, as $p\to\infty$ we have
$(h-x)/((p-1)(x-1))\to(h-1)/\ln h$, $((p-1)/p)^p\to e^{-1}$ and
$((h-1)/(h-x))^p\to h^{1/(h-1)}$. Hence $K(h^{1/p},p)\to S(h)$, and by \ref{m:b},
$m_{\exp}=\lim_pm_p\le S(\kappa_\cZ)$. The asymptotics follow from \ref{m:c}, \ref{m:d},
$K(h,p)\sim\frac{(p-1)^{p-1}}{p^p}h^{p-1}$ as $h\to\infty$, and $S(h)\sim h/(e\ln h)$.
\end{proof}

\begin{remark}\label{rem:fujii}
Fujii and Seo \cite{FS05} proved related reverse inequalities of Araki's inequality with the
constant $K$. In Bourin's inequality the unitary $V$ cannot be omitted: the operator inequality
$A^{p/2}\rho A^{p/2}\le K(\kappa^{1/p},p)(A^{1/2}\rho^{1/p}A^{1/2})^p$ fails in general. For instance,
let $p=2$, $\rho=\diag(4,1)$, so that $K(\kappa^{1/2},2)=K(2,2)=\frac98$, and $A=C^2$ with
$C=\bigl(\begin{smallmatrix}1&-1/2\\-1/2&1\end{smallmatrix}\bigr)>0$. Then the $(2,2)$ entry of
$\frac98(C\rho^{1/2}C)^2-C^2\rho C^2$ equals $-\frac12$; this is an exact computation. We do not know whether the trace inequality
\eqref{eq:bourintrace} holds in every finite von Neumann algebra, for instance in a II$_1$ factor
that is not hyperfinite. If it does, Proposition~\ref{prop:monop}\ref{m:e} holds for all finite
$\cM$.
\end{remark}

\subsubsection*{Weighted monotonicity}
\begin{proposition}[Item 25]\label{prop:wmon}
For item~25 ($\varphi(A^4)\le\varphi(AB^2A)$ for $0\le A\le B$), rank-one $A$ gives no violation, and an explicit
construction with $A$ of rank two and $B-A$ of rank one gives
\[
  V_n\ge g_{25}(\delta)=\frac{\delta^8}{4096}\bigl(1+O(\delta^2)\bigr)\qquad\text{in every dimension }n\ge2 .
\]
Conversely, $V_n\le\kappa-1=2\delta/(1-\delta)$, and in every finite von Neumann algebra the constant
$c_{25}(\rho)=\sup\varphi(A^4)/\varphi(AB^2A)$ over $0\le A\le B$ satisfies $c_{25}(\rho)\le\kappa_\cZ(\rho)$.
\end{proposition}
\begin{proof}
\emph{Upper bound.} For a finite trace $\tau$ and $0\le A\le B$ we have $\tau(A^4)\le\tau(AB^2A)$ (cf.~\cite{hoa2010}):
congruence by $A^{1/2}$ twice gives $A^3\le ABA$, so $\tau(A^4)\le\tau(A^{3/2}BA^{3/2})=\tau(A^2\cdot AB)\le
\tau(A^4)^{1/2}\tau(AB^2A)^{1/2}$ by the Cauchy--Schwarz inequality for $\tau$. Since also $AB^2A\ge0$,
$\varphi(A^4)\le\|\rho\|\tau(A^4)\le\|\rho\|\tau(AB^2A)\le\kappa(\rho)\varphi(AB^2A)$. Condition (R) holds with
$\sigma_z(A,B)=(z^{1/4}A,z^{1/4}B)$, which preserves $0\le A\le B$ and multiplies $A^4$ and $AB^2A$ by $z$, so
Lemma~\ref{lem:resc} gives $c_{25}\le\kappa_\cZ$; on $M_n$ this is $V_n\le\kappa-1$.
\emph{Lower bound.} If $A=\alpha uu^*$ has rank one and $A\le B$, then $u^*Bu\ge\alpha$ and $AB^2A=\alpha^2\|Bu\|^2uu^*\ge\alpha^2(u^*Bu)^2uu^*\ge
A^4$, so every positive $\varphi$ satisfies the inequality. For the construction take $\rho$ with $\Tr \rho=1$ and eigenvalues
$(1\pm\delta)/2$, and, in an orthonormal basis at the angle $\vartheta$ to its eigenbasis, $A=\mathrm{diag}(1,s)$, $0<s<1$,
and $B=A+tvv^*$ with $v=(\cos\phi,\sin\phi)$; in this basis $\rho$ has the entries $\rho_{11,22}=\frac12(1\pm\delta\cos2\vartheta)$
and $\rho_{12}=\frac\delta2\sin2\vartheta$. Then $\cR-\cL=tT_1+t^2T_2$ with $T_1=\varphi(A^2P_vA+AP_vA^2)$, $T_2=\varphi(AP_vA)$,
$P_v=vv^*$, and whenever $T_1<0$ the best $t=-T_1/(2T_2)$ gives $(\cL-\cR)/\cR=\Gamma/(\cL-\Gamma)$, $\Gamma=T_1^2/(4T_2)$. Let $g_{25}(\delta)$
be the maximum over $(\vartheta,s,\phi)$; it is attained, and by embedding $\C^2\subset\C^n$ (with $A$, $B$ extended by $0$)
it is a lower bound in every dimension. For the leading order put $\vartheta=\pi/4$, $\sqrt s=\sigma\delta$ and $\tan\phi=-\varrho
s^{-3/2}$. Then $T_1=\cos^2\phi\,(1+\varrho^2-(1+s)\varrho/\sigma)$, $T_2=\cos^2\phi\,(\frac12+\frac{\varrho^2}{2s}-\frac\varrho\sigma)$ and
$\cos^2\phi=s^3/(s^3+\varrho^2)$, so that
\[
  \frac{\cL-\cR}\cR=\bigl(\sigma^2(\varsigma-\varsigma^2)-\sigma^4\bigr)^2\delta^8+O(\delta^{10}),\qquad\varsigma=\sigma/\varrho,
\]
which is maximal for $\varsigma=\frac12$, $\sigma^2=\frac18$, with the value $\delta^8/4096$.
\end{proof}

\noindent The numerical maximisation gives $g_{25}(0.1)/(0.1^8/4096)=\TCWmonLead$. The same scaling for a general
exponent $r\ge0$ in $\varphi(A^{2r+2})\le\varphi(A^rB^2A^r)$ gives the order $\delta^{4(r+1)}$; for $r=0$ the supremum of
the leading coefficient is $\frac1{16}$, that of item~10.

\subsection{Projection inequalities}
For projections $P,Q\in\cM$ let
\[
c_{PQP}=\sup\frac{\varphi(PQP)}{\varphi(Q)},\quad c_{PQ}=\sup_{QP\ne0}\frac{\varphi(PQP)}{\varphi(QPQ)},\quad
c_s=\sup_{QP\ne0}\frac{\varphi(s(PQP))}{\varphi(s(QPQ))},\quad c_\vee=\sup\frac{\varphi(P\vee Q)}{\varphi(P)+\varphi(Q)}.
\]
They correspond to the characterisations $\varphi(PQP)\le\varphi(Q)$, $\varphi(PQP)=\varphi(QPQ)$ and
$\varphi(s(PQP))=\varphi(s(QPQ))$ in Bikchentaev's series \cite{bikchentaev2009rm,bikchentaev2010smj,bikchentaev2011mn,Bik15} and to the
inequality $\varphi(P\vee Q)\le\varphi(P)+\varphi(Q)$, the formal specialisation of Tikhonov's
subadditivity \cite{tikhonov2005} to $f=\one_{(0,\infty)}$ and projections (for traces it follows from
Kaplansky's formula). The constant for
$\varphi(|P-Q|)\le\varphi(P+Q)$ is $c_{\rm proj}=G_\cZ$ by Theorem~\ref{thm:subadd}.

\begin{theorem}[Items 17, 23 and 33]\label{thm:proj}
\begin{enumerate}
\item\label{p:a} $P\rho P\le G_\cZ^2\rho$ for every projection $P$, and $c_{PQP}=G_\cZ^2=\frac{(1+\kappa_\cZ)^2}{4\kappa_\cZ}$;
in $M_n(\C)$ the supremum is attained (on $M_n(\C)$ the operator inequality is contained in
\cite[Cor.~1.5]{Bou06}).
\item\label{p:b} $c_{PQ}=c_s=\kappa_\cZ$; the corresponding infima are $\kappa_\cZ^{-1}$. If
$\cM=M_n(\C)$ and $\rho$ is not scalar, neither $c_{PQ}$ nor $c_s$ is attained.
\item\label{p:c} $\varphi(P\vee Q)\le\min\{\varphi(Q)+\kappa_\cZ\varphi(P),\varphi(P)+\kappa_\cZ\varphi(Q)\}$
and $c_\vee=\frac{1+\kappa_\cZ}2$.
\end{enumerate}
\end{theorem}

\begin{proof}
\ref{p:a} By Lemma~\ref{lem:pTp} for $T=z_0\rho$ with $z_0$ as in Lemma~\ref{lem:resc},
$Pz_0\rho P\le\cosh^2\OLC(z_0\rho)z_0\rho=G_\cZ^2z_0\rho$; multiply by $z_0^{-1/2}$ on both sides. Then
$\varphi(PQP)=\tau(QP\rho PQ)\le G_\cZ^2\tau(Q\rho Q)=G_\cZ^2\varphi(Q)$. For the lower bound, in $M_2$ with
$\rho=\diag(\kappa,1)$ let $P=|w\rangle\langle w|$, $w=(1,1)/\sqrt2$, and $Q=|q\rangle\langle q|$ with
$q=\rho^{-1}w/\|\rho^{-1}w\|$. Then $PQP=|\langle w,q\rangle|^2P$ and the ratio is
$\langle\rho^{-1}w,w\rangle\langle\rho w,w\rangle=(\kappa+1)^2/(4\kappa)$, the Kantorovich
expression, which here attains its maximum $G^2$. Apply
Theorem~\ref{thm:transfer}\ref{t:low}.

\ref{p:b} With $X=QP$, $|X|^2=PQP$ and $|X^*|^2=QPQ$, so $c_{PQ}\le\kappa_\cZ$ by
Theorem~\ref{thm:modulus}. If $X=w|X|$, then $s(PQP)=w^*w$, $s(QPQ)=ww^*$, and
$\varphi(w^*w)\le\kappa_\cZ\varphi(ww^*)$ by Theorem~\ref{thm:modulus} with $U=w^*$, $Y=ww^*$. For
the lower bound, in $M_2$ with $\rho=\diag(\kappa,1)$, let $P,Q$ be the projections onto
$(\cos\theta,\sin\theta)$ and $(\sin\theta,\cos\theta)$, $0<\theta<\pi/4$. Then $PQP=\sin^22\theta\,P$,
$QPQ=\sin^22\theta\,Q$, $s(PQP)=P$, $s(QPQ)=Q$, and both ratios equal
$(\kappa\cos^2\theta+\sin^2\theta)/(\kappa\sin^2\theta+\cos^2\theta)\to\kappa$ as $\theta\to0$. Non-attainment:
equality in $\tau(\rho PQP)\le\|\rho\|\tau(PQP)=\|\rho\|\tau(QPQ)\le\kappa\tau(\rho QPQ)$ forces
$\operatorname{ran}PQP\subset E_M$ and $\operatorname{ran}QPQ\subset E_m$, the eigenspaces of
$M=\max\sigma(\rho)$ and $m=\min\sigma(\rho)$. Take $x\in\operatorname{ran}Q$ with $Px\ne0$, put
$T=QPQ$ and $y=Tx\in E_m$; $y\ne0$ since $\langle Tx,x\rangle=\|Px\|^2$. Then $Py=PQP x\in E_M\perp
E_m$, so $\|Py\|^2=\langle Py,y\rangle=0$, whereas $\langle Py,y\rangle=\langle T^3x,x\rangle=\|T^{3/2}x\|^2>0$. For $c_s$, equality in
$\tau(\rho\,s(PQP))\le\|\rho\|\tau(s(PQP))=\|\rho\|\tau(s(QPQ))\le\kappa\tau(\rho\,s(QPQ))$ forces
$\operatorname{ran}s(PQP)=\operatorname{ran}PQP\subset E_M$ and $\operatorname{ran}s(QPQ)=\operatorname{ran}QPQ\subset E_m$,
which gives the same contradiction.

\ref{p:c} By Kaplansky's formula \cite[Ch.~6]{KR2} there is a partial isometry $w$ with
$ww^*=P\vee Q-Q$ and $w^*w=P-P\wedge Q\le P$. By Theorem~\ref{thm:modulus},
$\varphi(P\vee Q)=\varphi(Q)+\varphi(ww^*)\le\varphi(Q)+\kappa_\cZ\varphi(P)$; symmetrise and average to
get $c_\vee\le(1+\kappa_\cZ)/2$. In $M_2$ with $\rho=\diag(\kappa,1)$, $P=e_{22}$ and $Q_\theta$ the
projection onto $(\sin\theta,\cos\theta)$ give $P\vee Q_\theta=1$ for $\theta\ne0$ and ratio
$(\kappa+1)/(1+\kappa\sin^2\theta+\cos^2\theta)\to(1+\kappa)/2$. Apply
Theorem~\ref{thm:transfer}\ref{t:low}.
\end{proof}

The five projection constants $c_{\rm proj},c_{PQP},c_\vee,c_{PQ},c_s$ take four different values,
$G_\cZ\le G_\cZ^2\le\frac{1+\kappa_\cZ}2\le\kappa_\cZ$; near the centre they behave like
$1+\frac12\LNC^2$, $1+\LNC^2$, $1+\LNC$ and $1+2\LNC$. The quadratic ones come from the Krein and
Wielandt inequalities, the linear ones from comparing equivalent projections.

\begin{proposition}[Item 22]\label{prop:item22}
For all projections $P,Q\in\cM$, $\varphi(\Re PQ)\le G_\cZ\,\varphi(|PQ|)$, and the constant is optimal. In
$M_n(\C)$ it is $G$ and attained by rank-one projections, so $V_n=G-1$ for item~22.
\end{proposition}

\begin{proof}
$\varphi(\Re PQ)=\re\varphi(PQ)\le|\varphi(PQ)|\le G_\cZ\varphi(|PQ|)$ by Theorem~\ref{thm:gardner}. For the lower
bound let $\rho=\diag(a,b)$ in $M_2$, $v=\sqrt{b/(a+b)}\,e_1+\sqrt{a/(a+b)}\,e_2$, $u=\rho v/\|\rho v\|$,
$P=uu^*$ and $Q=vv^*$. Then $PQ=(u^*v)\,uv^*$ with $u^*v=\langle\rho v,v\rangle/\|\rho v\|>0$, so
$|PQ|=(u^*v)Q$, $\varphi(PQ)=(u^*v)\|\rho v\|$ and $\varphi(|PQ|)=(u^*v)\langle\rho v,v\rangle$. The ratio is
$\|\rho v\|/\langle\rho v,v\rangle=\sqrt{ab}\,(a+b)/(2ab)=G(\rho)$, the equality case of
Lemma~\ref{lem:krein}. The constant is a violation constant (two projections,
$F_1=\frac12(X_1X_2+X_2X_1)$, $F_2=|X_1X_2|$), and Theorem~\ref{thm:transfer}\ref{t:low} and
Remark~\ref{rem:attain} complete the proof.
\end{proof}

\noindent The violation constants of items 17, 18, 22 and 23 are functions of $\kappa_\cZ$; on $M_n$ they do
not depend on the dimension. For items 19 and 20, the remaining projection inequalities of~\cite{bikchentaev2010smj}
and~\cite{bikchentaev2011mn} the normalisation contains $\varphi(1)$, and we determine the envelopes on
$M_n$.

\begin{theorem}[Items 19, 20 and 23 on $M_n$]\label{thm:projP}
Let $\rho>0$ in $M_n$, $n\ge2$, with $\OMC(\rho)=\delta$, and let $P,Q$ be projections. With $\cN$ as in
Table~\ref{tab:items}:
\begin{enumerate}
\item[(19)] $V_n=\delta$ for $\varphi(P+Q)\le\varphi(1+P\wedge Q)$ if $n$ is even, and $V_n=(n-1)\delta/(n-\delta)$ if $n$
is odd;
\item[(20)] for $\varphi(e^{QPQ})\le\varphi(e^P)$ and even $n$, $V_n=\max_{0<g<\pi}\lambda(\delta,g)$, where $\lambda(\delta,g)$ is
the larger root of the quadratic~\eqref{eq:quad} below, and $V_n=\frac{(e-1)^2}{e(e+1)}\delta^2+O(\delta^3)$;
\item[(23)] the value $V_n=\delta/(1-\delta)$ of $\varphi(P\vee Q)\le\varphi(P)+\varphi(Q)$
(Theorem~\ref{thm:proj}\ref{p:c}) is approached but not attained.
\end{enumerate}
The value is attained in 20 by rank-one $P,Q$, and approached in 19 and 23 as $P\ne Q$ tend to the same
eigenvector of $\rho$. For odd $n$, $V_n$ in item~20 lies between the $M_2$ value and the same maximum with
$(1-\delta)/(n-1)$ added to $\beta$ in~\eqref{eq:quad}.
\end{theorem}

\begin{proof}
\emph{Step 1: reduction to rank-one projections on $M_2$.} By Lemma~\ref{lem:halmos}, $\C^n=\bigoplus_jH_j$ with
$\dim H_j\in\{1,2\}$, every $H_j$ invariant under $P$ and $Q$. Every operator in $\cL$, $\cR$ and $\cN$ (namely $P+Q$,
$1+P\wedge Q$, $e^{QPQ}$, $e^P$, $P\vee Q$) is then block diagonal, with the block
obtained from $P_j,Q_j$; for the lattice operations because $\mathrm{ran}\,P\cap\mathrm{ran}\,Q=\bigoplus_j(\mathrm{ran}\,
P_j\cap\mathrm{ran}\,Q_j)$, the projections onto the $H_j$ commuting with $P$ and $Q$. Hence $\cL=\sum_jL_j$, $\cR=\sum_jR_j$ and
$\cN=\sum_jN_j$, where $\cL_j$ is the same expression for $P_j,Q_j$ and the functional $\Tr(\rho_j\,\cdot\,)$ on $H_j$, with the
compression $\rho_j$ of $\rho$ to $H_j$; its spectrum lies in $[b,a]$, so $\OMC(\rho_j)\le\delta$. On a one-dimensional $H_j$,
$P_j,Q_j\in\{0,1\}$ commute and $\cL_j\le \cR_j$ in all three items ($p+q\le1+pq$, $e^{pq}\le e^p$, $\max(p,q)\le p+q$).
On a two-dimensional $H_j$, $P_j\ne Q_j$ are rank-one projections. Since $\cN_j\ge0$, and $\cL_j\le \cR_j$ wherever $\cN_j=0$,
\[
  \frac{\cL-\cR}\cN\le\max_j\frac{\cL_j-\cR_j}{\cN_j},
\]
and the quotients do not change when $\rho_j$ is scaled. It therefore suffices to bound $(\cL-\cR)/\cN$ for
$\rho=\frac12(1+t\sigma_3)$, $0\le t\le\delta$, and $P=\frac12(1+\mu\cdot\sigma)$, $Q=\frac12(1+\nu\cdot\sigma)$ with unit
Bloch vectors $\mu\ne\nu$ in $\R^3$ ($\sigma=(\sigma_1,\sigma_2,\sigma_3)$ the Pauli matrices). For fixed $\mu,\nu$ the
quotient is a linear-fractional function of $t$ (all three quantities are affine in $\rho$) with positive denominator, hence
monotone on $[-\delta,\delta]$; conjugation by $\sigma_1$ maps $t$ to $-t$ and $(\mu,\nu)$ to reflected vectors. So
the bound for $t=\delta$ covers every $t\le\delta$, and the $M_2$ envelope bounds the envelope in every dimension.
In item~23, $\cN$ vanishes on blocks where $P=Q=0$, so the strict bound below holds on $M_n$. In 19 and
20, $\cN$ contains $\varphi(1)$, which we treat at the end.

\emph{Step 2: formulas on $M_2$.} Write $\varphi(x_0+x\cdot\sigma)=x_0+\delta x_3$, $\mu\cdot\nu=\cos2\theta$
($0<\theta\le\pi/2$) and $c=\cos\theta=|\langle u,v\rangle|$ for unit vectors $u,v$ with $P=uu^*$, $Q=vv^*$. Then
$P\wedge Q=0$, $P\vee Q=1$, $e^{QPQ}=1+(e^{c^2}-1)Q$ and $e^P=1+(e-1)P$. We use the following
elementary fact: \emph{for $\beta>|z|$, $\max_{|\xi|=1}(\alpha+y\cdot\xi)/(\beta+z\cdot\xi)$ is the unique zero $\lambda$ of
$\alpha-\lambda\beta+|y-\lambda z|$}, since $(\alpha+y\cdot\xi)-\lambda(\beta+z\cdot\xi)\le\alpha-\lambda\beta+|y-\lambda z|$
with equality for $\xi\parallel y-\lambda z$, and the right-hand side decreases strictly in $\lambda$.

(19) $(\cL-\cR)/\cN=\varphi(P+Q)-1=\frac\delta2(\mu_3+\nu_3)<\delta$, approached as $\mu,\nu\to e_3$.

(20) For fixed $g=2\theta$, the pair $(\mu_3,\nu_3)$ ranges over the image of the unit sphere under $\xi\mapsto(\mu\cdot
\xi,\nu\cdot\xi)$, a filled ellipse, and the linear-fractional quotient is maximal on its boundary, $\mu_3=\cos x$,
$\nu_3=\cos(x+g)$. With $K=e-1$ and $E=e^{\cos^2(g/2)}-1$,
\[
  \frac{\cL-\cR}\cN=\frac{\alpha+y\cdot\xi_x}{\beta+z\cdot\xi_x},\quad \alpha=\tfrac{E-K}2,\ y=\tfrac\delta2(E\cos g-K,\,-E\sin g),
  \ \beta=1+\tfrac K2,\ z=(\tfrac{\delta K}2,0),
\]
$\xi_x=(\cos x,\sin x)$. By the fact, its maximum $\lambda(\delta,g)$ is the larger root of
\begin{equation}\label{eq:quad}
  (\beta^2-|z|^2)\lambda^2-2(\alpha\beta-y\cdot z)\lambda+\alpha^2-|y|^2=0 .
\end{equation}
For small $\delta$ the maximiser has $g=2K\delta/e+O(\delta^2)$, $x=\pi/2$, and $\cL-\cR=-\frac e8g^2+\frac{K\delta}2g+O(\delta^3)$,
which gives $V=K^2\delta^2/(e(e+1))+O(\delta^3)$.

(23) $\cL/\cR=1/(1+\frac\delta2(\mu_3+\nu_3))<1/(1-\delta)$, approached as $\mu,\nu\to-e_3$.

\emph{Step 3: items 19 and 20 in dimension $n$.} For even $n$ take $\rho$ with the eigenvalues $a$ and $b$, each $n/2$
times, and $n/2$ copies of the $M_2$ extremiser, one in each block spanned by an eigenvector for $a$ and one for $b$;
this attains (approaches) the $M_2$ value, which is an upper bound by Step~1. For odd $n$ at least one block of
Lemma~\ref{lem:halmos} is one-dimensional, so there are $k'\le k=(n-1)/2$ two-dimensional blocks. In item~19,
$\cN\ge\varphi(1)=\Tr \rho$; by Step~2 a two-dimensional block contributes $\cL_j-\cR_j=\Tr \rho_j\cdot\frac{\OMC(\rho_j)}2(\mu_3+\nu_3)
<\lambda_1(\rho_j)-\lambda_2(\rho_j)$, and a one-dimensional one $\cL_j-\cR_j\le0$. By Ky Fan's maximum
principle~\cite[Sect.~III.1]{bhatia1997}, $\sum_j\lambda_1(\rho_j)$ is at most the sum of the $k'$ largest and
$\sum_j\lambda_2(\rho_j)$ at least the sum of the $k'$ smallest eigenvalues of $\rho$. The resulting bound for $(\cL-\cR)/\cN$ is a
linear-fractional function of the spectrum of $\rho$ and is maximal for $k'=k$ and the spectrum $a^{(k)},b^{(k+1)}$; hence
$(\cL-\cR)/\cN<k(a-b)/(k(a+b)+b)=(n-1)\delta/(n-\delta)$, which is approached for this spectrum with $P=1$, $Q=0$ on the remaining
eigenvector (a block with $\cL_j=\cR_j$). In item~20 the same configuration, with $P=Q=0$ on the remaining eigenvector, adds $b/k$ per block
to $\cN$, i.e.\ $(1-\delta)/(n-1)$ to $\beta$ (normalisation $\Tr \rho_j=1$); this gives the lower bound, and the $M_2$ value
the upper bound.
\end{proof}

\begin{proposition}[Item 14]\label{prop:holder}
For every $n\ge2$, $V_n\ge\delta^2/36+O(\delta^3)$ for item~14, and more generally
$V_n\ge(1-\frac2p)^2\delta^2/4+O(\delta^3)$ for H\"older's inequality $|\varphi(PQ)|\le\varphi(P)^{1/p}\varphi(Q)^{1/q}$.
\end{proposition}

\begin{proof}
Take $\rho=\frac12(1+\delta D)$, $D=\mathrm{diag}(1,-1)$, and real unit vectors $u,v$ at the angles $\vartheta\pm\varepsilon/2$ to
$e_1$, $P=uu^*$, $Q=vv^*$. Then $\varphi(PQ)=\frac12(\cos^2\varepsilon+\delta\cos\varepsilon\cos2\vartheta)$ and
$\varphi(P)^{1/p}\varphi(Q)^{1/q}=\frac12\bigl(1+\delta\cos2\vartheta+\delta(\frac1q-\frac1p)\varepsilon\sin2\vartheta\bigr)+O(\delta^2\varepsilon^2)$,
so $\cL-\cR=\frac12\bigl(-\varepsilon^2-\delta(1-\frac2p)\varepsilon\sin2\vartheta\bigr)+O(\delta^3)$ for $\varepsilon=O(\delta)$. The choice
$\vartheta=\pi/4$, $\varepsilon=-\frac\delta2(1-\frac2p)$ gives $\cL-\cR=\frac18(1-\frac2p)^2\delta^2+O(\delta^3)$ and $\cR=\frac12+O(\delta)$.
\end{proof}

\subsection{The canonical trace}\label{ssec:canon}
Items 11 and 24 characterise the canonical trace among the positive functionals with $\varphi(1)=n$
on $M_n$; they are not violation constants in the sense of Definition~\ref{def:viol}, and we
determine them on $M_n$.

\begin{proposition}[Item 11]\label{prop:det}
Let $\Tr \rho=n$. For item~11, $V(\rho)=1-\mathrm{GM}/\mathrm{AM}$ of the eigenvalues of $\rho$, attained at $A=\rho^{-1}$. Over
the densities with contrast $\delta$ the maximum is attained at a two-level spectrum:
\[
  V_n(\delta)=1-\min_{1\le m<n}\frac{n\,\ell^{(n-m)/n}}{m+(n-m)\ell},\qquad\ell=\frac1\kappa=\frac{1-\delta}{1+\delta}.
\]
In particular $V_2=1-\sqrt{1-\delta^2}$, and $V_n\uparrow1-1/S(\kappa)$ as $n\to\infty$, where
$S(\kappa)=\frac{(\kappa-1)\kappa^{1/(\kappa-1)}}{e\log \kappa}$ is the Specht ratio.
\end{proposition}
\begin{proof}
By the AM--GM inequality for the eigenvalues of $\rho^{1/2}A\rho^{1/2}$, $\varphi_\rho(A)/n\ge(\det \rho\det A)^{1/n}$, with
equality for $A=\rho^{-1}$, so $\inf_{A>0}\varphi_\rho(A)/(n(\det A)^{1/n})=(\det \rho)^{1/n}$ and $V(\rho)=1-(\det \rho)^{1/n}$, which
for $\Tr \rho=n$ is $1-\mathrm{GM}/\mathrm{AM}$. In the coordinates $x_i=\log w_i$, where $w_1,\dots,w_n$ are the
eigenvalues of $\rho$, $\mathrm{AM}/\mathrm{GM}=\frac1n\sum_i
e^{x_i-\bar x}$ is convex, so on the box of spectra in $[b,a]$ it is maximal at a vertex, a two-level spectrum with $m$
eigenvalues $a$ and $n-m$ eigenvalues $b$; this gives the formula. Its limit is Specht's theorem~\cite{specht1960}:
$\sup\mathrm{AM}/\mathrm{GM}$ over all finite families in $[b,a]$ is $S(\kappa)$.
\end{proof}

\begin{proposition}[Item 24]\label{prop:power}
Let $\Tr \rho=n$ and let $w_1,\dots,w_n$ be the eigenvalues of $\rho$, $w_{\min}=\min_iw_i$. For item~24,
$V(\rho)=w_{\min}^{-1/2}-1$, and
\[
  V_n(\delta)=\sqrt{\frac{\ell+n-1}{n\ell}}-1,\qquad\ell=\frac{1-\delta}{1+\delta};
\]
$V_2=(1-\delta)^{-1/2}-1\sim\delta/2$, $V_n\sim\frac{n-1}n\delta$, and $V_n\uparrow\sqrt \kappa-1$ as $n\to\infty$.
\end{proposition}
\begin{proof}
For $A\ge0$, $\varphi(A^2)=\Tr(A^{1/2}\rho A^{1/2}A)\le\|A\|\varphi(A)\le\Tr(A)\varphi(A)\le\varphi(A)^2/w_{\min}$, with equality for
the projection onto an eigenvector for $w_{\min}$. With $\Tr \rho=n$ and the extreme eigenvalues in the ratio $\kappa$,
$w_{\min}=nb/\sum_iw_i\ge nb/(b+(n-1)a)=n\ell/(\ell+n-1)$, with equality for the spectrum $\{a,\dots,a,b\}$. For
$n=2$, $(\ell+1)/(2\ell)=1/(1-\delta)$.
\end{proof}

\begin{remark}\label{rem:canon}
Items 11 and 24 characterise the canonical trace among the functionals with $\varphi(1)=n$, so they are violated also by
central densities on a direct sum (Corollary~\ref{cor:sectors}), and their envelopes grow with $n$, unlike all others with
a closed form.
\end{remark}

\subsection{Subadditivity of powers}\label{ssec:tik}
For $0<p<1$ let $c_p^T(\rho)=\sup_{A,B\in\cM_+}\varphi((A+B)^p)/(\varphi(A^p)+\varphi(B^p))$, and
\[
\tp(\kappa)=\max_{0\le x\le1}2^{p-1}\frac{\kappa x^p+(1-x)^p}{\kappa x+1-x}.
\]

\begin{lemma}\label{lem:rotfeld}
Let $\chi$ be a normal positive tracial functional on $\cM$, $A,B\in\cM_+$ and $0<p<1$. Then
$\chi((A+B)^p)\le\chi(A^p)+\chi(B^p)$.
\end{lemma}

\begin{proof}
For $s>0$ let $h_s(t)=t/(t+s)$ and $C=A+B$. Then $h_s(C)=(C+s)^{-1/2}(A+B)(C+s)^{-1/2}$ and, by the trace
property and $(C+s)^{-1}\le(A+s)^{-1}$,
$\chi((C+s)^{-1/2}A(C+s)^{-1/2})=\chi(A^{1/2}(C+s)^{-1}A^{1/2})\le\chi(A^{1/2}(A+s)^{-1}A^{1/2})=\chi(h_s(A))$;
similarly for $B$. Hence $\chi(h_s(A+B))\le\chi(h_s(A))+\chi(h_s(B))$. Integrate against
$\frac{\sin p\pi}\pi s^{p-1}ds$ over $(0,\infty)$: for $X\in\cM_+$,
$X^p=\frac{\sin p\pi}\pi\int_0^\infty s^{p-1}h_s(X)\,ds$ as a norm convergent integral.
\end{proof}

For matrices this is Rotfel'd's trace inequality \cite{Rot67}; the short proof above is included
for completeness.

\begin{theorem}[Items 12 and 34]\label{thm:tik}
\begin{enumerate}
\item $c_p^T(z\rho)=c_p^T(\rho)$ for $z\in\cZ_{++}$.
\item $\varphi((A+B)^p)\le\min\{\varphi(B^p)+\kappa_\cZ\varphi(A^p),\varphi(A^p)+\kappa_\cZ\varphi(B^p)\}$ for $A,B\in\cM_+$.
\item $\tp(\kappa_\cZ)\le c_p^T(\rho)\le\frac{1+\kappa_\cZ}2$, and $\tp(\kappa)>1$ for $\kappa>1$.
\end{enumerate}
\end{theorem}

\begin{proof}
(a) Use $(A,B)\mapsto(z^{1/p}A,z^{1/p}B)$. (b) $Y=(A+B)^p-B^p\ge0$ by the L\"owner--Heinz inequality and
$Z=A^p$. Lemma~\ref{lem:rotfeld} applied to $\chi=\tau(\one_Q\,\cdot)$ for central projections $Q$ gives
$\tau(\one_QY)\le\tau(\one_QZ)$, hence $\tau(uY)\le\tau(uZ)$, and by \eqref{eq:sandwich}
$\varphi(Y)\le\tau(uY)\le\tau(uZ)\le\kappa_\cZ\tau(lZ)\le\kappa_\cZ\varphi(Z)$. (c) The upper bound
follows from (b) by averaging. For the lower bound, in $M_2$ with $\rho=\diag(\kappa,1)$, let
$A,B$ be the projections onto $(\pm\sin\theta,\cos\theta)$ and $x=\sin^2\theta$. Then
$A+B=2\diag(x,1-x)$, $A^p=A$, $B^p=B$, and the ratio is $2^{p-1}(\kappa x^p+(1-x)^p)/(\kappa x+1-x)$;
apply Theorem~\ref{thm:transfer}\ref{t:low}. This function equals $1$ at $x=\frac12$ and its
logarithmic derivative there is $2(p-1)(\kappa-1)/(\kappa+1)<0$, so $\tp(\kappa)>1$.
\end{proof}

As $p\to0$ both bounds tend to $(1+\kappa_\cZ)/2$, the constant $c_\vee$ of
Theorem~\ref{thm:proj}\ref{p:c}. Numerically the optimisers in $M_2$ and $M_3$ are always the
symmetric pairs of rank-one projections used above, which suggests $c_p^T(\rho)=\tp(\kappa_\cZ)$. For
a non-homogeneous function $f$ in place of $t^p$, condition (R) fails and the constant of the
$f$-version need not be a function of $\kappa_\cZ$.

\begin{remark}\label{rem:opmon}
Part (b) of Theorem~\ref{thm:tik}, and hence the upper bound $\frac{1+\kappa_\cZ}2$ in (c), hold for
every operator monotone $f:[0,\infty)\to[0,\infty)$ with $f(0)=0$ in place of $t^p$. Such $f$ has
the form $f(t)=\beta t+\int_{(0,\infty)}\lambda h_\lambda(t)\,d\mu(\lambda)$ with $\beta\ge0$ and a
positive measure $\mu$ \cite[Ch.~V]{bhatia1997}. So the proof of Lemma~\ref{lem:rotfeld} gives
$\chi(f(A+B))\le\chi(f(A))+\chi(f(B))$, and $f(A+B)-f(B)\ge0$ by operator monotonicity. Hence the
proof of Theorem~\ref{thm:tik}(b) applies verbatim and gives
$\varphi(f(A+B))\le\min\{\varphi(f(B))+\kappa_\cZ\varphi(f(A)),\varphi(f(A))+\kappa_\cZ\varphi(f(B))\}$.
\end{remark}

\subsubsection*{Items 12 and 13}
Item~12 is Theorem~\ref{thm:tik} with $p=\frac12$: $\mathsf t_{1/2}(\kappa_\cZ)\le c\le\frac{1+\kappa_\cZ}2$. On
$M_2$ the computed envelope of item~12 (Table~\ref{tab:values}) agrees with $\mathsf t_{1/2}(\kappa)-1$ to
three decimal places at $\delta=0.1$, $0.5$ and $0.9$, and its leading coefficient $\frac12$
(Section~\ref{sec:numerics}) is that of $\mathsf t_{1/2}(\kappa)-1$; this supports the conjecture
$c_p^T=\tp(\kappa_\cZ)$ for $p=\frac12$. With the contrast of the operators prescribed, item~12 cannot be
violated at all below a sharp threshold (Proposition~\ref{prop:item12}).

For the Powers--St{\o}rmer inequality (item~13) only numerical values are available
(Section~\ref{sec:numerics}). Its normalisation is chosen as follows.

\begin{remark}\label{rem:13}
With $\cN=\cR$, item~13 would have $V=\infty$ for every $\delta>0$: for rank-one $A=uu^*$, $B=vv^*$ with overlap $c=|u^*v|\to0$,
$\cL\approx2c\,|u^*\rho u_\perp|$ while $\cR=2c^2u^*\rho u$. This is why we normalise by $\varphi(A+B)$.
\end{remark}

\subsection{The spectral geometric mean}\label{ssec:sgm}
For $A,B\in\cM_{++}$ the spectral geometric mean is $A\natural B=(A^{-1}\#B)^{1/2}A(A^{-1}\#B)^{1/2}$, where
$X\#Y=X^{1/2}(X^{-1/2}YX^{-1/2})^{1/2}X^{1/2}$ is the geometric mean. Let
$c_6(\rho)=\sup\{\varphi(A\natural B)/\sqrt{\varphi(A)\varphi(B)}:A,B\in\cM_{++}\}$.

\begin{proposition}[Item 6]\label{prop:sgm}
$G_\cZ\le c_6(\rho)\le\kappa_\cZ$ in every finite von Neumann algebra. In $M_n(\C)$, $G-1\le V_n\le\kappa-1$
for item~6, and the value $G$ of the quotient is approached by the explicit family in the proof.
\end{proposition}

\begin{proof}
\emph{Upper bound.} Let $Y=A^{1/2}(A^{-1}\#B)^{1/2}$. Then $A\natural B=Y^*Y$ and, since
\[
A^{-1}\#B=A^{-1/2}(A^{1/2}BA^{1/2})^{1/2}A^{-1/2},\qquad\text{we have}\qquad
YY^*=(A^{1/2}BA^{1/2})^{1/2}=|B^{1/2}A^{1/2}|.
\]
By the trace property and H\"older's inequality,
$\tau(A\natural B)=\tau(|B^{1/2}A^{1/2}|)\le\|B^{1/2}\|_2\|A^{1/2}\|_2=(\tau(A)\tau(B))^{1/2}$. Since
$\tau\le\|\rho^{-1}\|\varphi$ on $\cM_+$,
\[
\varphi(A\natural B)\le\|\rho\|\tau(A\natural B)\le\|\rho\|(\tau(A)\tau(B))^{1/2}\le\kappa(\rho)(\varphi(A)\varphi(B))^{1/2}.
\]
\emph{(R).} Let $z\in\cZ_{++}$. As central elements commute with $\cM$ and with the functional
calculus of its elements, $(sX)\#(tY)=(st)^{1/2}(X\#Y)$ for $s,t\in\cZ_{++}$; hence
$(zA)^{-1}\#(zB)=(z^{-1}A^{-1})\#(zB)=A^{-1}\#B$ and therefore $(zA)\natural(zB)=z(A\natural B)$. So the ratio for $z\rho$ at $(A,B)$
equals the ratio for $\rho$ at $(zA,zB)$, and $c_6(z\rho)=c_6(\rho)$. By Lemma~\ref{lem:resc},
$c_6(\rho)=c_6(z_0\rho)\le\kappa(z_0\rho)=\kappa_\cZ(\rho)$.
\emph{Lower bound in $M_2$.} Let $\rho=\diag(a,b)$, $e_\pm=(e_1\pm e_2)/\sqrt2$, $v=e_+$ and $x=\rho^{-1}v$. For
invertible $A$ and a rank-one $B=vv^*$, $\|v\|=1$, we have $A^{-1}\#vv^*=(v^*Av)^{-1/2}vv^*$, hence
$(A^{-1}\#B)^{1/2}=(v^*Av)^{-1/4}vv^*$ and $A\natural B=(v^*Av)^{1/2}vv^*$. With $A=xx^*+\varepsilon$,
\[
  \frac{\varphi(A\natural B)^2}{\varphi(A)\varphi(B)}=\frac{(v^*Av)(v^*\rho v)}{\Tr(\rho A)}\ \xrightarrow[\varepsilon\to0]{}\
  \frac{|v^*x|^2\,v^*\rho v}{x^*\rho x}=(v^*\rho^{-1}v)(v^*\rho v)=\tfrac14(a^{-1}+b^{-1})(a+b)=G^2 .
\]
For fixed $\varepsilon$ the quotient at $(A,vv^*+\varepsilon')$ tends to the quotient at $(A,vv^*)$ as
$\varepsilon'\to0$, by continuity; letting then $\varepsilon\to0$, the quotient approaches $G$ for
positive definite $A$ and $B$. (The limits must be taken in this order; for instance
$\varepsilon'=\varepsilon$ does not give $G$.)
\emph{Transfer.} Let $\rho\notin\cZ$ and let $\pi$, $t$ and $R_\varepsilon\to\diag(\alpha,\beta)$,
$\alpha/\beta=\kappa_\cZ(\rho)$, be as in the proof of Theorem~\ref{thm:transfer}\ref{t:low}, and $e=\pi(1)$.
If $e=1$ we use $\pi$ itself, a unital injective $*$-homomorphism, and drop the terms with $s$ below.
Otherwise the map $\pi'(X,\lambda)=\pi(X)+\lambda(1-e)$ is a unital injective $*$-homomorphism $M_2\oplus\C\to\cM$, so it
commutes with inverses, square roots and products, and $\pi'(A_0,s)\natural\pi'(B_0,s)=\pi'(A_0\natural B_0,s)$
for positive definite $A_0,B_0\in M_2$ and $s>0$. Hence
\[
\frac{\varphi(\pi'(A_0,s)\natural\pi'(B_0,s))}{\sqrt{\varphi(\pi'(A_0,s))\varphi(\pi'(B_0,s))}}
=\frac{t\Tr(R_\varepsilon(A_0\natural B_0))+s\varphi(1-e)}{\sqrt{(t\Tr(R_\varepsilon A_0)+s\varphi(1-e))(t\Tr(R_\varepsilon B_0)+s\varphi(1-e))}},
\]
which tends to the ratio for $\diag(\alpha,\beta)$ at $(A_0,B_0)$ as $s\to0$ and then $\varepsilon\to0$. With
the $M_2$ examples above, $c_6(\rho)\ge G(\diag(\alpha,\beta))=G_\cZ(\rho)$. For central $\rho$ both bounds are
$1$. In $M_n$ the statements follow with $\kappa_\cZ=\kappa$.
\end{proof}

\noindent Numerically, for $n=3$ and $4$ the computed envelope of item~6 exceeds $G-1$ at some $\delta$
(Section~\ref{sec:numerics}). This indicates that the constant on $M_n$, $n\ge3$, is larger than $G$, but
we do not use it as a proof. For $n=2$ we conjecture that the constant is $G$
(Conjecture~\ref{conj:values}).

\subsection{Further inequalities}\label{ssec:further}

\subsubsection*{The Alhasan--Fawwaz inequality}
Let
\[
a(\rho)=\sup\Bigl\{\frac{\varphi(A)}{\varphi(|A+iB|)}:A\in\cM_+,\ B\in\cM_{sa},\ A+iB\ne0\Bigr\}.
\]
Since $\varphi(A)$ and $\varphi(B)$ are real, $\varphi(A)=\re\varphi(X)$ for $X=A+iB$. So $a(\rho)$ is
the supremum of $\re\varphi(X)/\varphi(|X|)$ over accretive $X$, i.e.\ over $X$ with $\re X\ge0$. By
\cite{AF21}, $a(\rho)=1$ if and only if $\rho\in\cZ$; this also follows from \ref{af:a} below.

\begin{proposition}[Item 35]\label{prop:af}
\begin{enumerate}
\item\label{af:a} $a(z\rho)=a(\rho)$ for $z\in\cZ_{++}$, and
$a_{M_2}(\diag(\kappa_\cZ,1))\le a(\rho)\le G_\cZ$. If $\rho\notin\cZ$, then $a(\rho)>1$.
\item\label{af:b} If $\cM=M_n(\C)$ and $\rho$ is not scalar, the supremum is attained and
$a(\rho)<\cosh\OLC(\rho)$. In particular, $G_\cZ$ is never attained in $M_n(\C)$.
\item\label{af:c} If $X\in M_2(\C)$ and $\re X\ge0$, then $\re X\le2|X|$. Consequently
$a(\rho)<2$ for every $\rho\in M_2(\C)_{++}$.
\item\label{af:d} For $\rho=\diag(10^4,1,1)\in M_3(\C)$, $a(\rho)>2.68$. Hence $a$ is not a function
of $\kappa_\cZ$, even on finite factors: $\diag(10^4,1)\in M_2(\C)$ and $\diag(10^4,1,1)\in M_3(\C)$
both have $\kappa_\cZ=10^4$.
\end{enumerate}
\end{proposition}

\begin{proof}
\ref{af:a} (R) holds with $\sigma_z(A,B)=(zA,zB)$. The upper bound is Theorem~\ref{thm:gardner}:
$\varphi(A)=\re\varphi(X)\le|\varphi(X)|\le G_\cZ\varphi(|X|)$. The lower bound is
Theorem~\ref{thm:transfer}\ref{t:low}, with the test class $X_1\in\cM_+$, $X_2\in\cM_{sa}$ and
$F_1=X_1$, $F_2=|X_1+iX_2|$. Let $\rho\notin\cZ$. Choose $S\in\cM_{sa}$ with $C=[S,\rho]\ne0$; then
$C^*=-C$. For $\lambda>0$ let $R=-\lambda iC\in\cM_{sa}$ and $X_t=1+t(R+iS)$. For small $t>0$ we
have $\re X_t=1+tR\ge0$. Since $X_t^*X_t=1+2tR+t^2(R^2+S^2+i[R,S])$,
$|X_t|=1+tR+\frac{t^2}2(S^2+i[R,S])+O(t^3)$, and
\[
\frac{\varphi(\re X_t)}{\varphi(|X_t|)}=1-\frac{t^2}{2\varphi(1)}\,\varphi\bigl(S^2+i[R,S]\bigr)+O(t^3).
\]
Here $\varphi(i[R,S])=i\tau(R[S,\rho])=\lambda\tau(C^2)=-\lambda\tau(C^*C)$. So for
$\lambda>\varphi(S^2)/\tau(C^*C)$ the ratio exceeds $1$ for small $t$.

\ref{af:b} The ratio is invariant under positive scalings of $(A,B)$. On the compact set
$\{\|A+iB\|_2=1\}$ it is continuous, and its denominator is at least $\lambda_{\min}(\rho)\Tr|X|>0$, so
the supremum is attained. Suppose the ratio equals $\cosh\OLC(\rho)$ at $X=U|X|$, with $U$ unitary,
and let $h=|X|^{1/2}$. Then the chain
$\re\langle Uh,\rho h\rangle\le\|h\|_2\|\rho h\|_2\le\cosh\OLC(\rho)\langle\rho h,h\rangle$ from the proof of
Theorem~\ref{thm:gardner} consists of equalities. Let $M>m$ be the extreme eigenvalues of $\rho$,
$P_M,P_m$ the corresponding spectral projections and $E=\operatorname{ran}(P_M+P_m)$. Equality in
Lemma~\ref{lem:krein} for $L_\rho$ means $(L_\rho-m)(M-L_\rho)h=0$ and $\|\rho h\|_2^2=mM\|h\|_2^2$.
Hence $V:=\operatorname{ran}h\subset E$ and $M\|P_Mh\|_2^2=m\|P_mh\|_2^2$, and both norms are nonzero.
Equality in the Cauchy--Schwarz inequality gives $Uh=\lambda\rho h$, and
$\varphi(X)=\lambda\|\rho h\|_2^2\ge0$ gives $\lambda>0$. So $X=\lambda\rho K$ with $K=h^2=|X|$. Also
$h^2=hU^*Uh=\lambda^2h\rho^2h$, so the compression of $\rho^2=M^2P_M+m^2P_m$ to $V$ is
$\lambda^{-2}$. Thus $P_VP_MP_V=tP_V$ and $P_VP_mP_V=(1-t)P_V$ for some $t$, and $t\in(0,1)$
because $P_Mh\ne0\ne P_mh$. Take $x\in V$ with $\langle Kx,x\rangle>0$, and
$\xi=P_Mx-cP_mx$ with $t/(1-t)<c<Mt/(m(1-t))$. Then $P_V\xi=(t-c(1-t))x$ and
$P_V\rho\xi=(Mt-cm(1-t))x$. Since $K=P_VKP_V$,
\[
\langle(\re X)\xi,\xi\rangle=\lambda\re\langle K\xi,\rho\xi\rangle
=\lambda\bigl(t-c(1-t)\bigr)\bigl(Mt-cm(1-t)\bigr)\langle Kx,x\rangle<0,
\]
which contradicts $\re X\ge0$.

\ref{af:c} Let first $X$ be invertible, $X=UP$ with $P=|X|$, and $T=P^{-1/2}UP^{1/2}$. Then
$X=P^{1/2}TP^{1/2}$, so $\re X=P^{1/2}(\re T)P^{1/2}$ and $\re T\ge0$. As $T$ is similar to $U$, it
is diagonalisable with eigenvalues $\lambda_1,\lambda_2$ of modulus $1$. Let
$c=\frac12(\lambda_1+\lambda_2)$. The traceless matrix $T-c$ is unitarily similar to a matrix with
zero diagonal, and $\diag(1,-1)$ conjugates such a matrix to its negative. Hence the numerical
range satisfies $W(T)=2c-W(T)$. Since $W(T)\subset\{\re z\ge0\}$, we get $\re z\le2\re c$ on $W(T)$,
i.e.\ $\re T\le2\re c\le2$. If $\re c=1$, then $\lambda_1=\lambda_2=1$, so $T=1$ and $\re X=P$. In all
cases $\re X\le c'P$ for some $c'<2$. If $X\ne0$ is singular, then $X=|u\rangle\langle v|$.
Here $u=\mu v$ with $\re\mu\ge0$, since otherwise
$\langle(\re X)\xi,\xi\rangle=\re(\langle\xi,v\rangle\langle u,\xi\rangle)<0$ for a suitable $\xi$;
hence $\re X\le|X|$. Thus $\varphi(\re X)<2\varphi(|X|)$ for every accretive $X\ne0$, and $a(\rho)<2$
because the supremum is attained (by \ref{af:b}; for scalar $\rho$, $a=1$).

\ref{af:d} Let
\[
P=\begin{pmatrix}2&19&20\\19&637&491\\20&491&416\end{pmatrix},\qquad
U=\frac1{10485}\begin{pmatrix}9965&-1180&3040\\2020&9907&-2776\\-2560&3224&9643\end{pmatrix},
\]
and $X=UP$. The matrix $U$ is orthogonal; it is the Cayley transform $(1-N)(1+N)^{-1}$ of the
skew-symmetric $N$ with $N_{12}=\frac2{25}$, $N_{13}=-\frac7{50}$, $N_{23}=\frac3{20}$. The matrix $P$
is positive definite, since its leading principal minors are $2$, $913$ and $16006$. Hence
$|X|=P$. An exact computation in rational arithmetic, which can be checked by hand, shows that
the leading principal minors of $\re X$ are $\frac{11662}{2097}$, $\frac{6082302974}{109935225}$ and
$\frac{36535906}{7329015}$, so $\re X>0$. It also gives $\varphi(\re X)=\frac{39575293}{699}$ and
$\varphi(|X|)=\varphi(P)=21053$, so $\varphi(\re X)/\varphi(|X|)=\frac{39575293}{14716047}>2.68$. Together
with \ref{af:c} this proves the claim.
\end{proof}

The constant $2$ in \ref{af:c} is sharp, but not attained. Every $T\in M_2(\C)$ that is similar to a
unitary can be written as $T=P^{-1/2}UP^{1/2}$ with $P>0$ and $U$ unitary: if $T=SDS^{-1}$ with $D$
unitary and $S=|S^*|V$ is the polar decomposition, take $P=|S^*|^{-2}$ and $U=VDV^*$. Then $X=UP$
satisfies $|X|=P$ and $\re X=P^{1/2}(\re T)P^{1/2}$. For
$T_\epsilon=\bigl(\begin{smallmatrix}e^{i\epsilon}&2\cos\epsilon\\0&e^{-i\epsilon}\end{smallmatrix}\bigr)$,
$0<\epsilon<\pi/2$, the eigenvalues $e^{\pm i\epsilon}$ are distinct and unimodular, and
$\re T_\epsilon=\cos\epsilon\bigl(\begin{smallmatrix}1&1\\1&1\end{smallmatrix}\bigr)\ge0$ has largest
eigenvalue $2\cos\epsilon$. So the best constant $c_2$ in the operator inequality $\re X\le c_2|X|$ for accretive
$X\in M_2(\C)$ equals $2$.

\begin{remark}\label{rem:afnum}
Numerically (not used in the proofs), in $M_2(\C)$ with $\rho=\diag(\kappa,1)$ the constant $a$ is
$1.0546$, $1.1802$, $1.3746$, $1.7491$, $1.9721$ and $1.9972$ for $\kappa=2$, $4$, $10$, $10^2$, $10^4$
and $10^6$, while $G_\cZ$ is $1.0607$, $1.25$, $1.7393$, $5.05$, $50.005$ and $500.0005$. Also
$(a-1)/(G_\cZ-1)\to1$ as $\kappa\to1$, and the maximisers have $\re X$ of rank one. In $M_n(\C)$ the
value depends on the multiplicity of $\min\sigma(\rho)$ but not on that of $\max\sigma(\rho)$. For
instance $a(\diag(10,1,1))\approx1.3898$, whereas $a(\diag(10,10,1))\approx1.3746=a(\diag(10,1))$. For
$\diag(100,1,\dots,1)\in M_n(\C)$ with $n=2,\dots,6$ the values are $1.7491$, $1.9746$, $2.0117$,
$2.0156$ and $2.0158$. For $\diag(10^4,1,1)$ the value is $2.7157$. The variant with $B\ge0$ gives
smaller values, $1.0541$ and $1.1269$ for $\kappa=4,10$ in $M_2(\C)$. For the best constant $c_n$ in
$\re X\le c_n|X|$ for accretive $X\in M_n(\C)$ we have $c_2=2$ (proved above), and numerically
$c_3\ge2.94$.
\end{remark}

\subsubsection*{Exponential inequalities}
Bikchentaev's characterisation of traces by $\varphi(e^{PQP})\le\varphi(e^Q)$ concerns projections
$P,Q$ (see \cite{Bik12} for the precise statement). For traces it remains true for all $Q\ge0$ (Lemma~\ref{lem:bk}), but it fails for
self-adjoint $Q$: for $Q=-t1$ with $t>0$ and $0\ne P\ne1$,
$\tau(e^{PQP})=e^{-t}\tau(P)+\tau(1-P)>\tau(e^Q)$. Let
\[
c_{\exp}(\rho)=\sup\frac{\varphi(e^{PQP})}{\varphi(e^Q)},\qquad
b(\rho)=\sup\frac{\varphi(Pe^{PQP})}{\varphi(e^Q)}.
\]
In $c_{\exp}$, $P$ runs through the projections and $Q$ through $\cM_+$. Let $c_{\exp}^{\rm pr}$
be the same supremum with $Q$ also a projection. In $b$, $P\ne0$ runs through the projections and
$Q$ through $\cM_{sa}$. The quantity $b$ is invariant under $Q\mapsto Q+s$. Let $b_2(\kappa)$ denote
$b(\diag(\kappa,1))$ computed in $M_2(\C)$.

\begin{lemma}\label{lem:bk}
Let $\tau$ be a faithful normal finite trace on $\cM$, $P$ a projection and $Q\in\cM_{sa}$. Then
$\tau(Pe^{PQP})\le\tau(Pe^QP)\le\tau(e^Q)$. If $Q\ge0$, then $\tau(e^{PQP})\le\tau(e^Q)$.
\end{lemma}

\begin{proof}
The first inequality is Jensen's trace inequality for the unital compression to $P\cM P$
(cf.\ \cite{BK90}); we give a direct proof. Let $\varepsilon>0$, let $I_1,\dots,I_N$ be disjoint intervals of length $<\varepsilon$ covering
$\sigma(PQP)$, with left endpoints $\lambda_j$, and let $E_j=P\one_{I_j}(PQP)$. The $E_j$ are
projections, they commute with $PQP$, and $\sum_jE_j=P$. Moreover $E_jPQPE_j\ge\lambda_jE_j$ and
$E_je^{PQP}\le e^{\lambda_j+\varepsilon}E_j$. For $E_j\ne0$, $\omega_j=\tau(E_j\,\cdot\,E_j)/\tau(E_j)$ is
a normal state. Jensen's inequality for the spectral measure of $Q$ in $\omega_j$ gives
$\omega_j(e^Q)\ge e^{\omega_j(Q)}$, and $\omega_j(Q)=\tau(E_jPQPE_j)/\tau(E_j)\ge\lambda_j$. Hence
$\tau(Pe^QP)=\sum_j\tau(E_je^QE_j)\ge\sum_je^{\lambda_j}\tau(E_j)\ge e^{-\varepsilon}\sum_j\tau(E_je^{PQP})
=e^{-\varepsilon}\tau(Pe^{PQP})$; let $\varepsilon\to0$. The second inequality holds because $e^Q\ge0$.
If $Q\ge0$, then $e^Q\ge1$ and
$\tau(e^{PQP})=\tau(Pe^{PQP})+\tau(1-P)\le\tau(Pe^QP)+\tau((1-P)e^Q(1-P))=\tau(e^Q)$.
\end{proof}

\begin{proposition}[Items 20 and 36]\label{prop:bikexp}
\begin{enumerate}
\item\label{be:a} $c^{\rm pr}_{\exp}\le c_{\exp}\le\kappa_\cZ$, $b\le\kappa_\cZ$, and
$b\le c_{\exp}\le1+b$.
\item\label{be:b} $b(z\rho)=b(\rho)$ for $z\in\cZ_{++}$, and $b(\rho)\ge b_2(\kappa_\cZ)$; moreover
$b_2(\kappa)>1$ for $\kappa>1$. If $\rho\notin\cZ$, then $c^{\rm pr}_{\exp}(\rho)>1$.
\item\label{be:c} Neither $c_{\exp}$ nor $c^{\rm pr}_{\exp}$ is a function of $\kappa_\cZ$. For
$\kappa>1$ and $\rho_\lambda=\diag(\kappa,1)\oplus\lambda1_2\in M_2(\C)\oplus M_2(\C)$, with
$\tau=\Tr\oplus\Tr$, we have $\kappa_\cZ(\rho_\lambda)=\kappa$ for all $\lambda>0$. As $\lambda\to0$ both constants tend to
their values for $\diag(\kappa,1)$ in $M_2(\C)$, which exceed $b_2(\kappa)$ and $1$
respectively. As $\lambda\to\infty$, $c_{\exp}(\rho_\lambda)\to b_2(\kappa)$ and
$c^{\rm pr}_{\exp}(\rho_\lambda)\to1$.
\item\label{be:d} Let $c_T$ denote $c_{\exp}$ with $Q$ restricted by $0\le Q\le T$. Then
$c_T=1+\gamma_{\exp}T+O(T^2)$ as $T\to0$, where
$\gamma_{\exp}(\rho)=\sup_P\tau((P\rho P-\rho)_+)/\tau(\rho)$, and $\gamma_{\exp}>0$ for $\rho\notin\cZ$.
For $\rho_\lambda$ as in \ref{be:c},
$\gamma_{\exp}(\rho_\lambda)=\gamma_{\exp}(\diag(\kappa,1))\,(\kappa+1)/(\kappa+1+2\lambda)$.
\end{enumerate}
\end{proposition}

\begin{proof}
\ref{be:a} Let $E$ be a central projection. Lemma~\ref{lem:bk} in $E\cM$, with the trace
$\tau(E\,\cdot)$, gives $\tau(EPe^{PQP})\le\tau(Ee^Q)$, and $\tau(Ee^{PQP})\le\tau(Ee^Q)$ if $Q\ge0$.
Approximating $u=\ceil\rho$ by positive combinations of central projections, the same holds with
$u$ in place of $E$. Since $Pe^{PQP}=Pe^{PQP}P\ge0$, \eqref{eq:sandwich} gives
$\varphi(Pe^{PQP})\le\tau(uPe^{PQP})\le\tau(ue^Q)\le\kappa_\cZ\tau(le^Q)\le\kappa_\cZ\varphi(e^Q)$, and the
same argument applies to $e^{PQP}$. For $s\ge\|Q\|$ we have $Q+s\ge0$ and
$e^{P(Q+s)P}=(1-P)+e^sPe^{PQP}$, so
\begin{equation}\label{eq:shift}
\frac{\varphi(e^{P(Q+s)P})}{\varphi(e^{Q+s})}=\frac{\varphi(Pe^{PQP})}{\varphi(e^Q)}+e^{-s}\,\frac{\varphi(1-P)}{\varphi(e^Q)},
\end{equation}
which gives $b\le c_{\exp}$. Finally, for $Q\ge0$ we have $e^Q\ge1$, so
$\varphi(e^{PQP})=\varphi(1-P)+\varphi(Pe^{PQP})\le\varphi(1)+b\,\varphi(e^Q)\le(1+b)\varphi(e^Q)$.

\ref{be:b} Since $z$ commutes with $P$ and $Q$, $P(Q+\ln z)P=PQP+(\ln z)P$ with commuting
summands. Hence $Pe^{P(Q+\ln z)P}=zPe^{PQP}$; also $e^{Q+\ln z}=ze^Q$. So (R) holds with
$\sigma_z(P,Q)=(P,Q+\ln z)$. For the lower bound let $\pi$, $E_\pi=\pi(1)$, $t$ and $\varepsilon$ be as
in the proof of Theorem~\ref{thm:transfer}\ref{t:low}, so that $\varphi\circ\pi=t\Tr(R_\varepsilon\,\cdot)$.
For a projection $p\in M_2(\C)$ and $q\in M_2(\C)_{sa}$ put $P=\pi(p)$ and
$Q_L=\pi(q)-L(1-E_\pi)$. Then $Pe^{PQ_LP}=\pi(pe^{pqp})$ and $e^{Q_L}=\pi(e^q)+e^{-L}(1-E_\pi)$;
let $L\to\infty$ and then $\varepsilon\to0$.

To see that $b_2(\kappa)>1$, let $\rho=\diag(\kappa,1)$. Every pair $(P,Q)$ in $M_2(\C)$ with $P$ of
rank one is, up to a shift of $Q$, of the form $P=|w\rangle\langle w|$ and
$Q=-\mu|v_2\rangle\langle v_2|$, where $\mu\in\R$ and $v_1,v_2$ is an orthonormal basis. Put
$\beta_i=\langle\rho v_i,v_i\rangle$. With $x=|\langle w,v_2\rangle|^2$ the ratio is
$e^{-\mu x}\langle\rho w,w\rangle/(\beta_1+e^{-\mu}\beta_2)$. Its supremum over $\mu$ is
\begin{equation}\label{eq:bratio}
R(w,v_1)=\langle\rho w,w\rangle\,x^x(1-x)^{1-x}\beta_1^{x-1}\beta_2^{-x},
\end{equation}
attained at $e^{-\mu}=x\beta_1/((1-x)\beta_2)$ if $0<x<1$; for $x\in\{0,1\}$, $R=1$. Now choose
$v_1$ not an eigenvector of $\rho$ and $w=\sqrt{1-x}\,v_1+\sqrt x\,e^{i\chi}v_2$ with a suitable
$\chi$. Then $\langle\rho w,w\rangle=\beta_1+2\sqrt{x(1-x)}\,|\langle\rho v_1,v_2\rangle|+O(x)$, so
$R=1+2\beta_1^{-1}|\langle\rho v_1,v_2\rangle|\sqrt x+O(x|\ln x|)>1$ for small $x>0$.

To see that $c^{\rm pr}_{\exp}(\rho)>1$ for $\rho\notin\cZ$, let first $D\in M_2(\C)_{++}$ be
non-scalar. Choose a unit vector $v$ that is not an eigenvector of $D$, a unit vector $v'\perp v$,
and put $\eta_0=|\langle Dv,v'\rangle|>0$. For $x\in(0,1)$ let
$w=\sqrt{1-x}\,v+\sqrt x\,e^{i\chi}v'$ with $\re(e^{i\chi}\langle Dv',v\rangle)=\eta_0$. Take
$p=|w\rangle\langle w|$ and $q=|v\rangle\langle v|$. Then $pqp=(1-x)p$, and
\[
\Tr(De^{pqp})-\Tr(De^q)=(e^{1-x}-1)\langle Dw,w\rangle-(e-1)\langle Dv,v\rangle=2(e-1)\eta_0\sqrt x+O(x),
\]
which is $>0$ for small $x$. Apply this to $D=\diag(\alpha,\beta')$, where $\alpha=\ceil\rho(\omega_0)$
and $\beta'=\flr\rho(\omega_0)$ are as in Lemma~\ref{lem:corner}, and fix such $p,q$. By continuity
$\Tr(R_\varepsilon e^{pqp})>\Tr(R_\varepsilon e^q)$ for small $\varepsilon$. For $P=\pi(p)$ and
$Q=\pi(q)$ we have $e^{PQP}=\pi(e^{pqp})+(1-E_\pi)$ and $e^Q=\pi(e^q)+(1-E_\pi)$, so
\[
\frac{\varphi(e^{PQP})}{\varphi(e^Q)}=\frac{t\Tr(R_\varepsilon e^{pqp})+\varphi(1-E_\pi)}{t\Tr(R_\varepsilon e^q)+\varphi(1-E_\pi)}>1.
\]

\ref{be:c} $\kappa_\cZ(\rho_\lambda)=\kappa$ by Proposition~\ref{prop:char}\ref{ch:prop}, and
$\varphi=(\Tr\oplus\Tr)(\rho_\lambda\,\cdot)$. For
$P=P_1\oplus P_2$ and $Q=Q_1\oplus Q_2$ write the ratio as $(N_1+\lambda N_2)/(D_1+\lambda D_2)$. Then
$N_2\le D_2$ by Lemma~\ref{lem:bk}, and $D_2=\Tr e^{Q_2}\ge2$. The ratio is therefore at most
$\max\{N_1/D_1,1\}$. On the other hand, taking $P_2=Q_2=0$ gives
$(N_1+2\lambda)/(D_1+2\lambda)\to N_1/D_1$ as $\lambda\to0$. This proves the limits as $\lambda\to0$.
As $\lambda\to\infty$, for $Q\ge0$ we have $N_1\le\varphi_1(1)+b_2(\kappa)D_1$ by \ref{be:a} in
$M_2(\C)$. Since $b_2(\kappa)>1$, the ratio is at most $b_2(\kappa)+(\kappa+1)/(2\lambda)$; and
$c_{\exp}(\rho_\lambda)\ge b(\rho_\lambda)\ge b_2(\kappa)$ by \ref{be:a} and \ref{be:b}. For
projections $N_1\le e\,\varphi_1(1)$, so the ratio is at most $1+e(\kappa+1)/(2\lambda)$, while $P=1$
gives the value $1$. It remains to compare the limits. First,
$c^{\rm pr}_{\exp}(\diag(\kappa,1))>1$ by \ref{be:b}. Second,
$R$ in \eqref{eq:bratio} is continuous on the compact set of pairs of unit vectors $(w,v_1)$, and
rank-two $P$ gives the value $1<b_2(\kappa)$. So $b_2(\kappa)$ is attained at some $(P,Q)$ with
$0\ne P\ne1$, and then \eqref{eq:shift} gives $c_{\exp}(\diag(\kappa,1))>b_2(\kappa)$.

\ref{be:d} For $0\le Q\le T$, $e^{PQP}=1+PQP+O(T^2)$ and $e^Q=1+Q+O(T^2)$ uniformly in $P$. So
the ratio is $1+\tau(Q(P\rho P-\rho))/\tau(\rho)+O(T^2)$, and
$\sup_{0\le Q\le T}\tau(Q(P\rho P-\rho))=T\,\tau((P\rho P-\rho)_+)$. If $\rho\notin\cZ$, then
$P\rho P\not\le\rho$ for some $P$ by Theorem~\ref{thm:pap}, so $\gamma_{\exp}>0$. For $\rho_\lambda$ and
$P=P_1\oplus P_2$ we have $P\rho_\lambda P-\rho_\lambda=(P_1\rho_1P_1-\rho_1)\oplus\lambda(P_2-1)$ with
$\rho_1=\diag(\kappa,1)$. The second summand is $\le0$, and $\tau(\rho_\lambda)=\kappa+1+2\lambda$, which
gives the formula for $\gamma_{\exp}(\rho_\lambda)$.
\end{proof}

By \ref{be:d}, even the slope $\gamma_{\exp}$ depends on the weights of the central summands, not
only on $\kappa_\cZ$.

\begin{remark}\label{rem:expnum}
Numerically (not used in the proofs), for $\rho=\diag(\kappa,1)$ in $M_2(\C)$ the pair
$(c_{\exp},b)$ is $(1.0405,1.0178)$, $(1.1697,1.1030)$, $(1.5351,1.4203)$ and $(5.6545,5.5296)$ for
$\kappa=2$, $4$, $10$ and $100$. For $\kappa=10$ the supremum $c_{\exp}$ is reached at
$\|Q\|\approx4$, and $\gamma_{\exp}\approx0.2448$. In $M_3(\C)$ with $\kappa=10$, $c_{\exp}$ is $1.5232$,
$1.4552$ and $1.5015$ for the spectra $(10,1,1)$, $(10,10,1)$ and $(10,\sqrt{10},1)$. In all three
cases, and in $M_2(\C)\oplus M_2(\C)$, $b=1.4203=b_2(10)$. This suggests $b=b_2(\kappa_\cZ)$.
\end{remark}

\begin{remark}[Items 20 and 36]\label{rem:item20}
Item~20 is the constant $c^{\rm pr}_{\exp}$ with $P$ and $Q$ interchanged, so $V(\rho)=c^{\rm pr}_{\exp}(\rho)-1$ on
$M_n$. On $M_2$ every density of contrast $\delta$ is a multiple of a unitary conjugate of
$\diag(\kappa,1)$, so Theorem~\ref{thm:projP} gives $c^{\rm pr}_{\exp}(\diag(\kappa,1))=1+V_2(\delta)$
explicitly. Proposition~\ref{prop:bikexp}\ref{be:c} shows that on $M_2\oplus M_2$ the constant of
$\rho_\lambda=\diag(\kappa,1)\oplus\lambda1_2$ moves from this value ($\lambda\to0$) to $1$ ($\lambda\to\infty$),
although $\kappa_\cZ(\rho_\lambda)=\kappa$ is fixed. This is the dilution by other blocks that occurs
whenever $\cN$ contains $\varphi(1)$ (Remark~\ref{rem:dilute}). The same holds for item~36, and the
computed envelope of item~15 decreases with $n$ for the same reason (Section~\ref{sec:numerics}).
\end{remark}

\subsubsection*{The Peierls--Bogoliubov and Araki--Lieb--Thirring inequalities}
Bikchentaev \cite{bikchentaev2011pb} characterised tracial functionals by two inequalities. The first is the
Peierls--Bogoliubov inequality
\[
\varphi(e^{A+B})\ge\varphi(e^A)\exp\bigl(\varphi(Be^A)/\varphi(e^A)\bigr)
\]
for self-adjoint $A,B$. The second is the Araki--Lieb--Thirring inequality
\[
\varphi((B^{1/2}AB^{1/2})^{rq})\le\varphi((B^{r/2}A^rB^{r/2})^q)
\]
for $A,B\ge0$, with exponents $r>1$, $q>0$ (see \cite{bikchentaev2011pb} for the precise quantifiers). For
non-tracial $\varphi$ the number
$\varphi(Be^A)$ need not be real; we use its real part. With $|\varphi(Be^A)|$ instead, \ref{pb:a} below remains true, since
$|\varphi(Be^A)|\ge\re\varphi(Be^A)$. We also consider the symmetric form with
$\omega_A(B)=\varphi(e^{A/2}Be^{A/2})/\varphi(e^A)$, which coincides with the first form for traces.
Put
\[
c_{\rm PB}(\rho)=\sup_{A,B\in\cM_{sa}}\frac{\varphi(e^A)\exp(\omega_A(B))}{\varphi(e^{A+B})},\qquad
c_{r,q}(\rho)=\sup_{A,B\in\cM_+}\frac{\varphi((B^{1/2}AB^{1/2})^{rq})}{\varphi((B^{r/2}A^rB^{r/2})^q)}.
\]

\begin{proposition}[Items 16 and 37]\label{prop:pb}
\begin{enumerate}
\item\label{pb:a} For $\rho\in\cZ$ the supremum below (and $c_{\rm PB}$) equals $1$: $\varphi$ is then a trace, so the
Peierls--Bogoliubov inequality holds~\cite{bikchentaev2011pb} (and $\varphi(Be^A)$ is real), with equality at $B=0$. If $\rho\notin\cZ$, then
$\sup\varphi(e^A)\exp(\re\varphi(Be^A)/\varphi(e^A))/\varphi(e^{A+B})=+\infty$, already over
$A\in\cM_{sa}$ and $B\le0$.
\item\label{pb:b} $c_{\rm PB}(z\rho)=c_{\rm PB}(\rho)$ for $z\in\cZ_{++}$, and
$b(\rho)\le c_{\rm PB}(\rho)\le\kappa_\cZ(\rho)$. In particular $c_{\rm PB}(\rho)>1$ for $\rho\notin\cZ$.
\end{enumerate}
\end{proposition}

\begin{proof}
\ref{pb:a} For $Y\in\cM_{++}$ and $B\in\cM_{sa}$, $\re\varphi(BY)=\tau(B(\rho\circ Y))$ with
$\rho\circ Y=\frac12(\rho Y+Y\rho)$. We first find $Y\in\cM_{++}$ with $\rho\circ Y\not\ge0$. Take
$\eta_1,\eta_2$ and $\pi$ from Lemma~\ref{lem:corner}\ref{c:b} for $a=\rho$. The compression of
$\rho$ to $\mathcal K=\operatorname{span}\{\eta_1,\eta_2\}$ is $D=\diag(a_1,a_2)$ with $a_1\ne a_2$. Let
$g=(1,1)/\sqrt2$. Since $g$ and $Dg$ are linearly independent, there is $c\in\C^2$ with
$\langle c,g\rangle=1$ and $\langle c,Dg\rangle=-1$. For $\varepsilon>0$ put
$Y=\pi(|g\rangle\langle g|+\varepsilon1)+\varepsilon(1-E_\pi)\in\cM_{++}$, where $E_\pi=\pi(1)$, and $\xi=c_1\eta_1+c_2\eta_2$.
Since $Y\xi\in\mathcal K$,
$\langle(\rho\circ Y)\xi,\xi\rangle=\re\langle Y\xi,\rho\xi\rangle=\re\langle yc,Dc\rangle$ with
$y=|g\rangle\langle g|+\varepsilon1$, and this equals $-1+O(\varepsilon)<0$. Now let $F\ne0$ be the
spectral projection of $\rho\circ Y$ for $(-\infty,0)$, $\vartheta=-\tau(F(\rho\circ Y))>0$, $A=\ln Y$
and $B_s=-sF$ with $s>0$. Then $\re\varphi(B_sY)=s\vartheta$. Since $A+B_s\le A$, the monotonicity
of $\tau\circ\exp$ \cite{FK86} gives
$\varphi(e^{A+B_s})\le\|\rho\|\tau(e^{A+B_s})\le\|\rho\|\tau(e^A)$. Hence the ratio is at least
$\varphi(Y)e^{s\vartheta/\varphi(Y)}/(\|\rho\|\tau(Y))\to\infty$ as $s\to\infty$.

\ref{pb:b} (R) holds with $\sigma_z(A,B)=(A+\ln z,B)$, since
\[
ze^A=e^{A+\ln z},\qquad ze^{A/2}Be^{A/2}=e^{(A+\ln z)/2}Be^{(A+\ln z)/2},\qquad ze^{A+B}=e^{A+\ln z+B}.
\]
For the upper bound
note that the ratio is unchanged under $A\mapsto A+t$, so we may assume $\varphi(e^A)=1$. Put
$Y=e^A$, $\sigma=Y^{1/2}\rho Y^{1/2}$ and $H=A+B$. Then $\tau(\sigma)=1$ and
$\omega_A(B)=\tau(\sigma B)=\tau(\sigma H)-\tau(\sigma\ln Y)$. Also $\sigma\le\|\rho\|Y$, so
$\ln\sigma\le\ln\|\rho\|+\ln Y$ by the operator monotonicity of $\ln$. The Peierls--Bogoliubov
inequality for $\tau$ \cite{bikchentaev2011pb}, applied to $\ln\sigma$ and $H-\ln\sigma$, gives
$\ln\tau(e^H)\ge\tau(\sigma H)-\tau(\sigma\ln\sigma)$. Moreover $\varphi(e^H)\ge\|\rho^{-1}\|^{-1}\tau(e^H)$.
Hence the logarithm of the ratio is at most
$\tau(\sigma\ln\sigma)-\tau(\sigma\ln Y)+\ln\|\rho^{-1}\|\le\ln\kappa(\rho)$, and Lemma~\ref{lem:resc}
replaces $\kappa$ by $\kappa_\cZ$. For the lower bound let $P\ne0$ be a projection, $Q\in\cM_{sa}$,
$A_L=PQP-L(1-P)$ and $B_L=Q-A_L$. Then $e^{A_L+B_L}=e^Q$ and $e^{A_L}=Pe^{PQP}+e^{-L}(1-P)$. Also
$\omega_{A_L}(B_L)\to0$ as $L\to\infty$: in $e^{A_L/2}(Q-PQP+L(1-P))e^{A_L/2}$ the corner
$P\,\cdot\,P$ vanishes and the other terms are $O(Le^{-L/2})$. So the ratio tends to
$\varphi(Pe^{PQP})/\varphi(e^Q)$. Finally $b>1$ for $\rho\notin\cZ$ by
Proposition~\ref{prop:bikexp}\ref{be:b}.
\end{proof}

\begin{proposition}[Items 7 and 38]\label{prop:alt}
Let $r\ge1$ and $q>0$. Then $c_{r,q}(z\rho)=c_{r,q}(\rho)$ for $z\in\cZ_{++}$, and
$c_{r,q}(\rho)\le\kappa_\cZ$. If $r>1$ and $\rho\notin\cZ$, then $c_{r,q}(\rho)>1$ for every $q>0$, and
$\sup_{q>0}c_{r,q}(\rho)=\kappa_\cZ$.
\end{proposition}

\begin{proof}
(R) holds with $\sigma_z(A,B)=(z^{1/(rq)}A,B)$. The upper bound $\kappa$ follows from Kosaki's
inequality $\tau((B^{1/2}AB^{1/2})^{rq})\le\tau((B^{r/2}A^rB^{r/2})^q)$, valid for all $r\ge1$ and
$q>0$ \cite{Kos92}, together with $\|\rho^{-1}\|^{-1}\tau\le\varphi\le\|\rho\|\tau$ on $\cM_+$; then
apply Lemma~\ref{lem:resc}. All expressions are admissible, so Theorem~\ref{thm:transfer}\ref{t:low}
reduces the remaining claims to $M_2(\C)$ with $\rho=\diag(\kappa,1)$, $\kappa>1$. For a unit vector
$a$, $A=|a\rangle\langle a|$ and $B>0$, we have $(B^{1/2}AB^{1/2})^{rq}=\langle Ba,a\rangle^{rq}P_u$ and
$(B^{r/2}A^rB^{r/2})^q=\langle B^ra,a\rangle^qP_v$, where $P_u,P_v$ are the projections onto
$u=B^{1/2}a$ and $v=B^{r/2}a$. So, with $g(x)=\langle\rho x,x\rangle/\|x\|^2$, the ratio is
$(\langle Ba,a\rangle^r/\langle B^ra,a\rangle)^q\,g(u)/g(v)$. To show $c_{r,q}>1$, let $a$ be a unit
vector that is not an eigenvector of $\rho$, $d=(\rho-g(a))a\ne0$ (so $d\perp a$),
$H=-(|d\rangle\langle a|+|a\rangle\langle d|)$ and $B=1+\epsilon H$ with small $\epsilon>0$. Since
$Ha=-d\perp a$, $\langle Ba,a\rangle=1$ and $\langle B^ra,a\rangle=1+O(\epsilon^2)$. Moreover
$u=a-\frac\epsilon2d+O(\epsilon^2)$ and $v=a-\frac{r\epsilon}2d+O(\epsilon^2)$, so
$g(u)=g(a)-\epsilon\|d\|^2+O(\epsilon^2)$ and $g(v)=g(a)-r\epsilon\|d\|^2+O(\epsilon^2)$. Hence the ratio
is $1+(r-1)\epsilon\|d\|^2/g(a)+O(\epsilon^2)>1$ for small $\epsilon$. For the supremum over $q$ let
$A=|a\rangle\langle a|$ with
$\|a\|=1$, and $B=\diag(1,\gamma')$ with $\gamma'>0$. As $q\to0$ the ratio tends to $g(u)/g(v)$. For
$a\propto(1,\mu\gamma'^{-1/2})$ we get $u\propto(1,\mu)$ and $v\propto(1,\mu\gamma'^{(r-1)/2})$.
Letting $\mu\to0$ and $\gamma'\to\infty$ with $\mu\gamma'^{(r-1)/2}\to\infty$, $g(u)/g(v)$ tends
to $\kappa$.
\end{proof}

\begin{remark}\label{rem:pbnum}
Numerically (not used in the proofs), $c_{\rm PB}=b$ to six digits for $\rho=\diag(\kappa,1)$ in
$M_2(\C)$ with $\kappa=2$, $4$, $10$ and $100$, the common values being $1.0178$, $1.1030$, $1.4203$
and $5.5296$. For $\kappa_\cZ=10$, $c_{\rm PB}=1.4203$ also in $M_3(\C)$ for three spectra and in
$M_2(\C)\oplus M_2(\C)$. For $\kappa=10$ and $(r,q)=(2,1)$, $(2,\frac12)$, $(2,2)$ and $(4,\frac14)$,
multistart optimisation in $M_2(\C)$ gives the lower bounds $c_{r,q}\ge2.99$, $4.04$, $2.15$ and
$5.52$. The near-optimal pairs are close to singular, and these values need not be sharp; for
$(r,q)=(2,1)$ the value $2.99$ is superseded by the proved bound $G^2=3.025$ of
Proposition~\ref{prop:altlow}.
\end{remark}

\noindent Item~16 is the symmetric form, with the normalisation $\cN=\cR$, so $V(\rho)=c_{\rm PB}(\rho)-1$.
On $M_2$ the computed envelope of item~16 (Table~\ref{tab:values}) agrees with $b_2(\kappa)-1$ to within
$3\cdot10^{-3}$ at $\delta=0.1$, $0.5$ and $0.9$ and lies below it; together with Remark~\ref{rem:pbnum} this
supports $c_{\rm PB}=b=b_2(\kappa_\cZ)$. In the search the best points have large $\|A\|$ and $\|B\|$, as in
the lower bound construction of Proposition~\ref{prop:pb}\ref{pb:b}, where $L\to\infty$.

\begin{proposition}[Item 7]\label{prop:altlow}
$G_\cZ(\rho)^2\le c_{2,1}(\rho)\le\kappa_\cZ(\rho)$ in every finite von Neumann algebra. In $M_n(\C)$,
$G^2-1\le V_n\le\kappa-1$ for item~7, and the value $G^2$ of the quotient is approached by the explicit
family in the proof.
\end{proposition}

\begin{proof}
The upper bound is Proposition~\ref{prop:alt}. For the lower bound let $\rho=\diag(a,b)$ in $M_2$ and, in the basis
$f_1=e_+$, $f_2=e_-$, write $\rho_{11}=\rho_{22}=\frac{a+b}2$, $\rho_{12}=\frac{a-b}2$ for its entries. Take
$B=\eta f_1f_1^*+f_2f_2^*$ and $A=uu^*$ with $u=f_1+s\eta f_2$. Then $B^{1/2}AB^{1/2}=ww^*$ with $w=\sqrt\eta f_1+s\eta f_2$,
so $\cL=\|w\|^2\,w^*\rho w=\eta^2\rho_{11}+O(\eta^{5/2})$, and $Bu=\eta(f_1+sf_2)$, so $\cR=\|u\|^2\eta^2(\rho_{11}+2s\rho_{12}+s^2\rho_{22})$.
As $\eta\to0$, $\cL/\cR\to \rho_{11}/(\rho_{11}+2s\rho_{12}+s^2\rho_{22})$, which at $s=-\rho_{12}/\rho_{22}$ equals
$\rho_{11}\rho_{22}/(\rho_{11}\rho_{22}-\rho_{12}^2)=(a+b)^2/(4ab)=G(\rho)^2$. So $c_{2,1}^{M_2}(\diag(\kappa,1))\ge G^2$, and
Theorem~\ref{thm:transfer}\ref{t:low} (all expressions are admissible, as in Proposition~\ref{prop:alt})
gives $c_{2,1}(\rho)\ge G_\cZ^2$ for $\rho\notin\cZ$; for central $\rho$ both sides are~$1$.
\end{proof}

\noindent For $\kappa=10$ the bound is $G^2=121/40=3.025$; it exceeds the value $2.99$ found by the multistart
search of Remark~\ref{rem:pbnum}. On $M_n$ we conjecture that $G^2$ is the constant
(Conjecture~\ref{conj:values}).

\section{Gardner's inequality with two contrasts}\label{sec:gardner2}

This section works on matrix algebras and their direct sums. With the operator free, Gardner's
constant is $G_\cZ$ in every finite von Neumann algebra (Theorem~\ref{thm:gardner}); with the contrast of
$|X|$ prescribed as well, the proofs below use the spectral decomposition of finite matrices and
Lemma~\ref{lem:halmos}, and we do not know whether Theorem~A holds in finite von Neumann algebras with
$\delta$ replaced by the noncentrality (the rescaling $X\mapsto zX$ used for condition~(R) changes the
contrast of $|X|$). Theorem~C concerns the unnormalised quantities of~\cite{abed2024} on direct sums of
matrix algebras with the trace $\Tr$, for which no central version is needed.

\subsection{The normalised violation}

For Gardner's inequality both contrasts can be prescribed and the envelope found exactly. It depends on the two contrasts
only through their product.

\begin{theorem}\label{thm:gardnerD}
Let $\rho>0$ and $X$ be in $M_n$, $n\ge2$, with $\OMC(\rho)=\delta$ and $\OMC(|X|)=\delta'$, $\delta,\delta'\in[0,1]$,
$\delta\delta'<1$. Then
\[
  |\varphi_\rho(X)|\le\frac{1}{\sqrt{1-\delta^2\delta'^2}}\,\varphi_\rho(|X|),
\]
and for even $n$ the constant is attained: $V_n(\delta,\delta')=G(\delta\delta')-1=(1-\delta^2\delta'^2)^{-1/2}-1$
for Gardner's inequality. For odd $n=2k+1$,
\[
  V_n(\delta,\delta')=\max_{0\le\gamma\le1}\frac{k\sqrt{S^2+\gamma pq}+\ell\ell'}{k(S+\gamma p)+\ell\ell'}-1,\qquad
  S=\ell+\ell',\ q=(1+\ell)(1+\ell'),\ p=(1-\ell)(1-\ell'),
\]
with $\ell=\frac{1-\delta}{1+\delta}$, $\ell'=\frac{1-\delta'}{1+\delta'}$; it is smaller than $G(\delta\delta')-1$ for
$\delta'<1$ and tends to it as $n\to\infty$. For $\delta'=1$ both give $G-1$, the value of Theorem~\ref{thm:gardner}. The upper bound
$G(\delta\delta')$ holds also for items 2 and~3, with the contrasts of $|X|,|Y|$ and of $|A^*|$.
\end{theorem}
\begin{proof}
Write $X=UT$ with $T=|X|$. The supremum of $|\Tr(\rho UT)|=|\Tr(U\,T\rho)|$ over unitaries $U$ is $\|T\rho\|_1$, so we must bound
$\|T\rho\|_1/\Tr(T\rho)$. For fixed eigenbases, $\|T\rho\|_1-c\Tr(T\rho)$ is a convex function of the eigenvalues of $T$ (a norm of a
linear function minus a linear one), and likewise of the eigenvalues of $\rho$. By Lemma~\ref{lem:vertex}, applied to $T$ and
then to $\rho$, we may assume $\rho=a(\ell+(1-\ell)P)$ and $T=t(\ell'+(1-\ell')Q)$ with projections $P,Q$ and $\ell=\frac{1-\delta}{1+\delta}$,
$\ell'=\frac{1-\delta'}{1+\delta'}$. By Lemma~\ref{lem:halmos}, $\C^n$ splits into subspaces of dimension $1$ or $2$ invariant
under $P$ and $Q$, hence under $\rho$ and $T$, and $\|T\rho\|_1$ and $\Tr(T\rho)$ are sums over them. On a one-dimensional summand the
two contributions are equal. On a two-dimensional one, $\rho$ and $T$ have the eigenvalues $x_1,x_2$ and $y_1,y_2$, namely
$a,a\ell$ and $t,t\ell'$; with $S=x_1y_2+x_2y_1$, $q=(x_1+x_2)(y_1+y_2)$, $p=(x_1-x_2)(y_1-y_2)$ and $\gamma$ the squared overlap
of the leading eigenvectors, the proof of Proposition~\ref{prop:k2} gives $\|T\rho\|_1^2=S^2+\gamma pq$ and $\Tr(T\rho)=S+\gamma p$.
The function $\gamma\mapsto(S^2+\gamma pq)/(S+\gamma p)^2$ is maximal at $\gamma=S/q\in[0,1]$, with the value
\[
  \frac{q^2}{4S(S+p)}=\frac{(x_1+x_2)^2(y_1+y_2)^2}{4(x_1y_2+x_2y_1)(x_1y_1+x_2y_2)}=\frac1{1-\delta^2\delta'^2},
\]
the last step by inserting $x_2/x_1=\ell$, $y_2/y_1=\ell'$. A quotient of sums is at most the largest quotient of the
summands. Equality for even $n$: $n/2$ copies of the $2\times2$ configuration $\rho=\mathrm{diag}(1,\ell)$, $T$ with the
eigenvalues $1,\ell'$ at $\gamma=S/q$, and $U$ the adjoint polar factor of $T\rho$. For odd $n$ at least one summand is
one-dimensional and contributes the same amount $\mu\ge\ell\ell'$ (normalising the largest eigenvalues of $\rho$ and $T$
to $1$) to $\|T\rho\|_1$ and to $\Tr(T\rho)$; the quotient $(\sum_js_j+\mu)/(\sum_jt_j+\mu)$, where $s_j\ge t_j$ are the
contributions of at most $k$ two-dimensional summands, is largest for $k$ summands and $\mu=\ell\ell'$, and, by the argument
of Lemma~\ref{lem:vertex} applied to $s-ct$, for the same $\gamma$ in every summand, which gives the stated maximum; it is attained by
$\rho=\mathrm{diag}(1,\ell)^{\oplus k}\oplus\ell$ and $T$ correspondingly. Items 2 and~3 reduce to the first inequality as in
the proofs of Theorem~\ref{thm:subadd} and Corollary~\ref{cor:item3}, since $|U^*X|=|X|$ for unitary $U$.
\end{proof}

\noindent Thus the normalised violation of Gardner's inequality is $\frac12\delta^2\delta'^2+O(\delta^4\delta'^4)$: it
vanishes to second order in each contrast and is symmetric in the density and the operator. For a Gibbs state at inverse
temperature $\beta$ and an operator $|X|=e^{-\beta' K}$, the constant is $(1-\tanh^2(\beta\Delta E/2)\tanh^2(\beta'\Delta
E'/2))^{-1/2}$, where $\Delta E$ and $\Delta E'$ are the spectral spreads of $H$ and $K$. The two-dimensional computation in
the proof has a transparent form, which explains the dependence on the product.

\begin{remark}[Geometric form]\label{rem:geom}
On $\C^2$ let $\rho=w(1+\delta S)$ and $T=t(1+\delta'S')$ with $w,t>0$ and symmetries $S,S'$ of trace zero. The unitary $S'S$ has
determinant $1$, hence eigenvalues $e^{\pm i\alpha}$, and $\Tr(SS')=\Tr(S'S)=2\cos\alpha$. Then
\[
  \Tr(T\rho)=2wt\,\Re z,\qquad \|T\rho\|_1=2wt\,|z|,\qquad z=1+\delta\delta'e^{i\alpha}.
\]
The first identity follows from $\Tr(T\rho)=wt\Tr(1+\delta S+\delta'S'+\delta\delta'S'S)$. For the second, $\|M\|_1^2=\Tr(M^*M)+
2|\det M|$ for $2\times2$ matrices, and $\Tr(\rho^2T^2)=w^2t^2\Tr\bigl((1+\delta^2+2\delta S)(1+\delta'^2+2\delta'S')\bigr)=
2w^2t^2\bigl((1+\delta^2)(1+\delta'^2)+4\delta\delta'\cos\alpha\bigr)$ and $|\det T\rho|=w^2t^2(1-\delta^2)(1-\delta'^2)$ add up to
$4w^2t^2|z|^2$. So Gardner's quotient $\sup_U|\Tr(\rho UT)|/\Tr(\rho T)$ equals $|z|/\Re z=\sec\arg z$. As the relative position of
$\rho$ and $T$ varies, $z$ runs over the circle of radius $\delta\delta'$ about $1$, whose points have $|\arg z|\le\theta$,
$\sin\theta=\delta\delta'$, with equality for $\cos\alpha=-\delta\delta'$ (Figure~\ref{fig:angle}(c)). This gives the
$2\times2$ case of Theorem~\ref{thm:gardnerD}, with the constant $\sec\theta=(1-\delta^2\delta'^2)^{-1/2}$. Since
$\delta=\sin\psi(\rho)$ and $\delta'=\sin\psi(T)$, the angle $\theta$ is determined by $\sin\theta=\sin\psi(\rho)\sin\psi(T)$.
\end{remark}

\subsection{\texorpdfstring{The unnormalised violations of~\cite{abed2024}}{The unnormalised violations of Abed et al.}}\label{ssec:unnorm}
We return to the unnormalised quantity of~\cite{abed2024}, in which both operators are normalised to
norm one and both contrasts are prescribed. Let $\nu=\sum_c\lfloor k_c/2\rfloor$ be the largest number of mutually
orthogonal two-dimensional subspaces that fit into the blocks of $\M$.

\begin{thmC}
Let $\M=\bigoplus_{c=1}^CM_{k_c}(\mathbb K)$, $\mathbb K\in\{\R,\C\}$, with $\nu\ge1$. For $X\ge0$ and $Z$ in $\M$ with
$\|X\|=\|Z\|=1$, $\OMC(X)=\delta<1$ and $\OMC(|Z|)=\delta'$,
\[
  \sup\bigl(|\Tr(XZ)|-\Tr(X|Z|)\bigr)=\frac{\nu\,\delta^2\delta'^2}{(1+\delta)(1+\delta')},
\]
and the supremum is attained (Theorem~\ref{thm:gardner2}). The two other functionals of~\cite{abed2024},
$\Tr|XY|-\Tr(|X||Y|)$ and $\Tr(Y|X|Y)-\Tr|YXY|$, have the sharp bounds of Proposition~\ref{prop:exp2} and
Theorem~\ref{thm:exp3}.
\end{thmC}

The unnormalised violation behaves like the normalised one.

In~\cite[Sect.~5]{abed2024} the violation of Gardner's inequality was studied without normalisation: for $X\ge0$ with
$\|X\|=1$ (the density) and $Z$ with $\|Z\|=1$, the quantity
\[
  z(X,Z)=|\Tr(XZ)|-\Tr(X|Z|)
\]
was sampled against $\OMC(X)$ and $\OMC(|Z|)$. Writing $Z=U|Z|$ with a unitary (orthogonal, in the real case) $U$,
$z=|\Tr(XUY)|-\Tr(XY)$ with $Y=|Z|\ge0$, and $\sup_Uz=\|YX\|_1-\Tr(XY)$ by the duality of the trace norm, the supremum being
attained at the polar unitary of $YX$, which is real if $X,Y$ are. We first solve the $2\times2$ case.

\begin{proposition}\label{prop:k2}
Let $X,Y\ge0$ in $M_2(\C)$ have eigenvalues $x_1\ge x_2\ge0$ and $y_1\ge y_2\ge0$, $x_1,y_1>0$. For every unitary $U$,
\[
  |\Tr(XUY)|-\Tr(XY)\ \le\ \frac{(x_1-x_2)^2(y_1-y_2)^2}{4(x_1+x_2)(y_1+y_2)} .
\]
Let $S=x_1y_2+x_2y_1$, $q=(x_1+x_2)(y_1+y_2)$, and $\gamma=|\langle e_1,f_1\rangle|^2$ for unit eigenvectors $e_1$ of $X$ for
$x_1$ and $f_1$ of $Y$ for $y_1$. If $x_1>x_2$ and $y_1>y_2$, the bound is attained iff $|\Tr(XUY)|=\|YX\|_1$ and
$\gamma=\frac14+\frac S{2q}$; it is attained by real $X,Y,U$.
\end{proposition}
\begin{proof}
Let $p=(x_1-x_2)(y_1-y_2)$, so that $q-2S=p$. First, $|\Tr(XUY)|=|\Tr(YXU)|\le\|YX\|_1$. Expanding in the eigenbases,
$\Tr(XY)=\sum_{ij}x_iy_j|\langle e_i,f_j\rangle|^2=S+\gamma p$, since the matrix $(|\langle e_i,f_j\rangle|^2)$ is doubly
stochastic. The singular values of $YX$ satisfy $s_1^2+s_2^2=\Tr(X^2Y^2)$ and $s_1s_2=|\det YX|=x_1x_2y_1y_2$, so
$\|YX\|_1^2=\Tr(X^2Y^2)+2x_1x_2y_1y_2=S^2+\gamma pq$. With $\sqrt r\le r/q+q/4$ (equality iff $\sqrt r=q/2$),
\[
  \|YX\|_1-\Tr(XY)\le\frac{S^2+\gamma pq}q+\frac q4-S-\gamma p=\frac{(q/2-S)^2}q=\frac{p^2}{4q},
\]
with equality iff $\gamma pq=q^2/4-S^2=\frac p2(\frac q2+S)$. Every $\gamma\in[0,1]$ is realised by a real rotation, and the
polar unitary of a real $YX$ is real.
\end{proof}

\begin{theorem}\label{thm:gardner2}
Let $\M=\bigoplus_{c=1}^CM_{k_c}(\mathbb K)$, $\mathbb K\in\{\R,\C\}$, and $\nu=\sum_c\lfloor k_c/2\rfloor\ge1$. For $X\ge0$ and
$Y\ge0$ in $\M$ with $\|X\|=\|Y\|=1$, $\OMC(X)=\delta\in[0,1)$ and $\OMC(Y)=\delta'\in[0,1]$, and every unitary
(orthogonal) $U\in\M$,
\[
  |\Tr(XUY)|-\Tr(XY)\le\frac{\nu\,\delta^2\delta'^2}{(1+\delta)(1+\delta')},
\]
and the bound is attained. Equivalently, $\sup\{|\Tr(XZ)|-\Tr(X|Z|):\|X\|=\|Z\|=1,\ \OMC(X)=\delta,\ \OMC(|Z|)=\delta'\}$
has this value.
\end{theorem}
\begin{proof}
Let $\ell=(1-\delta)/(1+\delta)$ and $\ell'=(1-\delta')/(1+\delta')$ be the smallest eigenvalues. It suffices to prove the
bound in the complex algebra, since the maximisers below are real. By the remark above, $\sup_Uz=F(X,Y):=\|YX\|_1-\Tr(XY)$,
and the polar unitary lies in $\M$. For a fixed eigenbasis of $X$, $F$ is a convex function of the eigenvalues of $X$ (a
norm of a linear function of them minus a linear function). The admissible eigenvalues form the union of the faces
$\{x_i=1,\ x_j=\ell\}$ of the cube $[\ell,1]^K$, $K=\sum_ck_c$, and a convex function on a face attains its maximum at a vertex; so
$F(X,Y)\le F(X',Y)$ for some $X'=\ell+(1-\ell)P$ with a projection $P\in\M$, $P\ne0,1$, and likewise $F(X',Y)\le F(X',Y')$
with $Y'=\ell'+(1-\ell')Q$. By Lemma~\ref{lem:halmos}, applied in every block, $X'$ and $Y'$ are direct sums over subspaces of
dimension $1$ or $2$, and so are $\|Y'X'\|_1$ and $\Tr(X'Y')$. One-dimensional summands contribute $0$. On a
two-dimensional one, $X'=\mathrm{diag}(1,\ell)$ and $Y'$ has the eigenvalues $1,\ell'$, so by Proposition~\ref{prop:k2} it
contributes at most $(1-\ell)^2(1-\ell')^2/(4(1+\ell)(1+\ell'))=\delta^2\delta'^2/((1+\delta)(1+\delta'))$. A block of size
$k_c$ holds at most $\lfloor k_c/2\rfloor$ two-dimensional summands. Equality: $\nu$ blocks $\mathrm{diag}(1,\ell)$ and
$Y$ with spectrum $\{1,\ell'\}$ at the angle of Proposition~\ref{prop:k2}, $X=Y=1$ on the remaining one-dimensional blocks,
and the polar unitary, which makes all block traces positive.
\end{proof}

\begin{figure}[t]
\centering\includegraphics[width=\textwidth]{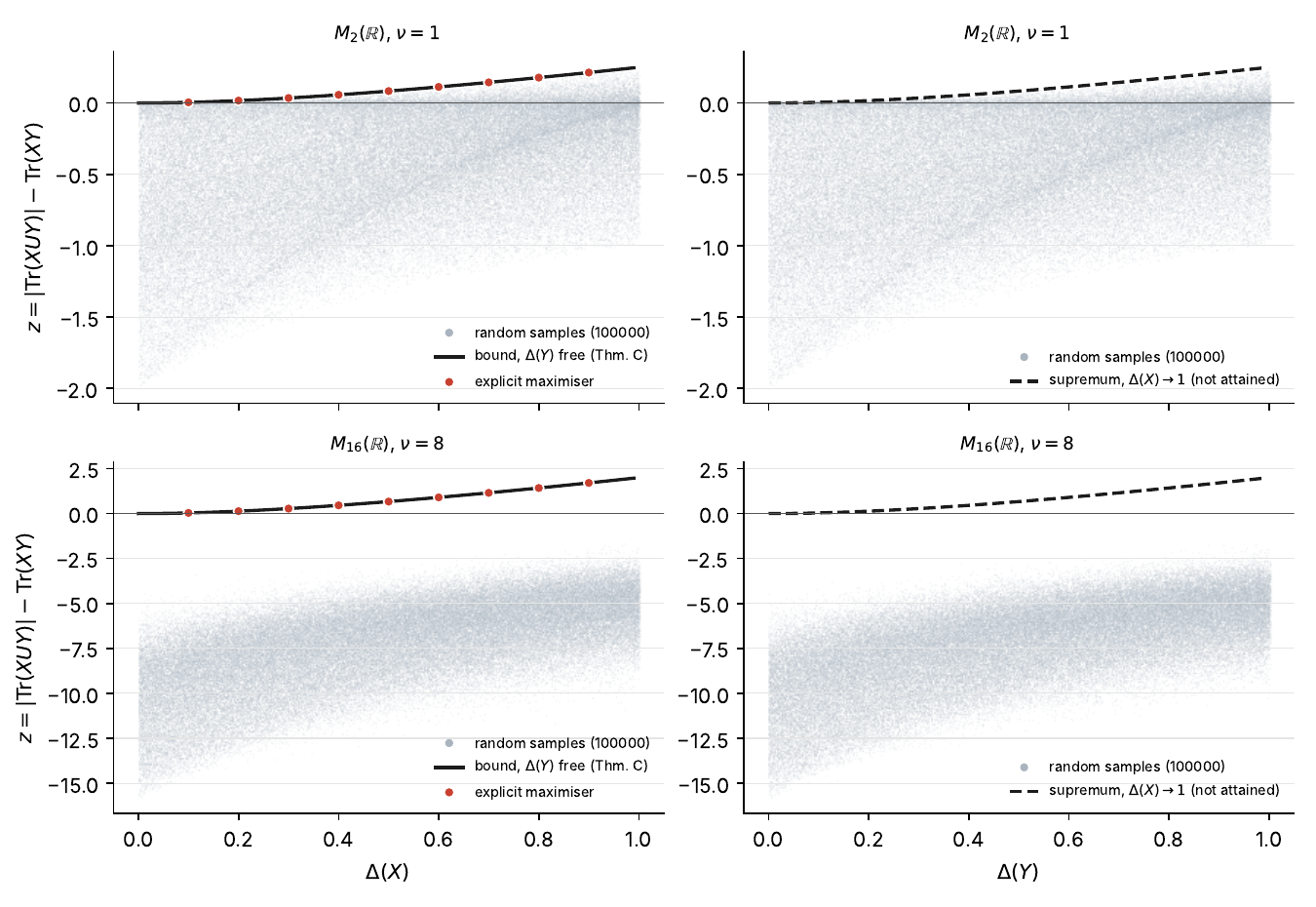}
\caption{The unnormalised Gardner violation $z=|\Tr(XUY)|-\Tr(XY)$ for $\|X\|=\|Y\|=1$ in $M_2(\R)$ (top, $\nu=1$) and
$M_{16}(\R)$ (bottom, $\nu=8$), against $\OMC(X)$ (left) and $\OMC(Y)$ (right). Grey: $100000$ random samples with
Haar-orthogonal eigenbases and $U$ and uniformly drawn contrasts, as in~\cite[Fig.~1]{abed2024}. Solid: the bound of
Theorem~\ref{thm:gardner2} with $\OMC(Y)$ free, $\nu\delta^2/(2(1+\delta))$; dots: $z$ at the explicit maximisers of its
proof, which agree with it to $3\times10^{-15}$ (printed by \texttt{trace\_\allowbreak envelopes\_\allowbreak
figures.py}). Dashed: the supremum along $\OMC(Y)$, approached as $\OMC(X)\to1$. In
$M_{16}$ no sample violates Gardner's inequality ($z>0$), although violations up to almost $2$ exist: the maximisers,
two-level spectra on aligned pairs of eigenvectors, form a set of negligible Haar measure. In the figure
$\Delta(\cdot)=\OMC(\cdot)$.}
\label{fig:gardner1d}
\end{figure}

\begin{figure}[t]
\centering\includegraphics[width=0.8\textwidth]{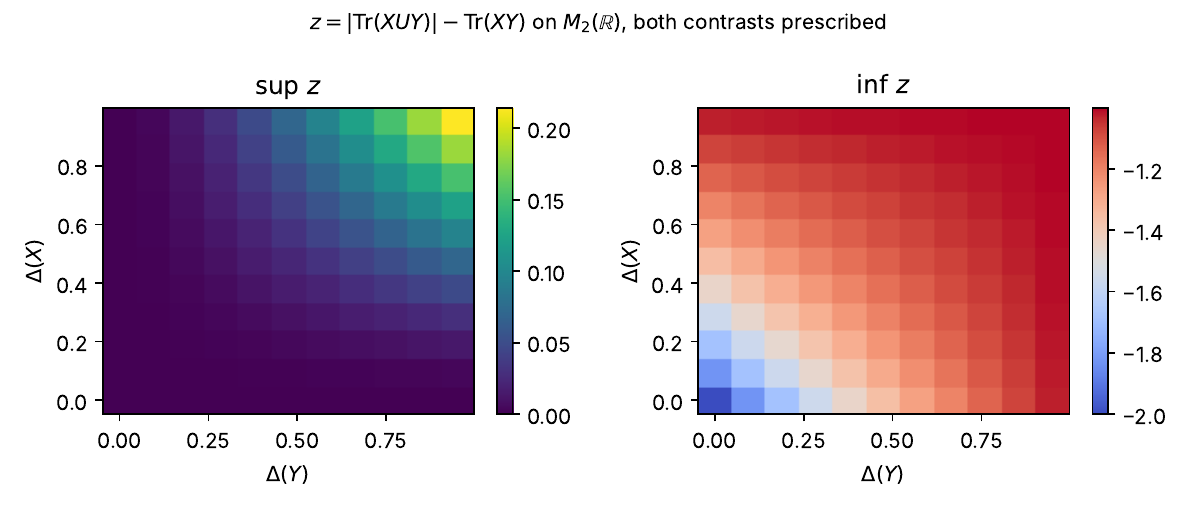}
\caption{$M_2(\R)$, both contrasts prescribed: the supremum (left) and infimum (right) of
$z=|\Tr(XUY)|-\Tr(XY)$ found by constrained optimisation on a grid of $\OMC(Y)$ (horizontal) and $\OMC(X)$
(vertical). The left panel agrees with $\delta^2\delta'^2/((1+\delta)(1+\delta'))$ (Theorem~\ref{thm:gardner2}) to
$\TCCGardnerTwoDDev$ at every grid point: a violation needs both operators to be far from scalar. In the
figure $\Delta(\cdot)=\OMC(\cdot)$.}
\label{fig:gardner2d}
\end{figure}

\begin{remark}\label{rem:gardner2}
With the contrast of $|Z|$ free, the supremum $\nu\delta^2/(2(1+\delta))$ is reached at $\delta'=1$, i.e.\ for singular
$Z$; the normalised counterpart is Theorem~\ref{thm:gardner}, $\sup_Z|\Tr(XZ)|/\Tr(X|Z|)=G(X)$. The bound is additive
over the blocks: it counts the mutually orthogonal two-dimensional subspaces that fit into them and does not depend on
the sizes of the blocks otherwise. Figures~\ref{fig:gardner1d} and~\ref{fig:gardner2d} compare it with random sampling,
the method of~\cite{abed2024}, and with numerical optimisation. Sampling stays far below the bound as soon as the blocks are
larger than $2\times2$, while the optimisation over both contrasts reproduces it.
\end{remark}

\subsection{\texorpdfstring{The other two functionals of~\cite{abed2024}}{The other two functionals of Abed et al.}}\label{sec:exp23}

Besides Gardner's inequality, \cite[Sect.~5]{abed2024} sampled two further unnormalised trace functionals,
\[
  z_2=\Tr|XY|-\Tr(|X|\,|Y|),\qquad z_3=\Tr(Y|X|Y)-\Tr|YXY|\quad(X=X^*,\ Y\ge0),
\]
with $|X|$ and $|Y|$ normalised to $\lambda_{\max}=1$ and their contrasts as parameters. Both are determined here in
$\M=\bigoplus_cM_{k_c}$, with $\nu=\sum_c\lfloor k_c/2\rfloor$ and $\ell=\frac{1-\delta}{1+\delta}$, $\ell'=\frac{1-\delta'}{1+\delta'}$
as before; note $1-\ell=\frac{2\delta}{1+\delta}$.

\begin{proposition}\label{prop:exp2}
Let $X=UX_0$, $Y=VY_0$ with $X_0,Y_0\ge0$ and unitaries $U,V\in\M$, so that $|X|=X_0$, $|Y|=Y_0$ and
$z_2=\|X_0VY_0\|_1-\Tr(X_0Y_0)$. For given spectra,
\[
  \sup z_2=-\inf z_2=\sum_c\sum_ix_{c,i}^\downarrow\bigl(y_{c,i}^\downarrow-y_{c,i}^\uparrow\bigr),
\]
with $x_{c,i}^\downarrow$ the eigenvalues of $X_{0,c}$ in decreasing order and $y_{c,i}^\downarrow$, $y_{c,i}^\uparrow$ those of
$Y_{0,c}$ in decreasing and increasing order. With $\OMC(|X|)=\delta$ and $\OMC(|Y|)=\delta'$ prescribed,
$\sup z_2=-\inf z_2=\nu(1-\ell)(1-\ell')=\frac{4\nu\delta\delta'}{(1+\delta)(1+\delta')}$, attained by two-level spectra.
\end{proposition}
\begin{proof}
$V$ and the eigenbases act sector by sector, and von Neumann's trace inequality $\|AB\|_1\le\sum_is_i(A)s_i(B)$, its lower
counterpart $\|AB\|_1\ge\sum_is_i(A)s_{k+1-i}(B)$ and the rearrangement inequality for $\Tr(X_0Y_0)$ give the first
formula; both extremes are attained by aligning or reversing the eigenbases. Pairing $i$ with $k+1-i$,
$\sum_ix_{c,i}^\downarrow(y_{c,i}^\downarrow-y_{c,i}^\uparrow)=\sum_{i\le k/2}(x_{c,i}^\downarrow-x_{c,k+1-i}^\downarrow)
(y_{c,i}^\downarrow-y_{c,k+1-i}^\downarrow)\le\lfloor k/2\rfloor(1-\ell)(1-\ell')$, with equality for two-level spectra.
\end{proof}

\noindent So $z_2$ is linear in each contrast, unlike Gardner's violation, which is quadratic in each. The third
functional is never negative:

\begin{proposition}\label{prop:exp3eq}
For $X=X^*$ and $Y\ge0$, $\Tr|YXY|\le\Tr(Y|X|Y)$, with equality if and only if $YX_+Y^2X_-Y=0$, in particular if $Y$
commutes with $X$.
\end{proposition}
\begin{proof}
With $X=X_+-X_-$ and $|X|=X_++X_-$, the triangle inequality for the trace norm gives
$\|YX_+Y-YX_-Y\|_1\le\Tr(YX_+Y)+\Tr(YX_-Y)$. For positive $A,B$, $\|A-B\|_1=\Tr A+\Tr B$ if and only if $AB=0$: with $P$
the projection onto the positive part of $A-B$, $\|A-B\|_1=\Tr A+\Tr B-2\Tr((1-P)A)-2\Tr(PB)$, and equality forces
$(1-P)A=0=PB$, so $AB=AP(1-P)B=0$; conversely, if $AB=0$ the supports are orthogonal and $|A-B|=A+B$. Apply this to
$A=YX_+Y$, $B=YX_-Y$. If $XY=YX$, then $YX_+Y^2X_-Y=Y^4X_+X_-=0$.
\end{proof}

\begin{proposition}\label{prop:k2b}
Let $X=X^*$ in $M_2(\C)$, let $|X|$ have eigenvalues $x_1\ge x_2\ge0$, $x_1>0$, and let $Y\ge0$ have eigenvalues
$y_1\ge y_2\ge0$. Then
\[
  \Tr(Y|X|Y)-\Tr|YXY|\le\frac{2x_1x_2}{x_1+x_2}\,(y_1-y_2)^2 .
\]
With $S=x_1y_2+x_2y_1$, $q=(x_1+x_2)(y_1+y_2)$ and $\gamma=|\langle e_1,f_1\rangle|^2$ for unit eigenvectors $e_1$ of $X$
for its eigenvalue of modulus $x_1$ and $f_1$ of $Y$ for $y_1$: if $x_2>0$ and $y_1>y_2$, the bound is attained if and only
if $X$ is indefinite and $\gamma=S/q$; it is attained by real matrices.
\end{proposition}
\begin{proof}
$z_3$ is invariant under $X\mapsto-X$ and vanishes for semidefinite $X$. Otherwise let $Xe_1=x_1e_1$, $Xe_2=-x_2e_2$,
$u=\langle e_1,Y^2e_1\rangle=\gamma y_1^2+(1-\gamma)y_2^2$ and $s=y_1^2+y_2^2$. Then $\Tr(Y|X|Y)=\Tr(|X|Y^2)=A:=(x_1-x_2)u+x_2s$,
and since $\det(YXY)=-x_1x_2y_1^2y_2^2\le0$, $\Tr|YXY|=\sqrt{\Tr(YXY)^2-4\det(YXY)}=\sqrt{L^2+K^2}$ with
$L=\Tr(XY^2)=(x_1+x_2)u-x_2s$ and $K=2\sqrt{x_1x_2}\,y_1y_2$. For the unit vector
$(\alpha,\beta)=(x_1-x_2,\,2\sqrt{x_1x_2})/(x_1+x_2)$ the Cauchy--Schwarz inequality gives $\sqrt{L^2+K^2}\ge\alpha L+\beta K$,
and $A-\alpha L-\beta K=\frac{2x_1x_2}{x_1+x_2}(y_1-y_2)^2$ for every $u$; equality holds iff $\beta L=\alpha K$, i.e.\
$u-y_2^2=(y_1-y_2)S/(x_1+x_2)$, i.e.\ $\gamma=S/q$. Real maximisers: diagonal $X$ and $Y$ rotated by $\arccos\sqrt\gamma$.
\end{proof}

\begin{theorem}\label{thm:exp3}
Let $\M=\bigoplus_{c=1}^CM_{k_c}(\mathbb K)$, $\mathbb K\in\{\R,\C\}$, with $\nu\ge1$, and let $|X|$ and $Y$ be normalised to
$\lambda_{\max}=1$, with $\OMC(|X|)=\delta\in[0,1)$ and $\OMC(Y)=\delta'\in[0,1]$. Then
\[
  \sup\bigl(\Tr(Y|X|Y)-\Tr|YXY|\bigr)=\begin{cases}(\nu-\delta)\bigl(\frac{2\delta'}{1+\delta'}\bigr)^2&\text{if every }k_c\text{ is even},\\[2pt]
  \nu\bigl(\frac{2\delta'}{1+\delta'}\bigr)^2&\text{if some }k_c\text{ is odd},\end{cases}
\]
and the supremum is attained by real direct sums of the $2\times2$ maximisers of Proposition~\ref{prop:k2b} and $1\times1$
blocks.
\end{theorem}

\noindent For odd blocks the eigenvalue $\ell$ of $|X|$ can sit on a one-dimensional summand, where it costs nothing, so
the supremum does not depend on $\delta$; for $k=3$, $C=1$: $(2\delta'/(1+\delta'))^2>(1-\delta)(2\delta'/(1+\delta'))^2$. The
proof reduces to two-level spectra and, with Lemma~\ref{lem:halmos}, to $2\times2$ blocks, except in one case: the
eigenvector of $|X|$ for $\ell$ may be spread over all two-dimensional blocks of its sector, so that $X$ is not block
diagonal. This case is settled by the following sharp bound, whose proof linearises the problem.

\begin{proposition}\label{prop:mblocks}
Let $X=X_+\oplus(-1)$ on $\C^m\oplus\C^m$, where $0<X_+\le1$ has eigenvalues $x_1,\dots,x_m$; let
$c_\theta=\mathrm{diag}(\cos\theta_j)$, $s_\theta=\mathrm{diag}(\sin\theta_j)$, $Q$ the projection onto the range of
$\binom{c_\theta}{s_\theta}:\C^m\to\C^{2m}$, $b\in[0,1)$ and $Y=b+(1-b)Q$ (a two-level operator). Then
\[
  \Tr(Y|X|Y)-\Tr|YXY|\ \le\ (1-b)^2\sum_{i=1}^m\frac{2x_i}{1+x_i},
\]
with equality for $X_+=\mathrm{diag}(x_1,\dots,x_m)$ and $\cos^2\theta_i=(1+bx_i)/((1+x_i)(1+b))$.
\end{proposition}
\begin{proof}
Let $\gamma=1-b^2$, $K=1+X_+$, $A=c_\theta^2$ (so $s_\theta^2=1-A$). In the decomposition $\C^m\oplus\C^m$,
$Y^2=b^2+\gamma Q$ has the blocks $b^2+\gamma A$, $\gamma c_\theta s_\theta$, $\gamma c_\theta s_\theta$, $b^2+\gamma s_\theta^2$, all diagonal.
For $\lambda>0$ the lower right block of $\lambda-XY^2$ is the invertible diagonal $E=\lambda+1-\gamma A$, and as
$(b^2+\gamma A)E-\gamma^2As_\theta^2=\lambda(b^2+\gamma A)+b^2$, the Schur complement formula gives
\[
  \det(\lambda-XY^2)=\det\bigl(\lambda E-X_+(\lambda(b^2+\gamma A)+b^2)\bigr)=\det\bigl((\lambda+1)(\lambda-b^2X_+)-\lambda\gamma KA\bigr)
  =p(\lambda),
\]
$p(\lambda)=\det\bigl((\lambda+1)(\lambda-b^2X_+)-\lambda\gamma K^{1/2}AK^{1/2}\bigr)$. The Schur complement of the lower right
block $\lambda$ gives $\det(\lambda-\Xi)=p(\lambda)$ as well, for the self-adjoint matrix
\[
  \Xi=\begin{pmatrix}\gamma K^{1/2}AK^{1/2}+b^2X_+-1 & bX_+^{1/2}\\ bX_+^{1/2} & 0\end{pmatrix},
\]
which is affine in $A$. Since $\det(\lambda-YXY)=\det(\lambda-XY^2)$, the polynomials agree and $YXY$ and $\Xi$ have
the same eigenvalues; hence $\Tr|YXY|=\|\Xi\|_1\ge\Tr(T\Xi)$ for every self-adjoint $T$ with $\|T\|\le1$.
Take
\[
  T=\begin{pmatrix}-\Lambda & \Gamma\\ \Gamma & \Lambda\end{pmatrix},\qquad \Lambda=\frac{1-X_+}{1+X_+},\quad \Gamma=\frac{2X_+^{1/2}}{1+X_+},
\]
so that $T^2=1$ ($\Lambda$ and $\Gamma$ commute and $\Lambda^2+\Gamma^2=1$). With $\Tr(Y|X|Y)=\Tr(|X|Y^2)=b^2\Tr X_++\gamma\Tr(X_+A)+mb^2+\gamma\Tr(1-A)$
and $K^{1/2}\Lambda K^{1/2}=1-X_+$, the terms containing $A$ cancel in $\Tr(|X|Y^2)-\Tr(T\Xi)$, and
\[
  \Tr(Y|X|Y)-\Tr|YXY|\le\sum_{i=1}^m\Bigl(b^2x_i+1+\frac{1-x_i}{1+x_i}(b^2x_i-1)-\frac{4bx_i}{1+x_i}\Bigr)=(1-b)^2\sum_{i=1}^m\frac{2x_i}{1+x_i}.
\]
For diagonal $X_+$ both operators are direct sums of $2\times2$ blocks, on which $X=\mathrm{diag}(x_i,-1)$, and
Proposition~\ref{prop:k2b} gives the value $(1-b)^2\,2x_i/(1+x_i)$ at the stated angle (there $\gamma=\sin^2\theta_i$).
\end{proof}

\begin{lemma}\label{lem:exp3}
Let $X=X^*$ be invertible, $P$ the projection onto the positive part of $X$, $S=2P-1$, and $Y\ge0$. Then
(a) if $|X|\le|X'|$, $X'=S|X'|$ and $|X'|$ commutes with $P$, then $z_3(X,Y)\le z_3(X',Y)$; (b) for fixed $X$, $z_3$ is a
convex function of $Y^2$.
\end{lemma}
\begin{proof}
With $A=YX_+Y$ and $B=YX_-Y$, $z_3=\Tr A+\Tr B-\|A-B\|_1$. (a) $A=Y|X|^{1/2}P|X|^{1/2}Y$ is non-decreasing in $|X|$ among
operators commuting with $P$, and so is $B$; by the triangle inequality $\|A+D-B\|_1\le\|A-B\|_1+\Tr D$ for $D\ge0$, so
$z_3$ does not decrease when $A$ or $B$ increases. (b) As $\|H\|_1=\max_{-1\le T\le1}\Tr(TH)$ for self-adjoint $H$, with
$T=1-2E$ we get $z_3=2\min_{0\le E\le1}\Tr\bigl(EA+(1-E)B\bigr)$. For invertible $Y$, $Z=|X|^{1/2}Y$ is invertible and
$E\mapsto F=ZEZ^*$ maps $\{0\le E\le1\}$ onto $\{0\le F\le \Gamma\}$, $\Gamma=ZZ^*=|X|^{1/2}Y^2|X|^{1/2}$, with $\Tr(EA)=\Tr(PF)$ and
$\Tr((1-E)B)=\Tr((1-P)(\Gamma-F))$. So $z_3=2\min\{\Tr(PF+(1-P)(\Gamma-F)):0\le F\le \Gamma\}$ is the partial minimum of a linear
function over a convex set of pairs $(F,\Gamma)$, hence convex in $\Gamma$, which is affine in $Y^2$; by continuity also at
singular $Y$.
\end{proof}

\begin{proof}[Proof of Theorem~\ref{thm:exp3}]
Write $a=\ell$, $b=\ell'$; it suffices to treat the complex algebra, since the maximisers are real. Let $\delta<1$, so that
$X$ is invertible; $z_3(-X,Y)=z_3(X,Y)$. Since $|X|$ commutes with $P$, it has a unit eigenvector $e$ for its eigenvalue $a$
in one sector and in $\mathrm{ran}\,P$ or in $\ker P$; replacing $X$ by $-X$ if necessary, $e\in\mathrm{ran}\,P$. Then
$|X|\le|X'|:=1-(1-a)ee^*$, and $X'=S|X'|=S-(1-a)ee^*$ has contrast $\delta$; by Lemma~\ref{lem:exp3}(a),
$z_3(X,Y)\le z_3(X',Y)$. By Lemma~\ref{lem:exp3}(b), a convex function of the eigenvalues of $Y^2$ on the faces
$\{y_i=1,\ y_j=b^2\}$ of $[b^2,1]^K$ is maximal at a vertex, so $z_3(X',Y)\le z_3(X',Y')$ with $Y'=b+(1-b)Q$, $Q\ne0,1$ a
projection in $\M$. Apply Lemma~\ref{lem:halmos} to $P$ and $Q$ in every sector. In a sector without $e$, $X'=S$ is a
direct sum over the subspaces, a one-dimensional summand contributes $0$, and a two-dimensional one at most $(1-b)^2$ by
Proposition~\ref{prop:k2b} ($|X|=1$ there): together at most $\lfloor k_c/2\rfloor(1-b)^2$. In the sector of $e$,
Lemma~\ref{lem:exp3}(a) with $|X'|\le1$ gives the same bound for $z_3$ restricted to it, and, if its size $k_c$ is odd or it
contains fewer than $k_c/2$ two-dimensional subspaces, at most $\lfloor k_c/2\rfloor(1-b)^2$, resp.\ $(k_c/2-1)(1-b)^2$.
Otherwise its $m=k_c/2$ two-dimensional subspaces fill it, with bases $f_j,f_j'$, $f_j\in\mathrm{ran}\,P$; the unitary that
multiplies $f_j'$ by $e^{-i\chi_j}$ commutes with $X'$ (which is $-1$ on $\ker P$) and removes the phases, and in the basis
$(f_1,\dots,f_m,f_1',\dots,f_m')$ we are in the setting of Proposition~\ref{prop:mblocks} with $X_+=1-(1-a)ee^*$, whose
eigenvalues are $a,1,\dots,1$: the sector contributes at most $(1-b)^2\bigl(\frac{2a}{1+a}+m-1\bigr)=(k_c/2-\delta)(1-b)^2$.
Summing over the sectors, $z_3\le(\nu-\delta)(1-b)^2$ if every $k_c$ is even and $z_3\le\nu(1-b)^2$ otherwise. Equality: if every
$k_c$ is even, $X=\mathrm{diag}(a,-1)\oplus\mathrm{diag}(1,-1)^{\oplus(\nu-1)}$ and $Y$ with eigenvalues $1,b$ in every block at
the angle of Proposition~\ref{prop:k2b}; if some $k_c$ is odd, $X=\mathrm{diag}(1,-1)$ on $\nu$ blocks, the eigenvalue $a$ on a
one-dimensional block of an odd sector, $1$ on the other one-dimensional blocks, and $Y$ as before ($Y=1$ on the
one-dimensional blocks).
\end{proof}

\noindent Before this proof was found, the cases $k=4$ and $k=6$ ($C=1$) had been verified by a rigorous branch-and-bound
computation over the angles of the reduced model (with gaps $10^{-7}$ and $10^{-5}$ that include floating-point
rounding), and numerical optimisation reproduces all suprema of $z_2$ and $z_3$. Theorems~\ref{thm:gardner2} and~\ref{thm:exp3} and Proposition~\ref{prop:exp2} thus settle
the three functionals of~\cite[Sect.~5]{abed2024} with both contrasts prescribed. Like Gardner's violation, they are
additive over the blocks: they count the mutually orthogonal two-dimensional subspaces that fit into the sectors.

\section{Characterisations of central elements}\label{sec:centre}

In this section $\cM$ is an arbitrary von Neumann algebra, unless finiteness is stated, and
$a\in\cM_{++}$. Normal functionals have the polar decomposition $\varphi=|\varphi|(\cdot\,u)$,
i.e.\ $\varphi(x)=|\varphi|(xu)$, with $|\varphi|\ge0$ and a partial isometry $u$ with
$u^*u=s(|\varphi|)$, and this decomposition is unique \cite[Ch.~III, \S4]{Tak79}; then
$|\varphi|(x)=\varphi(xu^*)$. Hermitian normal functionals have a unique Jordan decomposition
$\varphi=\varphi^+-\varphi^-$ with $\|\varphi\|=\|\varphi^+\|+\|\varphi^-\|$, and then
$|\varphi|=\varphi^++\varphi^-$ \cite[Ch.~III, \S4]{Tak79}. (The other convention
$\varphi=|\varphi|(u\,\cdot)$ is interchanged with ours by $\varphi\mapsto\varphi^*$, where
$\varphi^*(x)=\overline{\varphi(x^*)}$; this does not change any of the quantities below.) For $z\in\cZ_{++}$ put $\varphi_z=\varphi(z\,\cdot)$.

\begin{lemma}\label{lem:funct}
For $z\in\cZ_{++}$, $\varphi\mapsto\varphi_z$ is a bijection of the normal (Hermitian, positive)
functionals, $|\varphi_z|=|\varphi|_z$, $(\varphi_z)^\pm=(\varphi^\pm)_z$, and
$|\varphi_z+\omega_z|=|\varphi+\omega|_z$.
\end{lemma}

\begin{proof}
$\varphi_z(x)=|\varphi|(zxu)=|\varphi|_z(xu)$ and $s(|\varphi|_z)=s(|\varphi|)=u^*u$, since $z$ is central and
invertible; uniqueness gives $|\varphi_z|=|\varphi|_z$. The other claims are similar.
\end{proof}

\begin{theorem}[Item 40]\label{thm:pap}
$\lambda(a):=\min\{\lambda:pap\le\lambda a\text{ for all projections }p\in\cM\}=\sup_p\|a^{1/2}pa^{-1/2}\|^2=\cosh^2\LNC(a)$.
For $\cM=M_n(\C)$ the upper bound is contained in \cite[Cor.~1.5]{Bou06}.
\end{theorem}

\begin{proof}
$pap\le\lambda a$ iff $\|a^{-1/2}papa^{-1/2}\|\le\lambda$, and $a^{-1/2}papa^{-1/2}=(a^{1/2}pa^{-1/2})^*(a^{1/2}pa^{-1/2})$.
For $z\in\cZ_{++}$, $p(za)p=z^{1/2}papz^{1/2}$, so $\lambda(za)=\lambda(a)$, and Lemma~\ref{lem:pTp} with
$T=z_0a$ gives $\lambda(a)\le\cosh^2\LNC(a)$. For the lower bound, if $a\notin\cZ$, use
Lemma~\ref{lem:corner}\ref{c:b}: the compression of $a$ to $K=\operatorname{span}\{\eta_1,\eta_2\}$ is
$D=\diag(a_1,a_2)$ with $a_i\in I_i$, and $K$ is invariant under $p=\pi(|w\rangle\langle w|)$,
$w=(1,1)/\sqrt2$. For $k=c_1\eta_1+c_2\eta_2$, $\langle papk,k\rangle=|\langle c,w\rangle|^2\langle Dw,w\rangle$
and $\langle ak,k\rangle=\langle Dc,c\rangle$; the supremum over $c$ of the ratio is
$\langle Dw,w\rangle\langle D^{-1}w,w\rangle=(a_1+a_2)^2/(4a_1a_2)$ (the Kantorovich expression,
which here attains its maximum $G(D)^2$), which tends to $\cosh^2\LNC(a)$ as
$\varepsilon\to0$.
\end{proof}

\begin{theorem}[Item 39]\label{thm:nt2}
\[
C(a):=\sup_{0\ne\varphi\in\cM_*}\frac{|\varphi(a)|}{|\varphi|(a)}=\cosh\LNC(a).
\]
More strongly, $|\varphi(ay)|\le\cosh\LNC(a)\,|\varphi|(a)$ for every $\varphi\in\cM_*$ and every
contraction $y\in\cM$.
\end{theorem}

\begin{proof}
\emph{Upper bound.} Let $\omega=|\varphi|$ and $(\pi_\omega,\xi)$ its GNS pair. Then
\[
|\varphi(ay)|=|\omega(ayu)|=|\langle\pi_\omega(yu)\xi,\pi_\omega(a)\xi\rangle|\le\|\xi\|\,\|\pi_\omega(a)\xi\|\le\cosh\OLC(a)\,\omega(a)
\]
by Lemma~\ref{lem:krein}, since $\sigma(\pi_\omega(a))\subset\sigma(a)$. With $z_0$ as in
Lemma~\ref{lem:resc} and by Lemma~\ref{lem:funct},
\[
|\varphi(ay)|=|\varphi_{z_0^{-1}}(z_0ay)|\le\cosh\OLC(z_0a)\,|\varphi_{z_0^{-1}}|(z_0a)=\cosh\LNC(a)\,|\varphi|(a).
\]
\emph{Lower bound} ($a\notin\cZ$). In Lemma~\ref{lem:corner}\ref{c:b} let $y\in\C^2$ be a unit vector
with $|y_1|^2=a_2/(a_1+a_2)$, $|y_2|^2=a_1/(a_1+a_2)$, $x=Dy/\|Dy\|$, and $w\in M_2(\C)$ unitary
with $wy=x$. Put $\zeta=y_1\eta_1+y_2\eta_2$, $\omega=\omega_\zeta$ and $\varphi=\omega(\cdot\,\pi(w)s(\omega))$.
Since $\zeta\in e\cH$ with $e=f_1+f_2\in\cM$, $s(\omega)\le e$, so $u=\pi(w)s(\omega)$ satisfies
$u^*u=s(\omega)$ and $|\varphi|=\omega$. Since $s(\omega)\zeta=\zeta$ and $\pi(w)\zeta=x_1\eta_1+x_2\eta_2$,
\[
\frac{|\varphi(a)|}{|\varphi|(a)}=\frac{|\langle Dx,y\rangle|}{\langle Dy,y\rangle}=\frac{\|Dy\|}{\langle Dy,y\rangle}=\frac{a_1+a_2}{2\sqrt{a_1a_2}},
\]
which tends to $\cosh\LNC(a)$.
\end{proof}

In the finite case, writing $\varphi=\tau(\cdot\,h)$, Theorem~\ref{thm:nt2} is the $L^1$-version of
Theorem~\ref{thm:gardner} with the roles reversed: the density $a$ is fixed and the operator $h$
varies, and the extremal configuration is the same.

\begin{theorem}[Item 41]\label{thm:ntrest}
For the remaining items of \cite[Thm.~1 and Cor.~1]{novikov2015}:
\begin{enumerate}
\item\label{r:a} $\sup\varphi^+(a)/\varphi_1(a)$ over Hermitian normal $\varphi=\varphi_1-\varphi_2$ with
$\varphi_i\ge0$; $\sup\varphi^+(a)/\omega^+(a)$ over Hermitian normal $\varphi\le\omega$; and
$\sup(\varphi+\omega)^+(a)/(\varphi^+(a)+\omega^+(a))$ over Hermitian normal $\varphi,\omega$ are all equal to $\cosh^2\LNC(a)$;
\item\label{r:b} $\sup|\varphi+\omega|(a)/(|\varphi|(a)+|\omega|(a))=\cosh\LNC(a)$, both over Hermitian
and over all normal $\varphi,\omega$.
\end{enumerate}
\end{theorem}

\begin{proof}
\ref{r:a} Upper bounds: let $\lambda=\cosh^2\LNC(a)$ and $e=s(\varphi^+)$. Then
$\varphi^+(x)=\varphi(exe)$, and $eae\le\lambda a$ by Theorem~\ref{thm:pap}. Hence
$\varphi^+(a)=\varphi(eae)\le\varphi_1(eae)\le\lambda\varphi_1(a)$; for $\varphi\le\omega$,
$\varphi^+(a)\le\omega(eae)\le\omega^+(eae)\le\lambda\omega^+(a)$; and with $e'=s((\varphi+\omega)^+)$,
$(\varphi+\omega)^+(a)\le(\varphi^++\omega^+)(e'ae')\le\lambda(\varphi^+(a)+\omega^+(a))$. Lower bound: the
first supremum is a special case of the other two ($\omega=\varphi_1$, resp.\ $\varphi_1$ and
$-\varphi_2$), so it suffices to treat it. With the notation of Lemma~\ref{lem:corner}\ref{c:b}, for
a Hermitian $2\times2$ matrix $h$ let $\chi_h(x)=\sum_{i,j}h_{ij}\langle x\eta_i,\eta_j\rangle=\Tr(hC_x)$,
where $C_x$ is the matrix of the compression of $x$ to $K$; thus $\chi_h(a)=\Tr(hD)$ and
$\chi_{|g\rangle\langle g|}=\omega_{g_1\eta_1+g_2\eta_2}$. If
$h=\mu_+|g_+\rangle\langle g_+|-\mu_-|g_-\rangle\langle g_-|$ is the spectral decomposition, then
$\chi_h=\mu_+\omega_{g_+}-\mu_-\omega_{g_-}$ (vectors $g_\pm$ read in $K$), and the element
$x=\pi(|g_+\rangle\langle g_+|-|g_-\rangle\langle g_-|)$ of norm $1$ gives $\chi_h(x)=\mu_++\mu_-$, so
this is the Jordan decomposition. Take $w=(1,1)/\sqrt2$, $r\perp w$, $q=D^{-1}w/\|D^{-1}w\|$ and
$h_t=|q\rangle\langle q|-t|r\rangle\langle r|$, $\varphi_1=\chi_{|q\rangle\langle q|}$,
$\varphi_2=t\chi_{|r\rangle\langle r|}$. Then $\varphi^+(a)/\varphi_1(a)=\Tr(D(h_t)_+)/\langle Dq,q\rangle$.
As $t\to\infty$, $(h_t)_+\to|\langle q,w\rangle|^2|w\rangle\langle w|$ (in the basis $w,r$ the
positive eigenvalue of $h_t$ tends to $|\langle q,w\rangle|^2$ and its eigenvector to $w$), and the
ratio tends to the Kantorovich expression
$\langle D^{-1}w,w\rangle\langle Dw,w\rangle=(a_1+a_2)^2/(4a_1a_2)$.

\ref{r:b} Upper bound: if $\chi=\varphi+\omega=|\chi|(\cdot\,u)$, then
$|\chi|(a)=\varphi(au^*)+\omega(au^*)\le\cosh\LNC(a)(|\varphi|(a)+|\omega|(a))$ by
Theorem~\ref{thm:nt2}. Lower bound: let $P_\pm$ be the projections of the proof of
Theorem~\ref{thm:subadd} built for $D$ in place of $\diag(\kappa,1)$ (i.e.\ $s^2=a_2/(a_1+a_2)$,
$c^2=a_1/(a_1+a_2)$), $\varphi=\chi_{P_+}$, $\omega=-\chi_{P_-}$. Then $|\varphi|(a)+|\omega|(a)=\Tr(D(P_++P_-))=4a_1a_2/(a_1+a_2)$,
and $\varphi+\omega=\chi_{P_+-P_-}$ with $P_+-P_-=2sc\,\sigma$, whose Jordan decomposition (as above) gives
$|\varphi+\omega|=2sc(\omega_{\eta_1}+\omega_{\eta_2})$. The ratio is
$2sc(a_1+a_2)^2/(4a_1a_2)=(a_1+a_2)/(2\sqrt{a_1a_2})$.
\end{proof}

The convexity items (vii) and (viii) of \cite[Remark~1]{novikov2015} are not treated here.

\begin{theorem}[Item 42]\label{thm:l1}
\begin{enumerate}
\item\label{l:a} If $\cM$ is finite and $A\in\cM_{++}$, then
\[
\sup_{X\ne0}\frac{\tau(|AXA|)}{\tau(A|X|A)}=\kappa_\cZ(A),\qquad
\inf_{X\ne0}\frac{\tau(|AXA|)}{\tau(A|X|A)}=\kappa_\cZ(A)^{-1}.
\]
\item\label{l:b} For every normal $\varphi\ne0$, with $a^{1/2}\varphi a^{1/2}=\varphi(a^{1/2}\cdot a^{1/2})$,
\[
\kappa_\cZ(a)^{-1/2}\le\frac{\|a^{1/2}\varphi a^{1/2}\|}{|\varphi|(a)}\le\kappa_\cZ(a)^{1/2},
\]
and if $\cM$ is finite both bounds are sharp.
\end{enumerate}
\end{theorem}

\begin{proof}
\ref{l:a} Let $X=U|X|$. By H\"older's inequality,
\[
\tau(|AXA|)\le\|AU|X|^{1/2}\|_2\,\||X|^{1/2}A\|_2=\tau(A^2|X^*|)^{1/2}\tau(A^2|X|)^{1/2},
\]
and Theorem~\ref{thm:modulus} for the density $A^2$ gives
$\tau(A^2|X^*|)\le\kappa_\cZ(A^2)\tau(A^2|X|)$, where $\kappa_\cZ(A^2)=\kappa_\cZ(A)^2$. For the infimum,
$|X|=U^*X$ gives
\[
\tau(A|X|A)=\tau((AU^*A^{-1})(AXA))\le\|AU^*A^{-1}\|\,\tau(|AXA|)\le\kappa(A)\,\tau(|AXA|),
\]
and the ratio at $(zA,X)$ equals that at $(A,z^2X)$ for $z\in\cZ_{++}$, so $\kappa(A)$ may be
replaced by $\kappa_\cZ(A)$. Sharpness ($A\notin\cZ$): with Lemma~\ref{lem:corner}\ref{c:a},\ref{c:c} for $A$,
$X=v$ gives
$\tau(|AvA|)\ge\tau(v^*AvA)\ge(\alpha-2\varepsilon)\tau(f_2Af_2)\ge(\alpha-2\varepsilon)(\beta-\varepsilon)t$
and $\tau(A|v|A)=\tau(f_2A^2f_2)\le(\beta+2\varepsilon)^2t$; $X=v^*$ gives, since
$Av^*A=(AvA)^*$, $\tau(|Y^*|)=\tau(|Y|)$ and $AvA=(f_1A)^*(vf_2A)$,
$\tau(|Av^*A|)=\tau(|AvA|)\le\|Af_1\|_2\|vf_2A\|_2\le(\alpha+\varepsilon)(\beta+2\varepsilon)t$ and
$\tau(A|v^*|A)=\tau(f_1A^2f_1)\ge(\alpha-2\varepsilon)^2t$.

\ref{l:b} Let $\chi=a^{1/2}\varphi a^{1/2}$ and $\omega=|\varphi|$. By the Cauchy--Schwarz inequality,
\[
|\chi(x)|=|\omega(a^{1/2}xa^{1/2}u)|\le\omega(a^{1/2}xx^*a^{1/2})^{1/2}\omega(u^*au)^{1/2}\le\|x\|\,\omega(a)^{1/2}(\|a\|\omega(1))^{1/2},
\]
and $\omega(1)\le\|a^{-1}\|\omega(a)$, so $\|\chi\|\le\kappa(a)^{1/2}|\varphi|(a)$. Conversely
$|\varphi|(a)=\varphi(au^*)=\chi(a^{1/2}u^*a^{-1/2})\le\kappa(a)^{1/2}\|\chi\|$. By Lemma~\ref{lem:funct}
the ratio at $(za,\varphi)$ equals the ratio at $(a,\varphi_z)$, so $\kappa(a)$ may be replaced by
$\kappa_\cZ(a)$. If $\cM$ is finite and $\varphi=\tau(\cdot\,h)$ with $h=w|h|\in\cM$, then
$|\varphi|=\tau(|h|\,\cdot)$, $|\varphi|(a)=\tau(a^{1/2}|h|a^{1/2})$ and
$\|\chi\|=\tau(|a^{1/2}ha^{1/2}|)$, so sharpness follows from \ref{l:a} with $A=a^{1/2}$ and
$\kappa_\cZ(a^{1/2})=\kappa_\cZ(a)^{1/2}$.
\end{proof}

\begin{remark}\label{rem:centregen}
Theorems~\ref{thm:pap}, \ref{thm:nt2} and~\ref{thm:ntrest} and the bounds in
Theorem~\ref{thm:l1}\ref{l:b} hold in arbitrary von Neumann algebras, because the corner construction
in its vector form (Lemma~\ref{lem:corner}\ref{c:b}) needs no trace. Theorem~\ref{thm:l1}\ref{l:a}
and the sharpness in \ref{l:b} are proved for finite algebras only. All constants in this section
are finite because $a$ is invertible. We do not treat non-invertible $a$; the natural condition
for finiteness would then be $\flr a\ge c\ceil a$ for some $c>0$, which in a factor holds only for
$a=0$, but we do not prove that it is necessary.
\end{remark}

\section{Contrast versus noncentrality}\label{sec:compare}

Call a violation constant treated in finite von Neumann algebras in Section~\ref{sec:catalogue} or in
Section~\ref{sec:centre} \emph{exact} if it was shown
to be $f(\kappa_\cZ)$ for a function $f$ given in closed form; these are the constants of
Theorems~\ref{thm:gardner}, \ref{thm:subadd}, \ref{thm:modulus}, \ref{thm:pz2}, \ref{thm:ba}
($\gamma$), \ref{thm:young}, \ref{thm:mono2}, \ref{thm:proj}, \ref{thm:pap}, \ref{thm:nt2},
\ref{thm:ntrest}, \ref{thm:l1} (for finite $\cM$), of Corollaries~\ref{cor:item3}, \ref{cor:item5}, \ref{cor:item8}
and of Propositions~\ref{prop:youngsa}, \ref{prop:item21} and~\ref{prop:item22}; for an
infimum (as in
Theorems~\ref{thm:modulus} and~\ref{thm:l1}) we use its reciprocal, which is a supremum as in
Definition~\ref{def:viol}. In all these cases $f$ is continuous and strictly increasing on
$[1,\infty)$ with $f(1)=1$ ($f(1)=0$ for $\gamma$ and for item~21). The constant of item~37
(Proposition~\ref{prop:pb}\ref{pb:a}) is also exact, equal to $1$ on $\cZ$ and to $+\infty$ elsewhere; parts
\ref{cp:a} and~\ref{cp:b} below concern the exact constants with finite $f$, and parts \ref{cp:c}
and~\ref{cp:d} hold for item~37 as well.

\begin{theorem}\label{thm:compare}
Let $\rho\in\cM_{++}$ (with $\cM$ finite for the constants of Section~\ref{sec:catalogue} and of
Theorem~\ref{thm:l1}).
\begin{enumerate}
\item\label{cp:a} For every exact constant with finite $f$, $c(\rho)=f(\kappa_\cZ(\rho))\le f(\kappa(\rho))$: the
bound given by the contrast is valid.
\item\label{cp:b} $c(\rho)<f(\kappa(\rho))$ if and only if $\max_\Omega\ceil\rho$ and $\min_\Omega\flr\rho$
are not attained at a common point of $\Omega$. In a factor $c(\rho)=f(\kappa(\rho))$.
\item\label{cp:c} The loss is unbounded: in $\cM=M_2(\C)\oplus M_2(\C)$ with any faithful tracial state,
$\rho_\lambda=\diag(1,4)\oplus\lambda\diag(1,4)$ ($\lambda\ge1$) has Gardner constant
$\cosh\LNC(\rho_\lambda)=5/4$, while $\cosh\OLC(\rho_\lambda)=(1+4\lambda)/(2\sqrt{4\lambda})\to\infty$; and
every central $\rho$ has all constants equal to their tracial value, whatever its contrast. (Items 11
and 24, which characterise the canonical trace, are not violation constants in this sense and are
violated by every non-scalar central density, Remark~\ref{rem:sectors}.)
\item\label{cp:d} Let $\cM$ be neither a factor nor abelian. For every $t>1$ there are $\rho_1,\rho_2\in\cM_{++}$ with
$\kappa(\rho_1)=\kappa(\rho_2)=t$, $\kappa_\cZ(\rho_1)=1$ and $\kappa_\cZ(\rho_2)=t$. Consequently, for every
inequality treated in finite von Neumann algebras in Sections~\ref{sec:catalogue} and~\ref{sec:centre}
the violation constant is not a function
of the contrast on $\cM$, whereas every exact constant is a function of the noncentrality $\NC$
(equivalently, of $\kappa_\cZ$).
\end{enumerate}
\end{theorem}

\begin{proof}
\ref{cp:a} $\kappa_\cZ\le\kappa$ and $f$ is increasing.
\ref{cp:b} Let $u=\ceil\rho$, $l=\flr\rho$. By Corollary~\ref{cor:norm}, $\kappa=\max u/\min l$, and
$\kappa_\cZ=\max(u/l)$. Since $u(\omega)\le\max u$ and $l(\omega)\ge\min l$, $u(\omega)/l(\omega)=\kappa$
iff $u(\omega)=\max u$ and $l(\omega)=\min l$. As $f$ is strictly increasing, $c<f(\kappa)$ iff
$\kappa_\cZ<\kappa$. In a factor $\Omega$ is a point.
\ref{cp:c} $\kappa_\cZ(\rho_\lambda)=4$ by Proposition~\ref{prop:char}\ref{ch:prop} and $\kappa(\rho_\lambda)=4\lambda$;
$\cosh(\ln2)=5/4$. For central $\rho$, $\kappa_\cZ=1$.
\ref{cp:d} Since $\cM$ is not a factor, there is a central projection $P\ne0,1$; put
$\rho_1=P+t(1-P)\in\cZ$. Since $\cM$ is not abelian, there is a non-central projection $e$; then
$c(e)c(1-e)\ne0$ (otherwise $1-e\le1-c(e)$, i.e.\ $c(e)\le e$ and $e=c(e)$). Put $\rho_2=1+(t-1)e$. Its
spectrum is $\{1,t\}$, and by Lemma~\ref{lem:env}, $\ceil{\rho_2}=1+(t-1)c(e)$ and
$\flr{\rho_2}=1+(t-1)(1-c(1-e))$, so $\ceil{\rho_2}/\flr{\rho_2}=t$ on $c(e)c(1-e)$ and
$\kappa_\cZ(\rho_2)=t$. Every such violation constant equals
its tracial value ($1$, and $0$ for $\gamma$ and for item~21) iff $\kappa_\cZ=1$, and exceeds it otherwise: for the
exact ones this is the formula. For the others, namely Propositions~\ref{prop:pzk},
\ref{prop:youngpos}, \ref{prop:monop}, \ref{prop:af}, \ref{prop:bikexp}, \ref{prop:pb},
\ref{prop:alt} (with $r>1$), \ref{prop:sgm}, \ref{prop:altlow} and Theorem~\ref{thm:tik}, the upper bound equals $1$ at $\kappa_\cZ=1$
and the constant exceeds $1$ for $\kappa_\cZ>1$. For the Peierls--Bogoliubov inequality in the
form with $\varphi(Be^A)$ the constant is $1$ for central $\rho$ and $+\infty$ otherwise. Hence
$c(\rho_1)$ is the tracial value and $c(\rho_2)$ exceeds it, although $\OMC(\rho_1)=\OMC(\rho_2)$.
\end{proof}

Conversely, in a factor the contrast determines every exact constant, and in an abelian algebra
every functional is tracial and all constants equal their tracial values. Thus, for the exact
constants, the contrast gives the correct value for every $\rho$ precisely in factors (in an abelian
algebra the constants equal their tracial value $f(1)$, which is $<f(\kappa(\rho))$ unless $\rho$
is scalar), while the noncentrality gives it in every von Neumann algebra (every finite one for the
constants of Section~\ref{sec:catalogue} and of Theorem~\ref{thm:l1}). For the Petz--Zem\'anek constants $c_k$, $k\ge3$, even the noncentrality does
not suffice: by Propositions~\ref{prop:pz3} and~\ref{prop:pzk4}, $c_k$ is not a function of
$\kappa_\cZ=\kappa$ on $M_k(\C)$. Nor are the constants $a$, $c_{\exp}$ and $c^{\rm pr}_{\exp}$
functions of $\kappa_\cZ$ (Propositions~\ref{prop:af} and~\ref{prop:bikexp}). For the other
constants for which only bounds are proved, whether they are functions of $\kappa_\cZ$ is open
(Section~\ref{sec:open}).

\subsection{Finite direct sums}\label{ssec:sectors}
Let $\M=\bigoplus_{c=1}^CM_{k_c}$, with the trace $\Tr=\sum_c\Tr_c$ and $\varphi_\rho=\Tr(\rho\,\cdot\,)$,
$\rho=\bigoplus_c\rho_c\ge0$. The traces on $\M$ are the functionals $\sum_ct_c\Tr_c$, $t_c\ge0$; so $\varphi_\rho$ is a
trace iff $\rho$ is central, whatever its global contrast $\OMC(\rho)$. By Remark~\ref{rem:blocks},
$\NC(\rho)=\max_c\OMC(\rho_c)$ is the contrast inside the blocks, and for every exact constant
Theorem~\ref{thm:compare} gives $c_\M(\rho)=f(\max_c\kappa(\rho_c))=\max_cc_{M_{k_c}}(\rho_c)$. The following
corollary extends this, as an inequality, to the items of Table~\ref{tab:items} for which only matrix
envelopes are known; its proof is elementary.

\begin{corollary}[Direct sums]\label{cor:sectors}
Let $\rho=\bigoplus_c\rho_c>0$ in $\M$, and let $V_\M(\rho)$ and $V_{M_{k_c}}(\rho_c)$ be the suprema of the
normalised violation of one of the items 1--25 of Table~\ref{tab:items}, other than 11, 16 and 24, for
$\varphi_\rho$ on $\M$ and for $\varphi_{\rho_c}$ on $M_{k_c}$. Then
\[
  V_\M(\rho)\le\max_cV_{M_{k_c}}(\rho_c),
\]
with equality for items 1--10, 12--14, 17, 18, 21--23 and~25. In particular, if every $\rho_c$ is scalar, i.e.\
$\NC(\rho)=0$, none of the items 1--25 except 11 and 24 is violated, however different the scalars are; and
the envelope on $\M$ along the noncentrality $\NC(\rho)=\max_c\OMC(\rho_c)=\delta$ is at most the largest
envelope $V_{k_c}(\delta)$ of the blocks.
\end{corollary}
\begin{proof}
All variables of the items are in $\M$, hence block diagonal, and every functional calculus acts blockwise. Let $\cL_c$,
$\cR_c$, $\cN_c$ be the expressions of the item for the block variables and $\varphi_{\rho_c}$. For all items except 11, 16 and~24,
$\cL\le\sum_cL_c$ (with equality unless $\cL$ contains an absolute value of a functional, as in 1, 4 and~14),
$\cN=\sum_cN_c$, and $\cR=\sum_cR_c$, except in items 6 and~14, where $\cR=\sqrt{\sum_c\varphi_c(A_c)\sum_c\varphi_c(B_c)}\ge
\sum_cR_c$ by the Cauchy--Schwarz inequality and $\cR\ge\sum_cR_c$ by Hölder's inequality. Since $\rho>0$, $\cN_c=0$ forces the
variables of block~$c$ to vanish where they enter $\cL_c$, so $\cL_c\le \cR_c$ there. Hence
$(\cL-\cR)/\cN\le\sum_c(\cL_c-\cR_c)/\sum_cN_c\le\max_c(\cL_c-\cR_c)/\cN_c$, the maximum over the blocks with $\cN_c>0$ (for $\cN=\cR$ in items 6
and~14 use $\cL/\cR\le\sum_cL_c/\sum_cR_c$ instead). This proves the inequality. For the equality, in the
listed items $\cN$ is additive and the variables may vanish on all blocks but one without making $\cN$ vanish; restricting a
near-optimal configuration of one block and setting the variables to zero on the other blocks reproduces its quotient
(in 19, 20 and 15, $\cN$ contains $\varphi(1)$ and the other blocks dilute the violation). For item~6, whose
variables are invertible, set them equal to $\varepsilon1$ on the other blocks instead; their contributions to
$\cL$, $\cR$ and $\cN$ are $O(\varepsilon)$, and the quotient follows as $\varepsilon\to0$. If every $\rho_c$ is scalar, every
$V_{M_{k_c}}(\rho_c)$ vanishes, since $\varphi_{\rho_c}$ is then a trace on $M_{k_c}$; for item~16 the same conclusion holds
because the Peierls--Bogoliubov inequality holds for every trace on $\M$.
\end{proof}

\begin{remark}\label{rem:sectors}
The corollary is a statement about the characterisations, not about the inequalities: a non-central density is detected by
every characterisation of traces, but only through the variation \emph{inside} the blocks. Items 11 and 24 characterise
the \emph{canonical} trace $\Tr$ among the functionals with $\varphi(1)=\dim$, and central $\rho$ violate them; for a fixed
density their violations are $1-\mathrm{GM}/\mathrm{AM}$ and $\lambda_{\min}(\rho)^{-1/2}-1$, computed from the eigenvalues of $\rho$ normalised to
$\Tr \rho=\sum_ck_c$ (Propositions~\ref{prop:det} and~\ref{prop:power}; the proofs do not use the block structure).
\end{remark}

\begin{remark}[Dilution]\label{rem:dilute}
In items 15, 19 and 20 the normalisation contains $\varphi(1)$, and the other blocks dilute the violation:
a block on which the variables vanish contributes to $\cN$ and nothing positive to $\cL-\cR$, and the inequality of Corollary~\ref{cor:sectors} is in
general strict. For item~20, Proposition~\ref{prop:bikexp}\ref{be:c} shows that the violation tends to $0$
as the weight of a central summand grows while $\kappa_\cZ$ is fixed (Remark~\ref{rem:item20}); so its
constant is not a function of $\kappa_\cZ$. On $M_n$ the envelopes of items 19 and 20 depend on the parity of
$n$ (Theorem~\ref{thm:projP}). The numerical search on $M_2\oplus M_3$ confirms the corollary
(Section~\ref{sec:numerics}).
\end{remark}

\section{Contrasts of composite operators}\label{sec:composite}

The operators of Table~\ref{tab:items} enter the inequalities through composites: sums $A+B$ (items 2, 10, 12, 13),
congruences $B^{1/2}AB^{1/2}$ (items 6, 7), means (item~6), Jordan products $\Re PQ=\frac12(PQ+QP)$ (items 9, 21, 22),
and, behind every characterisation, the commutator $AB-BA$, on which a functional vanishes if and only if it is a trace.
This section describes how the contrast of a composite depends on the contrasts of its factors, and then uses the
answers to explain the thresholds observed in the numerical search. Two reparametrisations of the contrast make the
answers additive. For $A>0$ with extreme eigenvalues $M_A\ge m_A>0$ and condition number $\kappa_A=M_A/m_A$ let
\[
  \eta(A)=\operatorname{artanh}\OMC(A)=\tfrac12\log \kappa_A,\qquad
  \psi(A)=\arcsin\OMC(A)\in[0,\tfrac\pi2],\qquad \bar\lambda(A)=\frac{M_A+m_A}2 .
\]
The rapidity $\eta(A)=\OLC(A)$ is half the logarithm of the condition number, that is, half the Hilbert projective distance of
$A$ to $1$ and the Thompson distance of $A$ to the scalars (Section~\ref{sec:related}). The angle $\psi(A)$ is that of
Krein and Gustafson~\cite{krein1969,gustafson1968,gustafsonrao1997}, the largest angle between $x$ and $Ax$:
$\cos\psi(A)=\min_{x\ne0}\langle Ax,x\rangle/(\|Ax\|\|x\|)=2\sqrt{M_Am_A}/(M_A+m_A)$, and
$\sin\psi(A)=\min_{t>0}\|tA-1\|=\OMC(A)$, attained at $t=1/\bar\lambda(A)$ (Figure~\ref{fig:angle}(a)). For the density, $G=\cosh\eta(\rho)=\sec\psi(\rho)$, so the envelope $G-1$ of items 1--5, 18 and~22
(Theorem~\ref{thm:gardner}) is $\sec\psi(\rho)-1$. The angle between vectors is $\angle(x,y)=\arccos(\Re\langle x,y\rangle/
(\|x\|\|y\|))$, the angle of the real inner product $\Re\langle\cdot,\cdot\rangle$ on $\C^n=\R^{2n}$; it satisfies the
triangle inequality.

\begin{proposition}[Multiplicative composites: rapidities add]\label{prop:rapid}
Let $A,B>0$ in $M_n$ and $\eta_A=\eta(A)$, $\eta_B=\eta(B)$.
\begin{enumerate}
\item[(a)] $\OMC(A^p)=\tanh(|p|\eta_A)$ for real $p\ne0$, $\OMC(A\otimes B)=\tanh(\eta_A+\eta_B)$, and the compression of
$A$ to a subspace has contrast at most $\OMC(A)$.
\item[(b)] $\tanh|\eta_A-\eta_B|\le\OMC(A^{1/2}BA^{1/2})\le\tanh(\eta_A+\eta_B)$; the same bounds hold for the contrast of
$|AB|$, and for the spectrum of $AB$.
\item[(c)] For the geometric mean $A\#B=A^{1/2}(A^{-1/2}BA^{-1/2})^{1/2}A^{1/2}$,
$\tanh\frac{|\eta_A-\eta_B|}2\le\OMC(A\#B)\le\tanh\frac{\eta_A+\eta_B}2$.
\item[(d)] $\OMC(sA+tB)\le\max(\OMC(A),\OMC(B))$ for $s,t>0$, and the same holds for the harmonic mean. If
$\OMC(A),\OMC(B)>0$, then $\OMC(A+tB)=0$ for suitable $t>0$ and a suitable relative position.
\end{enumerate}
All bounds are attained by commuting operators, with the extreme eigenvectors aligned (upper bounds) or interchanged
(lower bounds), in every dimension $n\ge2$.
\end{proposition}

\noindent These facts are classical, and we include the short proof for the attainment statements. The upper bound in (b)
is the submultiplicativity of the condition number, $\kappa(A^{1/2}BA^{1/2})\le \kappa_A\kappa_B$. In the language of Hilbert's projective
metric $d_H$, which is invariant under congruences and under inversion, (b) reads
$|d_H(B,1)-d_H(1,A^{-1})|\le d_H(B,A^{-1})=d_H(A^{1/2}BA^{1/2},1)\le d_H(B,1)+d_H(1,A^{-1})$, i.e.\ it is the triangle
inequality~\cite{bushell1973,lemmens2012}; this is why rapidities, the halved distances to the scalars, add.
\begin{proof}
(a) Spectral mapping, the eigenvalues of $A\otimes B$ are the products of eigenvalues, and Cauchy interlacing.
(b) From $m_BA\le A^{1/2}BA^{1/2}\le M_BA$ and the monotonicity of the eigenvalues~\cite{bhatia1997},
$\lambda_k(A)m_B\le\lambda_k(A^{1/2}BA^{1/2})\le\lambda_k(A)M_B$ for every $k$. Hence the largest eigenvalue lies in
$[M_Am_B,M_AM_B]$ and the smallest in $[m_Am_B,m_AM_B]$, so the condition number lies in $[\kappa_A/\kappa_B,\kappa_A\kappa_B]$; since
$A^{1/2}BA^{1/2}$ is similar to $B^{1/2}AB^{1/2}$, it is also at least $\kappa_B/\kappa_A$. With $\OMC=\tanh(\frac12\log \kappa)$ this
is (b). For $|AB|$ use the singular values, $\sigma_k(A)\sigma_{\min}(B)\le\sigma_k(AB)\le\sigma_k(A)\sigma_{\max}(B)$;
$AB$ is similar to $A^{1/2}BA^{1/2}$. (c) The mean is monotone and homogeneous, so $\sqrt{m_Am_B}\le A\#B\le
\sqrt{M_AM_B}$, which gives the upper bound. From $B=(A\#B)A^{-1}(A\#B)$, whose spectrum is that of
$A^{-1/2}(A\#B)^2A^{-1/2}$, (a) and (b) give $\eta_B\le\eta_A+2\eta(A\#B)$, and symmetrically. (d) The mediant inequality
$\frac{sM_A+tM_B}{sm_A+tm_B}\le\max(\kappa_A,\kappa_B)$ bounds the condition number of $sA+tB$; for the harmonic mean apply it to
$A^{-1}+B^{-1}$. For the last claim take $B=c-A$ with $c>M_A$ in the eigenbasis of $A$ and with the intermediate
eigenvalues reflected; $\OMC(c-A)=(M_A-m_A)/(2c-M_A-m_A)$ takes every value in $(0,1)$. Attainment: for commuting
$A,B$ with aligned extreme eigenvectors all extreme eigenvalues multiply (add, for sums), and with interchanged ones the
quotients do.
\end{proof}

\paragraph{The Jordan product.} For $AB+BA$ the picture changes: it need not be positive, and angles, not rapidities,
add. The reason is geometric. A positive operator $A$ turns every vector by at most its angle: $\angle(x,Ax)\le\psi_A$,
where $\sin\psi_A=\OMC(A)$. Hence $ABx$ and $BAx$ differ from $x$ by at most $\psi_A+\psi_B$, and so does their sum
$(AB+BA)x$, as long as $\psi_A+\psi_B<\frac\pi2$. Then $(AB+BA)x$ stays in the half-space of $x$, i.e.\
$\langle(AB+BA)x,x\rangle>0$. If $\psi_A+\psi_B>\frac\pi2$, a suitable rotation of $B$ against $A$ turns some $x$ by more
than a right angle, and $AB+BA$ has a negative eigenvalue. In terms of contrasts: \emph{$AB+BA$ is positive for every
relative position if and only if the point $(\OMC(A),\OMC(B))$ lies in the unit disc}
(Figure~\ref{fig:ops-positivity}). The sufficiency of this criterion is Gustafson's~\cite{gustafson1968}, whose angle
argument we reproduce. Its sharpness is in the spirit of the work on eigenvalues of Jordan products by
Strang~\cite{strang1962}, Alikakos and Bates~\cite{alikakos1984,alikakos1985} and Conley, Pucci and
Serrin~\cite{conley2005}; we include the short proof of both directions, since we need the extremal configuration, and add
that the same angles bound the contrast of $AB+BA$.

\begin{figure}[!htbp]
\centering\includegraphics[width=\textwidth]{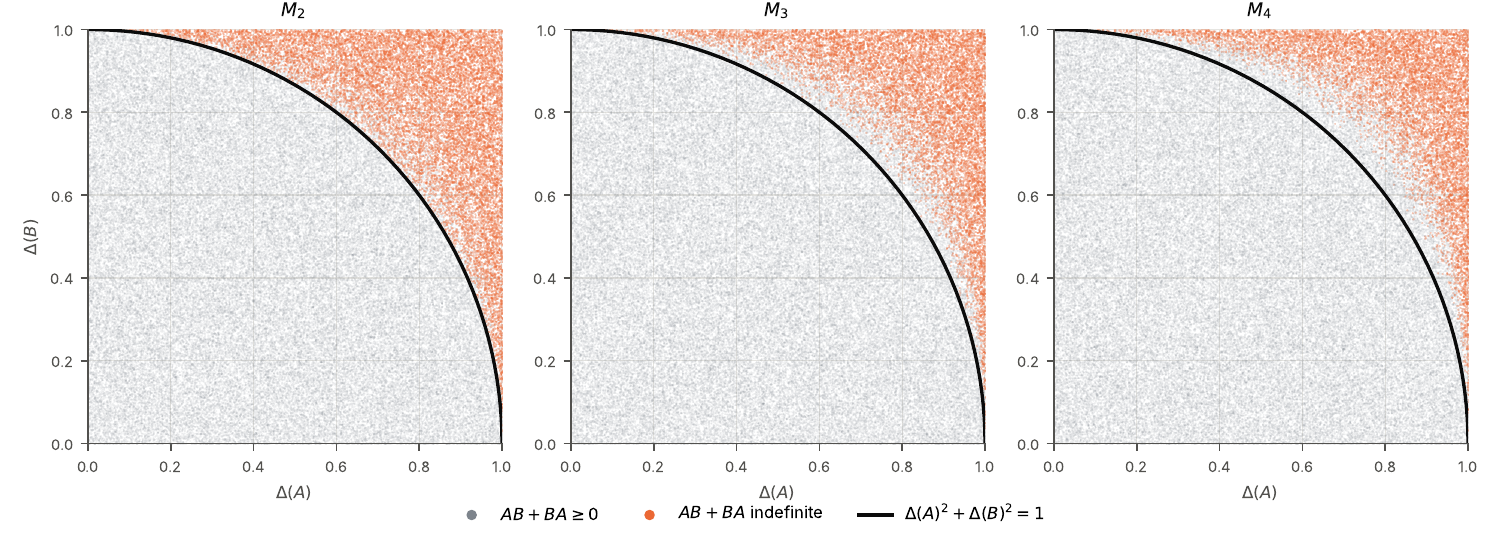}
\caption[]{Is $AB+BA$ positive? $\COSamples$ random pairs $A,B>0$ per dimension, with contrasts uniform in $[0,1]^2$ and
Haar random relative position: positive (grey) or indefinite (orange). No indefinite product occurs inside the quarter
circle $\OMC(A)^2+\OMC(B)^2=1$ (Theorem~\ref{thm:jordan}(a)); outside it occurs for some positions, more often the
further out. In the figure $\Delta(\cdot)=\OMC(\cdot)$.}
\label{fig:ops-positivity}
\end{figure}

\begin{theorem}[Jordan product]\label{thm:jordan}
Let $A,B>0$ in $M_n$, $n\ge2$, with contrasts $\delta_A,\delta_B$, angles $\psi_A=\arcsin\delta_A$, $\psi_B=\arcsin\delta_B$
and rapidities $\eta_A=\operatorname{artanh}\delta_A$, $\eta_B=\operatorname{artanh}\delta_B$.
\begin{enumerate}
\item[(a)] $AB+BA\ge0$ for every relative position of $A$ and $B$ if and only if $\delta_A^2+\delta_B^2\le1$. If
$\delta_A^2+\delta_B^2>1$, $AB+BA$ is indefinite for some relative position, in every dimension.
\item[(b)] If $\delta_A^2+\delta_B^2<1$, then $\tanh|\eta_A-\eta_B|\le\OMC(AB+BA)\le\sin(\psi_A+\psi_B)$, and the lower
bound is attained by commuting $A,B$ whose extreme eigenvectors are interchanged.
\end{enumerate}
\end{theorem}
\begin{proof}
(a) and the upper bound in (b). We use the real angle $\angle(x,y)$ defined above. By the formula for $\cos\psi$,
$\angle(x,Bx)\le\psi_B$ and $\angle(Bx,ABx)\le\psi_A$, so $\angle(x,ABx)\le\psi_A+\psi_B$ by the triangle inequality;
likewise $\angle(x,BAx)\le\psi_A+\psi_B$. If $\psi_A+\psi_B<\frac\pi2$, the cone $\{y:\angle(x,y)\le\psi_A+\psi_B\}$ is
convex and contains $(AB+BA)x$, so $\langle(AB+BA)x,x\rangle>0$; moreover the largest angle of $AB+BA$ is at most
$\psi_A+\psi_B$, and its sine is $\OMC(AB+BA)$. The case $\psi_A+\psi_B=\frac\pi2$ follows by continuity.

Converse in (a). On $M_2$ take $A=\mathrm{diag}(1,\ell_A)$ and $B$ with eigenvalues $1,\ell_B$ and eigenvector
$(\cos\theta,\sin\theta)$ for $1$, where $\ell_A=\frac{1-\delta_A}{1+\delta_A}$, $\ell_B=\frac{1-\delta_B}{1+\delta_B}$. With
$c=\cos^2\theta$ and $p=(1-\ell_A)(1-\ell_B)$ a direct computation gives
\[
  \tfrac12\Tr(AB+BA)=\ell_A+\ell_B+pc,\qquad\det(AB+BA)=4\ell_A\ell_B-p^2c(1-c).
\]
At $\theta=\pi/4$ the determinant is negative if and only if $16\ell_A\ell_B<p^2$; with
$\sqrt{\ell_A}=\sqrt{1-\delta_A^2}/(1+\delta_A)$ and $1-\ell_A=2\delta_A/(1+\delta_A)$ this reads $(1-\delta_A^2)(1-\delta_B^2)<\delta_A^2\delta_B^2$, i.e.\ $\delta_A^2+\delta_B^2>1$.
On $M_n$ add the direct summands $\ell_A1_{n-2}$ and $\ell_B1_{n-2}$.

Lower bound in (b). If $x,y$ are unit eigenvectors of $A$ for its extreme eigenvalues $M_A$, $m_A$, then
$\lambda_{\max}(AB+BA)\ge2\Re\langle Bx,Ax\rangle=2M_A\langle Bx,x\rangle\ge2M_Am_B$ and likewise
$\lambda_{\min}(AB+BA)\le2m_AM_B$. So the condition number of $AB+BA$ is at least $\kappa_A/\kappa_B$, and by symmetry at least
$\kappa_B/\kappa_A$.
\end{proof}

\noindent The angle argument in the proof of (a) works on any Hilbert space, so $\OMC(A)^2+\OMC(B)^2<1$
implies $AB+BA>0$ for positive invertible operators on any Hilbert space. Since $AB+BA>0$ is invariant
under $(A,B)\mapsto(z_1A,z_2B)$ with $z_1,z_2\in\cZ_{++}$, rescaling then shows that in every von Neumann algebra
$\NC(A)^2+\NC(B)^2<1$ implies $AB+BA>0$ (Remark~\ref{rem:upgrade}).

How large can $\OMC(AB+BA)$ be? The upper bound $\sin(\psi_A+\psi_B)$ is reached only at the boundary of the disc. The
true maximum depends on the dimension, because in $M_2$ both extreme eigenvalues of $AB+BA$ come from the same pair of
$2\times2$ matrices, while in $M_3$ they can come from different blocks.

\begin{proposition}[Largest contrast of the Jordan product]\label{prop:jordanmax}
In the notation of the proof above, let $\delta_A^2+\delta_B^2<1$.
\begin{enumerate}
\item[(a)] On $M_2$, $\displaystyle\max\OMC(AB+BA)=\max_{0\le c\le1}\Bigl(1-\frac{4\ell_A\ell_B-p^2c(1-c)}{(\ell_A+\ell_B+pc)^2}\Bigr)^{1/2}$.
\item[(b)] On $M_n$, $n\ge3$, $\max\OMC(AB+BA)\ge(2-\mu)/(2+\mu)$, where
$\mu=\min_{0\le c\le1}\bigl(\ell_A+\ell_B+pc-\sqrt{(\ell_A+\ell_B+pc)^2-4\ell_A\ell_B+p^2c(1-c)}\,\bigr)$ is the smallest eigenvalue of $AB+BA$ over the
$2\times2$ pairs of (a).
\end{enumerate}
Numerically, both values equal the commuting value $\tanh(\eta_A+\eta_B)$ for small contrasts and exceed it for large ones.
\end{proposition}
\begin{proof}
(a) On $M_2$ every pair is unitarily equivalent to the pair of the proof of Theorem~\ref{thm:jordan} for some $\theta$,
and $\OMC^2=1-4\det/\Tr^2$. (b) Take $A=1\oplus A_2$ and $B=1\oplus B_2$ (padded with ones) with the $2\times2$ pair
$A_2,B_2$ of smallest eigenvalue $\mu$. Then $AB+BA=2\oplus(A_2B_2+B_2A_2)$ has the largest eigenvalue $2$, since
$\|A_2B_2+B_2A_2\|\le2$, and the smallest eigenvalue $\mu$.
\end{proof}

For $\delta_A=\delta_B=0.6$: the commuting value is $\COJordanTanh$, the maximum on $M_2$ is $\COJordanTwo$, the value of (b)
for $n\ge3$ is $\COJordanThree$, and $\sin(\psi_A+\psi_B)=\COJordanSin$. A numerical optimisation over all relative
positions and intermediate eigenvalues (24 starts per point) finds on $M_3$ and on $M_4$ exactly the value of (b)
($\COJordanThreeOpt$ and $\COJordanFour$ at $\delta_A=\delta_B=0.6$, and largest difference $\COJordanDev$ on a grid
of $\delta_A$ with $\delta_B=0.6$; Figure~\ref{fig:ops-jordan}). Equality for
every $n\ge3$ would follow if the smallest eigenvalue of a Jordan product with spectra in $[\ell_A,1]$ and $[\ell_B,1]$
were always attained on a two-dimensional subspace. Strang~\cite{strang1962} bounded this smallest eigenvalue from below,
and Alikakos and Bates showed that his bound is best possible~\cite{alikakos1985}; we have not checked whether it
coincides with $\mu$.

\begin{proposition}[Commutators: three contrasts multiply]\label{prop:commutator}
Let $A,B>0$ in $M_n$.
\begin{enumerate}
\item[(a)] $\|AB-BA\|\le2\OMC(A)\OMC(B)\,\bar\lambda(A)\bar\lambda(B)$, with equality for suitable $A,B$ of any prescribed contrasts in
every $M_n$, $n\ge2$. In particular $\|AB-BA\|\le\frac{2\OMC(A)\OMC(B)}{(1+\OMC(A))(1+\OMC(B))}\|A\|\|B\|\le
\frac12\|A\|\|B\|$.
\item[(b)] For a state $\varphi_\rho$ ($\Tr \rho=1$) with $\OMC(\rho)=\delta$, $\kappa=\frac{1+\delta}{1-\delta}$,
\[
  |\varphi_\rho(AB-BA)|\le2\,c_n(\delta)\,\OMC(A)\OMC(B)\,\bar\lambda(A)\bar\lambda(B),\qquad
  c_n(\delta)=\begin{cases}\delta,&n\ \text{even},\\[2pt] \dfrac{k(\kappa-1)}{k(\kappa+1)+1},&n=2k+1,\end{cases}
\]
and the bound is attained for every $\delta$ and all prescribed contrasts of $A$ and $B$.
\end{enumerate}
\end{proposition}
\noindent Part (a) is a consequence of Kittaneh's inequality $\|XY-YX\|\le\|X\|\|Y\|$ for $X\ge0$~\cite{kittaneh2008},
applied to $X=A-m_A\ge0$ and $Y=B-\bar\lambda(B)$, whose norms are $2\OMC(A)\bar\lambda(A)$ and $\OMC(B)\bar\lambda(B)$;
equivalently, of the bound $\|AB-BA\|\le\frac12\|A\|\|B\|$ for positive operators applied to $A-m_A$ and $B-m_B$. Our
contribution is the sharp trace version (b), the commutator test of traciality with all three contrasts prescribed.

\begin{proof}
(a) Write $A=\bar\lambda(A)(1+E)$, $B=\bar\lambda(B)(1+F)$ with self-adjoint $E,F$, $\|E\|=\OMC(A)$, $\|F\|=\OMC(B)$.
Then $AB-BA=\bar\lambda(A)\bar\lambda(B)(EF-FE)$ and $\|EF-FE\|\le2\|E\|\|F\|$. Equality holds on $M_2$ for $E=\OMC(A)\sigma_z$,
$F=\OMC(B)\sigma_x$, since $\sigma_z\sigma_x-\sigma_x\sigma_z=2i\sigma_y$; in $M_n$ add a direct summand $\bar\lambda(A)\oplus\bar\lambda(B)$.
The second form uses $\|A\|=\bar\lambda(A)(1+\OMC(A))$. (b) $C=i(AB-BA)$ is self-adjoint with $\Tr C=0$ and $\|C\|\le K$, the bound
of (a). By von Neumann's trace inequality~\cite{bhatia1997}, $|\Tr \rho C|$ is at most the maximum of
$\sum_iw_ic_i$ over decreasingly ordered $w_i$ (the eigenvalues of $\rho$) and $c_i\in[-K,K]$ with $\sum_ic_i=0$, which is
$K\sum_{i\le k}(w_i-w_{n+1-i})$, $k=\lfloor n/2\rfloor$. Let $a=w_1$, $b=w_n=a/\kappa$ and $S=\sum_{i\le k}w_i\le ka$. Then
$\sum_{i\le k}(w_i-w_{n+1-i})\le S-kb$ and $1=\Tr \rho\ge S+(n-k)b$. Maximising $S-kb$ under $S\le k\kappa b$ and $S\le1-(n-k)b$ gives
the value at $k\kappa b=1-(n-k)b$, namely $k(\kappa-1)/(k\kappa+n-k)$, which is $c_n(\delta)$ for even and odd $n$. For equality take $k$
copies of the $2\times2$ pair of (a), $\rho$ diagonal in the eigenbasis of the $k$ blocks of $C$ with the eigenvalues $a$ on the
positive and $b$ on the negative eigenvectors of $C$, and, for odd $n$, the eigenvalue $b$ of $\rho$ and scalars $A$, $B$ on
the remaining one-dimensional block.
\end{proof}

So the most basic characterisation of the trace, $\varphi(AB)=\varphi(BA)$, fails exactly at the order of the product of
the three contrasts, $2\OMC(\rho)\OMC(A)\OMC(B)$ in even dimension; this is the simplest instance of the principle that a
violation needs a functional far from a trace and operators far from scalars. The odd dimensions are slightly weaker, as
for Gardner's inequality (Theorem~\ref{thm:gardnerD}), and $c_{2k+1}(\delta)\to\delta$ as $k\to\infty$.

Random samples ($\COSamples$ per bound and dimension, Haar random relative positions, the intermediate eigenvalues
uniform between the extreme ones) never exceed the bounds of this section on $M_2$, $M_3$ and $M_4$ (the largest value of
sample minus bound is $\COWorst$), and among $\COSamples$ random pairs per dimension with contrasts uniform in $[0,1]^2$ the
number of indefinite Jordan products with $\delta_A^2+\delta_B^2\le1$ is $\COIndefInside$. As in
Section~\ref{sec:montecarlo}, random positions rarely come close to the extremes in higher dimension: for the commutator
test on $M_4$ the samples stay far below the bound, which needs two aligned $2\times2$ blocks
(Figure~\ref{fig:ops-commutator}).

\subsection{Thresholds in the contrast of the operators}\label{sec:thresholds}

The numerical search of Section~\ref{sec:bothcontrasts} shows that several characterisations are not violated at all
unless the operators have contrast close to $1$. The composition rules above explain this for two of them. The simplest
example is the analogue of item~21 for positive operators, $\varphi(\Re AB)\ge0$: it holds for every positive functional
as long as $\OMC(A)^2+\OMC(B)^2\le1$, because then $\Re AB=\frac12(AB+BA)\ge0$ by Theorem~\ref{thm:jordan}, and for
operators of equal contrast $\delta'$ it can fail only for $\delta'>1/\sqrt2$. Projections have the angle $\frac\pi2$,
$\Re PQ$ is indefinite in general, and accordingly item~21 is violated by every non-tracial functional, with the
envelope $\frac12\tan^2(\psi/2)$ of Theorem~B.

For subadditivity of the square root (item~12) the same mechanism works one level deeper, through the link between
Jordan products and subadditivity found by Moslehian and Najafi~\cite{moslehian2011}.

\begin{proposition}[Sharp contrast threshold for item 12]\label{prop:item12}
Let $A,B>0$ with $\OMC(A^{1/2})^2+\OMC(B^{1/2})^2\le1$, i.e.
\[
  \Bigl(\frac{\OMC(A)}{1+\sqrt{1-\OMC(A)^2}}\Bigr)^2+\Bigl(\frac{\OMC(B)}{1+\sqrt{1-\OMC(B)^2}}\Bigr)^2\le1 .
\]
Then $(A+B)^{1/2}\le A^{1/2}+B^{1/2}$, and hence $\varphi((A+B)^{1/2})\le\varphi(A^{1/2})+\varphi(B^{1/2})$ for every
positive functional $\varphi$. In particular, if $\OMC(A)=\OMC(B)=\delta'\le2\sqrt2/3\approx0.943$, item~12 is not
violated for any density, so $V_n(\delta,\delta')\le0$ for all $\delta$ and $n$. The threshold is sharp: if
$\delta'>2\sqrt2/3$, then for every $n\ge2$ item~12 is violated by operators of contrast $\delta'$ and every density of
contrast $\delta$ close enough to $1$.
\end{proposition}
\begin{proof}
By Proposition~\ref{prop:rapid}(a), $\OMC(A^{1/2})=\tanh(\eta_A/2)=\OMC(A)/(1+\sqrt{1-\OMC(A)^2})$, which gives the
two forms of the hypothesis. By Theorem~\ref{thm:jordan}(a) applied to $A^{1/2}$ and $B^{1/2}$,
$A^{1/2}B^{1/2}+B^{1/2}A^{1/2}\ge0$, hence
\[
  (A^{1/2}+B^{1/2})^2=A+B+A^{1/2}B^{1/2}+B^{1/2}A^{1/2}\ge A+B ,
\]
and, the square root being operator monotone, $A^{1/2}+B^{1/2}\ge(A+B)^{1/2}$. A positive functional preserves this
inequality. For $\OMC(A)=\OMC(B)=\delta'$ the hypothesis reads $\tanh(\eta/2)\le1/\sqrt2$ with $\tanh\eta=\delta'$, i.e.\
$\delta'=2t/(1+t^2)$ at $t=1/\sqrt2$, which is $2\sqrt2/3$.

Sharpness. Let $\delta'>2\sqrt2/3$, so that $t=\OMC(A^{1/2})>1/\sqrt2$, and on $M_2$ let $A=\mathrm{diag}(1,\ell')$ and
$B=RAR^*$ with $R$ the rotation by $\pi/4$. The reflection $F$ in the line at angle $\pi/8$ interchanges $A$ and $B$, hence
$A^{1/2}$ and $B^{1/2}$, so it commutes with $A+B$, with $S=A^{1/2}+B^{1/2}$ and with $J=A^{1/2}B^{1/2}+B^{1/2}A^{1/2}$.
Since $F$ has the simple eigenvalues $\pm1$, these three matrices are diagonal in its eigenbasis $u_1,u_2$, and
$S^2=A+B+J$ gives $\langle u_i,Su_i\rangle^2=\langle u_i,(A+B)u_i\rangle+\langle u_i,Ju_i\rangle$. The pair $A^{1/2}$,
$B^{1/2}$ is the configuration of the converse in Theorem~\ref{thm:jordan}(a) with $\delta_A=\delta_B=t$, so $J$ is
indefinite since $2t^2>1$, and for the $u_i$ with $\langle u_i,Ju_i\rangle<0$ we get
$\langle u_i,(A^{1/2}+B^{1/2})u_i\rangle<\langle u_i,(A+B)^{1/2}u_i\rangle$. On $M_n$ replace $A,B$ by $A\oplus1_{n-2}$ and
$B\oplus1_{n-2}$, which have the same contrast, and let $\rho$ have $u_i$ as eigenvector for its largest eigenvalue and all
other eigenvalues equal. As $\delta\to1$, $\varphi_\rho\to n\langle u_i,\cdot\,u_i\rangle$, so
$\varphi_\rho((A+B)^{1/2})>\varphi_\rho(A^{1/2})+\varphi_\rho(B^{1/2})$ for all $\delta$ close enough to $1$.
\end{proof}

\begin{figure}[t]
\centering\includegraphics[width=\textwidth]{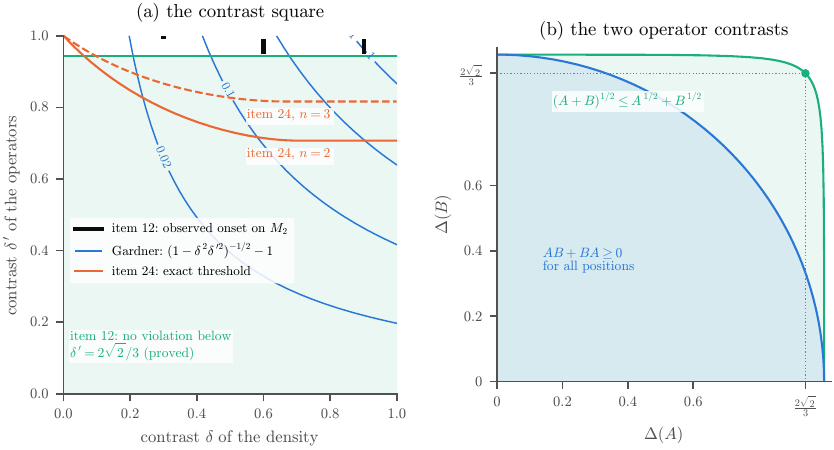}
\caption{(a) The contrast square: level curves of the Gardner envelope $(1-\delta^2\delta'^2)^{-1/2}-1$ (blue; the
values $0.02$, $0.1$, $0.3$, $1$), the band in which item~12 cannot be violated (green,
Proposition~\ref{prop:item12}), the exact violation thresholds of item~24 for $n=2$ and $n=3$ (orange,
Proposition~\ref{prop:item24}), and, in black, the interval of grid values of $\delta'$ in which the numerical search on $M_2$
first finds a violation of item~12, at $\delta=0.3$, $0.6$ and $0.9$. (b) The plane of the two operator contrasts:
$AB+BA\ge0$ for all relative positions inside the quarter disc (blue, Theorem~\ref{thm:jordan}), and
$(A+B)^{1/2}\le A^{1/2}+B^{1/2}$ in the larger green region (Proposition~\ref{prop:item12}). In the figure
$\Delta(\cdot)=\OMC(\cdot)$, $\delta$ is the contrast of the density and $\delta'$ that of the operators.}
\label{fig:square}
\end{figure}

The proved threshold is consistent with the computation: on $M_2$ and $M_3$ the search finds no violation of item~12 for
$\delta'\le\TCCSubaddNone$ and $\delta\in\{0.3,0.6,0.9\}$, and the first violation above $10^{-4}$ at
$\delta'=$ \TCCSubaddOnset; all these grid values lie above $2\sqrt2/3$ (Figure~\ref{fig:square}(a)). Unlike item~24 below,
whose whole threshold curve we find for $n=2$, here sharpness is proved only in the limit $\delta\to1$; at
$\delta=0.3$, $0.6$ and $0.9$ the search finds the onset of violations higher. The criterion
of Moslehian and Najafi~\cite[Theorem~2.1]{moslehian2011} is that $AB+BA\ge0$ if and only if $f(A+B)\le f(A)+f(B)$ for
\emph{all} non-negative operator monotone functions $f$. For the square root alone the weaker hypothesis
$A^{1/2}B^{1/2}+B^{1/2}A^{1/2}\ge0$ suffices, and the proof of sharpness shows that in the extremal configuration it is
also necessary.

For the power-mean inequality $\varphi(A^2)^{1/2}\le\varphi(A)$ with $\varphi(1)=n$ (item~24) a Kantorovich-type argument
gives the envelope with both contrasts prescribed in closed form, and with it a sharp threshold.

\begin{proposition}[Item 24 with both contrasts]\label{prop:item24}
Let $n\ge2$, $\OMC(\rho)=\delta<1$, $\Tr \rho=n$, and let $A\ge0$ have $\OMC(A)=\delta'$. Put $\ell=\frac{1-\delta}{1+\delta}$,
$\ell'=\frac{1-\delta'}{1+\delta'}$,
\[
  x_0=n\ell'+(1-\ell')\frac{n\ell}{n-1+\ell},\qquad x_*=\frac{2n\ell'}{1+\ell'},\qquad
  f(x)=\frac{\sqrt{(1+\ell')x-n\ell'}}x .
\]
Then $\varphi_\rho(A^2)^{1/2}/\varphi_\rho(A)\le f(\max(x_0,x_*))$, so $V_n(\delta,\delta')\le f(\max(x_0,x_*))-1$, with equality
if $x_0\ge x_*$ or $(n-1)\ell\ell'\le1$, in particular always for $n=2$. Consequently:
\begin{enumerate}
\item[(a)] item~24 is violated for no density if and only if $\delta'\le\sqrt{(n-1)/n}$; for $\delta'>\sqrt{(n-1)/n}$ it is
violated for every $\delta$ close enough to $1$;
\item[(b)] for $n=2$ the violation threshold is the curve $f(\max(x_0,x_*))=1$ in the $(\delta,\delta')$-square, and it equals
$1/\sqrt2$ for all $\delta\ge1/\sqrt2$.
\end{enumerate}
\end{proposition}
\begin{proof}
Both sides are homogeneous of degree one in $A$, so let $\|A\|=1$; then the spectrum of $A$ lies in $[\ell',1]$ and contains
both end points. From $(1-A)(A-\ell')\ge0$ we get $A^2\le(1+\ell')A-\ell'$, hence, with $x=\varphi(A)$,
\[
  \frac{\varphi(A^2)^{1/2}}{\varphi(A)}\le\frac{\sqrt{(1+\ell')x-n\ell'}}x=f(x),
\]
with equality for two-level $A$. If $e$ is a unit eigenvector of $A$ for the eigenvalue $1$, then $A\ge\ell'+(1-\ell')ee^*$,
so $x\ge n\ell'+(1-\ell')\langle \rho e,e\rangle\ge n\ell'+(1-\ell')\lambda_{\min}(\rho)$, and $\lambda_{\min}(\rho)\ge n\ell/(n-1+\ell)$
for $\Tr \rho=n$ (proof of Proposition~\ref{prop:power}); so $x\ge x_0$. The function $f$ increases on $[n\ell'/(1+\ell'),x_*]$
and decreases on $[x_*,\infty)$, since $\frac d{dx}f(x)^2=x^{-3}\bigl(2n\ell'-(1+\ell')x\bigr)$; hence $f(x)\le f(\max(x_0,x_*))$.
For equality take $A=\ell'+(1-\ell')ee^*$ with a unit vector $e$, so that $x=n\ell'+(1-\ell')\langle \rho e,e\rangle$. If
$x_0\ge x_*$, let $\rho$ have the spectrum $\{a,\dots,a,b\}$ and $e$ an eigenvector for $b$; then $x=x_0$. Otherwise we need
$\langle \rho e,e\rangle=n\ell'/(1+\ell')$; the values $\langle \rho e,e\rangle$ for the spectra $\{a,\dots,a,b\}$ and
$\{a,b,\dots,b\}$ (normalised to $\Tr \rho=n$) fill the interval $[n\ell/(n-1+\ell),\,n/(1+(n-1)\ell)]$, whose left end is
at most $n\ell'/(1+\ell')$ because $x_0\le x_*$, and whose right end is at least $n\ell'/(1+\ell')$ if $(n-1)\ell\ell'\le1$.

(a) Since $f(\max(x_0,x_*))\le f(x_*)=(1+\ell')/(2\sqrt{n\ell'})$, there is no violation if $(1+\ell')^2\le4n\ell'$, i.e.\
$\ell'\ge(\sqrt n-\sqrt{n-1})^2$, i.e.\ $\delta'\le\sqrt{(n-1)/n}$. Conversely, as $\delta\to1$ we have $\ell\to0$, so
$x_0\to n\ell'<x_*$ and $(n-1)\ell\ell'\le1$, and the envelope tends to $f(x_*)-1>0$ when $(1+\ell')^2>4n\ell'$. (b) For $n=2$
equality always holds; $x_0\le x_*$ means $1-\delta\le2\ell'/(1+\ell')$, and at $\ell'=(\sqrt2-1)^2$ this is
$\delta\ge1/\sqrt2$.
\end{proof}

\noindent The computed envelopes of item~24 on $M_2$ and $M_3$ with both contrasts prescribed agree at every grid point
with Proposition~\ref{prop:item24} to $\TCCPowerMeanDev$ (absolute difference). The threshold rises above $\sqrt{(n-1)/n}$ for smaller $\delta$ and
tends to $1$ as $\delta\to0$ (Figure~\ref{fig:square}(a)): the less tracial the functional, the more moderate the
operators that detect it. For monotonicity (item~10) an argument of this kind is not available: for $0\le A\le B=A+tvv^*$ with a unit vector $v$
that is not an eigenvector of $A$, $B^2-A^2=t\bigl((A+\frac t2)vv^*+vv^*(A+\frac t2)\bigr)$ is indefinite for every $t>0$,
however small the contrasts of $A$ and $B$, so the operator inequality fails and only the functional can prevent a
violation. Correspondingly, the numerically observed onset of violations of item~10 depends on $\delta$: on $M_2$ the
first grid value of $\delta'$ with a violation above $10^{-4}$ is \TCCMonOnsetTwo.

\begin{figure}[!htbp]
\centering\includegraphics[width=\textwidth]{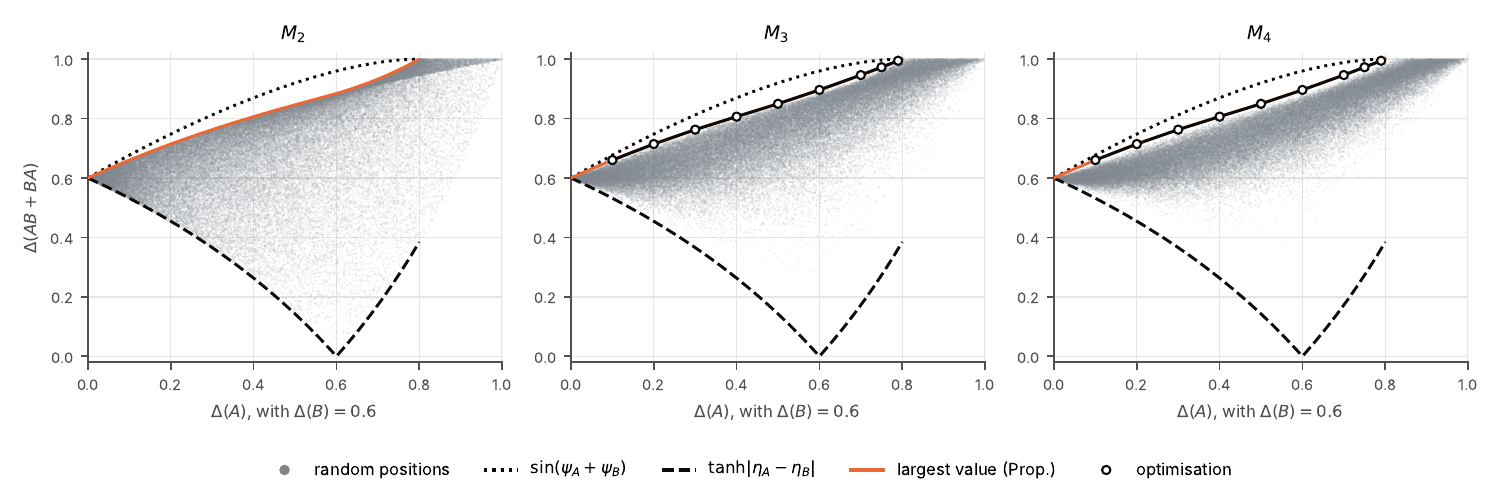}
\caption[]{Contrast of the Jordan product $AB+BA$ against $\OMC(A)<0.8$, with $\OMC(B)=0.6$ (for $\OMC(A)\le0.8$ it
is positive, Theorem~\ref{thm:jordan}(a)): the bounds of Theorem~\ref{thm:jordan}(b), the largest value of
Proposition~\ref{prop:jordanmax} (orange; on $M_3$ and $M_4$ the value $(2-\mu)/(2+\mu)$), and the largest value found by
optimisation (markers). In the figure $\Delta(\cdot)=\OMC(\cdot)$.}
\label{fig:ops-jordan}
\end{figure}
\begin{figure}[!htbp]
\centering\includegraphics[width=\textwidth]{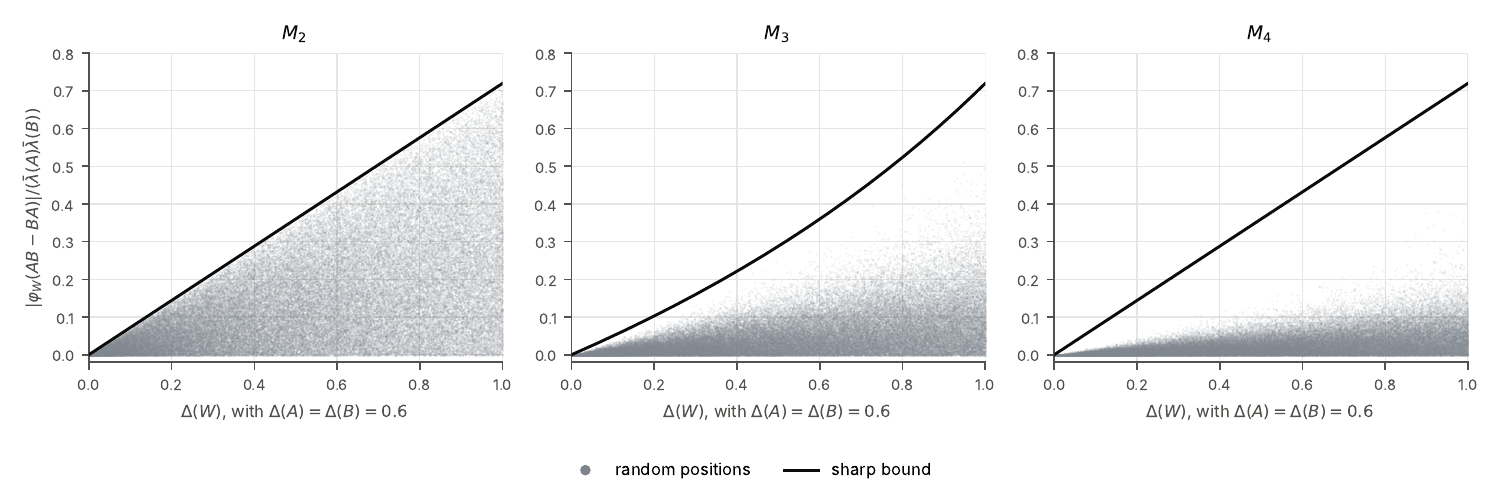}
\caption[]{The commutator test $|\varphi_\rho(AB-BA)|/(\bar\lambda(A)\bar\lambda(B))$ for random states $\varphi_\rho$ and random
$A,B$ against $\OMC(\rho)$, with $\OMC(A)=\OMC(B)=0.6$, and the sharp bound $2c_n(\delta)\cdot0.36$ of
Proposition~\ref{prop:commutator}(b). In the figure $\varphi_W=\varphi_\rho$, $\Delta(W)=\OMC(\rho)=\delta$ and
$\Delta(\cdot)=\OMC(\cdot)$.}
\label{fig:ops-commutator}
\end{figure}

\section{Numerical evidence}\label{sec:numerics}

\subsection{Method}\label{sec:method}
The envelopes were computed by constrained gradient ascent with the library of~\cite{lib} (the programs
\texttt{trace\_\allowbreak characterizations.py} and \texttt{trace\_\allowbreak contrast.py} of its repository). The
random sampling (\texttt{trace\_\allowbreak montecarlo.py}), the computations for composite operators
(\texttt{contrast\_\allowbreak operations.py}) and Figure~\ref{fig:gardner1d} (\texttt{trace\_\allowbreak
envelopes\_\allowbreak figures.py}) are plain NumPy/SciPy programs in the same repository that do not use the library.
The density $\rho$ is diagonal with a spectrum parametrised so that $\OMC(\rho)=\delta$
holds exactly, the other eigenvalues being free. Positive variables have a free contrast in $[0,1]$, so rank-deficient
operators are included; there are self-adjoint and unitary variables (polar factors, and projections $UDU^*$ with every rank
pattern spread over the starts) and positive scalars. In item~2 the search runs over self-adjoint $X,Y$ only, which
contain the extremiser of Corollary~\ref{cor:item5}. Half of the starts begin at structured points (two-level spectra,
rotations in pairs of eigenvectors of $\rho$). The lattice operations are the spectral projections of $P+Q$ for the
eigenvalue $2$ ($P\wedge Q$) and for its range ($P\vee Q$), computed by functional calculus with an indicator function. In
items 19 and~23 the suprema are approached as $P$ and $Q$ merge, where $P\wedge Q$ and $P\vee Q$ jump; there the search
maximises $(\cL-\cR)/\cN-10^{-12}/\|P-Q\|_{\mathrm F}^2$, which keeps $P$ and $Q$ about $10^{-3}$ apart, while the reported
values are those of $(\cL-\cR)/\cN$. A grid of $\TCGridSize$ contrasts times $\TCStarts$ starts forms one batch per item
and dimension ($\TCSteps$ gradient steps, $n=\TCDims$).

The best final point for every contrast is recomputed with $\TCDigits$-digit arithmetic (mpmath), after the unitaries
are re-orthonormalised and the contrast is recomputed from the stored eigenvalues. Every reported value is therefore the
exact violation at an exactly admissible point, and so a rigorous \emph{lower} bound for the supremum; we call it the
\emph{computed envelope}. It is not an upper bound: where no theorem is available, the true envelope may be larger. The
double-precision and high-precision values agree to $\TCCertDev$ (relative), and at $\delta=0$ no value exceeds
$\TCVzero$. For items 12--16, independent NumPy/SciPy searches (BFGS from $\TCSciStarts$ random starts on $M_2$ at
$\delta=0.3,0.5,0.9$) give values such that the ratio of the batched search to them lies in
$[\TCSciRatioMin,\TCSciRatioMax]$; the ratios above $1$ come from item~13, where BFGS stops earlier. The whole computation takes $\TCSeconds$~s on $\TCWorkers$ CPU processes.

\subsection{Envelopes with the operators free}

\begin{table}[h]
\centering\scriptsize
\setlength{\tabcolsep}{4pt}
\begin{tabular}{llllll}
\toprule
items & $V^\star(\delta)$ & as $\delta\to0$ & $V^\star(0.5)$ & search $V(0.5)$ & status\\
\midrule
\TCClassRows
\bottomrule
\end{tabular}
\caption[]{The items with a proved envelope or lower bound, ordered by the order of vanishing at the trace. $V^\star$:
the sharp value or, where marked $\ge$, a proved lower bound; ``search'': the computed envelope on $M_2$ (a range over the
items of a row). Items 11 and 24 are given for $n=2$.}
\label{tab:classes}
\end{table}

\begin{table}[h]
\centering\scriptsize
\setlength{\tabcolsep}{1.8pt}
\begin{tabular}{rlccccccrc}
\toprule
 & inequality $\cL\le \cR$ & ref. & $V(0.1)$ & $V(0.5)$ & $V(0.9)$ & $V^\star(0.5)$ & order & slope & $n$-dep.\\
\midrule
\TCRows
\bottomrule
\end{tabular}
\caption[]{Computed envelopes (certified lower bounds) on $M_2$ at $\delta=0.1,0.5,0.9$ and the proved values $V^\star$;
$^a$: sharp for $n=2$ only, $^b$: a lower bound, $^c$: approached but not attained. ``order'': leading order of $V^\star$
as $\delta\to0$; ``slope'': $\log_2(V(0.2)/V(0.1))$ of the search; ``$n$-dep.'': the largest $|V_n/V_2-1|$ for $n=3,4$ and
$\delta\ge0.1$ (for item~7 it reflects search shortfalls, for item~13 a genuine increase with $n$; items 19 and 20 differ
for odd $n$, Theorem~\ref{thm:projP}). ``n.r.'': no violation resolved by the search (\TCMode{} run on \TCDevice).}
\label{tab:values}
\end{table}

Tables~\ref{tab:classes} and~\ref{tab:values}, Figure~\ref{fig:ladder} and the top rows of
Figures~\ref{fig:mc-gardner}--\ref{fig:mc-wmon} summarise the computation. Where an envelope is known exactly and
attained (items 1--5, 8, 11, and 9 for $n=2$; for item~10, exact by Theorem~\ref{thm:mono2} and
Remark~\ref{rem:g10}, see Section~\ref{sec:numcompare}), the search reproduces it to $\TCSharpDev$ for $\delta\ge0.05$ and
$n=2,3,4$; for items 17, 18, 20--22 and 24 it does so to $\TCSharpDevNew$, and for the approached suprema of 19 and 23 to
$\TCSharpDevApproached$. The values of 17, 18 and 21--23 on $M_3$ and $M_4$ agree with those on $M_2$ to $\TCNDevNew$, as
Theorem~\ref{thm:proj}, Corollary~\ref{cor:item5} and Propositions~\ref{prop:item21} and~\ref{prop:item22} predict, and on $M_3$ the search finds the odd-dimensional value $\TCExpOddHalf$ of item~20 at
$\delta=0.5$ ($M_2$: $\TCRefExpqpqHalf$). These agreements are a check of the method rather than new information.

The new information concerns the items without a proved envelope. All five items 12--16 fail quadratically, with
$V\approx0.5\,\delta^2$ (items 12 and 13), $\delta^2/36$ (item~14, equal to the lower bound of
Proposition~\ref{prop:holder} to leading order), $0.064\,\delta^2$ (item~15) and $0.09\,\delta^2$ (item~16). Item~15
decreases with $n$, because the identity part of $e^{P/2}e^Qe^{P/2}$ enters $\cN$, whereas item~13 increases with $n$ at
large $\delta$ (on $M_3$ at $\delta=0.9$ the value $\TCVPowersNineThree$ against $\TCVPowersNine$ on $M_2$); items 12,
14 and~16 give the same values for $n=2,3,4$ (to $\TCNDev$). In item~16 the best points have large $\|A\|$ and $\|B\|$, and the supremum is approached as the
scale grows. Among items 6, 7 and 25 (marked $^b$ in Table~\ref{tab:values}) and item~10 (exact,
Theorem~\ref{thm:mono2}), whose computed envelopes were compared with the proved bounds, no computed value exceeds the bound (the largest relative
difference, search minus bound, is $\TCBoundExcess$) except for item~6 in dimension $n\ge3$: on $M_2$ the search never
exceeds $G-1$, but for $n\ge3$ it finds larger violations, $V\ge\TCSGMExcessV$ at $n=\TCSGMExcessN$,
$\delta=\TCSGMExcessDelta$, where $G-1=\TCSGMExcessKappa$. At $n=3$, $\delta=0.9$ the present run reaches only
$\TCVSgmNineThree<G-1=\TCRefSgmNine$, but a point found by an earlier run and stored in the test suite of the
repository gives $V>2.1$ there (recomputed with 60 digits). For item~6 we therefore only know $G-1\le V_n\le \kappa-1$ (Proposition~\ref{prop:sgm}). Items 9
and 10 give the same values for $n=3,4$ as for $n=2$, to $\TCNDevProj$ and $\TCNDevMon$.

Except for items 6 ($n\ge3$), 11 ($A=\rho^{-1}$) and 25 ($A$ of rank two with eigenvalue ratio $\approx\delta^2/8$), the best
points are rank-one operators or matrix units in the span of eigenvectors of $\rho$ for its extreme eigenvalues, and they do
not commute with $\rho$. The search has two limitations. The proved lower bounds of items 6 and 7 are approached only as
an operator becomes singular; gradient ascent reaches $\TCSGMRatioHalf$ and $\TCALTRatioHalf$ of the lower bound at $\delta=0.5$ and
finds no violation of item~6 for $\delta\le0.2$. Item~25 is resolved for $\delta\ge\TCWmonFirst$; the search reaches
$g_{25}$ to $10^{-3}$ for $\delta\ge\TCWmonReach$ and does not exceed it. Finally, on $M_2\oplus M_3$ with
$\rho=1\oplus s1$ and $\rho=\mathrm{diag}(1,\ell)\oplus s1$, $s\in\{\frac14,4\}$, and intra-block contrast
$\delta\in\{0.3,0.5,0.7\}$, the computation confirms Corollary~\ref{cor:sectors}: for central $\rho$ the largest value found for the items that
hold for every trace is $\TCSectorCentral$; for items 11 and 24 the search reproduces the exact values for the fixed
density to $\TCSectorDet$; with intra-block contrast no value exceeds the $M_2$ envelope (largest relative difference
$\TCSectorExcess$), and where Corollary~\ref{cor:sectors} gives equality the search reaches it to $10^{-3}$ in $\TCSectorReached$ of
$\TCSectorTotal$ cases.

\subsection{Both contrasts}\label{sec:bothcontrasts}

For the 14 items with operator variables that are not projections, a second computation prescribes the contrast
$\delta'$ of all operator variables as well (in items 10 and 25 the variables of prescribed contrast are $A$ and $B-A$;
$\TCCStarts$ starts per grid point, $\TCCSteps$ gradient steps, best point
recomputed with $\TCCDigits$ digits; $\TCCSeconds$~s on four CPU processes). The bottom rows of
Figures~\ref{fig:mc-gardner}--\ref{fig:mc-wmon} (Appendix~\ref{app:mc}) show the computed $V(0.9,\delta')$. At
$\delta'=0$ no item is violated, as Proposition~\ref{prop:scalar} requires, but beyond that the dependence on $\delta'$
takes several forms. For item~1 the values follow Theorem~\ref{thm:gardnerD} to $\TCCGardnerDev$ (relative, $n=2,3$,
all grid points). For items 2 and~3 the theorem gives only the upper bound $G(\delta\delta')-1$; on $M_2$ the
search reaches it (to $\TCCGardnerDevEven$), and on $M_3$ it reaches the odd-dimensional value of item~1 up to the
relative differences $\TCCTriangleDevOdd$ (item~2) and $\TCCRealpartDevOdd$ (item~3). Item~8 follows
Proposition~\ref{prop:youngtwo} on $M_2$; here the violation grows smoothly with both contrasts. Items 6, 10, 12, 13, 24
and~25, in contrast, show no violation above $10^{-4}$ for $\delta'<\TCCThresholdMin$: at $\delta=0.9$ the first grid
value of $\delta'$ with a violation above $10^{-4}$ is \TCCThresholds\ (``--'': none on the grid). For items 12 and 24
Section~\ref{sec:thresholds} proves such thresholds, and for item~24 the computed values agree with the exact envelope of
Proposition~\ref{prop:item24}; for items 6, 10, 13 and 25 the thresholds remain observations. The determinant inequality
(item~11) behaves differently again: it is violated most strongly when the contrast of $A$ is close to that of $\rho$, since
its extremiser is $\rho^{-1}$, and in $M_2$ items 4 and~7 also have an interior maximum in $\delta'$ (for item~4 at the
contrast of $|A|$ for the extremiser $A=\sqrt a\,e_1e_2^*+\sqrt b\,e_2e_1^*$), which the additional dimension of $M_3$
removes. Finally, item~16 does not see the contrast of the operators at all: it is invariant under $A\mapsto A+c$,
$B\mapsto B+c$, which makes both contrasts arbitrarily small without changing the quotient, so $V(\delta,\delta')$ equals the
envelope with free operators for every $\delta'>0$ and jumps to $0$ at $\delta'=0$.

The same computation restricted to two-level operators of contrast $\delta'$ reproduces the violations found with
general spectra on $M_{\TCCTwoLevelDim}$ (to $1\%$, wherever they exceed $10^{-4}$) except in the items and at the
grid values \TCCTwoLevelShort. In item~4 the extremisers on $M_3$ apparently need a third eigenvalue (with the contrast
free, the extremiser has the eigenvalues $\sqrt a,\sqrt b,0$; Section~\ref{sec:projsuffice}), and in items 6 and~7 the two-level values fall further behind as $\delta'\to1$. In item~25 two-level operators give no violation
at $\delta=0.9$, $\delta'=0.99$ and $0.999$, where general spectra give $1.9\times10^{-4}$ and $1.5\times10^{-3}$; in
item~11 they fall short most at $\delta=0.6$, $\delta'=0.9$ ($0.017$ against $0.045$), and in item~2 by about
$10\%$ at $\delta=0.3$. In item~13 at $\delta=0.9$ two-level operators give no violation for $\delta'=0.9$--$0.99$,
where general spectra give $0.019$--$0.11$, but this is a failure of the restricted search rather than evidence: on
$M_3$ it returns the same values at every $\delta'\le0.99$, i.e.\ it does not move away from its starting points. On
$M_2$, where every operator of contrast $\delta'$ is two-level, the two searches should agree; they differ in item~7
($\delta=0.3$, $\delta'=0.9$) and in item~13 ($\delta=0.9$, $\delta'=0.9$ and $0.95$, where the restricted search finds
the larger value), which indicates the reliability of the search near $\delta'=1$. In item~10 the two classes differ only
where neither violates the inequality.

\subsection{Projections and two-level operators}\label{sec:projsuffice}

For the items to which the vertex principle (Lemma~\ref{lem:vertex}) applies, two-level operators attain the supremum,
and in eleven items the variables are projections by definition. To see how far this extends, we repeated the search for
the other 14 items with the positive variables (the moduli, for the others) restricted to two-level operators
$c(\ell+(1-\ell)P)$ with free $\ell\in[0,1]$, $c>0$ and a projection $P$ of any rank (for self-adjoint variables
$\alpha+\beta P$ with free real $\alpha,\beta$), and to scaled projections $cP$ (Table~\ref{tab:projsuffice}).

\begin{table}[h]
\centering\scriptsize
\setlength{\tabcolsep}{3pt}
\begin{tabular}{r|ccc|ccc|ccc|ccc}
\toprule
 & \multicolumn{6}{c|}{$M_2$} & \multicolumn{6}{c}{$M_3$}\\
 & \multicolumn{3}{c|}{$\delta=0.5$} & \multicolumn{3}{c|}{$\delta=0.9$} & \multicolumn{3}{c|}{$\delta=0.5$} &
 \multicolumn{3}{c}{$\delta=0.9$}\\
item & all & two-level & proj. & all & two-level & proj. & all & two-level & proj. & all & two-level & proj.\\
\midrule
\TCCRowsA
\bottomrule
\end{tabular}
\caption[]{Computed envelopes (certified lower bounds) over all operators, over two-level operators and over scaled
projections, for the items whose variables are not projections. In $M_2$ every operator is two-level, so there the middle
column only checks the restricted parametrisation. ``--'': item~11 needs an invertible $A$.}
\label{tab:projsuffice}
\end{table}

The result is an observation with clear exceptions, not a general principle. Scaled projections attain the computed
envelope for items 1, 2, 3, 8, 10, 12, 16 (to within $\TCCProjGap\%$) and 24 in $M_2$ and $M_3$, and for item~13 in $M_2$.
The restricted searches are less reliable than the unrestricted one: two-level operators include the scaled projections,
so every entry of the two-level column below the projection column is a failure of the restricted search, as are, on
$M_2$, where every operator is two-level, the values $0$ for item~6 and $0.0329$ for item~13 at $\delta=0.5$. In items 6
and~7 the supremum is approached only in a degenerate limit in which an operator becomes a projection and both sides of
the inequality vanish; restricting the search to projections (item~6) brings it closer to the proved lower bound
$G-1$ than the unrestricted search, on $M_2$ and on $M_3$, and restricting it to two-level operators (item~7) brings
it closer to $G^2-1$ on $M_2$ (but not on $M_3$); for item~7 the violation at projections themselves is zero. Two-level operators that are not projections are needed for item~4 (the extremiser has
$|A|$ with the eigenvalues $\sqrt a,\sqrt b$), item~11 ($A=\rho^{-1}$) and item~25 (rank-one $A$ gives no violation,
Proposition~\ref{prop:wmon}). In $M_3$ the extremiser of item~4 lives on a two-dimensional subspace and vanishes on its
complement, so it has the three eigenvalues $\sqrt a,\sqrt b,0$ and the restricted search misses it. Finally,
general operators give strictly larger violations than the restricted searches find in two cases: in $M_3$ at
$\delta=0.9$ item~13 is violated by $0.3048$, against $0.2197$ for projections and for $M_2$, and for item~6 with
$n\ge3$ the unrestricted search of the previous sections exceeds $G-1$, while the restricted searches stay below
it. In short, projections and their affine images are a good first test of a characterisation, but they do not always
realise the largest violation.

\subsection{Random sampling}\label{sec:montecarlo}

Random sampling is the method of~\cite[Sect.~5]{abed2024}. We extended it to all items: for $n=\TMCDims$ and each item
we draw $\TMCSamples$ random functionals and operators ($\rho$ with a uniformly distributed contrast and the other
eigenvalues uniform in between, positive operators $U\mathrm{diag}(\lambda)U^*$ with Haar-random $U$, uniformly distributed
contrast and random scale, arbitrary operators $V|X|$ with Haar-random $V$, self-adjoint ones with random signs, and
Haar-random projections), and a second sample in which every positive operator (every modulus) is replaced by a scaled
projection of random rank. Appendix~\ref{app:mc} shows all samples together with the computed envelopes, and
Table~\ref{tab:montecarlo} the fraction of violating samples.

\begin{table}[h]
\centering\small
\setlength{\tabcolsep}{4pt}
\begin{tabular}{rl|rr|rr|rr}
\toprule
 & & \multicolumn{2}{c|}{$M_2$} & \multicolumn{2}{c|}{$M_3$} & \multicolumn{2}{c}{$M_4$}\\
item & inequality & random & proj. & random & proj. & random & proj.\\
\midrule
\TMCRows
\bottomrule
\end{tabular}
\caption[]{Random sampling, $\TMCSamples$ samples per item, class and dimension: the percentage of samples that violate
the inequality, for random operators and for scaled projections. ``--'': not applicable (the variables are projections,
or an invertible operator is needed).}
\label{tab:montecarlo}
\end{table}

No sample exceeds the computed envelope by more than the interpolation between grid values allows: the largest ratio of
a sampled violation to the envelope, interpolated linearly in $\OMC(\rho)$ at the sample's own contrast, is
$\TMCWorstRel$ over all items and dimensions. Ratios slightly above $1$ are possible because the linear interpolation
underestimates the envelope where it is concave between grid values, and for the items without a closed form the
computed envelope is only a lower bound. This is an independent consistency check of the computed values. The samples also show why sampling alone is
a poor test of a characterisation. Random operators rarely come close to the envelope, and for monotonicity,
subadditivity, the Powers--St{\o}rmer inequality and the weighted monotonicity fewer than $0.1\%$ of the samples (or none)
violate the inequality at all, so a Monte Carlo study alone would suggest that these inequalities hold for every
$\varphi_\rho$. Sampling also deteriorates with the dimension: for Gardner's inequality the fraction of violating random
samples drops from $4.4\%$ on $M_2$ to below $0.1\%$ on $M_3$ and $M_4$, and the best sample stays far below the envelope
(Figure~\ref{fig:mc-gardner}), because the extremal configurations, rank-one operators on a two-dimensional subspace
aligned with the extreme eigenvectors of $\rho$, have a Haar measure that decreases rapidly with $n$. Sampling scaled
projections helps: the fraction of violating samples rises, for Gardner's inequality from $4.4\%$ to $19\%$ on $M_2$ and
from below $0.1\%$ to $0.6\%$ on $M_4$, for subadditivity from below $0.1\%$ to $21\%$, $9\%$ and $5\%$ on $M_2$, $M_3$,
$M_4$, and for the power-mean inequality from $7\%$ to $50\%$ on $M_2$, and the best sample comes close to the envelope
(Figures~\ref{fig:mc-gardner}, \ref{fig:mc-subadd} and~\ref{fig:mc-lp}). The exceptions are those of
Section~\ref{sec:projsuffice}: items 4 and~7, whose extremisers are two-level operators that are not projections, and
items 6 and~11, which need invertible operators.

\subsection{Comparison with the constants of Section~\ref{sec:catalogue}}\label{sec:numcompare}
Several computed envelopes can now be compared with constants proved in finite von Neumann
algebras. For item~10 the computed envelope equals $(G-1)^2/(4G)$ (Theorem~\ref{thm:mono2} and
Remark~\ref{rem:g10}) to the digits of Table~\ref{tab:values}, and for $n=3,4$ it agrees with $n=2$ to
$\TCNDevMon$, as the theorem predicts. For item~12 it agrees with the lower bound
$\mathsf t_{1/2}(\kappa)-1$ of Theorem~\ref{thm:tik} to three decimal places (Section~\ref{ssec:tik}),
and for item~16 it lies below and within $3\cdot10^{-3}$ of the lower bound $b_2(\kappa)-1$ of
Proposition~\ref{prop:pb}. For item~7 the search reaches $\TCALTRatioHalf$ of the proved lower bound
$G^2-1$ at $\delta=0.5$; the multistart value $c_{2,1}\ge2.99$ for $\kappa=10$ of Remark~\ref{rem:pbnum} is
also below the proved $G^2=3.025$. For item~6 the search exceeds $G-1$ for $n\ge3$ but stays below the
upper bound $\kappa-1$ of Proposition~\ref{prop:sgm}.

\subsection{Other numerical observations}\label{sec:numother}
The numerical statements in Remarks~\ref{rem:pznum}, \ref{rem:youngstar}, \ref{rem:afnum},
\ref{rem:expnum} and~\ref{rem:pbnum}, after Conjecture~\ref{conj:young} and after
Theorem~\ref{thm:tik}, concern quantities not covered by the certified computation: the constants of
items 27, 30, 34--36 and~38; the constant $c_{2,1}$ of item~7 at $\kappa=10$; and the constants $c_{\rm PB}$
(item~16) and $b$ at the values $\kappa=2,4,10,100$ on $M_2$ and on $M_3$ and $M_2\oplus M_2$, where no
certified computation was made. They were obtained by multistart optimisation in double precision on $M_2$, $M_3$, $M_4$ and
$M_2\oplus M_2$; unlike the computed envelopes of this section they are not recomputed at exactly
admissible points, so they are indications, not certified lower bounds. No proof uses them.

\section{Conclusions and open problems}\label{sec:open}

A characterisation of the trace by an inequality becomes quantitative once the distance of a functional
from the traces is measured by the Michelson contrast of its density on a factor, and by its
noncentrality on a general finite von Neumann algebra. For Gardner's inequality the answer is as simple
as it could be: the sharp constant is the secant of the angle whose sine is the product of the contrasts
of the density and of the operator, and with the operator free it is $\cosh\LNC(\rho)$. For the other
characterisations the finite proved matrix envelopes vanish at the trace to the orders $\delta$,
$\delta^2$ or $\delta^4$ (item~37 is $+\infty$ for every non-central $\rho$, i.e.\ non-scalar on $M_n$), and the weighted monotonicity (item~25) has
a proved lower bound of order $\delta^8$, with the classical reverse-inequality constants as exact
values where we could determine them. Of the 38 trace characterisations, 18 (items 1--5, 8, 10, 17, 18, 21--23, 26, 28, 29, 31, 33 and~37, the last being
$+\infty$) have constants that are explicit functions of $\kappa_\cZ$ in every finite von Neumann algebra,
five more are
exact on matrix algebras (items 11, 19, 24; item~9 for $n=2$; item~20 for even $n$), 13 are bounded, and
items 13 and 15 are known only numerically; on direct sums only the contrast inside the blocks counts. The rules by which
contrasts compose (rapidities add for products and means, angles add for Jordan products, three
contrasts multiply in the commutator test) explain some of the thresholds in the contrast of the
operators that the computation reveals, and prove them for subadditivity of the square root and for the
power-mean inequality.

\subsection{Matrix envelopes}\label{ssec:openmat}
The computation supports the following values. (The value $(G-1)^2/(4G)$ of item~10, which the search
had also suggested, is proved in Theorem~\ref{thm:mono2} and Remark~\ref{rem:g10}.)

\begin{conjecture}\label{conj:values}
$V_n=G^2-1$ for item~7 and $V_n=g_{25}(\delta)$ for item~25 in every dimension; $V_2=G-1$ for item~6; the value of
Proposition~\ref{prop:abed} holds for item~9 in every dimension; and the lower bound of
Proposition~\ref{prop:youngtwo} is the two-contrast envelope of item~8 for even $n$.
\end{conjecture}

Several questions remain open. What is the threshold of item~12 for $\delta<1$, where the search places it
above the value $2\sqrt2/3$ of Proposition~\ref{prop:item12}, and what are the thresholds of items 6, 10, 13
and~25, which the search shows but which no operator inequality of the kind used in
Section~\ref{sec:thresholds} can explain for item~10? What is the envelope of item~6 for $n\ge3$, where the
search exceeds $G-1$, and does it grow with $n$? Do items 12--16 have closed forms? The search gives the
leading coefficients $\frac12$, $\frac12$, $\frac1{36}$, $0.064$ and $0.09$; for item~14 the coefficient
$(1-\frac2p)^2/4$ is a proved lower bound, and for items 12 and 16 the computed envelopes agree with the
proved lower bounds $\mathsf t_{1/2}(\kappa)-1$ and $b_2(\kappa)-1$ (Section~\ref{sec:numcompare}), which
suggests that these are the values. Finally, the same questions can be asked for other
characterisations, such as those of~\cite{moslehian2025jmaa} with the arithmetic mean.

\subsection{Von Neumann algebras}\label{ssec:openvn}
We first record two observations that bear on problems~(7) and~(8) below.

\begin{remark}\label{rem:modular}
Let $\cM$ be finite and $\varphi=\tau(\rho\,\cdot)$ with $\rho\in\cM_{++}$. In the standard form
$L^2(\cM,\tau)$, with $Jh=h^*$ and the cyclic and separating vector $\rho^{1/2}$, the Tomita operator
is $S(x\rho^{1/2})=x^*\rho^{1/2}=J\,L_{\rho^{1/2}}R_{\rho^{-1/2}}(x\rho^{1/2})$. Hence the modular
operator is $\Delta_\varphi=L_\rho R_{\rho^{-1}}$ \cite[Ch.~9]{KR2}. For $h=Y\rho^{1/2}$ we have
$\langle\Delta_\varphi h,h\rangle=\|\rho^{1/2}Y\|_2^2=\varphi(|Y^*|^2)$ and $\|h\|_2^2=\varphi(|Y|^2)$. So
Theorem~\ref{thm:modulus} with $q=2$ gives $\|\Delta_\varphi\|=\|\Delta_\varphi^{-1}\|=\kappa_\cZ(\rho)$.
Conversely, let $\varphi$ be a faithful normal state with cyclic vector $\xi_\varphi$ on an arbitrary
von Neumann algebra $\cM$, and suppose $\|\Delta_\varphi\|<\infty$. Then $\cM$ is finite. Indeed,
$\varphi(xx^*)=\|Sx\xi_\varphi\|^2=\|\Delta_\varphi^{1/2}x\xi_\varphi\|^2\le\|\Delta_\varphi\|\varphi(x^*x)$
for $x\in\cM$. If $\cM$ were not finite, it would contain a nonzero properly infinite projection
$p$. Such a $p$ is the sum of infinitely many mutually orthogonal projections $f_k\sim p$
\cite[Ch.~6]{KR2}. With $v_k^*v_k=p$ and $v_kv_k^*=f_k$ we get
$\varphi(p)=\varphi(v_k^*v_k)\le\|\Delta_\varphi\|\varphi(f_k)$, so $N\varphi(p)\le\|\Delta_\varphi\|\varphi(p)$
for every $N$, a contradiction. Since $\varphi$ is faithful, $\cM$ is also $\sigma$-finite, so it carries a
faithful normal tracial state \cite[Ch.~8]{KR2}, i.e.\ it is finite in the sense of
Section~\ref{sec:setting}. The condition $\|\Delta_\varphi\|<\infty$ does not require $\rho$ to
have a bounded inverse. For example, $\rho=\bigoplus_ic_i\diag(2,1)$ in $\bigoplus_iM_2(\C)$ with
$c_i\downarrow0$ and $\tau=\sum_it_i\operatorname{tr}_i$ (normalised traces $\operatorname{tr}_i$, $t_i>0$,
$\sum_it_i=1$) gives $\|\Delta_\varphi\|=2$. For weights
the finiteness conclusion fails. If $h\in B(\cH)$ is positive and invertible and $\dim\cH=\infty$,
then $\varphi=\Tr(h\,\cdot)$ is a faithful normal semifinite weight with $\Delta_\varphi=L_hR_{h^{-1}}$
on the Hilbert--Schmidt class and $\|\Delta_\varphi\|=\kappa(h)<\infty$, but $B(\cH)$ is not finite.
\end{remark}

\begin{remark}\label{rem:cstar}
Let $A$ be a unital $C^*$-algebra, $\rho\in A$ positive and invertible, and
$\kappa_\cZ^A(\rho)=\inf_z\kappa(z\rho)$ over positive invertible central $z$. Let $\Sigma$ be the
maximal ideal space of the centre, and $J_\zeta$ the Glimm ideal generated by the maximal ideal at
$\zeta\in\Sigma$. A central element has scalar image in $A/J_\zeta$, and in $A/\mathfrak p$ for every primitive
ideal $\mathfrak p$; each $\mathfrak p$ contains some $J_\zeta$. Hence $\kappa(z\rho)\ge\kappa(\rho+J_\zeta)\ge\kappa(\rho+\mathfrak p)$
for $\mathfrak p\supseteq J_\zeta$. Since the irreducible representations of $A$ separate its points,
$\|a\|=\sup_{\mathfrak p}\|a+\mathfrak p\|$, and as $\|a+\mathfrak p\|\le\|a+J_\zeta\|\le\|a\|$ for
$\mathfrak p\supseteq J_\zeta$, also $\|a\|=\sup_\zeta\|a+J_\zeta\|$ for $a\in A$. Also
$\|\Ad_\rho\|=\|\Ad_{z\rho}\|\le\kappa(z\rho)$. Moreover $\Ad_\rho$ induces $\Ad_{\rho+\mathfrak p}$ on $A/\mathfrak p$, and
by Kaplansky's density theorem the latter has norm $\kappa(\rho+\mathfrak p)$ (as on $B(\cH)$, using rank-one
operators). Thus
\[
\sup_{\mathfrak p}\kappa(\rho+\mathfrak p)\le\|\Ad_\rho\|\le\kappa_\cZ^A(\rho),\qquad\sup_\zeta\kappa(\rho+J_\zeta)\le\kappa_\cZ^A(\rho).
\]
We sketch why equality holds in the last inequality. We use that $\zeta\mapsto\|a+J_\zeta\|$ is upper
semicontinuous on $\Sigma$ for every $a\in A$; this is part of the Dauns--Hofmann representation of
$A$ by sections over $\Sigma$ \cite{DH68}, and we have not checked the precise reference. Then
$\zeta\mapsto\kappa(\rho+J_\zeta)$ is upper semicontinuous on the compact space $\Sigma$, so its
supremum is a maximum. Let $t,s>0$ with $ts=\max_\zeta\kappa(\rho+J_\zeta)$. The
Kat\v{e}tov--Tong insertion theorem \cite{Kat51,Ton52} gives a continuous $z$ on $\Sigma$, that is, a
central element, with $\|\rho^{-1}+J_\zeta\|/s\le z(\zeta)\le t/\|\rho+J_\zeta\|$. Then
$\|z\rho\|\le t$ and $\|(z\rho)^{-1}\|\le s$.
\end{remark}

\begin{enumerate}[label=(\arabic*)]
\item \emph{Young's inequality for positive operators.} Prove Conjecture~\ref{conj:young}, in
particular $c_2^+=\cosh\LNC$. Only the positivity of $B$ matters. For $B=B^*$ the constant is
$(4G_\cZ^2-3)^{1/2}>G_\cZ$ (Proposition~\ref{prop:youngsa}); so if the conjecture holds, the positivity of $B$
cannot be weakened to self-adjointness. By
Remark~\ref{rem:youngstar}, $c_2^+=G_\cZ$ is equivalent to \eqref{eq:youngstar}, which for rank-one $B$
is the Kantorovich inequality. For $p>2$, find $c_p^+$; is it a function of $\kappa_\cZ$?
\item \emph{Monotonicity.} Is $m_p(\rho)=\mpp(\kappa_\cZ(\rho))$, i.e.\ is the $2\times2$ problem
extremal for every $p>1$? Find a closed form of $\mpp$ for $p\ne2$. Does the trace inequality
\eqref{eq:bourintrace} hold in every finite von Neumann algebra (Remark~\ref{rem:fujii})? If so,
the bounds $m_p\le K(\kappa_\cZ^{1/p},p)$ and $m_{\exp}\le S(\kappa_\cZ)$ of
Proposition~\ref{prop:monop}\ref{m:e} hold without the assumption that $\cM$ is AFD.
\item \emph{Petz--Zem\'anek for $k\ge3$.} By Propositions~\ref{prop:pzk} and~\ref{prop:pzG},
$G_\cZ\le c_k\le\kappa_\cZ^{1-1/k}$. By Propositions~\ref{prop:pz3} and~\ref{prop:pzk4}, $c_k$ is not a
function of $\kappa_\cZ$. Determine $c_k$; which spectral data of $\rho$ govern it? Numerically
$c_k>G_\cZ$ for $k\ge3$, and $c_k$ is increasing in $k$ (Remark~\ref{rem:pznum}); prove this.
Proposition~\ref{prop:pzG} gives $\liminf_kc_k\ge c_\infty$; is $\lim_kc_k=c_\infty$?
\item \emph{Subadditivity.} Is $c_p^T=\tp(\kappa_\cZ)$? Numerically (not used in the proofs) no
configuration exceeding $\tp(\kappa_\cZ)$ was found in $M_2(\C)$, $M_3(\C)$ or
$M_2(\C)\oplus M_2(\C)$ for $p\in\{0.1,0.25,0.5,0.75\}$. For an operator monotone $f\ge0$ with
$f(0)=0$ the constant of the $f$-version is at most $\frac{1+\kappa_\cZ}2$ (Remark~\ref{rem:opmon}).
Since (R) fails, on what does it depend? For $p=\frac12$ (item~12) the certified matrix computation agrees
with $\mathsf t_{1/2}(\kappa)$ (Section~\ref{sec:numcompare}).
\item \emph{The Alhasan--Fawwaz inequality.} For $A\ge0$ and $B=B^*$ \cite{AF21}, the constant
$a(\rho)$ satisfies $1<a\le G_\cZ$ for $\rho\notin\cZ$. It is not a function of $\kappa_\cZ$, it is
$<2$ on $M_2(\C)$, and it is $<\cosh\OLC$ on every $M_n(\C)$ (Proposition~\ref{prop:af}).
Numerically it depends on the multiplicity of $\min\sigma(\rho)$ (Remark~\ref{rem:afnum}). Determine
$a$. Is $a$ bounded on $M_n(\C)$ independently of $\rho$? Is $a=G_\cZ$ in II$_1$ factors? The same
questions arise for the variant with $B\ge0$.
\item \emph{Exponential and Araki--Lieb--Thirring inequalities.} Bikchentaev's characterisation
of traces by $\varphi(e^{PQP})\le\varphi(e^Q)$ concerns projections \cite{Bik12}. For traces it holds
for all
$Q\ge0$, but not for all $Q=Q^*$. Its constants lie in $(1,\kappa_\cZ]$ and are not functions of
$\kappa_\cZ$ (Proposition~\ref{prop:bikexp}\ref{be:c}). The shift-invariant constant $b$ satisfies
$b_2(\kappa_\cZ)\le b\le\kappa_\cZ$ (Proposition~\ref{prop:bikexp}). For the Peierls--Bogoliubov
inequality of \cite{bikchentaev2011pb} in the form with $\varphi(Be^A)$ the constant is infinite for $\rho\notin\cZ$. In the
symmetric form, $b\le c_{\rm PB}\le\kappa_\cZ$ (Proposition~\ref{prop:pb}). Is
$c_{\rm PB}=b=b_2(\kappa_\cZ)$, as the numerics in Remarks~\ref{rem:expnum} and~\ref{rem:pbnum} and the computed envelope of item~16
(Section~\ref{sec:numcompare}) suggest? For the Araki--Lieb--Thirring inequality of \cite{bikchentaev2011pb} (exponents $r>1$ and $q>0$,
all $A,B\ge0$), $1<c_{r,q}\le\kappa_\cZ$ and $\sup_qc_{r,q}=\kappa_\cZ$ (Proposition~\ref{prop:alt});
find $c_{r,q}$; for $(r,q)=(2,1)$, $G_\cZ^2\le c_{2,1}\le\kappa_\cZ$ (Proposition~\ref{prop:altlow}). For the spectral
geometric mean of~\cite{bikchentaev2026sgm}, $G_\cZ\le c_6\le\kappa_\cZ$ (Proposition~\ref{prop:sgm}); determine $c_6$.
\item \emph{Weights and non-finite algebras.} Several characterisations hold for normal semifinite
weights \cite{gardner1979,stolyarov2002}. For a faithful normal state, $\|\Delta_\varphi\|$ extends $\kappa_\cZ$
(Remark~\ref{rem:modular}). Do the results of Sections~\ref{sec:catalogue} and~\ref{sec:centre} hold with
$\|\Delta_\varphi\|$ in place of $\kappa_\cZ$ whenever $\|\Delta_\varphi\|<\infty$? This includes
$\varphi=\tau(\rho\,\cdot)$ with $\rho$ or $\rho^{-1}$ unbounded. For faithful normal states (and
positive functionals), $\|\Delta_\varphi\|<\infty$ forces $\cM$ to be finite. For weights it does not:
$\varphi=\Tr(h\,\cdot)$ on $B(\cH)$ with $h$ positive and invertible has $\|\Delta_\varphi\|=\kappa(h)$
(Remark~\ref{rem:modular}). For faithful normal semifinite weights with bounded $\Delta_\varphi$, do
the results hold with $\|\Delta_\varphi\|$ in place of $\kappa_\cZ$?
\item \emph{$C^*$-algebras.} Central envelopes need not exist in a $C^*$-algebra, but
$\kappa^A_\cZ(\rho)$ is defined, and by Remark~\ref{rem:cstar} it is (modulo the sketch there)
$\max_\zeta\kappa(\rho+J_\zeta)$. Do the results extend with $\kappa^A_\cZ$? Can the inequality
$\sup_P\kappa(\rho+P)\le\kappa^A_\cZ(\rho)$ be strict? For inner derivations, the analogous
comparison between $\|D_a\|$ and the distance from $a$ to the centre is governed by the
structure of the primitive ideal space \cite{Arc78,Som93}, which suggests that it can be.
\item \emph{Two contrasts beyond matrices.} Does Theorem~A hold in finite von Neumann algebras, with
$\delta$ replaced by $\NC(\rho)$ and $\delta'$ by a suitable central version of the contrast of $|X|$? The
rescaling argument of Lemma~\ref{lem:resc} does not apply, since $X\mapsto zX$ changes the contrast of
$|X|$. The same question arises for the matrix-only items 9, 11, 13--15, 19, 24 and~25.
\end{enumerate}

\section*{Declarations}

\noindent\textbf{Author contributions.} All authors contributed to the conception of the work,
the proofs, the computations and the writing, and approved the final manuscript.

\noindent\textbf{Funding.} No funding was received for this work.

\noindent\textbf{Competing interests.} The authors declare no competing interests.

\noindent\textbf{Data availability.} No data sets were generated or analysed apart from the numerical
computations of Section~\ref{sec:numerics}, which are reproduced by the programs described below. No
proof depends on a floating-point computation: the computed envelopes are certified lower bounds that
support the conjectures and illustrate the theorems, and the further numerical observations of
Section~\ref{sec:numother} are not used in any proof. The proof of Proposition~\ref{prop:af}\ref{af:d}
uses an exact computation in rational arithmetic, which can be checked by hand.

\noindent\textbf{Code availability.} The library~\cite{lib}, the programs that compute all numbers,
tables and figures of Sections~\ref{sec:gardner2}, \ref{sec:composite} and~\ref{sec:numerics} and of
Appendix~\ref{app:mc}, and the tests are available at
\url{https://github.com/andre-a-no/von_neumann_type_i} (MIT licence), in the version that accompanies the
arXiv posting of this paper.

\noindent\textbf{Reproducing the computations.} The programs are in the directory \texttt{scripts} of
the repository and run in this order, since later ones read the results of earlier ones:
\texttt{trace\_\allowbreak characterizations.py --mode full} (the envelopes with free operators,
Tables~\ref{tab:classes} and~\ref{tab:values}; $\TCSeconds$~s), \texttt{trace\_\allowbreak contrast.py
--mode full} (both contrasts and the restricted searches, Table~\ref{tab:projsuffice}; $\TCCSeconds$~s),
then \texttt{trace\_\allowbreak montecarlo.py} (Table~\ref{tab:montecarlo} and Appendix~\ref{app:mc}),
\texttt{contrast\_\allowbreak operations.py} (Section~\ref{sec:composite}),
\texttt{trace\_\allowbreak envelopes\_\allowbreak figures.py} (Figure~\ref{fig:gardner1d}) and
\texttt{trace\_\allowbreak paper\_\allowbreak schematics.py} (Figures~\ref{fig:angle}, \ref{fig:ladder}
and~\ref{fig:square}). Figure~\ref{fig:gardner2d} comes from \texttt{bounds\_\allowbreak search.py --mode
full} of the pipeline of~\cite{lib}. The programs write the numbers as \LaTeX{} macros and tables, which
the source of this paper reads; the figures label the contrast of the density $\Delta(W)$, which is
$\OMC(\rho)$ in the notation of this paper.

{\sloppy

}

\clearpage
\appendix
\section{Random samples and computed envelopes, item by item}\label{app:mc}

Each figure shows one inequality, with one panel per dimension. In the figures the contrast of the
density is labelled $\Delta(W)$, which is $\delta=\OMC(\rho)$, and the contrast $\delta'$ of the operators is
labelled $\delta'$. Top row: the normalised violation $(\cL-\cR)/\cN$
of the random samples (grey) and, where the variables are not projections, of the
scaled-projection samples (orange) against $\OMC(\rho)$ on $M_2$, $M_3$, $M_4$; black: the computed envelope $V(\delta)$,
a certified lower bound for the supremum, joined by straight lines. Bottom row (items with operator variables): samples
with $\OMC(\rho)=0.9$ and all operator variables of contrast $\delta'$, against $\delta'$, on $M_2$ and $M_3$, with the
computed $V(0.9,\delta')$. Samples far below zero are cut off. Where an envelope is concave between grid points (item~11
below its maximum at $\delta'=\delta$ and after its steep drop near $\delta'=1$), a few samples lie above the straight
segments; on $M_2$ they lie below the closed form $1-(w_1\ell'+w_2)/(2\sqrt{\ell'})$, $\ell'=\frac{1-\delta'}{1+\delta'}$,
$w=\frac{2}{1+\ell}(1,\ell)$.

\begin{figure}[!htbp]
\centering
\includegraphics[width=\textwidth]{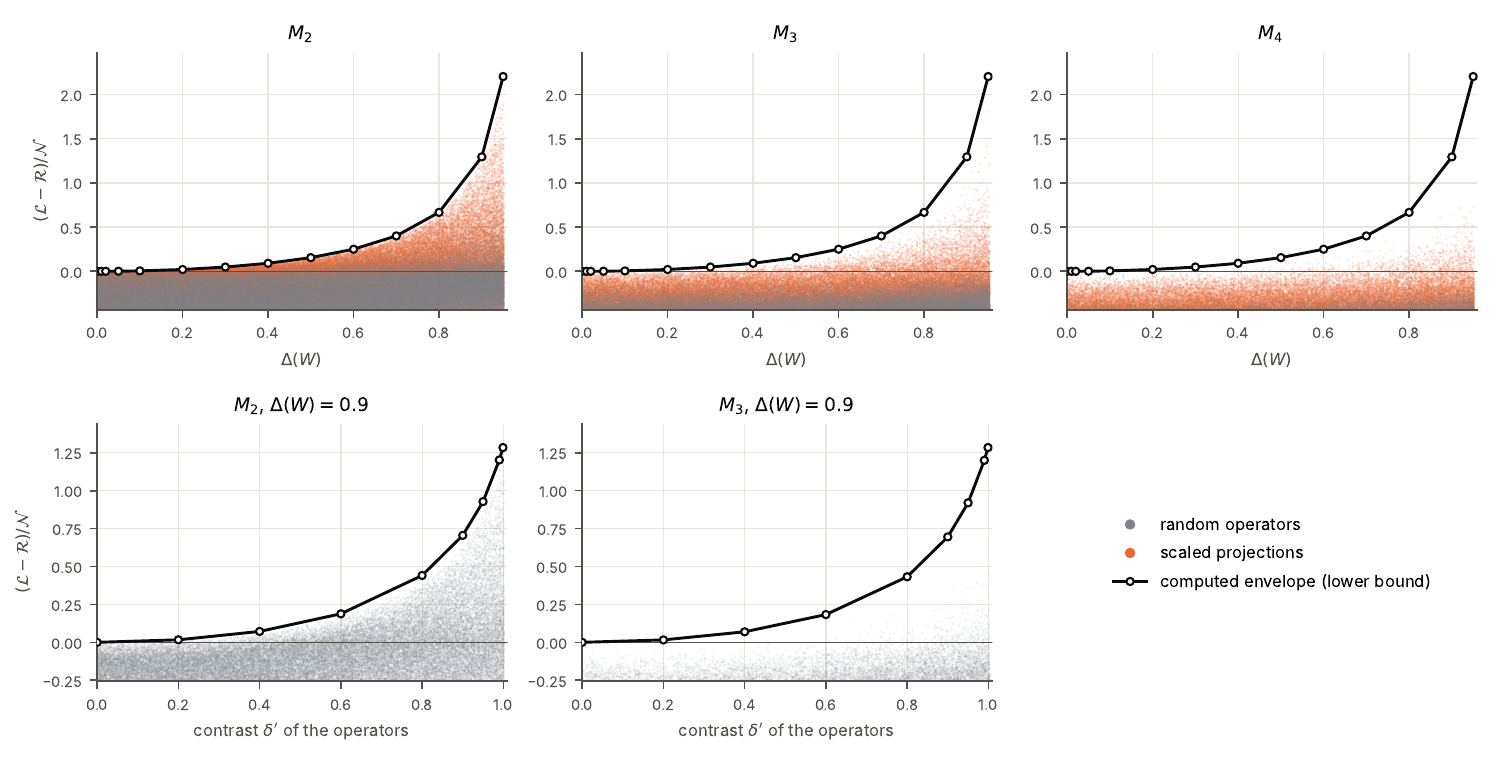}
\caption[]{Item 1, $|\varphi(X)|\le\varphi(|X|)$: random samples and the computed envelope (a certified lower bound); top: against the contrast $\delta=\OMC(\rho)$ of the density, bottom: against the contrast $\delta^\prime$ of the operators at $\delta=0.9$. In the figure $\Delta(W)=\OMC(\rho)=\delta$.}
\label{fig:mc-gardner}
\end{figure}

\begin{figure}[!htbp]
\centering
\includegraphics[width=\textwidth]{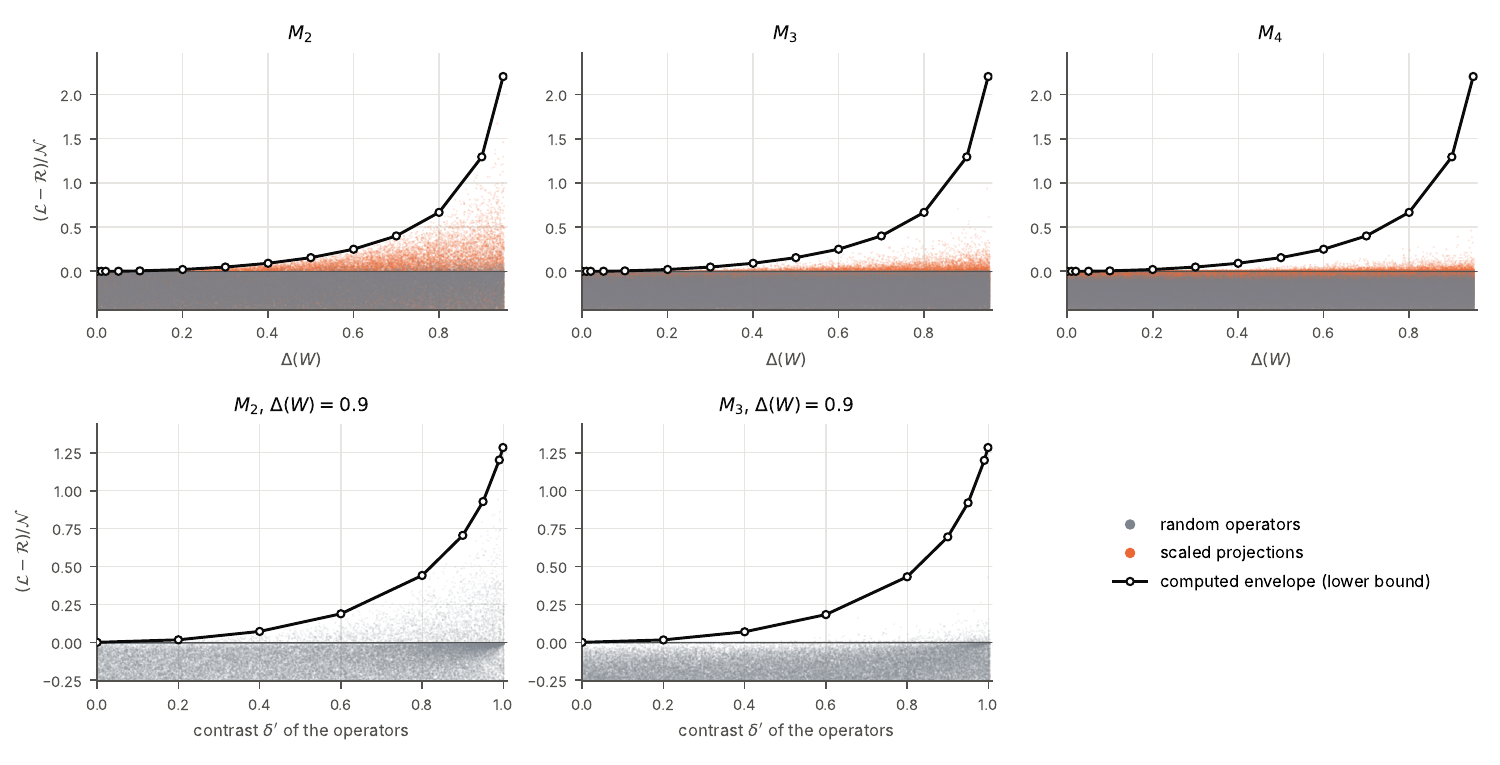}
\caption[]{Item 2, $\varphi(|X+Y|)\le\varphi(|X|)+\varphi(|Y|)$: random samples and the computed envelope (a certified lower bound); top: against the contrast $\delta=\OMC(\rho)$ of the density, bottom: against the contrast $\delta^\prime$ of the operators at $\delta=0.9$. In the figure $\Delta(W)=\OMC(\rho)=\delta$.}
\label{fig:mc-triangle}
\end{figure}

\begin{figure}[!htbp]
\centering
\includegraphics[width=\textwidth]{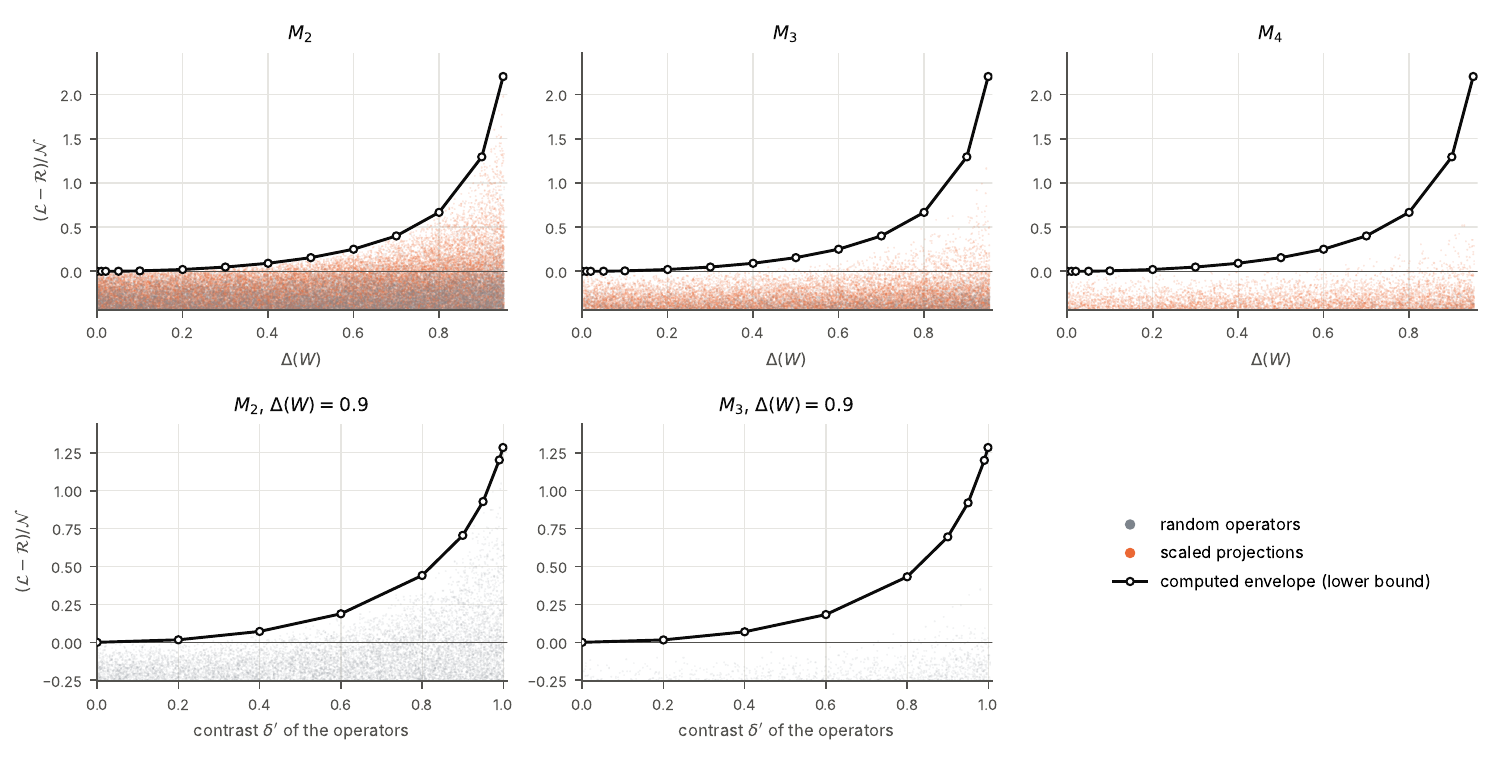}
\caption[]{Item 3, $\varphi(A+A^*)\le2\varphi(|A^*|)$: random samples and the computed envelope (a certified lower bound); top: against the contrast $\delta=\OMC(\rho)$ of the density, bottom: against the contrast $\delta^\prime$ of the operators at $\delta=0.9$. In the figure $\Delta(W)=\OMC(\rho)=\delta$.}
\label{fig:mc-realpart}
\end{figure}

\begin{figure}[!htbp]
\centering
\includegraphics[width=\textwidth]{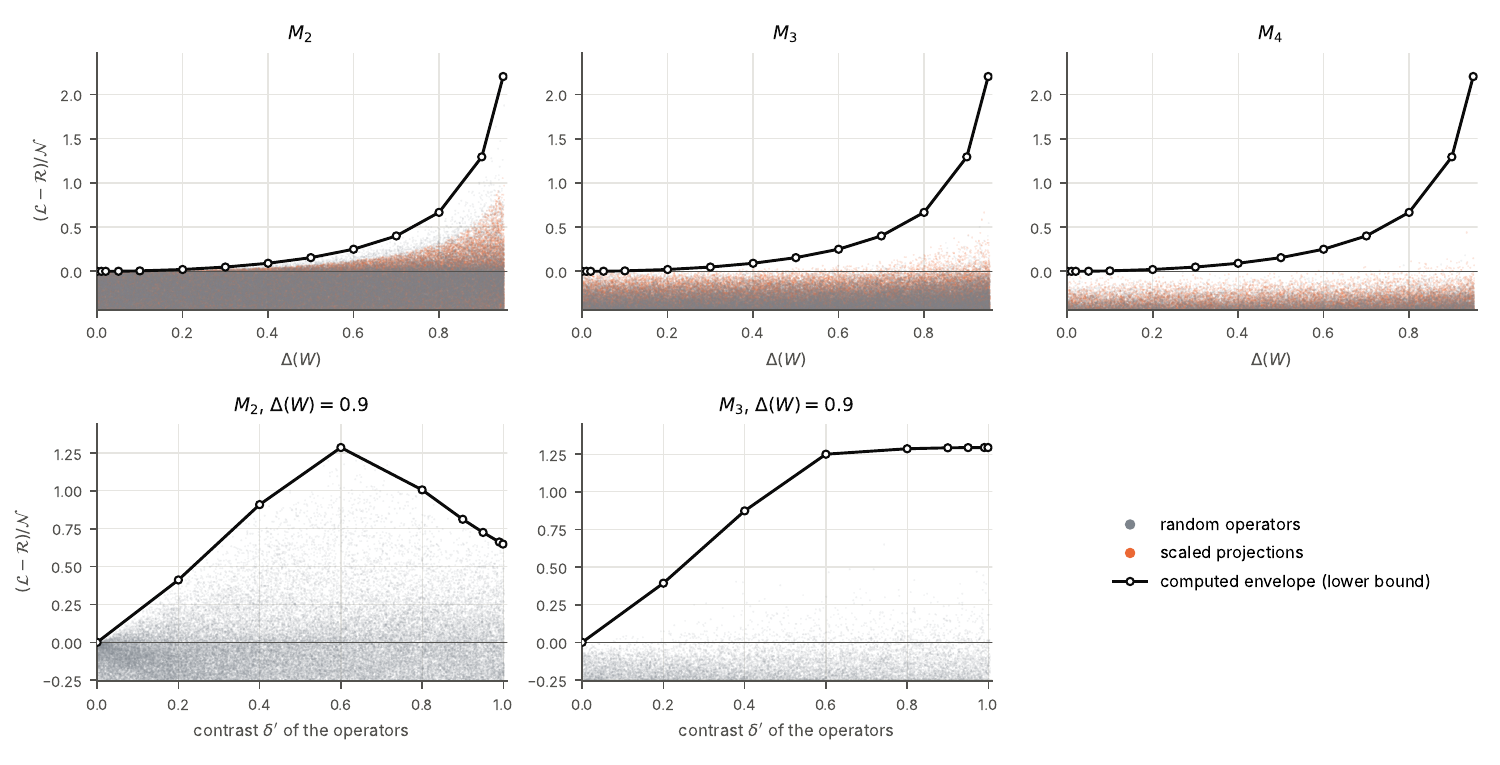}
\caption[]{Item 4, $|\varphi(A^2)|\le\varphi(A^*A)$: random samples and the computed envelope (a certified lower bound); top: against the contrast $\delta=\OMC(\rho)$ of the density, bottom: against the contrast $\delta^\prime$ of the operators at $\delta=0.9$. In the figure $\Delta(W)=\OMC(\rho)=\delta$.}
\label{fig:mc-square}
\end{figure}

\begin{figure}[!htbp]
\centering
\includegraphics[width=\textwidth]{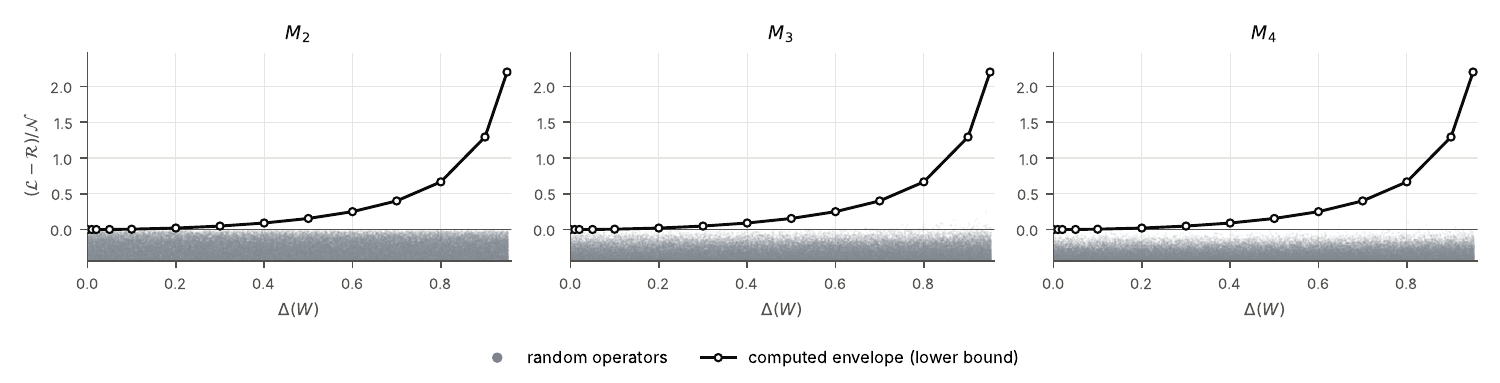}
\caption[]{Item 5, $\varphi(|P-R|)\le\varphi(|P-Q|)+\varphi(|Q-R|)$: random samples and the computed envelope (a certified lower bound) against the contrast $\delta=\OMC(\rho)$ of the density, labelled $\Delta(W)$ in the figure.}
\label{fig:mc-metric}
\end{figure}

\begin{figure}[!htbp]
\centering
\includegraphics[width=\textwidth]{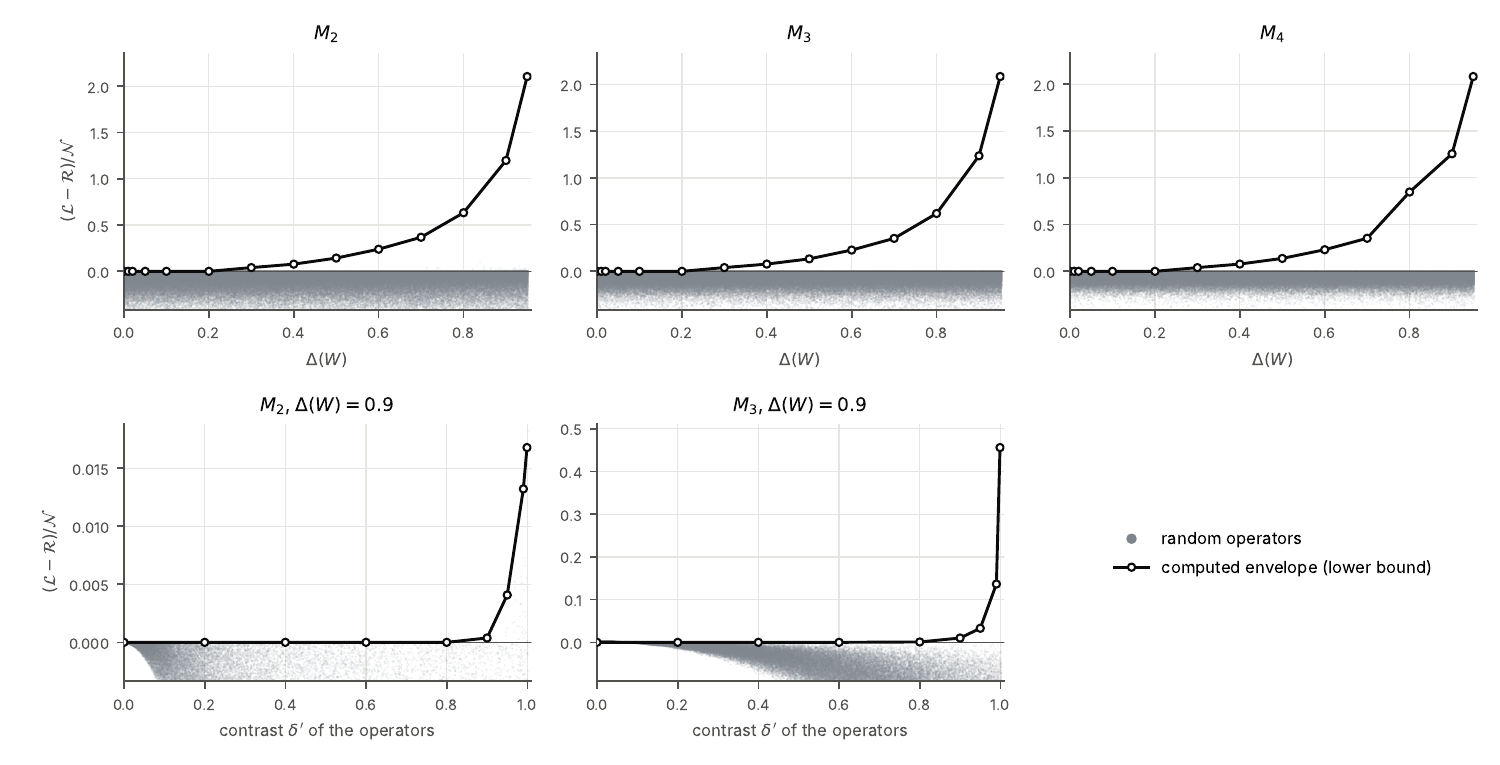}
\caption[]{Item 6, $\varphi(A\natural B)\le\sqrt{\varphi(A)\varphi(B)}$: random samples and the computed envelope (a certified lower bound); top: against the contrast $\delta=\OMC(\rho)$ of the density, bottom: against the contrast $\delta^\prime$ of the operators at $\delta=0.9$. In the figure $\Delta(W)=\OMC(\rho)=\delta$.}
\label{fig:mc-sgm}
\end{figure}

\begin{figure}[!htbp]
\centering
\includegraphics[width=\textwidth]{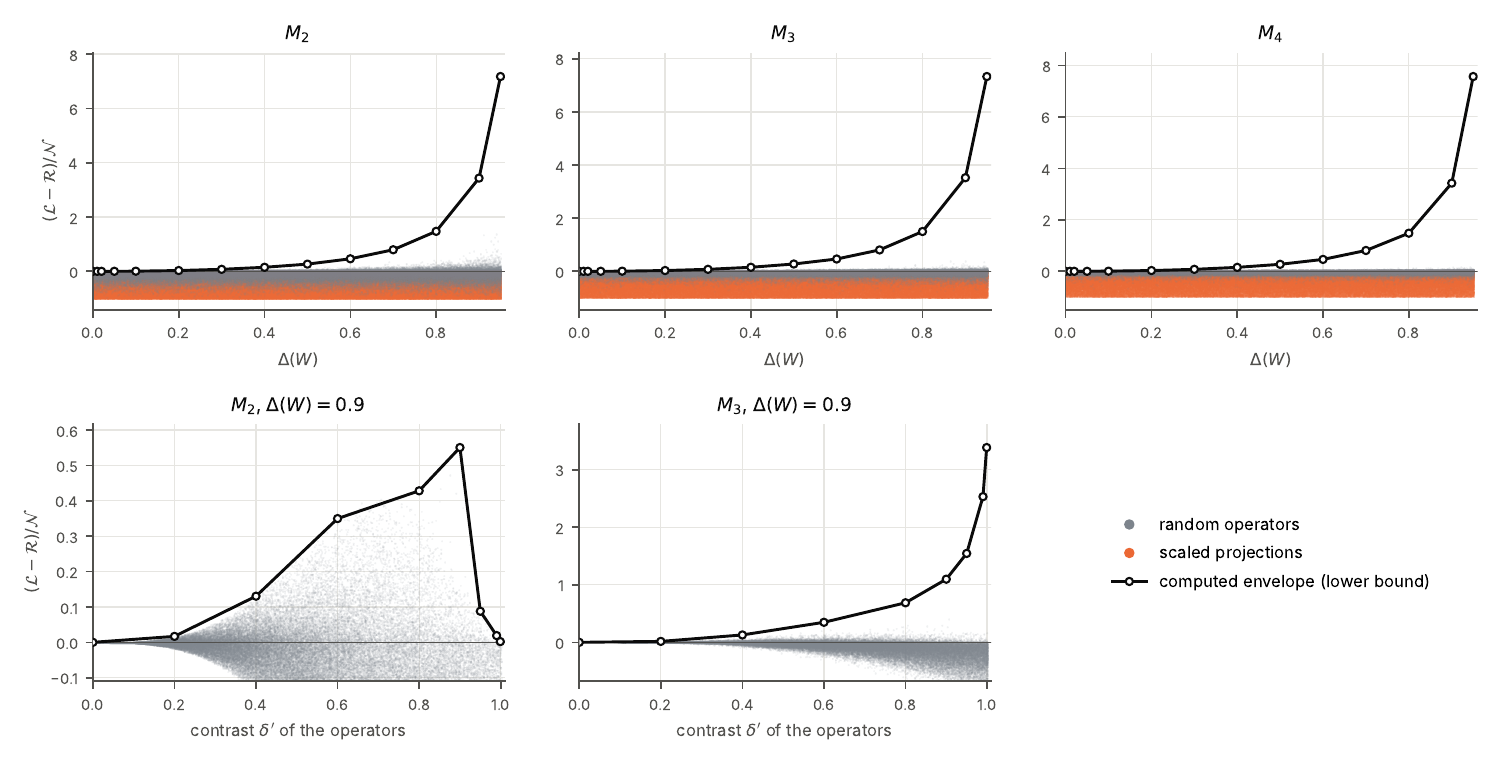}
\caption[]{Item 7, $\varphi((B^{1/2}AB^{1/2})^2)\le\varphi(BA^2B)$: random samples and the computed envelope (a certified lower bound); top: against the contrast $\delta=\OMC(\rho)$ of the density, bottom: against the contrast $\delta^\prime$ of the operators at $\delta=0.9$. In the figure $\Delta(W)=\OMC(\rho)=\delta$.}
\label{fig:mc-alt}
\end{figure}

\begin{figure}[!htbp]
\centering
\includegraphics[width=\textwidth]{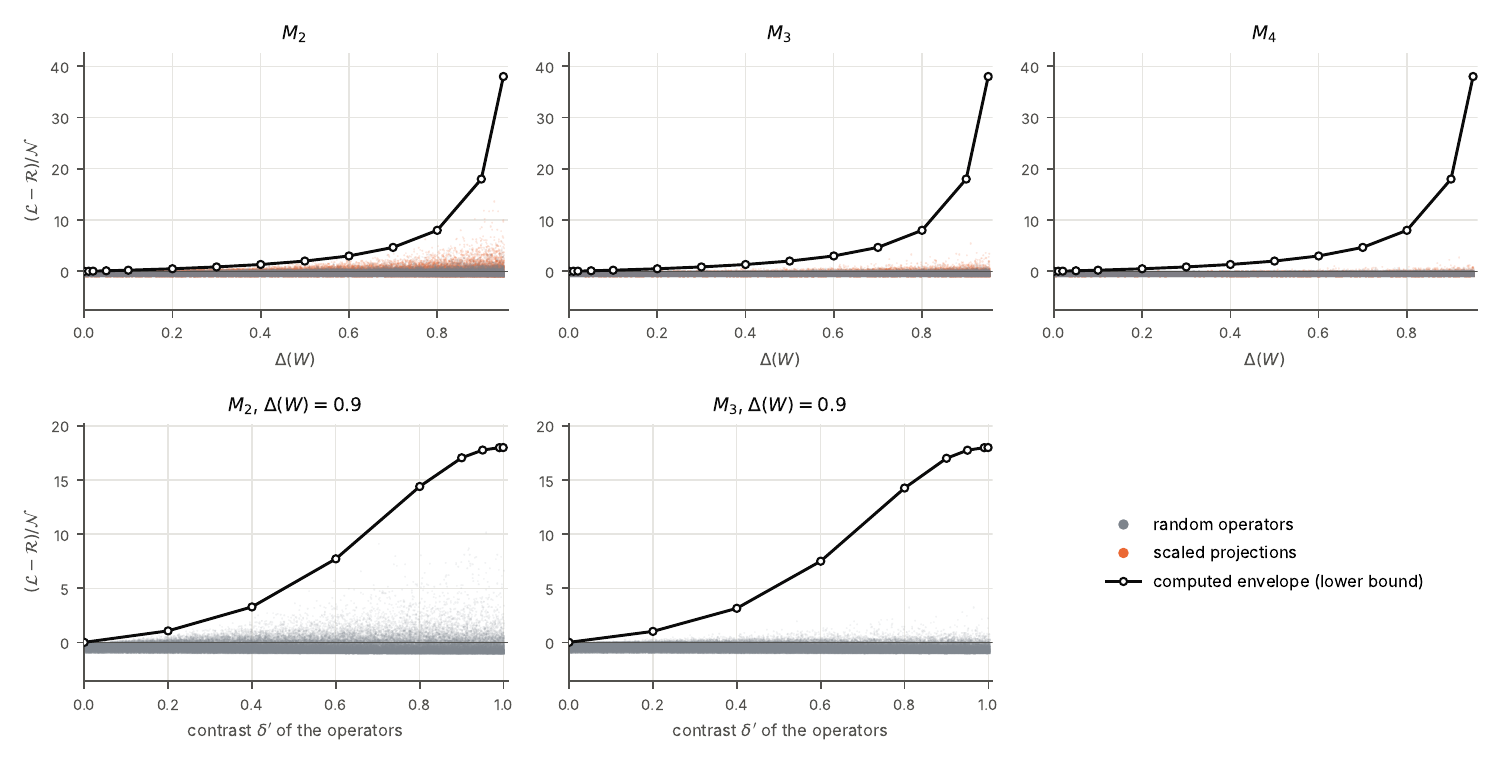}
\caption[]{Item 8, $\varphi(|AB^*|)\le\frac12\varphi(|A|^2)+\frac12\varphi(|B|^2)$: random samples and the computed envelope (a certified lower bound); top: against the contrast $\delta=\OMC(\rho)$ of the density, bottom: against the contrast $\delta^\prime$ of the operators at $\delta=0.9$. In the figure $\Delta(W)=\OMC(\rho)=\delta$.}
\label{fig:mc-young}
\end{figure}

\begin{figure}[!htbp]
\centering
\includegraphics[width=\textwidth]{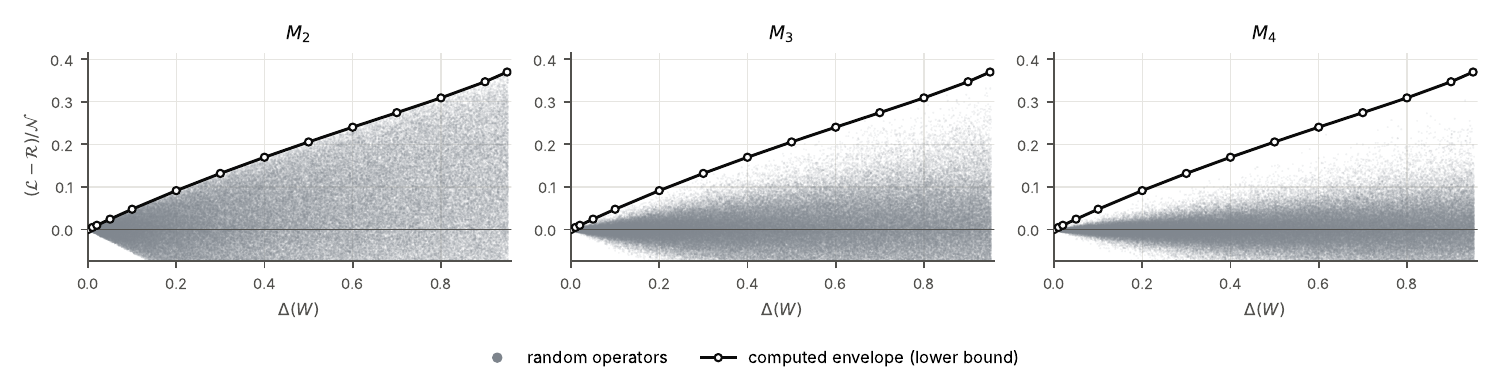}
\caption[]{Item 9, $\varphi(PQ+QP)\le2\varphi(QPQ)$: random samples and the computed envelope (a certified lower bound) against the contrast $\delta=\OMC(\rho)$ of the density, labelled $\Delta(W)$ in the figure.}
\label{fig:mc-projections}
\end{figure}

\begin{figure}[!htbp]
\centering
\includegraphics[width=\textwidth]{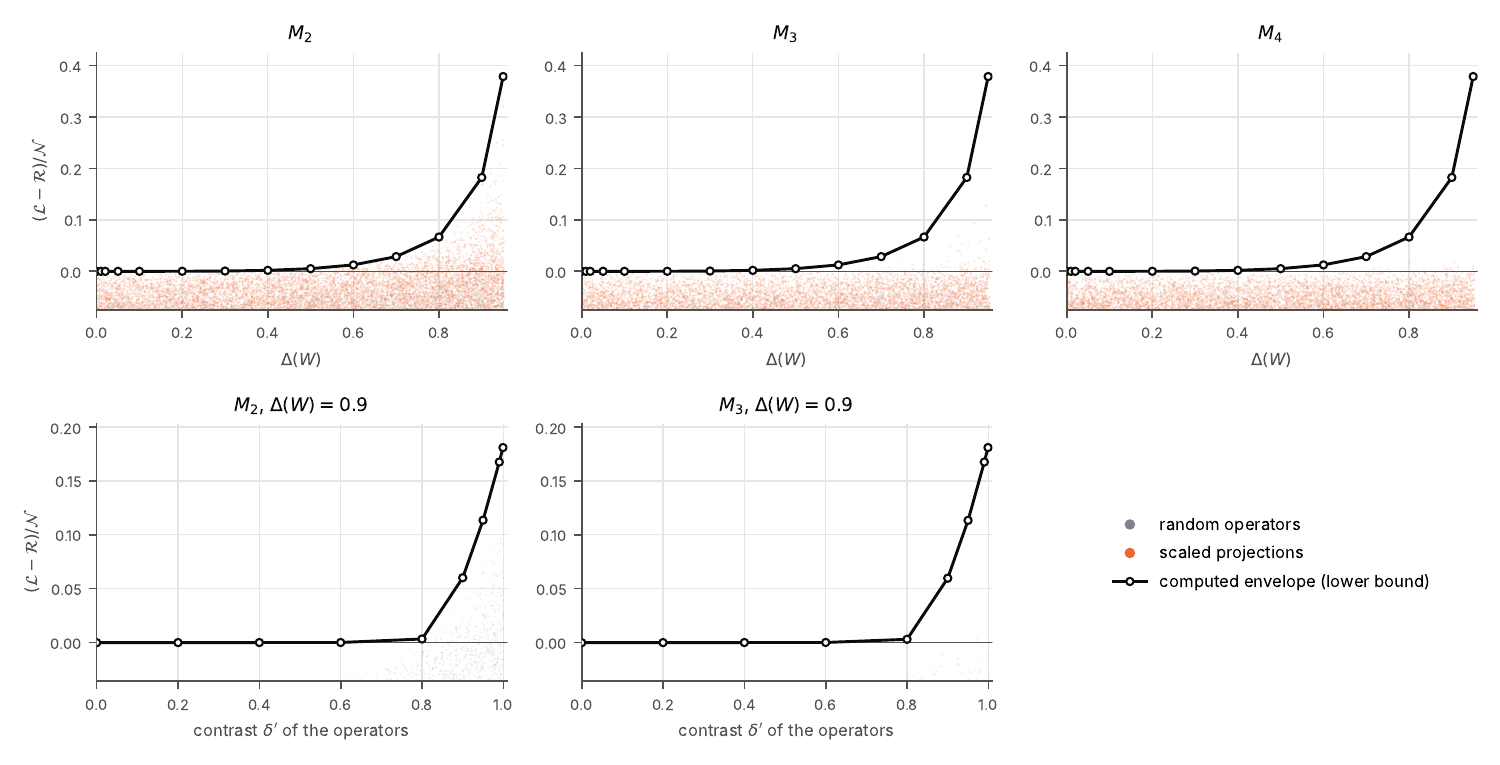}
\caption[]{Item 10, $0\le A\le B\Rightarrow\varphi(A^2)\le\varphi(B^2)$: random samples and the computed envelope (a certified lower bound); top: against the contrast $\delta=\OMC(\rho)$ of the density, bottom: against the contrast $\delta^\prime$ of the operators at $\delta=0.9$. In the figure $\Delta(W)=\OMC(\rho)=\delta$.}
\label{fig:mc-monotone}
\end{figure}

\begin{figure}[!htbp]
\centering
\includegraphics[width=\textwidth]{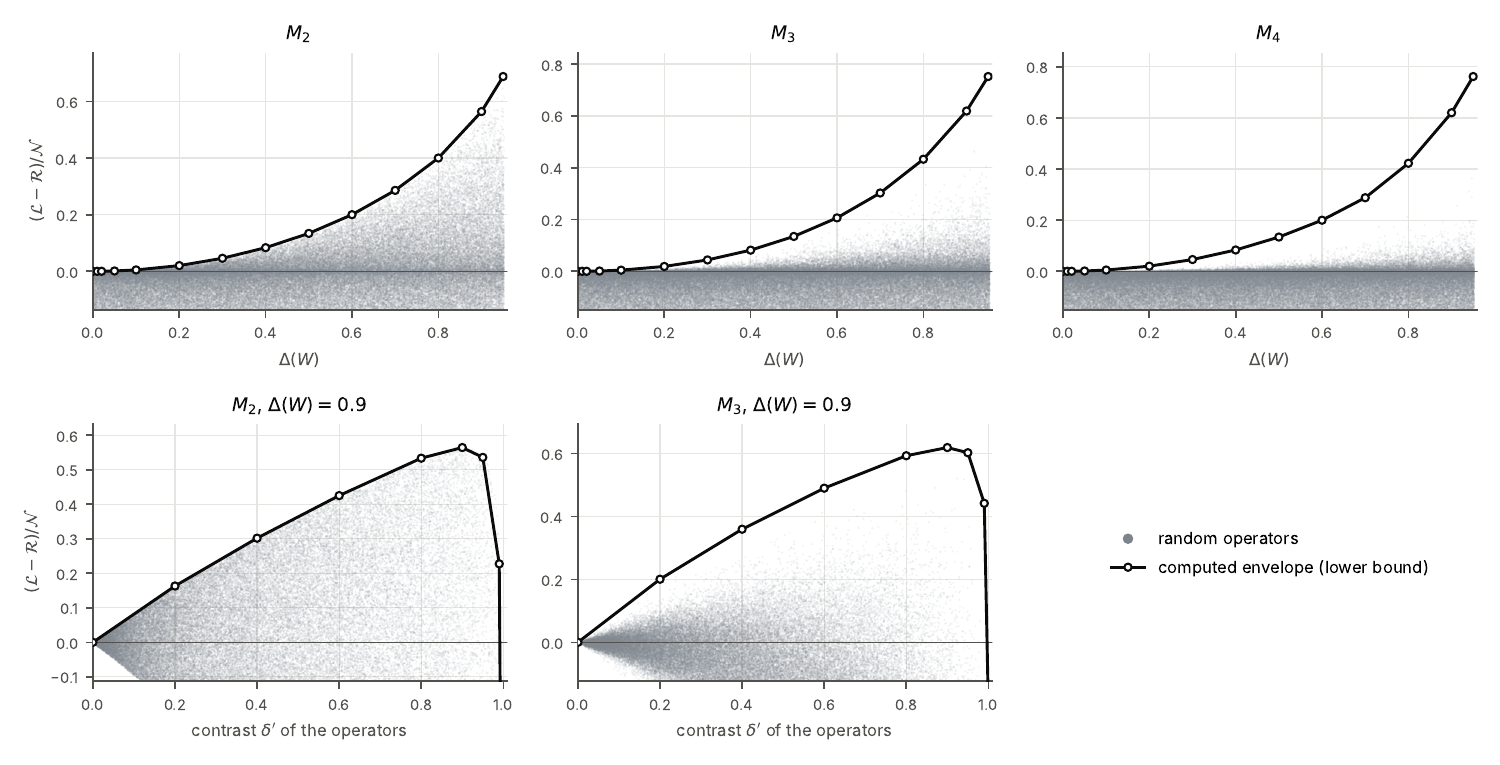}
\caption[]{Item 11, $(\det A)^{1/n}\le\frac1n\varphi(A)$, $\varphi(1)=n$: random samples and the computed envelope (a certified lower bound); top: against the contrast $\delta=\OMC(\rho)$ of the density, bottom: against the contrast $\delta^\prime$ of the operators at $\delta=0.9$. In the figure $\Delta(W)=\OMC(\rho)=\delta$.}
\label{fig:mc-det}
\end{figure}

\begin{figure}[!htbp]
\centering
\includegraphics[width=\textwidth]{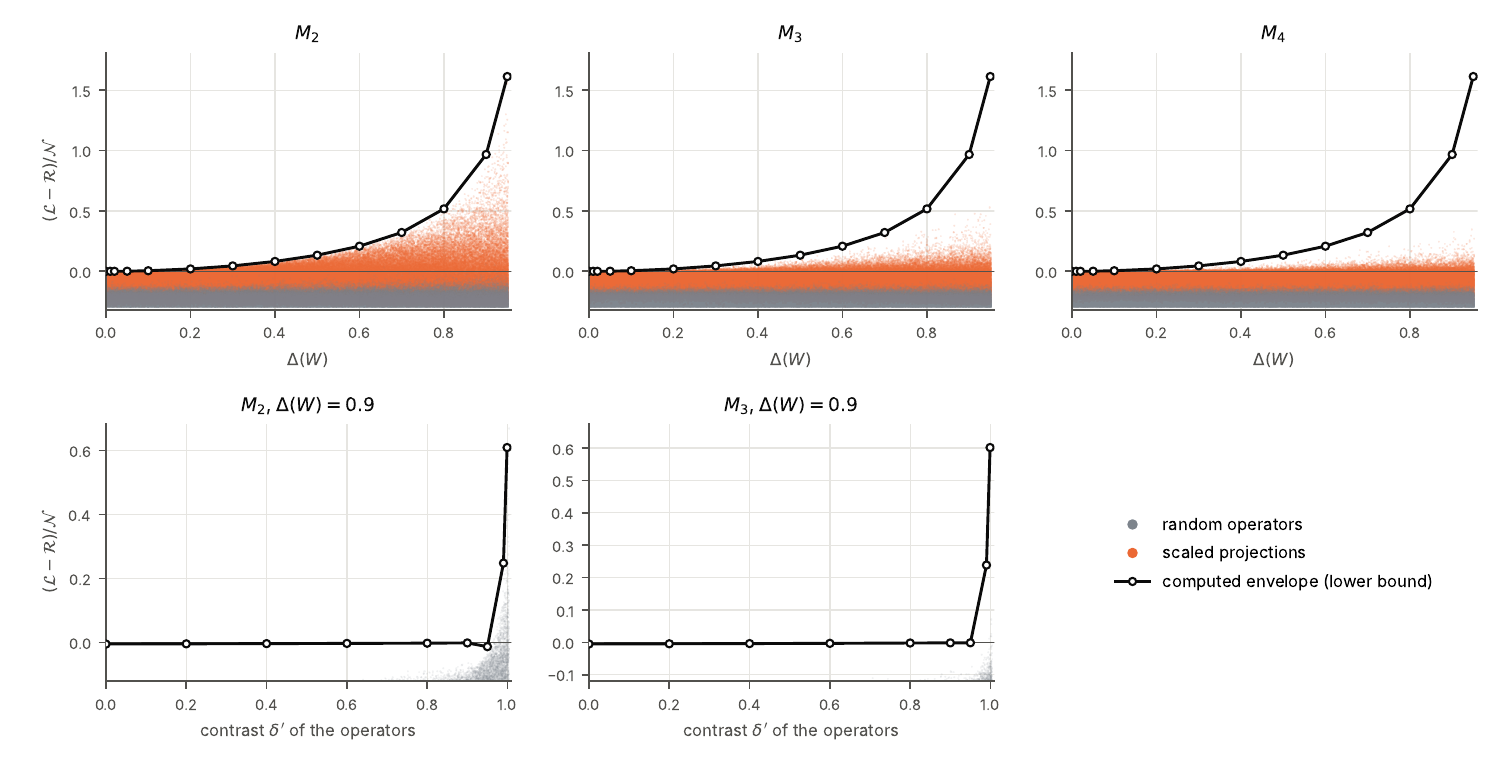}
\caption[]{Item 12, $\varphi((A+B)^{1/2})\le\varphi(A^{1/2})+\varphi(B^{1/2})$: random samples and the computed envelope (a certified lower bound); top: against the contrast $\delta=\OMC(\rho)$ of the density, bottom: against the contrast $\delta^\prime$ of the operators at $\delta=0.9$. In the figure $\Delta(W)=\OMC(\rho)=\delta$.}
\label{fig:mc-subadd}
\end{figure}

\begin{figure}[!htbp]
\centering
\includegraphics[width=\textwidth]{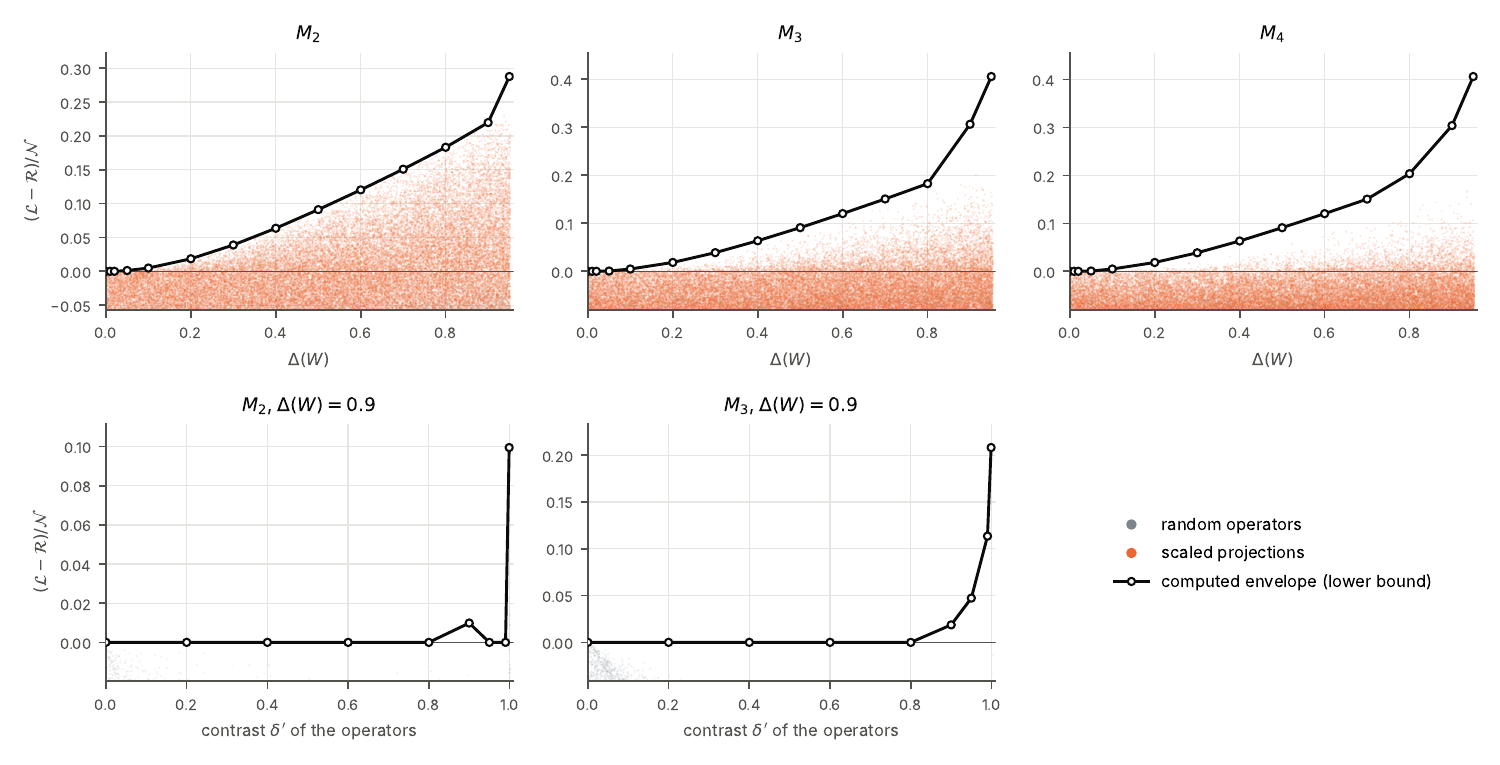}
\caption[]{Item 13, $\varphi(A+B-|A-B|)\le2\varphi(A^{1/4}B^{1/2}A^{1/4})$: random samples and the computed envelope (a certified lower bound); top: against the contrast $\delta=\OMC(\rho)$ of the density, bottom: against the contrast $\delta^\prime$ of the operators at $\delta=0.9$. In the figure $\Delta(W)=\OMC(\rho)=\delta$.}
\label{fig:mc-powers}
\end{figure}

\begin{figure}[!htbp]
\centering
\includegraphics[width=\textwidth]{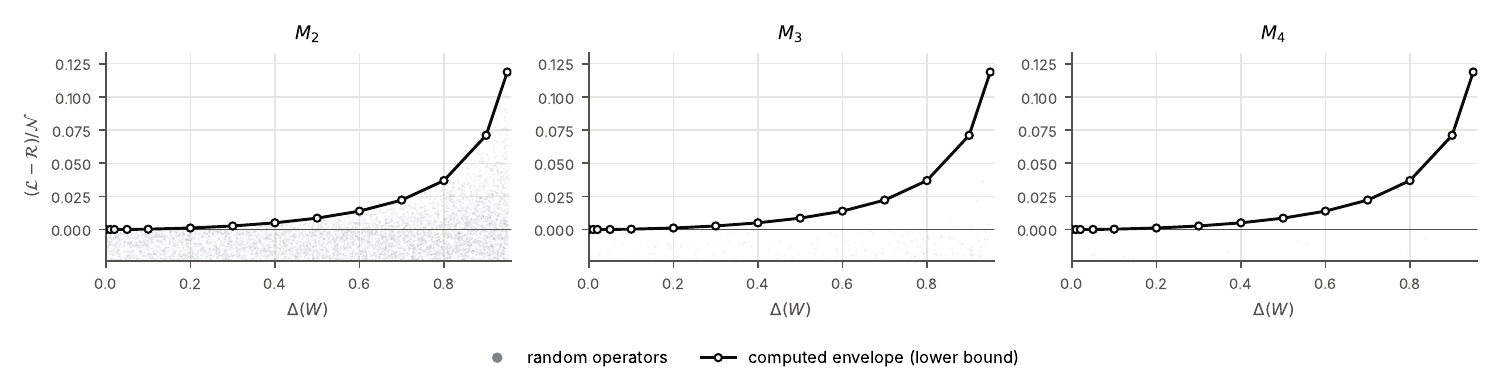}
\caption[]{Item 14, $|\varphi(PQ)|\le\varphi(P)^{1/3}\varphi(Q)^{2/3}$: random samples and the computed envelope (a certified lower bound) against the contrast $\delta=\OMC(\rho)$ of the density, labelled $\Delta(W)$ in the figure.}
\label{fig:mc-holder}
\end{figure}

\begin{figure}[!htbp]
\centering
\includegraphics[width=\textwidth]{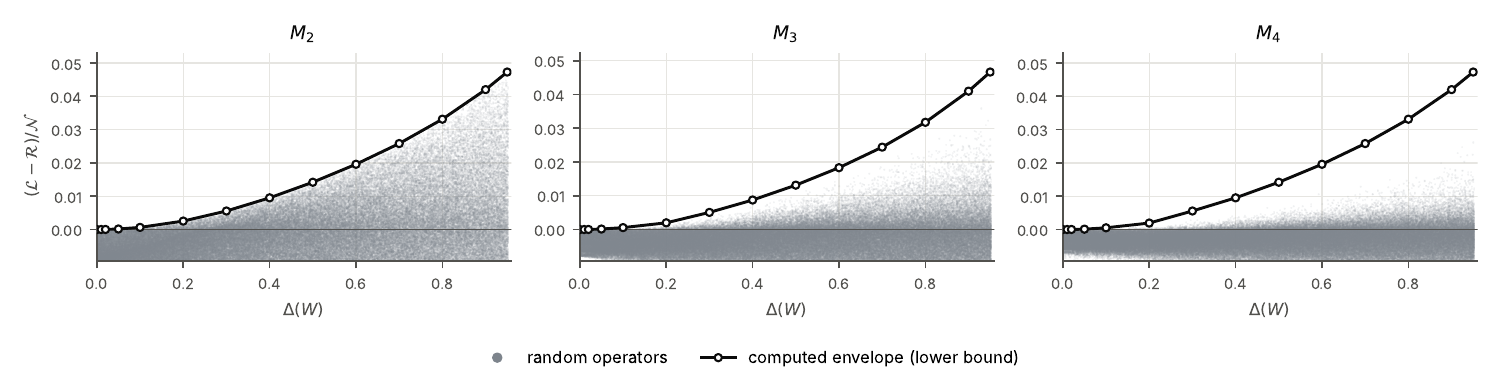}
\caption[]{Item 15, $\varphi(e^{P+Q})\le\varphi(e^{P/2}e^Qe^{P/2})$: random samples and the computed envelope (a certified lower bound) against the contrast $\delta=\OMC(\rho)$ of the density, labelled $\Delta(W)$ in the figure.}
\label{fig:mc-gt}
\end{figure}

\begin{figure}[!htbp]
\centering
\includegraphics[width=\textwidth]{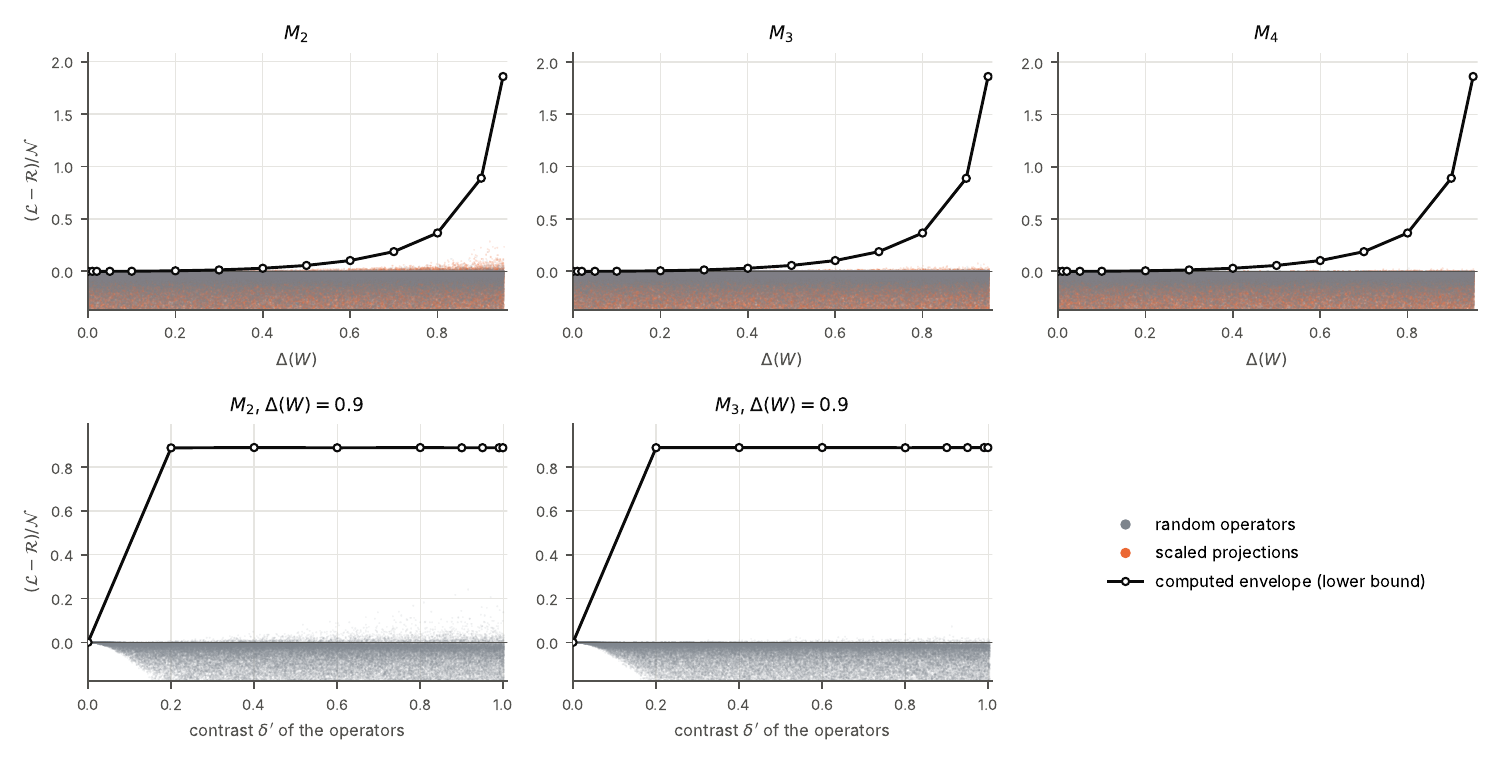}
\caption[]{Item 16, $\varphi(e^A)e^{\varphi(e^{A/2}Be^{A/2})/\varphi(e^A)}\le\varphi(e^{A+B})$: random samples and the computed envelope (a certified lower bound); top: against the contrast $\delta=\OMC(\rho)$ of the density, bottom: against the contrast $\delta^\prime$ of the operators at $\delta=0.9$. In the figure $\Delta(W)=\OMC(\rho)=\delta$.}
\label{fig:mc-pb}
\end{figure}

\begin{figure}[!htbp]
\centering
\includegraphics[width=\textwidth]{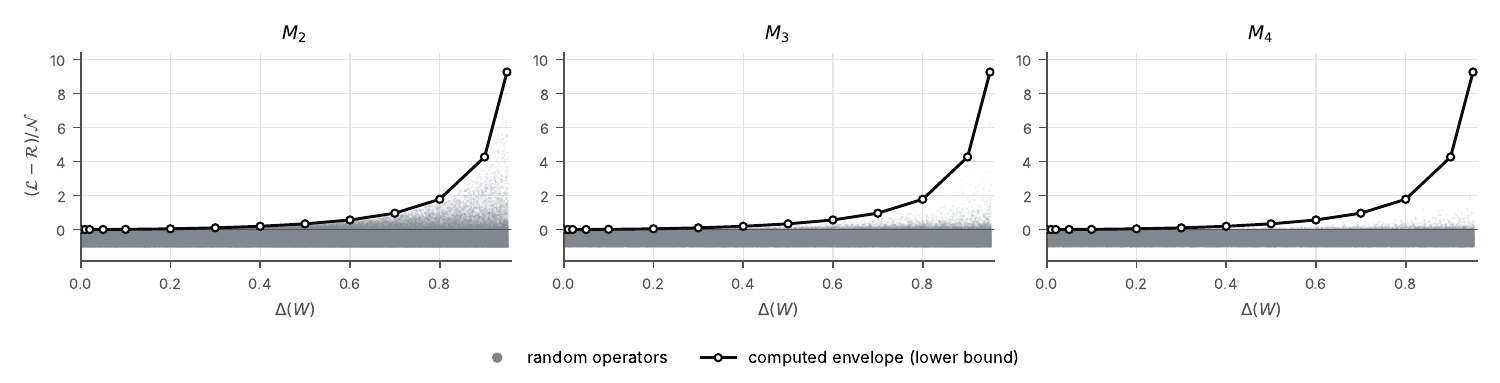}
\caption[]{Item 17, $\varphi(PQP)\le\varphi(Q)$: random samples and the computed envelope (a certified lower bound) against the contrast $\delta=\OMC(\rho)$ of the density, labelled $\Delta(W)$ in the figure.}
\label{fig:mc-pqp}
\end{figure}

\begin{figure}[!htbp]
\centering
\includegraphics[width=\textwidth]{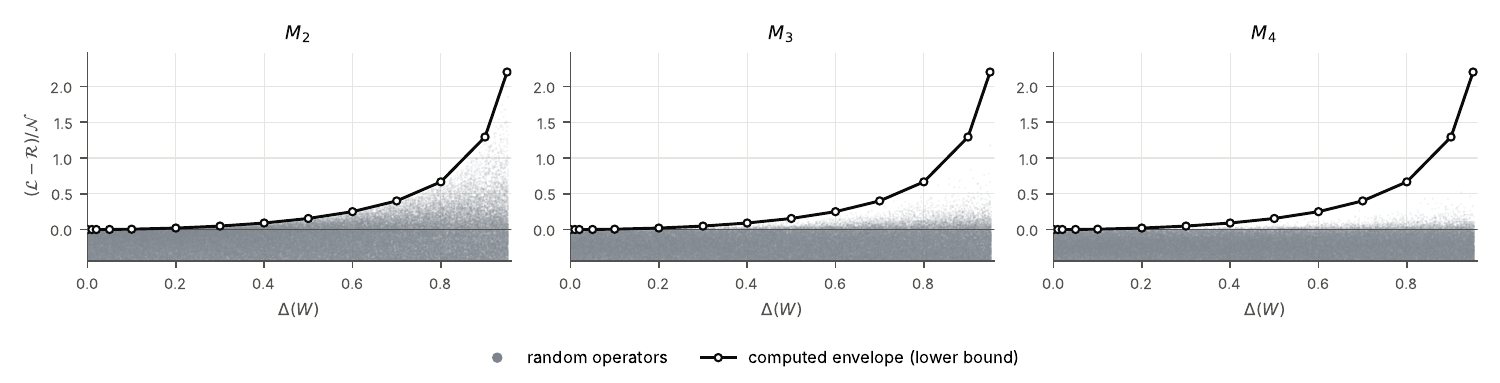}
\caption[]{Item 18, $\varphi(|P-Q|)\le\varphi(P+Q)$: random samples and the computed envelope (a certified lower bound) against the contrast $\delta=\OMC(\rho)$ of the density, labelled $\Delta(W)$ in the figure.}
\label{fig:mc-absdiff}
\end{figure}

\begin{figure}[!htbp]
\centering
\includegraphics[width=\textwidth]{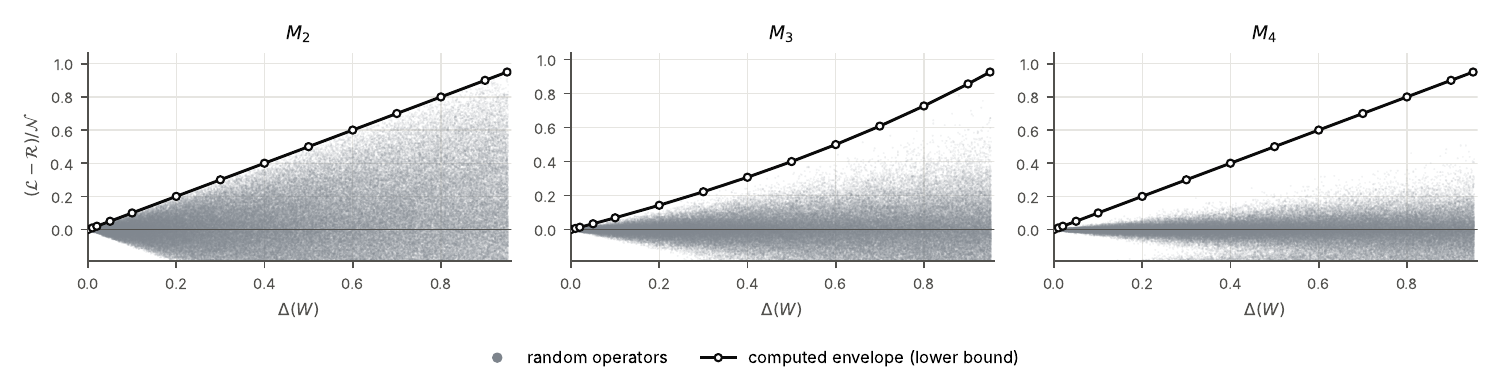}
\caption[]{Item 19, $\varphi(P+Q)\le\varphi(1+P\wedge Q)$: random samples and the computed envelope (a certified lower bound) against the contrast $\delta=\OMC(\rho)$ of the density, labelled $\Delta(W)$ in the figure.}
\label{fig:mc-meet}
\end{figure}

\begin{figure}[!htbp]
\centering
\includegraphics[width=\textwidth]{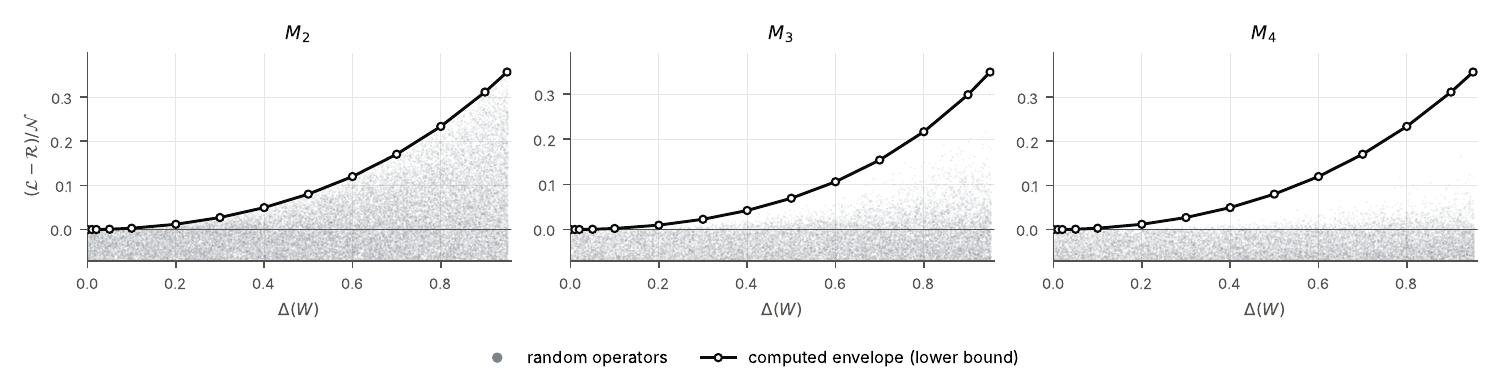}
\caption[]{Item 20, $\varphi(e^{QPQ})\le\varphi(e^P)$: random samples and the computed envelope (a certified lower bound) against the contrast $\delta=\OMC(\rho)$ of the density, labelled $\Delta(W)$ in the figure.}
\label{fig:mc-expqpq}
\end{figure}

\begin{figure}[!htbp]
\centering
\includegraphics[width=\textwidth]{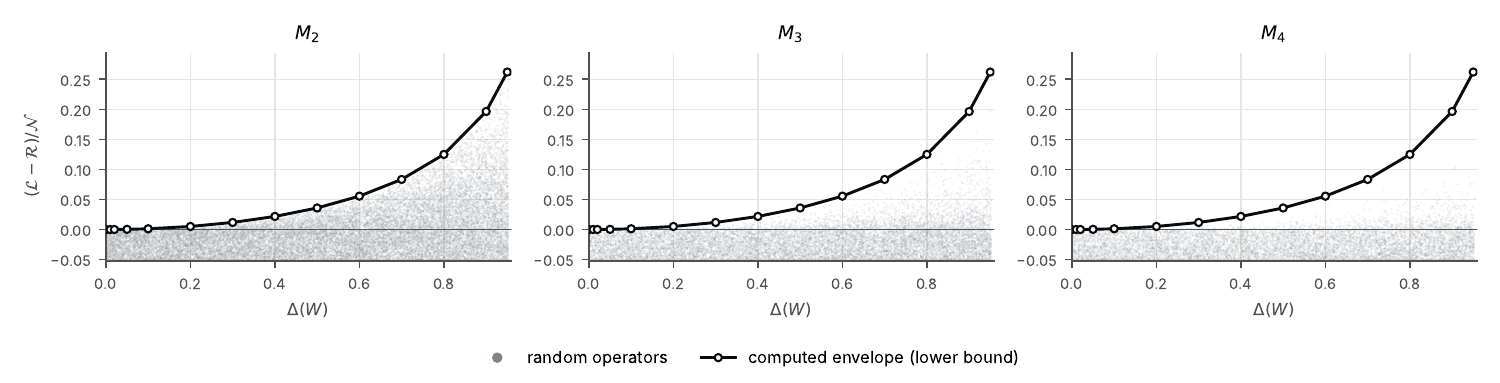}
\caption[]{Item 21, $\varphi(\mathrm{Re}\,PQ)\ge0$: random samples and the computed envelope (a certified lower bound) against the contrast $\delta=\OMC(\rho)$ of the density, labelled $\Delta(W)$ in the figure.}
\label{fig:mc-repos}
\end{figure}

\begin{figure}[!htbp]
\centering
\includegraphics[width=\textwidth]{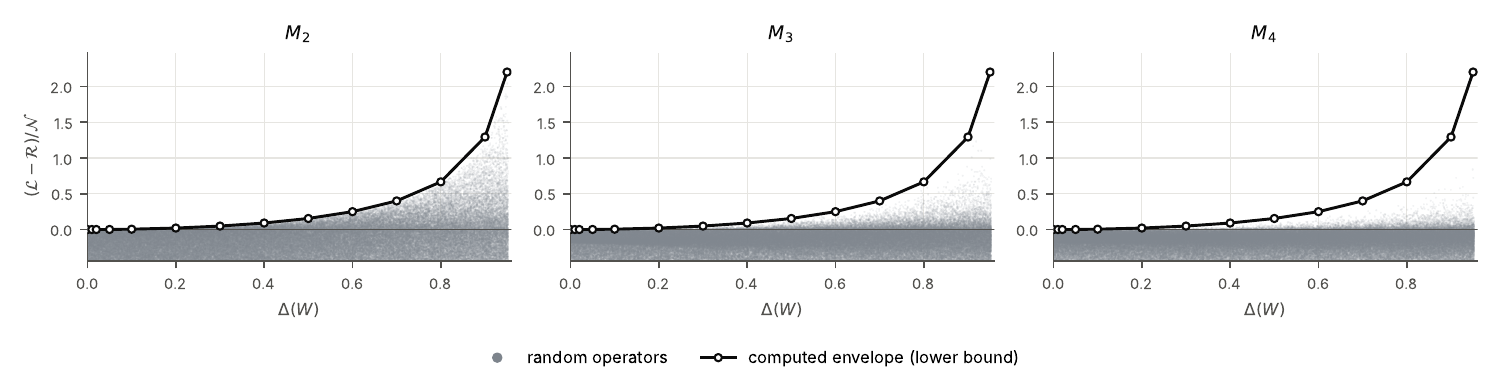}
\caption[]{Item 22, $\varphi(\mathrm{Re}\,PQ)\le\varphi(|PQ|)$: random samples and the computed envelope (a certified lower bound) against the contrast $\delta=\OMC(\rho)$ of the density, labelled $\Delta(W)$ in the figure.}
\label{fig:mc-reabs}
\end{figure}

\begin{figure}[!htbp]
\centering
\includegraphics[width=\textwidth]{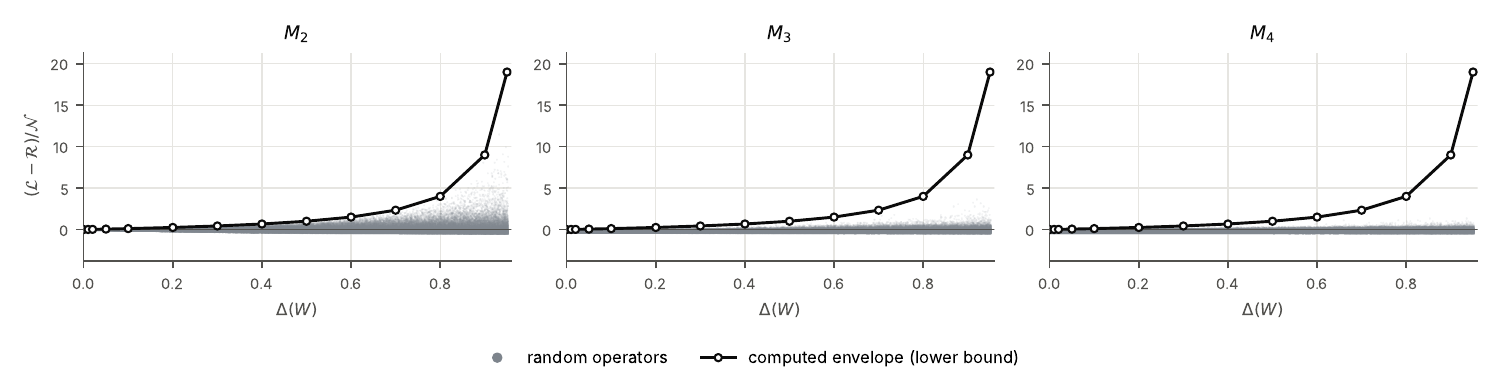}
\caption[]{Item 23, $\varphi(P\vee Q)\le\varphi(P)+\varphi(Q)$: random samples and the computed envelope (a certified lower bound) against the contrast $\delta=\OMC(\rho)$ of the density, labelled $\Delta(W)$ in the figure.}
\label{fig:mc-join}
\end{figure}

\begin{figure}[!htbp]
\centering
\includegraphics[width=\textwidth]{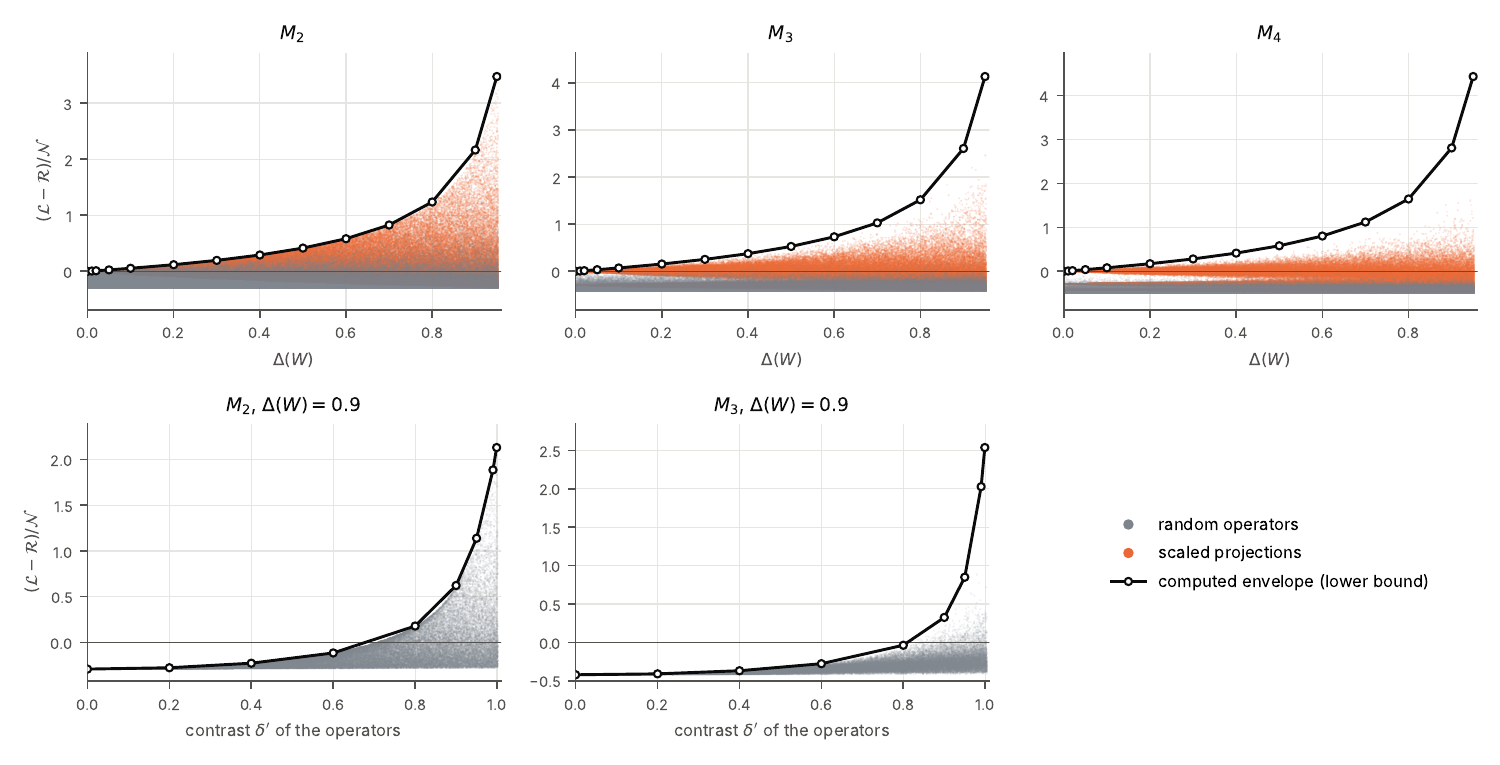}
\caption[]{Item 24, $\varphi(A^2)^{1/2}\le\varphi(A)$, $\varphi(1)=n$: random samples and the computed envelope (a certified lower bound); top: against the contrast $\delta=\OMC(\rho)$ of the density, bottom: against the contrast $\delta^\prime$ of the operators at $\delta=0.9$. In the figure $\Delta(W)=\OMC(\rho)=\delta$.}
\label{fig:mc-lp}
\end{figure}

\begin{figure}[!htbp]
\centering
\includegraphics[width=\textwidth]{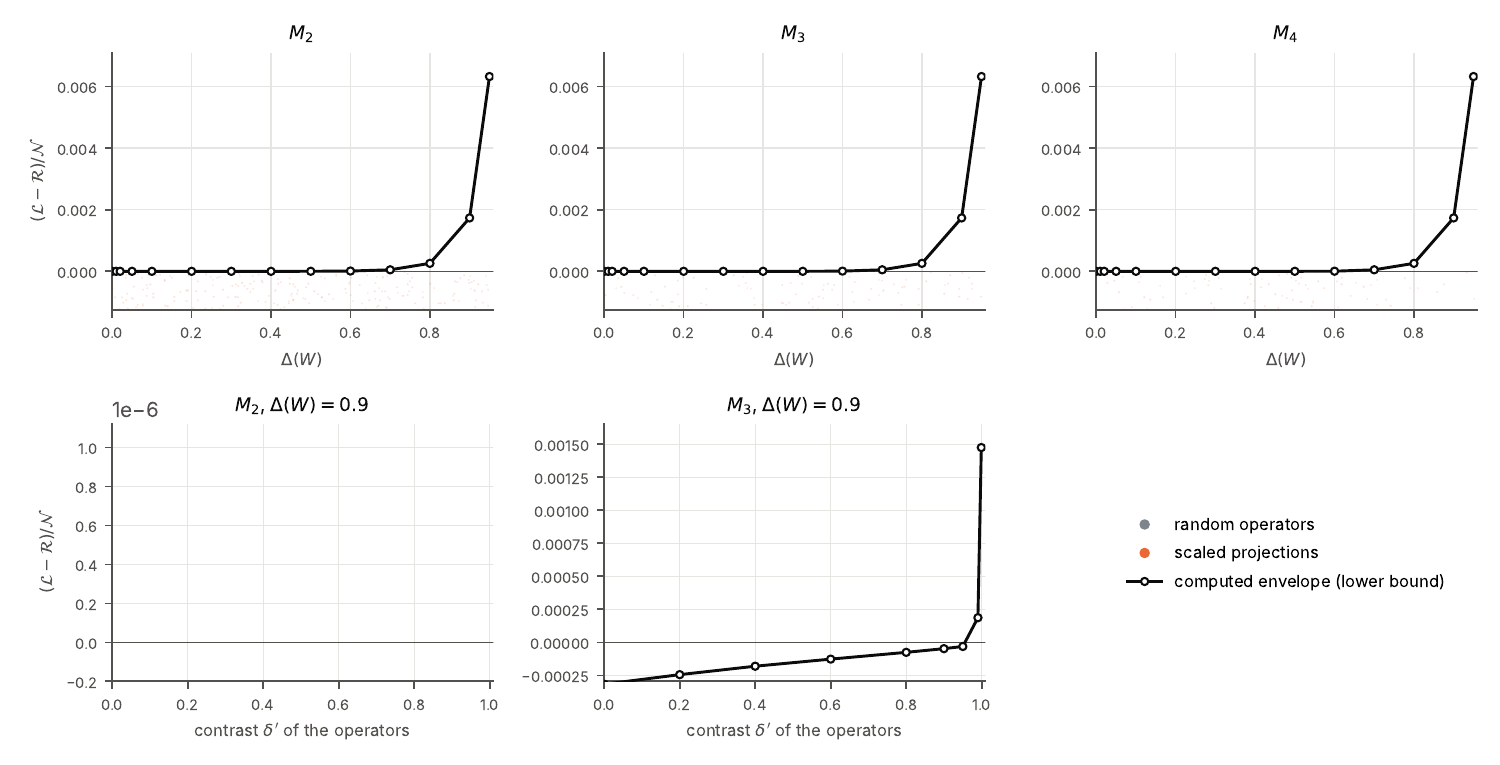}
\caption[]{Item 25, $0\le A\le B\Rightarrow\varphi(A^4)\le\varphi(AB^2A)$: random samples and the computed envelope (a certified lower bound); top: against the contrast $\delta=\OMC(\rho)$ of the density, bottom: against the contrast $\delta^\prime$ of the operators at $\delta=0.9$. In the figure $\Delta(W)=\OMC(\rho)=\delta$.}
\label{fig:mc-wmon}
\end{figure}

\end{document}